\documentclass[oneside,10pt,
chapterprefix, DIV= 11,
]{scrbook}

\usepackage{lineno,hyperref}
\usepackage{amssymb} 
\usepackage{dsfont} 
\usepackage{amsmath} 
\usepackage{xfrac} 
\usepackage{cases} 
\usepackage{subcaption} 
\usepackage{enumitem} 
\usepackage{mathtools} 
\usepackage[ruled,vlined, linesnumberedhidden]{algorithm2e} 
\SetKwInOut{Input}{Input} 
\usepackage{amsthm}
\usepackage{cancel}
\usepackage{empheq}
\usepackage[dvipsnames]{xcolor}
\usepackage{booktabs}
\usepackage{yfonts} 

\usepackage{xcolor}
\definecolor{caaorange}{RGB}{255, 102, 0}
\definecolor{caagreen}{RGB}{38, 159, 120}
\definecolor{caablue}{RGB}{11, 62, 117}
\definecolor{caagrey}{RGB}{102, 102, 102}

\definecolor{faublue}{RGB}{0,56,101}
\definecolor{faulblue}{RGB}{114,159,207}
\definecolor{faullblue}{RGB}{221,229,240}

\definecolor{faunat}{RGB}{0,155,119}
\definecolor{faulnat}{RGB}{170,207,189}
\definecolor{faullnat}{RGB}{229,239,234}

\newcommand\colResultsOut{faublue}
\newcommand\colResultsIn{faullblue!25}

\usepackage{tikz}

\DeclareRobustCommand{\meddiamond}{%
  \mathbin{\text{\scalebox{.75}{\rotatebox[origin=c]{45}{$\Box$}}}}%
}

\usepackage[framemethod=tikz]{mdframed}

\newmdtheoremenv[
  backgroundcolor=\colResultsIn,
  linecolor=\colResultsOut,
  leftmargin=0pt,
  innertopmargin=-3pt,
  innerleftmargin=5pt,
  innerrightmargin=5pt,
  nobreak=true,
  ]{theorem}{Theorem}[section]
  
\newtheorem{definition}[theorem]{Definition}
\newtheorem{remark}[theorem]{Remark}

\newmdtheoremenv[
  backgroundcolor=\colResultsIn,
  linecolor=\colResultsOut,
  leftmargin=0pt,
  innertopmargin=-3pt,
  innerleftmargin=5pt,
  innerrightmargin=5pt,
  nobreak=true,
  ]{lemma}[theorem]{Lemma}

\newmdtheoremenv[
  backgroundcolor=\colResultsIn,
  linecolor=\colResultsOut,
  leftmargin=0pt,
  innertopmargin=-3pt,
  innerleftmargin=5pt,
  innerrightmargin=5pt,
  nobreak=true,
  ]{proposition}[theorem]{Proposition}

\newtheorem{assumption}{Assumption}

\newmdtheoremenv[
  backgroundcolor=\colResultsIn,
  linecolor=\colResultsOut,
  leftmargin=0pt,
  innertopmargin=-3pt,
  innerleftmargin=5pt,
  innerrightmargin=5pt,
  nobreak=true,
  ]{corollary}[theorem]{Corollary}

\setenumerate[1]{font=\normalfont, itemsep = 1pt}
\setitemize[1]{itemsep = 1pt}

\usepackage{appendix}
\AtBeginEnvironment{subappendices}{%
\chapter*{Appendix}
\addcontentsline{toc}{chapter}{Appendices}
\counterwithin{figure}{section}
\counterwithin{table}{section}
}

\hypersetup{
	colorlinks,
	citecolor=blue!50!black,
	filecolor=black,
	linkcolor=blue!50!black,
	urlcolor=blue!50!black
}

\makeatletter
\@addtoreset{chapter}{part}
\makeatother

\newcommand\Ni{\mathtt{N_i}}
\newcommand\Nx{\mathtt{N_x}}
\newcommand\Ntot{\mathtt{N_{tot}}}
\newcommand\Nf{\mathtt{N_{f}}}
\newcommand\Ne{\mathtt{N_{e}}}
\newcommand\he{\mathtt{h_{e}}}
\usepackage{bm}
\def\bpsi{{\boldsymbol \psi}}
\newcommand\Nt{\mathtt{N_{t}}}
\newcommand\htt{\mathtt{h_{t}}}

\begin{document}

\begin{titlepage}
\centering
\vspace*{1in}
\begin{Large}\bfseries
Control and stabilization of\par
geometrically exact beams\par
\end{Large}
\vspace{0.25in}
\begin{Large}
Steuerung und Stabilisierung von\par
geometrisch exakten Balken\par
\end{Large}
\vspace{4.5cm}
\begin{large}
\textit{Der Naturwissenschaftlichen Fakult\"at}\par
\end{large}
\vspace{0.2in}
\begin{large}
\textit{der}\par
\textit{Friedrich-Alexander-Universit\"at}\par
\textit{Erlangen-N\"urnberg}\par
\end{large}
\vspace{0.2in}
\begin{large}
zur\par
Erlangung des Doktorgrades Dr. rer. nat.\par
\end{large}
\vfill
%

\begin{large}
vorgelegt von\par
Charlotte Margot Rodriguez\par 
\end{large}
\end{titlepage}


\frontmatter

\vspace*{10cm}
\begin{center}
Als Dissertation genehmigt von der Naturwissenschaftlichen Fakult\"at\par
der Friedrich-Alexander-Universit\"at Erlangen-N\"urnberg
\end{center}
\vfill
Tag der m\"undlichen Pr\"ufung: 9 Dezember 2021 \par
\noindent Vorsitzende/r des Promotionsorgans: Prof. Dr. Wolfgang Achtziger \par
\noindent Gutachter/in: Prof. Dr. Günter Leugering \par
\hspace{1.925cm}Prof. Dr. Marius Tucsnak \par
\hspace{1.925cm}Dr. Andrew Wynn

\tableofcontents

\chapter{Acknowledgements}

First, I would like to thank my supervisor G\"unter Leugering for his constant optimism and belief that the math will work, but also for his kindness, generosity and patience. I esteem very much his love for mathematics, and appreciated the liberty and encouragements I received from him to carry out my research. Also thank you to Barbara for her kind and beautiful words.

I am grateful to Marius Tucsnak for all his help during the Master degree and after, for his guidance and generosity.
Thank you also to Enrique Zuazua, for always asking questions during my presentations, that permitted to deepen my understanding of the basis of my work, and for insisting that the numerical aspect is also important.

I would like to thank all the jury members and reviewers for accepting to take part and interest in my work.

\medskip

\noindent Thank you to the people here at FAU, colleagues and friends: Martin Gugat, Wigand Rathman, Doris Ederer, Wolfgang Achtziger, Michi (the first person I met in FAU when I was looking for Chair), Michele (my first office mate, for welcoming me and helping me so much when I arrived at the Department), Anja (I enjoyed very much our trips around Erlangen and am very grateful for  your help with settling in Erlangen and of course for your Käzespätzle recipe!), Daniel, Thomas, Falk, Christian, Yan (we explored the many Christmas markets of Erlangen as well as New Delhi together), Francisco, Ngoc (always here to help or listen to us), Marius, Christophe (thank you for creating the DNN group and its various other names, I think you helped weaving a strong social link at the Department), Alexei (very interesting introduction to Mexican food!), Arefeh (always with a smile for everyone), Nico, Tobias, Hannes, Lukas and many other nice encounters. I also enjoyed our exploration of the world's cooking styles and hiking trips around Erlangen.

Vielen Dank an Alex für deine Hilfe bei der deutschen Übersetzung. 
Thank you also Leon for your help with the Doctorate procedure.

Yue, Dani\"el, Adeel, I don't know how I would have managed the last week without you! Adeel, I admire very much your patience, open-mindedness and kindness. Thank you Dani\"el, for your friendship but also your very helpful comments and great course in numerics. And thank you Yue for your cooking lessons, kindness, curiosity, generosity, and great teamwork!

\medskip

\noindent Thank you also to Rafael Palacios, Andrew Wynn, Marc Artola, Arturo Mu\~noz-Sim\'on and their team at the Department of Aeronautics of Imperial College London. Our discussions were always enjoyable and eye opening. 
I also think that they influenced greatly my thesis. 
Marc, thanks to you I also got almost the last push needed to understand how to invert the transformation (the very last one being Professor Leugering's strong belief that is indeed invertible).

We were welcomed so warmly at the Mathematics Department of IIT Delhi, and it will remain a strong memory. Thank you very much to Mani Mehra and her family, as well as her students Vaibhav, Nitin, Abhishek. 
I also appreciated my stay at Deusto and  am grateful for meeting there Domenec, Abdennebi, Dario and many others.

\medskip

\noindent
Thank you to the ConFlex consortium, George Weiss, Xiaowei Zhao, Hans Zwart, Birgit Jacob and all other members of ConFlex.
Thank you to the ESRs, Marc, Arturo, Andrea, Borjan, Fatemeh, Nathanael, Juan, Jincheng, Pei, Pietro, Mir, Shantanu, Vikram and Gast\'on, we could not see each other very often, but our meeting were always full of good spirit and fun.
And thank you Andrea for explaining to me the engineering vocabulary and concepts.

Also thank you to the community of Math Stack Exchange and Math Overflow for the selfless help they gave me just for the sake of solving mathematical problems.

\medskip

\noindent Je voudrais remercier mes parents J\'er\^ome and Esp\'erance, mon frère Alexandre, ainsi que ma famille, Ma\"it\'e Mamie, Lucien Papy, \'Eric, Val\'erie, Laure, Lucas, Christophe, Angelina, Thomas, Louise, d'avoir toujours été là et être rest\'es unis aussi loin que je me souvienne. Merci aussi à mes amies Alexandra, Sandrine et Dilan. Tous ensembles vous étiez une bouteille de force toujours présente à mes côtés. I would like to thank my partner Borjan, for his unwavering support despite the difficulties, without whom I would not be writing these words.


\begin{flushright} Charlotte Rodriguez\\
Erlangen, the 31st of August 2021
\end{flushright}

\vfill

%
%
%

\begin{center} 
\begin{tikzpicture}
\node at (0, 0) {\includegraphics[scale=0.18]{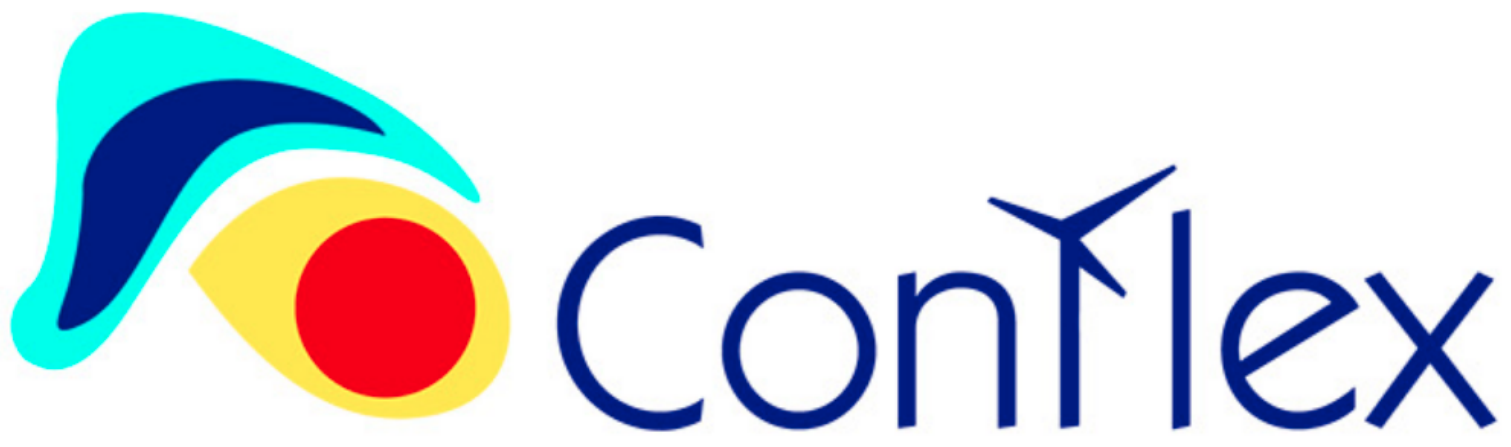}};

\node at (2.5, 0) {\includegraphics[scale=0.18]{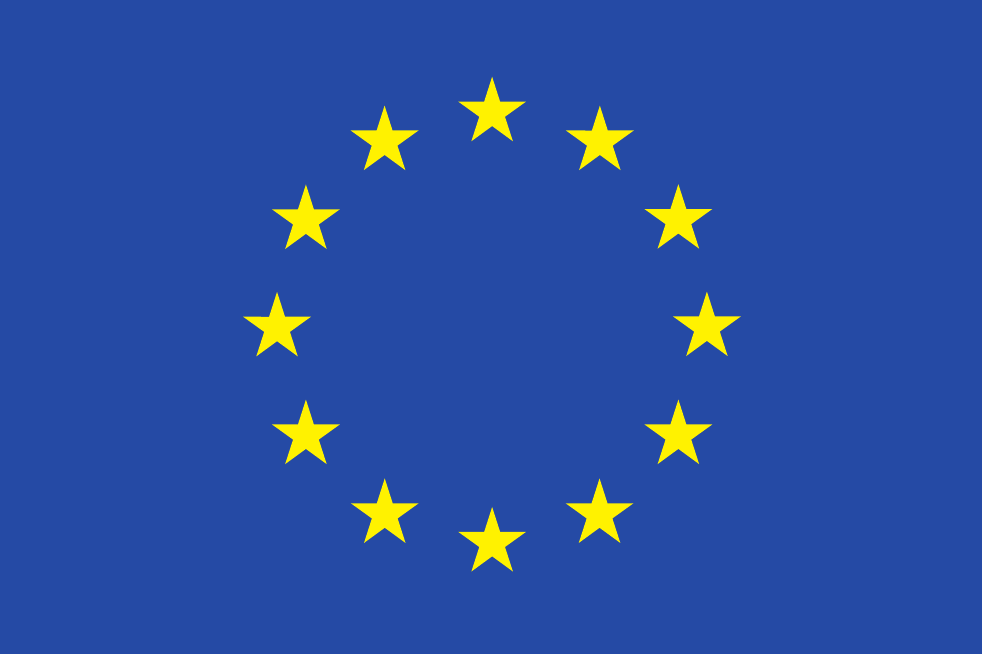}};

\node at (8.7, -0.025) {
\begin{minipage}{10cm}
This project has received funding from the European Union’s Horizon 2020 research and innovation programme under the Marie Sklodowska-Curie grant agreement No.765579-ConFlex.
\end{minipage}
};
\end{tikzpicture}
\end{center}

\chapter{Abstract}

In this thesis, we study well-posedness, stabilization and control problems involving freely vibrating beams that may undergo motions of large magnitude -- i.e. large displacements of the reference line and large rotations of the cross sections. Such beams, shearable and very flexible, are often called \emph{geometrically exact beams} and are especially needed in modern highly flexible light-weight structures, where one cannot neglect these large motions. The mathematical model is represented by a nonlinear governing system to account for such motions, and one thus also speaks of \emph{geometric nonlinearities}. The constitutive law however is linear, and allows for possibly composite, anisotropic materials.

\medskip

\noindent We view these beams from two perspectives. The first perspective is one in which the beam is described in terms of the position of its reference line and the orientation of its cross sections (expressed in some fixed coordinate system). This is the generally encountered model, due to Eric Reissner and Juan C. Simo. Of second order in time and space, it is a quasilinear system of six equations. 
The second perspective is one in which the beam is rather described by intrinsic variables -- velocities and strains or internal forces and moments -- which are moreover expressed in a \emph{moving} coordinate system attached to the beam. This system, proposed in its most general form by Dewey H. Hodges, consists of twice as many equations, but is of first order in time and space, hyperbolic and only semilinear (quadratic).
While the first model has a \emph{wave-like} form, the intrinsic beam model may be seen as part of the \emph{Hamiltonian} framework in continuum mechanics.

\medskip

\noindent 
Looking at the definition of the state of the latter model, we can see that both perspectives are linked by a nonlinear transformation.
The questions of well-posedness, stabilization and control are addressed for geometrically exact beams (and networks of such beams) governed by the intrinsic model, while by using the transformation we also prove that the existence and uniqueness of a classical solution to the intrinsic model implies that of a classical solution to the model written in terms of positions and rotations. In particular, this enables us to deduce corresponding well-posedness, stabilization and control results for the latter model. We also also address these questions for networks of beams attached to each other by means of rigid joints.

\medskip

\noindent The stabilization is realised by means of velocity feedback controls applied at the boundary. The first step consists in proving local in time existence and uniqueness of solutions in $C_t^0H_x^1$ (and $C_t^0H_x^2$ for more regular initial data) for general networks. Then, by means of \emph{quadratic Lyapunov functionals} we show that, first for a single beam controlled at one end, and then for a star-shaped network controlled at all simple nodes, one can achieve local exponential stability of the zero steady state for the $H^1$ and $H^2$ norms. Other than the quadratic nonlinearity, the main difficulty in finding such a functional lies in the fact that the linearized system is not homogeneous and thus one has to take into account not only the boundary conditions but also the governing system in order to find the functional.

The control problem we address herein is often called nodal profile control. It consists in steering the network states to given profiles \emph{at a prescribed node}, over a prescribed time interval, by means of sufficiently many controls actuating along the remainder of the nodes.
This notion is not hindered by the presence of cycles in the network. We first prove the semi-global in time existence and uniqueness of $C_{x,t}^1$ solutions -- i.e., for arbitrarily large time intervals, provided
that the initial and boundary data are small enough -- and then use the so-called \emph{constructive method} developed by Tatsien Li and collaborators to prove local exact controllability of nodal profiles for an A-shaped network.

We conclude with perspectives of future research and interesting open problems.


\chapter{Zusammenfassung}
In dieser Arbeit untersuchen wir Existenz und Eindeutigkeit von einer speziellen Klasse von partiellen Differentialgleichungen (PDG), die unter Anderem freischwingende Balken (die gro\ss en Auslenkungen unterworfen sind) modellieren. Weiter betrachten wir zugeh\"orige Steuerungs- und Stabiliserungsprobleme, wobei die Steuerung am Rand des Balkens lokalisiert ist (boundary control).
Die spezielle Modellklasse, die auch manchmal als ``geometric exact beam models'' bezeichnet wird, erm\"oglicht es, Balken zu modellieren, die scherbar und sehr flexibel sind, und insbesondere in modernen hochflexiblen Leichtbaustrukturen eingesetzt werden, bei denen gro\ss e Bewegungen nicht vernachl\"assigt werden k\"onnen.
Das resultierende mathematische Modell ist ein System von nichtlinearen, partiellen Differentialgleichung, das konstitutive ``Gesetz''  jedoch linear, und erlaubt zusammengesetzte, anisotrope Materialien zu ber\"ucksichtigen.

\medskip

\noindent Wir betrachten das genannte Balkenmodel aus zwei Perspektiven: Einmal wird der Balken durch die Position seiner Bezugslinie und die Ausrichtung seiner Querschnitte (ausgedr\"uckt in einem festen Koordinatensystem) beschrieben. Diese Beschreibung wird in der Literatur oft verwendet und geht auf Eric Reissner und Juan C. Simo zur\"uckgeht.
Die zugrundeliegende Dynamik ist dann durch ein System von quasilinearen partiellen Differentialgleichungen zweiter Ordnung in Raum (eindimensional) und Zeit gegeben (insgesamt $6$ Gleichungen).
In der zweiten Perspektive wird der Balken durch intrinsische Variablen -- das sind Geschwindigkeiten und Dehnungen oder innere Kr\"afte und Momente -- beschrieben, die in einem beweglichen, mit dem Balken verbundenen, Koordinatensystem ausgedr\"uckt werden. Diese von Dewey H. Hodges vorgeschlagene Beschreibung besteht dann aus doppelt so vielen Gleichungen ($12$ an der Zahl), ist daf\"ur aber in Zeit und Raum von erster Ordnung, hyperbolisch und nur semilinear (quadratisch).
Die erste Beschreibung (das System von 6 PDGs von zweiter Ordnung in Raum und Zeit) l\"asst sich als ein System von nichtlinearen Wellengleichungen interpretieren, w\"ahrend die zweite Beschreibung, die intrinsische mit 12 PDG, eher an eine Hamiltonische Formulierung erinnert.

\medskip

\noindent Die beiden oben eingef\"uhrten Beschreibungen oder Perspektiven k\"onnen durch eine nichtlineare Transformation ineinander \"uberf\"uhrt worden. Existenz und Eindeutigkeit, Stabilisierung und Steue-rung werden dann f\"ur die geometrisch exakten Balkenmodelle (und Netzwerke solcher Balken) studiert, wobei wir die einfachere intrinsische Formulierung, die ``nur'' semilinear ist, nutzen. Die Resultate lassen sich dann mit der entsprechenden inversen Transformation (deren Wohldefiniertheit zun\"achst gezeicht werden muss) auf das quasilineare System zur\"uck\"ubersetzen, was eine Interpretation in den entsprechenden Koordinaten zul\"asst. Wir erweitern diese Fragen danach auch auf Netzwerke von Balken, die durch starre Verbindungen miteinander gekoppelt sind.

\medskip

\noindent Die Stabilisierung wird mit Hilfe von Geschwindigkeitsr\"uckkopplungskontrollen am Balkenanfang oder -ende realisiert. Dazu beweisen wir in einem ersten Schritt die zeitlich lokale Existenz und Eindeutigkeit von L\"osungen im Raum $C_t^0H_x^1$ (und $C_t^0H_x^2$ f\"ur glattere Anfangsdaten) auf allgemeinen Netzwerken. Dann zeigen wir mit Hilfe der Lyapunov Theorie, dass man zun\"achst f\"ur einen einzelnen Balken, der an einem Ende gesteuert wird, und dann f\"ur ein sternf\"ormiges Netz, das an allen einfachen Knoten gesteuert wird, lokale exponentielle Stabilit\"at des station\"aren Nullzustands f\"ur die $H^1$- und $H^2$-Normen erreichen kann. Abgesehen von der quadratischen Nichtlinearit\"at der Dynamiken liegt die Hauptschwierigkeit bei der Suche nach einer geeigneten Lyapunov Funktion darin, dass das linearisierte System nicht homogen ist, und man daher nicht nur die Randbedingungen, sondern das gesamte beschreibende System bei der Wahl der Lyapunov Funktion ber\"ucksichtigen muss.

Unser Steuerungsproblem, das man auf dem allgemeinen Netzwerk (einschlie\ss lich ``Kreisen'') als Knotenprofilsteuerung ansehen kann, besteht darin, das System so zu steuern, dass es an einem Knoten des Netzes mit Hilfe von ausreichend vielen Steuerungen, die an anderen Knoten angewandt werden, bestimmte Profile \"uber ein bestimmtes Zeitintervall erreicht.  Dazu beweisen wir zun\"achst die semi-globale zeitliche Existenz und Eindeutigkeit von $C^1_{x,t}$-L\"osungen -- d.h.\ die Existenz und Eindeutigkeit f\"ur beliebig gro\ss e Zeitintervalle, sofern die Anfangs- und Randdaten klein genug sind -- und verwenden dann die von Tatsien Li, et al.\  entwickelte so genannte ``konstruktive Methode'', um die lokale exakte Steuerbarkeit von Knotenprofilen f\"ur ein A-f\"ormiges Netzwerk zu beweisen.

Abschlie\ss end werden Perspektiven f\"ur k\"unftige Forschung aufgezeigt und interessante offene Probleme angesprochen.

\chapter{Preface}
This thesis consists of three published/accepted peer-reviewed papers, and an exposition of these papers (given in Part \ref{part:exposition}). Reprints of the published articles are included in Part \ref{part:reprints}.

\begin{enumerate}[label={\normalfont \textbf{[A\arabic*]}}]
\item \label{A:SICON} 
Charlotte Rodriguez, and G\"unter Leugering. ``Boundary feedback stabilization for the intrinsic geometrically exact beam model''. In: \textit{SIAM Journal on Control and Optimization} 58 (6), pp. 3533--3558 (2020). \texttt{DOI}: \href{https://doi.org/10.1137/20M1340010}{\texttt{10.1137/20M1340010}}.

\item \label{A:MCRF}
Charlotte Rodriguez. ``Networks of geometrically exact beams: well-posedness and stabilization''. In: \textit{Mathematical Control and Related Fields} (2021). \texttt{DOI}: \href{https://www.aimsciences.org/article/doi/10.3934/mcrf.2021002}{\texttt{10.3934/mcrf.2021002}}. Advance online publication.

\item \label{A:JMPA}
G\"unter Leugering, Charlotte Rodriguez, and Yue Wang. ``Nodal profile control for networks of geometrically exact beams''. In: \textit{Journal de Mathématiques Pures et Appliquées} 155, pp. 111--139 (2021). \texttt{DOI}: \href{https://doi.org/10.1016/j.matpur.2021.07.007}{\texttt{10.1016/j.matpur.2021.07.007}}.
\end{enumerate}
The exposition is divided in six chapters providing a comprehensive introduction to the mathematical models of interest here, a summary of papers' results with ideas of the proofs of the main results, and a final summary and outlook.
The summary of the papers, from Chapter \ref{ch:wellposedness} to Chapter \ref{ch:control},
is not a sequential presentation of each paper. It is organized around three topics -- modelling and wellposedness (Chapter \ref{ch:wellposedness}), stabilization (Chapter \ref{ch:stab}) and control (Chapter \ref{ch:control}) -- in order to give a general overview to the reader.
Chapter \ref{ch:wellposedness} is built upon all three papers \ref{A:SICON}-\ref{A:JMPA}, Chapter \ref{ch:stab} upon \ref{A:SICON} and \ref{A:MCRF}, and Chapter \ref{ch:control} upon \ref{A:JMPA}.
Each chapter is accompanied with an Appendix consisting of preliminaries or proofs that could not be included in the aforementioned papers and are thus given here for the sake of completeness. 

\medskip

\noindent Charlotte Rodriguez’s contributions to the co-authored publications are summarized as follows:

\begin{itemize}
\item[\ref{A:SICON}] GL proposed the idea of studying the intrinsic formulation of the mathematical model of the geometrically exact beams (i.e., the IGEB model), before studying the transformation between the GEB and IGEB models. GL proposed to stabilize a beam by means of velocity feedback controls applied at one end. CR chose the Lyapunov theory for the stabilization analysis and developed a method to build an appropriate quadratic Lyapunov functionals and conceived independently a strategy to invert the transformation. Hence, CR was responsible for working out the proofs of this paper and wrote the article with guidance on its organization from GL. CR was corresponding author for this publication.

\item[\ref{A:JMPA}] Based on CR's expertise on models for geometrically exact beams, GL proposed the idea of studying nodal profile control for networks of geometrically exact beams containing loops, and suggested the theory of Tatsien Li and his collaborators as a basis. CR was responsible for the proofs of well-posedness and invertibility of the transformation between the GEB and IGEB formulation, for general networks with loops.
CR and YW worked out together the proof of controllability for the A-shaped network, as well as the Algorithm proposed in view of extending the controllability proof to other networks and settings. CR was responsible for writing most of the article, and was corresponding author for this publication.

\end{itemize}

\noindent Further preprints or accepted peer-reviewed conference papers -- written in the same period of time -- by the author, are not included in the thesis but listed below. They complement the topics covered in this work.

\begin{enumerate}[label={\normalfont \textbf{[P\arabic*]}}]

\item \label{P:CDC}
Marc Artola, Charlotte Rodriguez, Andrew Wynn, Rafael Palacios and G\"unter Leugering. ``Optimisation of Region of Attraction Estimates for the Exponential Stabilisation of the Intrinsic Geometrically Exact Beam Model''. In: \textit{IEEE Conference on Decision and Control} (2021).
Accepted. \texttt{arXiv}: \href{https://arxiv.org/abs/2110.06002}{\texttt{2110.06002}}

\item \label{P:ssm}
G\"unter Leugering, Charlotte Rodriguez, Yue Wang. ``Exact controllability of networks of elastic strings springs and masses''. 2021. In preparation.

\end{enumerate}

\mainmatter

\part{Exposition}
\label{part:exposition}

\chapter{Introduction}
\label{ch:intro}

\section{Motivation}
\label{sec:motivation}

%



Beam models describing the three-dimen\-sional motions of thin elastic bodies (much bigger in one dimension than the other two) have found many applications in civil, mechanical and aerospace engineering. This is also true for multi-link flexible structures such as large spacecraft structures, trusses, robot arms, solar panels, antennae \cite{chen_serial_EBbeams, LLS, flotow_spacecraft}, generally modeled by networks of interconnected beams.
Depending on the assumptions made on the beam, there are various partial differential equation (PDE) models for flexible beams, such as the well-known Euler-Bernoulli and Timoshenko models where for the former the \emph{cross sections} remain perpendicular to the \emph{reference line} (or \emph{centerline}), while they may also rotate for the latter -- one then speaks of \emph{shearing}.

\medskip

\noindent 
However, nowadays there is a growing interest in modern highly flexible light-weight structures -- for instance robotic arms \cite{grazioso2018robot}, flexible aircraft wings \cite{Palacios2010aero} or wind turbine blades \cite{Munoz2020, wang2014windturbine} -- which exhibit \emph{motions of large magnitude}, not negligible in comparison to the overall dimensions of the object.
To capture such a behavior, one needs a beam model which is \emph{geometrically exact} (sometimes also called \emph{geometrically nonlinear}), in the sense that the governing system presents nonlinearities in order to also represent large motions -- i.e., large displacements of the centerline and large rotations of the cross sections. 

\begin{figure}[h!] \centering
\includegraphics[scale=0.75]{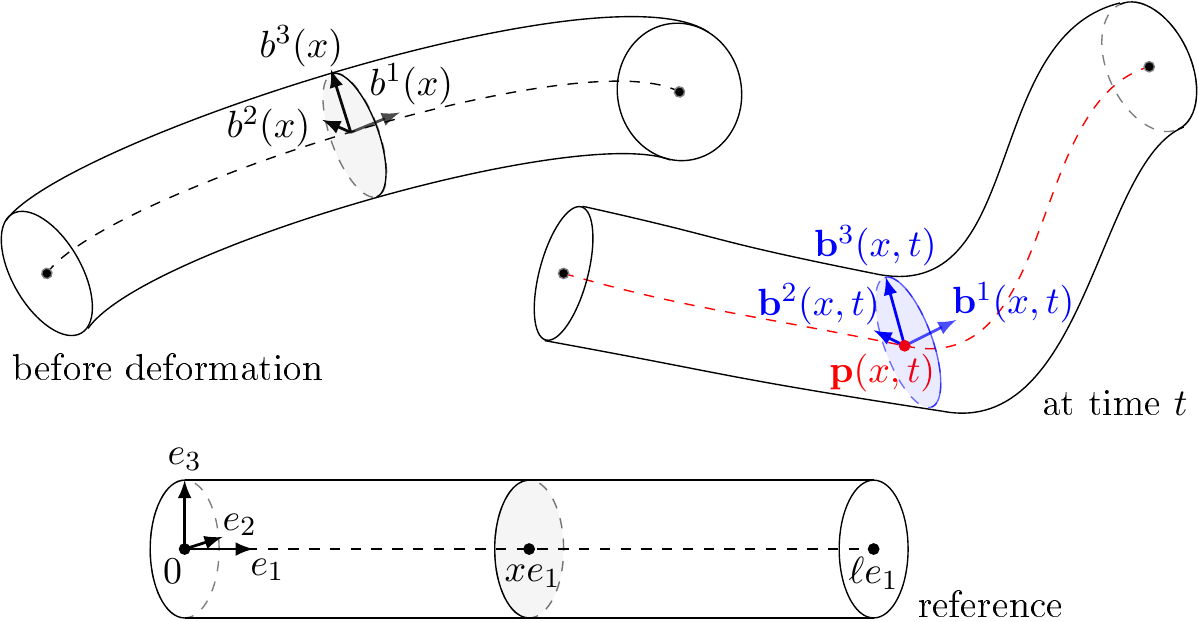}
\caption{The \emph{straight reference} beam (bottom), the beam \emph{before deformation} characterized by the curvature $\Upsilon_c = \mathrm{vec}(R^\intercal \frac{\mathrm{d}}{\mathrm{d}x}R)$ where $R = \begin{bmatrix}b^1 & b^2 & b^3 \end{bmatrix}$ (upper left), and the beam \emph{at time $t$} described by the state variables $\mathbf{p}$ and $\mathbf{R} = \begin{bmatrix}
\mathbf{b}^1 & \mathbf{b}^2 & \mathbf{b}^3
\end{bmatrix}$ (upper right).}
\label{fig:beam_unknown}
\end{figure}

This beam model, similarly to the Euler-Bernoulli and Timoshenko systems, is one-dimensional with respect to the spatial variable $x$ and accounts for linear-elastic material laws, meaning that the \emph{strains} (which are the local changes in the shape of the material) are assumed to be small.
As for the Timoshenko system, models for geometrically exact beams account for shear deformation as well.
Moreover, the geometrical and material properties of the beam may vary along the beam (indeed, we will see that the coefficients of the system depend on $x$) and the material may be anisotropic.
As a matter of fact, the Euler-Bernoulli and Timoshenko equations can be derived from geometrically exact beam models under appropriate simplifying assumptions \cite[Section IV]{Artola2021damping}.

\medskip

\noindent The mathematical model for geometrically exact beams may be written in terms of the position of the centerline of the beam and the orientation of its cross sections, with respect to a \emph{fixed} coordinate system $\{e_j\}_{j=1}^3$ (the standard basis of $\mathbb{R}^3$ here). This is the commonly known \emph{Geometrically Exact Beam model}, or GEB, which originates from the work of Reissner \cite{reissner1981finite} and Simo \cite{simo1985finite}. The governing system is quasilinear, consisting of six equations. One may draw a parallel with the wave equation as the GEB model is of second order both in space and time.

The state is then $(\mathbf{p}, \mathbf{R})$, expressed in $\{e_j\}_{j=1}^3$ and composed of the centerline's position $\mathbf{p}(x,t) \in \mathbb{R}^3$ and cross sections' orientation\footnote{The special orthogonal group $\mathrm{SO}(3)$ is the set of unitary real matrices of size $3$ and with a determinant equal to $1$, also called \emph{rotation} matrices.} given by the columns $\{\mathbf{b}^j\}_{j=1}^3$ of $\mathbf{R}(x,t) \in \mathrm{SO}(3)$. 
We refer to Fig. \ref{fig:beam_unknown} for visualisation.
For a beam of length $\ell>0$, set in $(0, \ell)\times(0, T)$, the governing system reads\footnote{
For any $u \in \mathbb{R}^3$, $\widehat{u}$ denotes the skew-symmetric matrix equivalent to the vector cross multiplication by $u$, that arises when treating the cross product as a linear map in the second argument, while $\mathrm{vec}(\cdot)$ permits to recover $u = \mathrm{vec}(\widehat{u})$ (see Section \ref{sec:notation} for more detail).}
\begin{align}
\label{eq:GEB_pres}
\begin{bmatrix}
\partial_t & \mathbf{0}\\
(\partial_t \widehat{\mathbf{p}}) & \partial_t
\end{bmatrix} \left[ \begin{bmatrix}
\mathbf{R} & \mathbf{0}\\ \mathbf{0} & \mathbf{R}
\end{bmatrix}
\mathbf{M} v \right] = \begin{bmatrix}
\partial_x & \mathbf{0} \\ (\partial_x \widehat{\mathbf{p}}) & \partial_x
\end{bmatrix} \left[\begin{bmatrix}
\mathbf{R} & \mathbf{0}\\ \mathbf{0} & \mathbf{R}
\end{bmatrix} z\right] + \begin{bmatrix}
\overline{\phi} \\ \overline{\psi}
\end{bmatrix},
\end{align}
given\footnote{The set of positive definite symmetric matrices of size $n$ is denoted $\mathbb{S}_{++}^n$.} external forces and moments $\overline{\phi}(x,t), \overline{\psi}(x,t) \in \mathbb{R}^3$, the mass matrix $\mathbf{M}(x)\in \mathbb{S}_{++}^{6}$, the flexibility (or compliance) matrix $\mathbf{C}(x) \in \mathbb{S}_{++}^{6}$ and the curvature before deformation $\Upsilon_c(x) \in \mathbb{R}^3$, and where $v, z$ depend on $(\mathbf{p}, \mathbf{R})$:
\begin{equation} \label{eq:def_v_z}
v
= \begin{bmatrix}
\mathbf{R}^\intercal \partial_t \mathbf{p}\\ \mathrm{vec}\left( \mathbf{R}^\intercal \partial_t \mathbf{R} \right)
\end{bmatrix}, \quad 
s= \begin{bmatrix}
\mathbf{R}^\intercal \partial_x \mathbf{p}  - e_1 \\ 
\mathrm{vec}\left(\mathbf{R}^\intercal \partial_x \mathbf{R} \right) - \Upsilon_c
\end{bmatrix}, \quad z = \mathbf{C}^{-1} s.
\end{equation}

On the other hand, the mathematical model can also be written in terms of so-called \emph{intrinsic} variables -- namely, velocities and internal forces/moments, or equivalently velocities and strains -- expressed in a \emph{moving} coordinate system attached to the beam (the moving basis $\{\mathbf{b}^j\}_{j=1}^3$).
This yields the \emph{Intrinsic Geometrically Exact Beam model}, or IGEB, which is introduced by Hodges \cite{hodges1990, hodges2003geometrically}. The governing system then counts twelve equations.
The state is 
\begin{align*}
y = \begin{bmatrix}
v\\z
\end{bmatrix},
\end{align*}
expressed in the moving basis and composed of the linear and angular velocities $v(x, t) \in \mathbb{R}^6$ and internal forces and moments $z(x, t) \in \mathbb{R}^6$. Set in $(0, \ell)\times(0, T)$, the governing system reads
\begin{align}
\label{eq:IGEB_pres} 
\begingroup 
\setlength\arraycolsep{2.5pt}
\renewcommand*{\arraystretch}{1}   
\begin{bmatrix}
\mathbf{M} & \mathbf{0}\\
\mathbf{0} & \mathbf{C}
\end{bmatrix}
\endgroup
\partial_t y
-
\begingroup 
\setlength\arraycolsep{2.5pt}
\renewcommand*{\arraystretch}{1}
\begin{bmatrix}
\mathbf{0} & \, \mathbf{I}\,\\
\, \mathbf{I}\, & \mathbf{0}
\end{bmatrix}
\endgroup
\partial_x y
- 
\begingroup 
\setlength\arraycolsep{2pt}
\renewcommand*{\arraystretch}{0.9}
\begin{bmatrix}
\mathbf{0} & 
\begin{matrix}
\widehat{\Upsilon}_c & \mathbf{0}\\
\widehat{e}_1 & \widehat{\Upsilon}_c
\end{matrix}
\\
\begin{matrix}
\widehat{\Upsilon}_c & \widehat{e}_1 \\
\mathbf{0}& \widehat{\Upsilon}_c
\end{matrix} & \mathbf{0}
\end{bmatrix}
\endgroup
y 
=
-
\begingroup 
\setlength\arraycolsep{2.5pt}
\renewcommand*{\arraystretch}{1}
\begin{bmatrix}
\begin{matrix}
\widehat{v}_2 & \mathbf{0}\\
\widehat{v}_1 & \widehat{v}_2
\end{matrix} & 
\begin{matrix}
 \mathbf{0} & \widehat{z}_1 \\
\widehat{z}_1 & \widehat{z}_2
\end{matrix}\\
\mathbf{0} & 
\begin{matrix}
\widehat{v}_2 & \widehat{v}_1\\
\mathbf{0} & \widehat{v}_2
\end{matrix}
\end{bmatrix}
\endgroup
\begin{bmatrix}
\mathbf{M} v\\ \mathbf{C}z
\end{bmatrix}
+ \begin{bmatrix}
\overline{\Phi} \\ \overline{\Psi} \\ \mathbf{0} \\ \mathbf{0}
\end{bmatrix}
\end{align}
denoting by $v_1,z_1$ and $v_2,z_2$ the first and last three components of $v,z$, and where $\overline{\Phi}(x,t), \overline{\Psi}(x,t) \in \mathbb{R}^3$ are the external forces and moments expressed in the moving basis.

An interesting feature of the IGEB model is that it falls into the class of one-dimensional first-order hyperbolic systems and is moreover only semilinear. Therefore, from a mathematical perspective, one gains access to the broad literature which has been developed on such system -- see notably Li and Yu \cite{Li_Duke85}, Bastin and Coron \cite{BC2016} -- beyond the context of beam models.
Due to its less compound nature, the IGEB formulation is used in aeroelastic modelling and engineering, notably in the context of very light-weight and slender aircraft aiming to remain airborne over long time horizons, and that consequently exhibit great flexibility \cite{Palacios2017modes, Palacios2011intrinsic, Palacios2010aero}; see also \cite{Artola2020aero, Artola2019mpc, Artola2021damping} where the authors additionally take into account structural damping.

\medskip

\noindent On another hand, as pointed out in \cite[Sec. 2.3.2]{weiss99}, one may see the GEB model and IGEB model as being related by a \emph{nonlinear transformation} (which we define in \eqref{eq:transfo}). In this work, we will keep track of this link between both models, studying mathematically the latter, and then deducing corresponding results for the GEB model.

As commonly done in solid mechanics, both the GEB and IGEB models are \emph{Lagrangian descriptions} of the beam (as opposed to the \emph{Eulerian description}), in the sense that the independent variable $x$ is attached to matter ($x$ is a label sticking to the particles of the beam's centerline throughout the deformation history) rather than being attached to an inertial frame of reference.

The IGEB model can also be seen as the beam dynamics being formulated in the \emph{Hamiltonian} framework in continuum mechanics (see notably \cite[Sections 5, 6]{Simo1988}), while the GEB model corresponds to the \emph{Lagrangian} framework.
Then, taking into account the interactions of the beam with its environment, one may study the IGEB model from the perspective of \emph{Port-Hamiltonian Systems} (see \cite{Maschke1992} for the finite dimension setting and \cite{Schaft2002},\cite[Chap. 7]{Zwart2012bluebook} for the infinite dimensions setting), as in  \cite{Macchelli2007, Macchelli2009} and \cite[Section 4.3.2]{Macchelli2009book}. See also the case of the Timoshenko model in \cite{Macchelli2004Timo, Mattioni2020Timo}.

\medskip

\noindent Though we presented the models in their most general form, with the presence of external forces and moments $\overline{\phi}, \overline{\psi}$ or $\overline{\Phi}, \overline{\Psi}$, we will assume that the beams are \emph{freely vibrating}. This means that the external forces and moments -- which could represent gravity or aerodynamic forces for instance -- are set to zero.


\section{Structure of this exposition}

In Chapter \ref{chap:mathmodel}, we explain how the position of the beam as well as its material and geometrical properties are described, we present the GEB  and IGEB models in detail, and the nonlinear transformation by which they are linked, and finally we introduce the mathematical model for networks of beams linked by rigid joints when the beams dynamics are given in terms of positions and rotation matrices.

Chapter \ref{ch:wellposedness} is concerned with providing the model for networks of beams described by the IGEB model, giving well-posedness results for the latter, and inverting the transformation between the two descriptions (for a single beam and networks).

Building upon the this, Chapter \ref{ch:stab} and Chapter \ref{ch:control} present the stabilization and controllability results. Local exponential stabilization by means of feedback controls is proved for a single beam controlled at one node and a star-shaped network controlled at all simple nodes, and the removal of one of the controls in the latter case is then discussed. Local exact controllability of nodal profiles is proved for an A-shaped network where one controls the internal forces and moments at both simple nodes to achieve a given profile at one of the multiple nodes, and the extension to more general networks and scenarios is then discussed.

Part \ref{part:exposition} is concluded in Chapter \ref{ch:conclusion} that summarizes the main
results and gives an outlook on future research questions.
Finally, in Part \ref{part:reprints} are found reprints of the author’s published/accepted peer-reviewed papers, which constitute this thesis.

\section{Notation}
\label{sec:notation}

Let $m, n \in \{1, 2, \ldots\}$ and $M \in \mathbb{R}^{n\times n}$. 
We use the notation $\|M\| = \sup_{|\xi| = 1} |M \xi|$, where $|\, . \,|$ is the Euclidean norm. The inner product in $\mathbb{R}^n$ is denoted $\langle \cdot \,, \cdot \rangle$.
%
%
We denote by $\{e_j\}_{j=1}^3 = \{(1, 0, 0)^\intercal, (0, 1, 0)^\intercal, (0, 0, 1)^\intercal\}$ the standard basis of $\mathbb{R}^3$.

\medskip

\noindent \textbf{Matrices.}
Here, the identity and null matrices are denoted by $\mathbf{I}_n \in \mathbb{R}^{n \times n}$ and $\mathbf{0}_{n, m} \in \mathbb{R}^{n \times m}$, and we use the abbreviation $\mathbf{0}_{n} = \mathbf{0}_{n, n}$. If there is no confusion, we omit the subscript and write $\mathbf{I}$ and $\mathbf{0}$ instead. 
%
The transpose and determinant $M$ are denoted by $M^\intercal$ and $\mathrm{det}(M)$, respectively. 
%
The symbol $\mathrm{diag}(\, \cdot \, , \ldots, \, \cdot \, )$ denotes a (block-)diagonal matrix composed of the arguments.
We denote by $\mathbb{S}^n$ (and $\mathbb{D}^n$), by $\mathbb{S}_+^n$ (and $\mathbb{D}_+^n$) and by $\mathbb{S}^n_{++}$ (and $\mathbb{D}^n_{++}$) the sets of symmetric (resp. diagonal), positive semi-definite symmetric (resp. positive semi-definite diagonal) and positive definite symmetric (resp. positive definite diagonal) real matrices of size $n$, respectively.
Then $\mathbb{S}_-^n$ (and $\mathbb{D}_-^n$) and $\mathbb{S}_{--}^n$ (and $\mathbb{D}_-^n$) denote the analogous spaces for negative definite and negative semi-definite matrices.

\medskip

\noindent \textbf{Functional spaces.}
In the norms' subscripts, when there is no ambiguity, we use the abbreviations $C_x^m = C^m([0, \ell_i]; \mathbb{R}^d)$ and $H_x^m = H^m(0, \ell_i; \mathbb{R}^{d})$, as well as $C_t^m = C^m(I; \mathbb{R}^d)$ and $C_{x,t}^m = C^m([0, \ell_i]\times I; \mathbb{R}^d)$ for the appropriate $\ell_i>0$, time interval $I$ and dimension $d \in \{1, 2, \ldots\}$.
Furthermore, we denote $\mathbf{H}^m_x = \prod_{i=1}^N H^m(0, \ell_i; \mathbb{R}^{12})$, $\mathbf{C}^m_{x} = \prod_{i=1}^{N} C^m([0, \ell_i]; \mathbb{R}^{12})$ and $\mathbf{L}^2_{x} = \prod_{i=1}^{N} L^2(0, \ell_i; \mathbb{R}^{12})$, these spaces being endowed with the associated product norms.

\medskip 

\noindent \textbf{Cross product.}
The cross product between any $u, \zeta \in \mathbb{R}^3$ is denoted $u \times \zeta$, and we shall also write $\widehat{u} \,\zeta = u \times \zeta$, meaning that $\widehat{u}$ is the skew-symmetric matrix 
\begin{linenomath}
\begin{equation*}
\widehat{u} = \begin{bmatrix}
0 & -u_3 & u_2 \\
u_3 & 0 & -u_1 \\
-u_2 & u_1 & 0
\end{bmatrix}, 
\end{equation*}
\end{linenomath}
and for any skew-symmetric $\mathbf{u} \in \mathbb{R}^{3 \times 3}$, the vector $\mathrm{vec}(\mathbf{u}) \in \mathbb{R}^3$ is such that $\mathbf{u} = \widehat{\mathrm{vec}(\mathbf{u})}$. 
%

%
%
%
%
\chapter{The mathematical models}
\label{chap:mathmodel}

After having given details about the variables and parameters that permit to describe the position and properties (geometry, material) of the beam in Section \ref{sec:mechanical_setting}, we give a more complete description of the GEB and IGEB models in Sections \ref{sec:pres_GEB} and \ref{sec:pres_IGEB}. As already mentionned in the introduction, the latter model originates from \cite{hodges2003geometrically} and since the notation used here differs significantly from this reference, we provide the correspondence with that of \cite{hodges2003geometrically} in Appendix \ref{ap:hodges}. 
After this, in Section \ref{sec:pres_transfo}, we present in more detail the transformation linking these two models, and finally introduce the model considered here for networks of beams in Section \ref{sec:networks}.
Let $T>0$.

\section{Mechanical setting}
\label{sec:mechanical_setting}

\subsection{Description of the beam}
\label{sec:decription_beam}

\begin{figure}\centering
\includegraphics[scale=0.75]{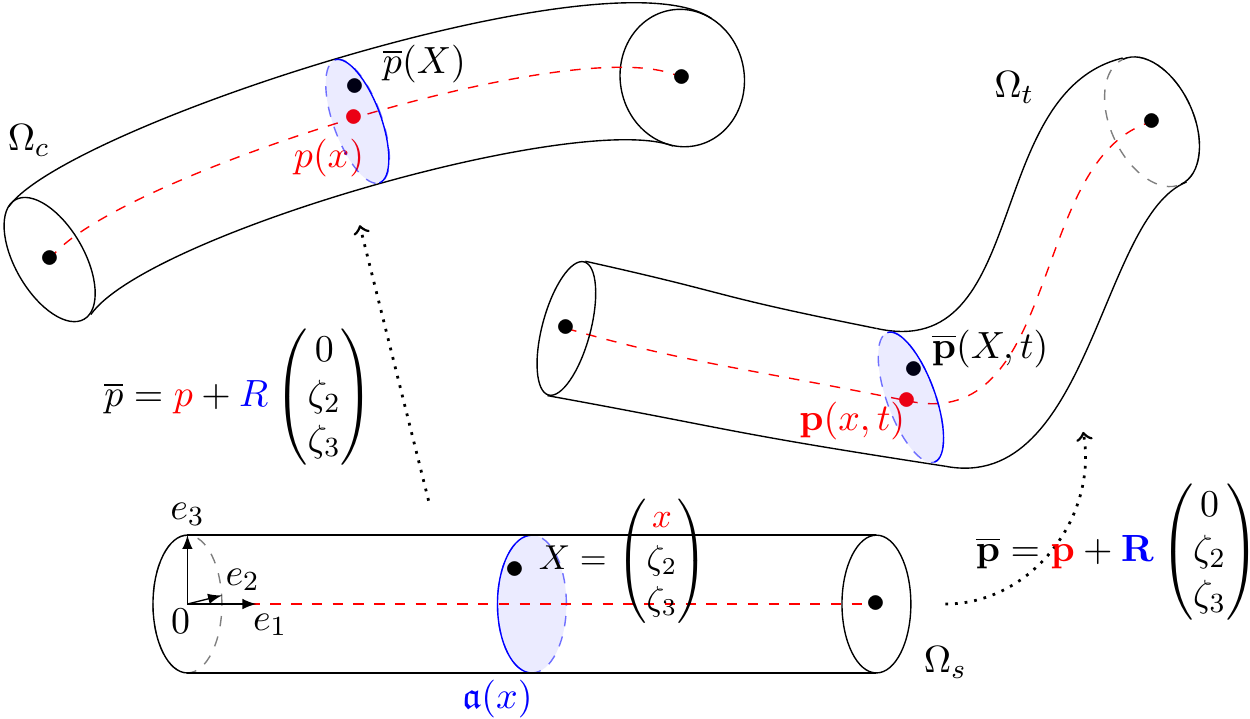}
\caption{Description of the beam in its different configurations $\Omega_s $, $\Omega_c $ and $\Omega_t $..}
\label{fig:beam_position}
\end{figure}

The beam is idealized as a \emph{reference line} -- that we also call \emph{centerline} -- and a family of \emph{cross sections}.
At rest (before deformation) the position of the centerline $p \colon [0, \ell] \rightarrow \mathbb{R}^3$ and the orientation of the cross sections, are both known. The latter is given by the columns $\{b^j\}_{j=1}^3$ of a rotation matrix $R \colon [0, \ell] \rightarrow \mathrm{SO}(3)$. We assume that $b^1 = \frac{\mathrm{d}p_i}{\mathrm{d}x}$, meaning that $p$ is parametrized by its arclength.

At any time $t \in [0, T]$, the position $\mathbf{p} \colon [0, \ell]\times [0, T] \rightarrow \mathbb{R}^3$ of the centerline and the orientation of the cross sections, given by the columns $\{\mathbf{b}^j\}_{j=1}^3$ of a rotation matrix $\mathbf{R} \colon [0, \ell]\times[0, T] \rightarrow \mathrm{SO}(3)$, are both unknown.
As shear deformation is allowed, $\mathbf{b}^1$ is not necessarily tangent to the centerline.

\medskip

\noindent Let $\Omega_{s} \subset \mathbb{R}^3$ be the straight, untwisted beam whose centerline is located at $x e_1$ for $x \in [0, \ell]$; it may be written as $\Omega_{s}  = \bigcup_{x \in [0,\ell]}\mathfrak{a} (x)$ where $\mathfrak{a} (x)$ is the cross section intersecting the centerline at $x e_1$.
Then, the beam before deformation takes the form $\Omega_c  = \{\overline{p} (X) \colon X \in \Omega_s \}$ while the beam at time $t>0$ takes the form $\Omega_t  = \{\overline{\mathbf{p}} (X, t) \colon X \in \Omega_s \}$, where $\bar{p} $ and $\bar{\mathbf{p}} $ are defined by $\overline{p} (X) = p (x) + R (x)(\zeta_2 e_2 + \zeta_3 e_3)$ and 
\begin{align*}
\overline{\mathbf{p}} (X,t) = \mathbf{p} (x,t) + \mathbf{R} (x,t)(\zeta_2 e_2 + \zeta_3 e_3),
\end{align*}
using the notation $X = (x,\zeta_2,\zeta_3)^\intercal$ for any $X \in \Omega_{s}$.
We call $\Omega_s $, $\Omega_c $ and $\Omega_t $ the  \emph{straight-reference} configuration, \emph{curved-reference} configuration and \emph{current} configuration of the beam (see Fig. \ref{fig:beam_position}), respectively. One may also call $\Omega_c$ and $\Omega_t$ the undeformed and the deformed beam, respectively.

\begin{remark}
The curvature before deformation $\Upsilon_c$ is expressed in terms of $R$ as $\Upsilon_c = \mathrm{vec}\left(R^\intercal \frac{\mathrm{d}}{\mathrm{d}x}R\right)$. If the beam is straight and untwisted with centerline $p(x) = xe_1$ before deformation, then $R$ is the identity matrix and $\Upsilon_c = 0$.
\end{remark}

\subsection{Mass matrix, flexibility matrix and constitutive law}
\label{sec:mass_flex_materialLaw}

While the mass matrix $\mathbf{M}$ relates the linear momentum $P$ and angular momentum $H$, to the velocities by 
\begin{align} \label{eq:def_momenta}
\begin{bmatrix}
P\\H
\end{bmatrix} = \mathbf{M}v,
\end{align}
the flexibility matrix $\mathbf{C}$ relates the stresses $z$ (i.e. vector of \emph{internal forces and moments} expressed in the moving basis) to the strains $s$. As we consider linear-elastic materials, the latter relationship, already introduced in \eqref{eq:def_v_z}, reads 
\begin{align} \label{eq:stress_strain_rel}
z = \mathbf{C}^{-1} s.
\end{align}

\begin{remark}[Other unknown for the IGEB model]
In view of this, for the IGEB model, one also has the possibility to choose the unknown state $y$ consisting of velocities $v$ and strains $s$ -- as in \ref{A:SICON} --, or velocities $v$ and stresses $z$ -- as in \ref{A:MCRF} and \ref{A:JMPA}. Both obtained first-order systems will have similar properties. In this exposition we focus on the case where $y$ consists of velocities and stresses to keep the presentation uniform.
\end{remark}

In general, for beams made of linear-elastic material, but without additional assumptions on the material and geometry, the mass and flexibility matrices are positive definite (possibly positive semi-definite for the former), symmetric and dependent on $x$.
In the specific case of prismatic isotropic beams with sectional principal axis aligned with the body-attached basis, both the mass and flexibility matrices are, in addition, constant and diagonal. More precisely, they have the form
\begin{align}\label{eq:def_M_C_iso}
\mathbf{C} = \mathrm{diag}(S_1, S_2)^{-1}, \qquad \mathbf{M} = \rho \, \mathrm{diag}\big( a \mathbf{I}_3, \ J \big),
\end{align} 
where $J \in \mathbb{R}^{3\times 3}$, called the \textit{inertia matrix}, and $S_1, S_2 \in \mathbb{R}^{3 \times 3}$, are positive definite diagonal matrices defined by
\begin{align} \label{eq:def_J_S1_S2}
J &= \mathrm{diag}\big((I_2+I_3)k_1, \ I_2, \ I_3 \big)\qquad \text{and} \qquad
\left\{\begin{aligned}
S_1 &=  a \, \mathrm{diag}(E, k_2 G, k_3 G) \\
S_2 &= J \, \mathrm{diag}(G, E, E).
\end{aligned}\right.
\end{align}
These matrices are expressed in terms of the following geometrical and material parameters: density $\rho>0$, cross section area $a>0$, shear modulus $G>0$, Young modulus $E>0$, area moments of inertia $I_2, I_3>0$, shear correction factors $k_2>0, \ k_3 >0$, and the factor $k_1>0$ that corrects the polar moment of area.

Our stance here is that both $\mathbf{M}$ and $\mathbf{C}$ have values in $\mathbb{S}_{++}^6$, and are at least $C^1$ with respect to $x$ if they are space-dependent. In \ref{A:SICON} we assume that the mass and flexibility matrices are defined by \eqref{eq:def_M_C_iso}-\eqref{eq:def_J_S1_S2}, while in the \ref{A:MCRF} and \ref{A:JMPA} we remain with the general assumptions.

\medskip

\noindent We work here with the constitutive law \eqref{eq:stress_strain_rel} in which the stresses are a linear function of the strains. However, this could be changed for instance to include viscous damping. In this case, rather than defining $z$ as just $\mathbf{C}^{-1}s$, one can choose 
\begin{align} \label{eq:constLaw_fullKV}
z = z_\dagger + C_\tau \partial_t z_\dagger, \quad z_\dagger:=\mathbf{C}^{-1}s,
\end{align}
where $C_\tau \in \mathbb{R}^{6 \times 6}$ characterizes the damping of the viscoelastic material (see \cite{Artola2021damping}),
and thus inject this new expression of $z$ in \eqref{eq:GEB_pres}. To change the constitutive law in IGEB model, one should only replace $z$ by its new expression \eqref{eq:constLaw_fullKV} in the first six equations of \eqref{eq:IGEB_pres} and replace $z$ by $z_\dagger$ in the last six ones. Indeed, the last six governing equations are in fact derived directly from the definition of the strains $s$ (see Proposition \ref{prop:transfoGov} for more detail) and consequently the term $z$ there only represents there the quantity $\mathbf{C}^{-1}s$ independently of the chosen consititutive law. The state would become
\begin{align*}
y = \begin{bmatrix}
v \\ z_\dagger
\end{bmatrix}.
\end{align*}
This is notably considered in \cite{Artola2020aero, Artola2019mpc, Artola2021damping} in contexts of model reduction, optimal control and aeronautics, where the authors moreover use the last six equations of \eqref{eq:IGEB_pres} to approximate the term $C_\tau \partial_t z_\dagger$ in \eqref{eq:constLaw_fullKV} by $C_\tau \mathbf{C}^{-1}(\partial_x v - \mathbf{E}^\intercal v)$, neglecting any nonlinear term that appears with the new constitutive law.

\subsection{Body-attached variables}
The set $\{\mathbf{b} ^j\}_{j=1}^3$ can be seen as a body-attached (moving with time) basis, with origin $\mathbf{p}$. Hence, here, we consider two kinds of coordinate systems: $\{e_j\}_{j=1}^3$ which is fixed in space and time, and the body-attached basis $\{\mathbf{b} ^j\}_{j=1}^3$.

We then make the difference between two kinds of vectors in $\mathbb{R}^3$: global and body-attached.
Consider two vectors $u := \sum_{j=1}^3 u_j e_j$ and $U:=\sum_{j=1}^3 U_j e_j$ of $\mathbb{R}^3$, 
the former being a \emph{global} vector and the latter being the \emph{body-attached} representation of $u$. By this, we mean that the components of $u$ are its coordinates with respect to the global basis $\{e_j\}_{j=1}^3$, while the components of $U$ are coordinates of the vector $u$ with respect to the body-attached basis $\{\mathbf{b} ^j\}_{j=1}^3$. In other words $u = \sum_{j=1}^3 U_j \mathbf{b} ^j $. Both vectors are then related by the identity $u = \mathbf{R}  U$ since $\mathbf{b} ^j = \mathbf{R}  e_j$, and we may also call $u$ the \emph{global} representation of $U$. In fact, in Section \ref{sec:motivation}, we already encountered variables related in such a way: the external forces and moments
\begin{align} \label{eq:relation_phiPhi_psiPsi}
\overline{\phi} = \mathbf{R}\overline{\Phi}, \quad \overline{\psi} = \mathbf{R}\overline{\Psi}.
\end{align}

Similar considerations also hold for the set $\{b^j\}_{j=1}^3$ of columns of $R$, that we may then call a body-attached basis for the undeformed beam.

\medskip

\noindent Let $(V, W)$, $(\Phi, \Psi)$ and $(\Gamma, \Upsilon)$ denote the first and last three components of $v$, $z$ and $s$ (defined in \eqref{eq:def_v_z}), respectively. In other words, $V$ is the linear velocity vector, $W$ the angular velocity vector, $\Psi$ contains the internal forces, $\Psi$ the internal moments, $\Gamma$ the linear strains, $\Upsilon$ the curvature (or angular strains), and one has
\begin{align} \label{eq:def_VWPhiPsi}
v = \begin{bmatrix}
V \\ W
\end{bmatrix}, \quad z = \begin{bmatrix}
\Phi \\ \Psi
\end{bmatrix}, \quad s = \begin{bmatrix}
\Gamma \\ \Upsilon
\end{bmatrix}.
\end{align}
All six variables, which have values in $\mathbb{R}^3$, are body-attached variables in the above sense. Thus, one of the specificities of the state of the IGEB model is that it is made of body-attached -- rather than global -- variables.
For the global representations (i.e., expressed in the fixed basis) of $\Phi$ and $\Psi$, we employ the following notation
\begin{align} \label{eq:def_phi_psi}
\phi = \mathbf{R} \Phi, \quad \psi = \mathbf{R} \Psi.
\end{align}

\begin{remark}
Note that $\Upsilon$, called ``curvature'' here, is in more precise terms the curvature change relative to the curved reference configuration $\Omega_c$. The curvature change relative to $\Omega_s$ then takes the form $\Upsilon + \Upsilon_c$. 
\end{remark}

\section[The Geometrically Exact Beam model (GEB)]{The Geometrically Exact Beam model}
\label{sec:pres_GEB}

Let us now describe the GEB model which, as mentioned above, is written in terms of the position $\mathbf{p}$ of the centerline and of the rotation matrix $\mathbf{R}$ whose columns $\{\mathbf{b}^j\}_{j=1}^3$ give the orientation of the cross sections. The governing system \eqref{eq:GEB_pres} consists of six equations, and is of order two and quasilinear. 
For instance, a freely vibrating beam which is clamped at $x=0$ and free at the other end, is described by the system 
\begin{align}\label{eq:GEBtyp}
\begin{dcases}
\begingroup 
\setlength\arraycolsep{3pt}
\renewcommand*{\arraystretch}{0.9}
\begin{bmatrix}
\partial_t & \mathbf{0}\\
(\partial_t \widehat{\mathbf{p}}) & \partial_t
\end{bmatrix} \left[ \begin{bmatrix}
\mathbf{R} & \mathbf{0}\\ \mathbf{0} & \mathbf{R}
\end{bmatrix}
\mathbf{M} v \right] = \begin{bmatrix}
\partial_x & \mathbf{0} \\ (\partial_x \widehat{\mathbf{p}}) & \partial_x
\end{bmatrix} \begin{bmatrix}
\phi \\ \psi
\end{bmatrix}
\endgroup
&\text{in }(0, \ell)\times(0, T)\\
(\mathbf{p}, \mathbf{R})(0, t) = (f^\mathbf{p}, f^\mathbf{R}) &t \in (0, T)\\
\phi(\ell, t) = \mathbf{0}, \quad \psi(\ell, t) = \mathbf{0} &t \in (0, T)\\
(\mathbf{p}, \mathbf{R}, \partial_t \mathbf{p}, \mathbf{R}W)(x,0) = (\mathbf{p}^0, \mathbf{R}^0, \mathbf{p}^1, w^0)(x) &x \in (0, \ell),
\end{dcases}
\end{align}
where we recall that $v, z, \phi, \psi$ are functions of the unknowns defined in \eqref{eq:def_v_z} and \eqref{eq:def_VWPhiPsi}-\eqref{eq:def_phi_psi}.
Here, $(f^\mathbf{p}, f^\mathbf{R}) \in \mathbb{R}^3 \times \mathrm{SO}(3)$ are some given (constant) boundary data, and the initial data are $\mathbf{p}^0(x)$, $\mathbf{p}^1(x)$, $w^0(x) \in \mathbb{R}^3$ and $\mathbf{R}^0(x) \in \mathrm{SO}(3)$.

\medskip

\noindent \textbf{More general boundary conditions.}
More generally, in this work we consider two kinds of boundary conditions (at $x = 0$ or $x=\ell$): Neumann-type conditions
\begin{align}
\label{eq:GEB_genBC_Neu}
\nu(x) \left[ \begin{smallmatrix} 
\phi(x, t) \\ \psi(x, t)
\end{smallmatrix} \right] = f(t), \qquad t \in (0, T)
\end{align}
with ``external load'' $f(t) \in \mathbb{R}^6$, and Dirichlet-type conditions
\begin{align}
\label{eq:GEB_genBC_Diri}
(\mathbf{p}, \mathbf{R})(x, t) = (f^\mathbf{p}, f^\mathbf{R})(t), \qquad t \in (0, T).
\end{align}
with boundary data $(f^\mathbf{p}, f^\mathbf{R})(t) \in \mathbb{R}^3 \times  \mathrm{SO}(3)$.
Here, $\nu$ denotes the external unit normal vector -- i.e. $\nu(0) = -1$ and $\nu(\ell) = 1$.


\section[The Intrinsic Geometrically Exact Beam model (IGEB)]{The Intrinsic Geometrically Exact Beam model}
\label{sec:pres_IGEB}

We turn to the description of the IGEB model written in terms of the linear and angular velocities $v$ and internal forces and moments $z$ (expressed in the body-attached basis).
We have seen that the governing system is of the form \eqref{eq:IGEB_pres}. Thus, introducing the matrix $\mathbf{E}(x) \in \mathbb{R}^{6 \times 6}$ containing the information on curvature and twist at rest, and the matrix $Q^\mathcal{P}(x) \in \mathbb{S}_{++}^{12}$, defined by
\begin{align} \label{eq:def_boldE_calP}
\mathbf{E} = 
\begingroup 
\setlength\arraycolsep{3pt}
\renewcommand*{\arraystretch}{0.95}
\begin{bmatrix}
\widehat{\Upsilon}_c & \mathbf{0}\\
\widehat{e}_1 & \widehat{\Upsilon}_c
\end{bmatrix}
\endgroup
, \quad Q^\mathcal{P} = \mathrm{diag}(\mathbf{M}, \mathbf{C}),
\end{align}
the IGEB system with boundary conditions similar to \eqref{eq:GEBtyp} reads (freely vibrating beam)
\begin{align}
\label{eq:IGEBtyp}
\begin{dcases}
\partial_t y + A(x) \partial_x y + \overline{B}(x) y = \overline{g}(x, y) &\text{in }(0, \ell)\times(0, T)\\
v(0, t) = \mathbf{0} &\text{for }t \in (0, T)\\
z(\ell, t) = \mathbf{0} &\text{for }t \in (0, T)\\
y(x, 0) = y^0(x) &\text{for }x \in (0, \ell),
\end{dcases}
\end{align}
where the coefficients $A,\overline{B}$ and the source $\overline{g}$ depend on $\mathbf{M}, \mathbf{C}$ and $R$.
To obtain the governing system here, we have just left-multiplied \eqref{eq:IGEB_pres} by the inverse of $Q^\mathcal{P}$.
More precisely, $A, \overline{B} \colon [0, \ell] \rightarrow \mathbb{R}^{12 \times 12})$ are defined by
\begin{linenomath}
\begin{align} \label{eq:def_A_barB}
A = - (Q^\mathcal{P})^{-1} \begin{bmatrix}
\mathbf{0} & \mathbf{I}_6\\
\mathbf{I}_6 & \mathbf{0}
\end{bmatrix}, \quad \overline{B} = (Q^\mathcal{P})^{-1} \begin{bmatrix}
\mathbf{0} & - \mathbf{E}\\
\mathbf{E}^\intercal & \mathbf{0}
\end{bmatrix}
\end{align}
\end{linenomath}
The function $\overline{g} \colon [0, \ell]\times \mathbb{R}^{12} \rightarrow \mathbb{R}^{12}$ is defined by 
\begin{align} \label{eq:def_barg}
\overline{g}(x, u) = Q^\mathcal{P}(x)^{-1} \mathcal{G}(u) Q^\mathcal{P}(x) u
\end{align}
for all $x \in [0, \ell]$ and $u=(u_1^\intercal, u_2^\intercal, u_3^\intercal, u_4^\intercal)^\intercal \in \mathbb{R}^{12}$ with each $u_j \in \mathbb{R}^3$, where the map  $\overline{\mathcal{G}}$ is defined by
\begin{linenomath}
\begin{align} \label{eq:def_calG}
\mathcal{G}(u) = - 
\begin{bmatrix}
\begin{matrix}
\widehat{u}_2 & \mathbf{0}\\
\widehat{u}_1 & \widehat{u}_2
\end{matrix} & 
\begin{matrix}
 \mathbf{0} & \widehat{u}_3 \\
\widehat{u}_3 & \widehat{u}_4
\end{matrix}\\
\mathbf{0} & 
\begin{matrix}
\widehat{u}_2 & \widehat{u}_1\\
\mathbf{0} & \widehat{u}_2
\end{matrix}
\end{bmatrix}.
\end{align}
\end{linenomath}

Let us now give a more complete description of the coefficients in the governing system.

\medskip

\noindent \textbf{The matrix $A(x)$.}  
For all $x \in [0, \ell]$ the matrix $A(x)$ is hyperbolic, in the sense that it has real eigenvalues only, with twelve associated independent eigenvectors. This property relies on the fact that the mass and flexibility matrices have values in the set of positive definite symmetric matrices.
Indeed, let $\Theta \colon [0, \ell] \rightarrow \mathbb{S}_{++}^6$ be defined by
\begin{align} \label{eq:def_Theta}
\Theta = (\mathbf{C}^{\sfrac{1}{2}} \mathbf{M}\mathbf{C}^{\sfrac{1}{2}})^{-1}.
\end{align}
Since $\Theta(x)$ is symmetric as well as positive definite one may write it as 
\begin{align} \label{eq:Theta_diagonalization}
\Theta = U^\intercal D^2 U
\end{align}
for $U, D \colon [0, \ell] \rightarrow \mathbb{R}^{6 \times 6}$ such that $D(x)$ is a positive definite diagonal matrix containing the square roots of the eigenvalues of $\Theta(x)$ as diagonal entries, while $U(x)$ is unitary.

Relying on this observation, one may also see that the  $A(x)$, defined in \eqref{eq:def_A_barB}, has solely real eigenvalues $\{\lambda^k (x)\}_{k=1}^{12}$: six positive ones which are the square roots of the eigenvalues of $\Theta(x)$ (i.e. the diagonal entries of $D$), and six negative ones which are equal to the former but with a minus sign:
\begin{align} \label{eq:sign_eigval}
\lambda^k(x) <0 \ \text{ if } \ k\leq 6, \qquad \lambda^k(x) >0 \ \text{ if } \ k\geq 7.
\end{align}
More precisely, one has the following proposition.

\begin{proposition} \label{prop:A_hyperbolic}
For all $x \in [0, \ell]$, the matrix $A(x)$ may be diagonalized as follows. One has $A = L^{-1} \mathbf{D} L$ in $[0, \ell]$, where $\mathbf{D}, L \colon [0, \ell] \rightarrow \mathbb{R}^{12 \times 12}$ are defined by
\begin{linenomath}
\begin{equation} \label{eq:def_bfD_L}
\mathbf{D} = \mathrm{diag}(-D, D), \qquad L = \begin{bmatrix}
U \mathbf{C}^{-\sfrac{1}{2}} & DU \mathbf{C}^{\sfrac{1}{2}} \\
U \mathbf{C}^{-\sfrac{1}{2}} & - DU \mathbf{C}^{\sfrac{1}{2}}
\end{bmatrix},
\end{equation}
\end{linenomath}
where $U(x)\in\mathbb{R}^{6 \times 6}$ is unitary, $D(x) \in \mathbb{D}_{++}^6$ and both fulfill \eqref{eq:Theta_diagonalization}.
Moreover, the inverse $L^{-1} \colon [0, \ell] \rightarrow \mathbb{R}^{12 \times 12}$ is given by
\begin{linenomath}
\begin{align} \label{eq:inverseL}
L^{-1} = \frac{1}{2} \begin{bmatrix}
\mathbf{C}^{\sfrac{1}{2}} U^\intercal & \mathbf{C}^{\sfrac{1}{2}} U^\intercal \\
\mathbf{C}^{-\sfrac{1}{2}} U^\intercal D^{-1} & - \mathbf{C}^{-\sfrac{1}{2}} U^\intercal D^{-1}
\end{bmatrix}.
\end{align}
\end{linenomath}
\end{proposition}

A more detailed proof than in \ref{A:SICON}-\ref{A:MCRF}-\ref{A:JMPA} is provided in Appendix \ref{ap:proof_A_hyperbolic}.

\begin{remark} \label{rem:diag}
Let us make some remarks about this diagonalization.
\begin{enumerate}
\item In Proposition \ref{prop:A_hyperbolic}, we diagonalize $A$ without concern for the regularity of its eigenvalues and eigenvectors. 
However, latter on for the well-posedness study we will make additional assumptions on $\Theta$ in order to ensure that such $D$, $U$ regular enough with respect to $x$ exist (see also \ref{A:MCRF}-\ref{A:JMPA}).
\item \label{remItem:diag_moreExplicit}
If the mass and flexibility matrices are diagonal, one may always take $U$ as the identity matrix, and the formula for $L$ and its inverse are then more explicit.
\end{enumerate}
\end{remark}

\noindent \textbf{The linear lower-order term.}
The matrix $\overline{B}(x)$ is indefinite and, up to the best of our knowledge, may not be assumed arbitrarily small.
This implies not only that the linearized system \eqref{eq:IGEB_pres} is not homogeneous, but also that \eqref{eq:IGEB_pres} cannot be seen as the perturbation of a system of conservation laws.
Indeed, in some specific cases, such as a beam which is straight before deformation (hence, $\Upsilon_c = \mathbf{0}$) and with constant diagonal mass and flexibility matrices (as in \eqref{eq:def_M_C_iso}) we can easily compute the norm of $\overline{B}$. It suffices to look for the square root of the largest eigenvalue of $\overline{B}\,\overline{B}^\intercal$ (i.e., the largest singular value of $\overline{B}$) 
\begin{align*}
\overline{B}\,\overline{B}^\intercal
 = 
 \begin{bmatrix}
 \mathbf{M}^{-1}\mathbf{E}\mathbf{E}^\intercal \mathbf{M}^{-1} & \mathbf{0} \\
 \mathbf{0} & \mathbf{C}^{-1} \mathbf{E}^\intercal \mathbf{E} \mathbf{C}^{-1} 
 \end{bmatrix}.
\end{align*}
Since in that case
\begin{align*}
\mathbf{E}\mathbf{E}^\intercal = \begin{bmatrix}
\mathbf{0} & \mathbf{0}\\
\mathbf{0} & I_\dagger
\end{bmatrix}, \quad 
\mathbf{E}^\intercal \mathbf{E} = 
\begin{bmatrix}
I_\dagger & \mathbf{0}\\
\mathbf{0} & \mathbf{0}
\end{bmatrix}, \quad \text{with } I_\dagger = 
\begin{bmatrix}
0 & 0 & 0\\
0 & 1 & 0\\
0 & 0 & 1
\end{bmatrix},
\end{align*}
one has $\overline{B}\,\overline{B}^\intercal = \mathrm{diag}\left( \mathbf{0}_3, (\rho J)^{-2} I_\dagger, S_1^2 I_\dagger, \mathbf{0}_3 \right)$ with $J, S_1$ defined in \eqref{eq:def_J_S1_S2}. Thus, the eigenvalues of $\overline{B}\,\overline{B}^\intercal$ are: $0$ with multiplicity $8$, as well as $(\rho I_2)^{-2}$, $(\rho I_3)^{-2}$, $(ak_2G)^2$ and $(ak_3G)^2$, and therefore
\begin{align*}
\|\overline{B}\| = \max \left\{ \frac{1}{\rho I_2}, \frac{1}{\rho I_3}, ak_2G, ak_3G \right\}.
\end{align*}

\medskip

\noindent \textbf{The nonlinear term.}
%
%
One sees that $\overline{g}$ is a quadratic nonlinearity (in the sense that its components are quadratic forms on $\mathbb{R}^{12}$ with respect to the second argument). It has the same regularity as the mass and flexibility matrices $\mathbf{M}, \mathbf{C}$ with respect to its first argument, and is $C^\infty$ with respect to its second argument. Moreover, $\overline{g}(x, \cdot)$ is \emph{locally} Lipschitz in $\mathbb{R}^{12}$ for any $x \in [0, \ell]$, and $\overline{g}$ is \emph{locally} Lipschitz in $H^1(0, \ell; \mathbb{R}^{12})$, but no global Lipschitz property is available.

Finally, $\overline{g}$ is quasi-dissipative, in the sense of the following proposition. Let $\mathcal{X} := L^2(0, \ell; \mathbb{R}^{12})$ be endowed with the inner product
\begin{align*}
\langle \varphi , \psi \rangle_\mathcal{X} = \int_0^\ell \langle \varphi(x) \,, Q^\mathcal{P}(x) \psi(x) \rangle dx,
\end{align*}
and the induced norm $\|\cdot\|_\mathcal{X}$, with $Q^\mathcal{P}$ defined by \eqref{eq:def_boldE_calP}.

\begin{proposition} \label{prop:dissip_barg}
Let $\mathcal{C} = H^1(0, \ell; \mathbb{R}^{12})$, and $\mathcal{C}_\alpha = \{\varphi \in \mathcal{C} \colon \|\varphi\|_\mathcal{C} \leq \alpha \}$, for any $\alpha > 0$. Then, for any $\alpha > 0$ there exists $\omega_\alpha \in \mathbb{R}$ such that $\overline{g} - \omega_\alpha \mathbf{I}_{12}$ is dissipative on $\mathcal{C}_\alpha$ in the sense that 
\begin{align*} 
\left \langle \overline{g}(\cdot, \varphi) - \overline{g}(\cdot, \psi), \varphi - \psi \right \rangle_\mathcal{X} \leq \omega_\alpha \|\varphi - \psi\|_\mathcal{X}^2, \qquad  \text{for all }\varphi, \psi \in \mathcal{C}_\alpha.
\end{align*}
\end{proposition}

Indeed, due to \eqref{eq:dissip_barg_idInnerProd} in Lemma \ref{lem:dissip_barg} below, one obtains that
\begin{align*}
\left \langle \overline{g}(\cdot, \varphi) - \overline{g}(\cdot, \psi), \varphi - \psi \right \rangle_\mathcal{X} \leq \left(\sup_{x \in [0, \ell]}\|\mathcal{G}(\psi(x))Q^\mathcal{P}(x)\| \right) \|\varphi - \psi\|_\mathcal{X}^2
\end{align*}
holds for any $\varphi, \psi \in H^1(0, \ell; \mathbb{R}^{12})$, thereby yielding Proposition \ref{prop:dissip_barg}.
The proof of Lemma \ref{lem:dissip_barg} can be found in Appendix \ref{ap:proof_lem_dissip_barg}.

\begin{lemma} \label{lem:dissip_barg}
For any $x \in [0, \ell]$ and all $a, b \in \mathbb{R}^{12}$, the following two identities hold:
\begin{align}
&\label{eq:id_yr}
a^\intercal \mathcal{G}(b) = -b^\intercal \mathcal{G}(a),\\
&\label{eq:dissip_barg_idInnerProd}
\left \langle \overline{g}(x,a) - \overline{g}(x,b), Q^\mathcal{P}(x) (a-b) \right \rangle = - \left \langle a-b, \mathcal{G}(b)Q^\mathcal{P}(x) (a-b) \right \rangle.
\end{align}
\end{lemma}

Let us stress that $Q^\mathcal{P}(x)$ is nothing but the matrix characterizing the energy of the beam $\mathcal{E}^\mathcal{P} = \int_0^\ell \left\langle y \,, Q^\mathcal{P} y \right \rangle dx$ (also discussed in Section \ref{sec:stab_energy}; see \eqref{eq:1b_def_calEP}), which should be conserved when given appropriate boundary conditions. It is thus not unexpected that the term $\left\langle y \,, Q^\mathcal{P} \overline{g}(y) \right \rangle$, which also writes as $\left\langle y \,, \mathcal{G}(y)y \right \rangle$, is equal to zero; see for instance \cite[Sec. II, III]{Artola2021damping}.


\medskip

\noindent \textbf{More General boundary conditions.}
We will also consider boundary conditions of the form
\begin{align}\label{eq:IGEB_genBC}
v(x,t) = q^D(t), \quad \text{or} \quad z(x,t) = q^N(t), \qquad t \in (0, T),
\end{align}
with ``Dirichlet data'' $q^D(t)\in \mathbb{R}^6$ and ``Neumann data'' $q^N(t) \in \mathbb{R}^6$.

\section{Transformations between the GEB and IGEB models}
\label{sec:pres_transfo}

Finally, as mentionned in Chapter \ref{ch:intro}, one can see the GEB model \eqref{eq:GEB_pres} and IGEB model \eqref{eq:IGEB_pres} as being related by the nonlinear transformation $\mathcal{T}$ defined by
\begin{equation} \label{eq:transfo}
\mathcal{T} (\mathbf{p}, \mathbf{R}) = 
\begin{bmatrix} v \\ z \end{bmatrix}
\end{equation}
where $v, z$ are the functions of $\mathbf{p}$ and $\mathbf{R}$ defined in \eqref{eq:def_v_z}.
More precisely, in the following proposition, we see that there is an equivalence between the first six governing equations of the IGEB model and the governing system of the GEB model, while the last six equations of the former model are directly derived from the definition of the strains, sometimes referred to as ``compatibility conditions''.

\begin{proposition} \label{prop:transfoGov}
If $y = \mathcal{T}(\mathbf{p}, \mathbf{R})$ for some $(\mathbf{p}, \mathbf{R}) \in C^2([0, \ell]\times[0, T];\mathbb{R}^3 \times \mathrm{SO}(3))$, where $\mathcal{T}$ is defined by \eqref{eq:transfo}, then $y$ fulfills the last six equations of \eqref{eq:IGEB_pres}. Moreover, $y$ satisfies the first six equations of \eqref{eq:IGEB_pres} if and only if $(\mathbf{p}, \mathbf{R})$ satisfies \eqref{eq:GEB_pres}.
\end{proposition}

The proof is provided in Appendix \ref{ap:proof_prop_tranfoGov}. 

%

\medskip

\noindent \textbf{Initial and boundary data.} The initial conditions of both systems can also be linked by this transformation. One has
\begin{align} \label{eq:rel_inidata}
y^0 = \begin{bmatrix}
v^0 \\ z^0
\end{bmatrix}, \quad
v^0 = \begin{bmatrix}
(\mathbf{R}^0)^{\intercal} \mathbf{p}^1 \\
(\mathbf{R}^0 )^{\intercal} w^0
\end{bmatrix}
, \quad
z^0 = \mathbf{C}^{-1} 
\begin{bmatrix}
(\mathbf{R}^0)^{\intercal} \frac{\mathrm{d}}{\mathrm{d}x} \mathbf{p}^0 - e_1\\
\mathrm{vec}\left( (\mathbf{R}^0)^{\intercal} \frac{\mathrm{d}}{\mathrm{d}x}\mathbf{R}^0 \right) - \Upsilon_c
\end{bmatrix}.
\end{align}
As for the boundary conditions (see \eqref{eq:GEB_genBC_Neu} and \eqref{eq:GEB_genBC_Diri}), one has
\begin{align} \label{eq:def_qD}
q^D = \begin{bmatrix}
(f^\mathbf{R})^\intercal \frac{\mathrm{d}}{\mathrm{d}t}f^\mathbf{p}\\
(f^\mathbf{R})^\intercal \frac{\mathrm{d}}{\mathrm{d}t}f^\mathbf{R}
\end{bmatrix},
\end{align}
and 
\begin{align} \label{eq:def_qN}
q^N = \begin{bmatrix}
\mathbf{R}(x, \cdot)^\intercal &  \mathbf{0}\\
\mathbf{0} & \mathbf{R}(x,\cdot)^\intercal
\end{bmatrix} f.
\end{align}
Observe in \eqref{eq:def_qN} that if one wants boundary data $q^N$ for the IGEB model related to the data $f$ for the GEB model, but not dependent on the state $\mathbf{R}$ of the latter model, then the expression of $f$ in the body-attached basis should also be available.

\begin{remark}[``Equivalence'' between the GEB and IGEB systems]
\label{rem:equiv_geb_igeb}
In view of this, the systems \eqref{eq:GEBtyp} and \eqref{eq:IGEBtyp} are equivalent if their initial data additionally fulfill \eqref{eq:rel_inidata}. However, this is facilitated by our choice of clamped and free boundary conditions, and also by the fact that we consider freely vibrating beams. 
Concerning the latter point, if one wants to take into account the presence of external forces and moments applied to the beam -- e.g., gravity \cite[eq. (4)]{Artola2019mpc} or aerodynamic forces \cite[eq. (12)]{Palacios2011intrinsic} -- then it should be taken into account that these forces may be functions of $x$, but might also only be available in terms of $\mathbf{p}$ and $\mathbf{R}$.
\end{remark}

\begin{figure} \centering
  \begin{subfigure}{0.2\textwidth}
  \centering
    \includegraphics[height=3.25cm]{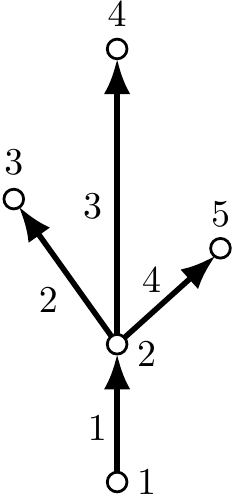}
    \caption{Star-shaped}
  \end{subfigure}%
  \begin{subfigure}{0.24\textwidth}
    \centering
    \includegraphics[height=3.25cm]{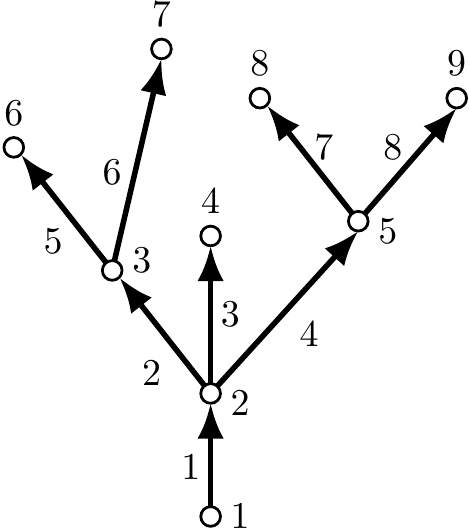}
    \caption{Tree-shaped} 
  \end{subfigure}%
  \begin{subfigure}{0.28\textwidth}
    \centering
    \includegraphics[height=3.25cm]{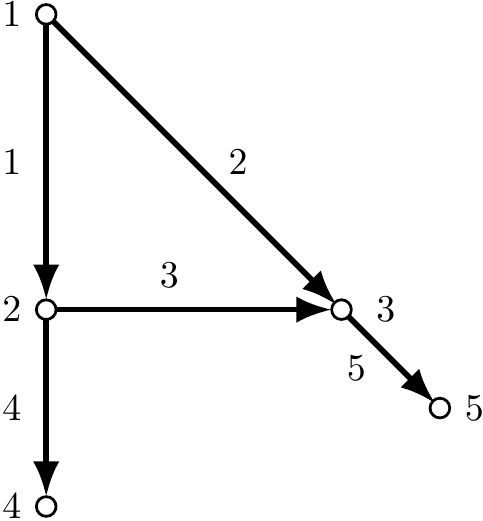}
    \caption{A-shaped} 
    \label{subfig:AshapedNetwork}
  \end{subfigure}%
\caption{Some oriented graphs representing beam networks.}
\label{fig:exple_networks}
\end{figure}

\section{Networks}
\label{sec:networks}

To represent a collection of $N$ beams, indexed by $i \in \mathcal{I} = \{1, \ldots, N\}$, attached in a certain manner to each other at their tips, we use an oriented graph containing $N$ edges, such as in Fig \ref{fig:exple_networks}. 

\medskip

\noindent \textbf{The edges.} Any edge $i$ is identified with the interval $[0, \ell_i]$, which is the spatial domain for the beam model in question (GEB or IGEB). 
The endpoints $x=0$ and $x=\ell_i$ are called \textit{initial point} and \textit{ending point} of this edge.
Therefore, just as for the beams, the \emph{edges} are indexed by $i \in \mathcal{I}$.
Then, the states $(\mathbf{p}_i, \mathbf{R}_i)$ and $y_i$, the given data characterizing the beams $\mathbf{M}_i$, $\mathbf{C}_i$ and $R_i$ (hence also $\Upsilon_{ci}$ and $\mathbf{E}_i$), and thus $A_i$ (and its eigenvalues $\{\lambda_i^k\}_{k=1}^{12}$, $D_i$ and eigenvectors $L_i$), $\overline{B}_i$ and $\overline{g}_i$, the initial data $\mathbf{p}_i^0$, $\mathbf{R}_i^0$, $\mathbf{p}_i^1$ and $w_i^0$, all have subindex $i$. 

\medskip

\noindent \textbf{The nodes.} The \textit{nodes} are indexed by\footnote{$\#$ denotes the set cardinality. We will also use the indexes $\mathcal{N} = \{0, 1, \ldots, \# \mathcal{N}\}$ in some cases.} $n \in \mathcal{N} = \{1, \ldots, \#\mathcal{N}\}$, which is then partitioned as $\mathcal{N} = \mathcal{N}_S \cup \mathcal{N}_M$, where $\mathcal{N}_S$ is the set of indexes of \emph{simple nodes}, while $\mathcal{N}_M$ is the set of indexes \emph{multiple nodes} -- where several edges meet. 
Then, the boundary data -- which are introduced latter on -- have subindex $n$.
The set of simple nodes is in addition partitioned as $\mathcal{N}_S = \mathcal{N}_S^D \cup \mathcal{N}_S^N$, where $\mathcal{N}_S^D$ contains the simple nodes with prescribed \emph{Dirichlet} boundary conditions (i.e., the centerline's position and the cross section's orientation in the case of the GEB model, or the velocities in the case of the IGEB model, are prescribed), while $\mathcal{N}_S^N$ contains the simple nodes at which \emph{Neumann} boundary conditions are enforced (i.e., the internal forces and moments are prescribed).

We interchangeably use the expressions ``node of index $n$'' (and ``edge of index $i$''), and ``node $n$'' (resp. ``edge $i$'') for short.

\medskip

\noindent \textbf{Connections and orientation of the edges.} For any $n\in\mathcal{N}$, we denote by $\mathcal{I}^n$ the set of indexes of edges incident with the node $n$, by $k_n = \# \mathcal{I}^n $ the \emph{degree} of the node $n$, and by $i^n$ the index\footnote{Defining $i^n$ as the \emph{smallest} element of $\mathcal{I}^n$, and not the \emph{largest} for example, is an arbitrary choice and is of no influence here.}
\begin{linenomath}
\begin{align} \label{eq:def_in}
i^n = \min_{i\in \mathcal{I}^n} i.
\end{align}
\end{linenomath}
Note that in the case of a simple node, $\mathcal{I}^n = \{i^n\}$.
The orientation of each beam is given by the variables $\mathbf{x}_i^n$ and $\tau_i^n$ defined as follows.
For any $i \in \mathcal{I}^n$, we denote by $\mathbf{x}_i^n$ the end of the interval $[0, \ell_i]$ which corresponds to the node $n$, while $\tau_i^n$ is the outward pointing normal at $\mathbf{x}_i^n$: 
\begin{linenomath}
\begin{align*}
\tau_i^n = 
\left\{ 
\begin{array}{l}
-1 \qquad \text{if } \mathbf{x}_i^n = 0,\\
+1 \qquad \text{if } \mathbf{x}_i^n = \ell_i.
\end{array}
\right.
\end{align*}
\end{linenomath}
As described in Fig. \ref{fig:exple_networks}, each edge $i$ is represented by an arrow and each node $n$ by a circle. The arrowhead is at the ending point $x=\ell_i$; see Fig. \ref{fig:xin}.

We also denote by $s_n \in \{0, \ldots, k_n\}$  (and by $k_n-s_n$) the number of beams ending (resp. starting) at the node $n$; see Fig. \ref{fig:NM_notation}. More precisely, we suppose that 
\begin{align} \label{eq:notation_I^n}
\mathcal{I}^n = \{i_1, \ldots, i_{k_n}\} \quad \text{with} \quad 
\left\{\begin{array}{l}
i_1 < i_2 < \ldots < i_{s_n} \\ 
i_{s_n+1} < i_{s_n+2} < \ldots < i_{k_n}
\end{array}\right.
\end{align}
and that $\tau_{i_\alpha}^n = -1$ for all $\alpha \in \{1, \ldots, s_n\}$, while $\tau_{i_\alpha}^n = +1$ for all $\alpha \in \{s_n+1, \ldots, k_n\}$. This is not to be confused with the notation $i^n$ introduced in \eqref{eq:def_in}.

\begin{figure}\centering
    
\begin{subfigure}{0.4\textwidth} \centering
    \vspace{0.4cm}
    \includegraphics[width=3cm]{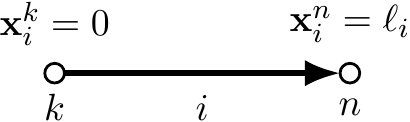}
    \vspace{0.3cm}
    \caption{Orientation of an edge $i$ starting and ending at the nodes $k$ and $n$, respectively.}
    \label{fig:xin}
\end{subfigure}%
\qquad
\begin{subfigure}{0.4\textwidth} \centering
    \includegraphics[scale=0.6]{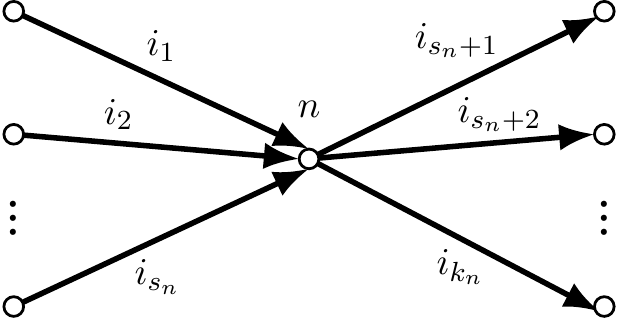}
    \caption{Notation for the elements of $\mathcal{I}^n$ at a multiple node $n$.}
    \label{fig:NM_notation}
\end{subfigure}%

\caption{Orientation of edges and notation at multiple nodes.}
\end{figure}

\medskip

\noindent \textbf{The transmission conditions.} 
For PDEs on networks, one must also specify the behaviour of the system at the intersection of several edges by imposing so-called \emph{transmission} of \emph{interface} conditions. 

Our stance is that at any multiple node $n \in \mathcal{N}_M$, the beams incident with this node remain attached to each other -- this concerns the position of their centerlines -- without changing the respective angles between them -- this concerns the orientation of their cross sections -- at all times. One then speak of a \emph{rigid joint}. In other words, the position of their centerlines must coincide, as specified by \eqref{eq:nGEBgen_cont_p} below.
As the orientation of the cross sections before deformation is specified by the (given) function $R_i$, the absence of relative motion between the incident beams is enforced by
\eqref{eq:nGEBgen_rigid} below, which states that the change of orientation $\mathbf{R}_iR_i^\intercal$ (from the configuration $\Omega_c$ to $\Omega_t$) is the same for all incident beams. See also \cite[Subsection 2.4]{strohm_dissert}.
Furthermore, at this multiple node $n$, we require the internal forces $\phi_i$ and moments $\psi_i$ exerted by incident beams $i\in\mathcal{I}_n$ to be balanced with the external load $f_n$ applied at this node, which reads as
\eqref{eq:nGEBgen_Kir} below 
and is also called the \emph{Kirchhoff} condition. 

Thus, in general we obtain a system with unknown state $(\mathbf{p}_i, \mathbf{R}_i)_{i \in \mathcal{I}}$ of the form
\begin{linenomath}
\begin{subnumcases}{\label{eq:nGEBgen}}
\nonumber
\begingroup 
\setlength\arraycolsep{2pt}
\renewcommand*{\arraystretch}{0.9}
\begin{bmatrix}
\partial_t & \mathbf{0}\\
(\partial_t \widehat{\mathbf{p}}_i) & \partial_t
\end{bmatrix} \left[ \begin{bmatrix}
\mathbf{R}_i & \mathbf{0}\\ \mathbf{0} & \mathbf{R}_i
\end{bmatrix}
\mathbf{M}_i v_i \right] 
\endgroup
\\
\label{eq:nGEBgen_gov}
\begingroup 
\setlength\arraycolsep{2pt}
\renewcommand*{\arraystretch}{0.9}
\hspace{2.25cm} = \begin{bmatrix}
\partial_x & \mathbf{0} \\ (\partial_x \widehat{\mathbf{p}}_i) & \partial_x
\end{bmatrix} \begin{bmatrix}
\phi_i \\ \psi_i
\end{bmatrix}
\endgroup
&$\text{in } (0, \ell_i)\times(0, T), \, i \in \mathcal{I}$\\
\label{eq:nGEBgen_cont_p}
\mathbf{p}_i(\mathbf{x}_i^n, t) = \mathbf{p}_{i^n}(\mathbf{x}_{i^n}^n, t) &$t \in (0, T), \, i \in \mathcal{I}^n, \, n \in \mathcal{N}_M$\\
\label{eq:nGEBgen_rigid}
(\mathbf{R}_i R_{i}^\intercal)(\mathbf{x}_i^n, t) = (\mathbf{R}_{i^n} R_{i^n}^\intercal)(\mathbf{x}_{i^n}^n, t) &$t \in (0, T), \, i \in \mathcal{I}^n, \, n \in \mathcal{N}_M$\\
\label{eq:nGEBgen_Kir}
{\textstyle \sum_{i\in\mathcal{I}^n}} \tau_i^n \left[ \begin{smallmatrix} 
\phi_i \\ \psi_i
\end{smallmatrix} \right] (\mathbf{x}_i^n, t) = f_n (t)
&$t \in (0, T), \, n \in \mathcal{N}_M$\\
\label{eq:nGEBgen_nSN}
\tau_{i^n}^n \left[ \begin{smallmatrix} 
\phi_{i^n} \\ \psi_{i^n}
\end{smallmatrix} \right] (\mathbf{x}_{i^n}^n, t) = f_n (t) &$t \in (0, T), \, n \in \mathcal{N}_S^N$\\
\label{eq:nGEBgen_nSD}
(\mathbf{p}_{i^n}, \mathbf{R}_{i^n})(\mathbf{x}_{i^n}^n, t) = (f_n^\mathbf{p}, f_n^\mathbf{R})(t) &$t \in (0, T), \, n \in \mathcal{N}_S^D$\\
\label{eq:nGEBgen_IC0}
(\mathbf{p}_i, \mathbf{R}_i)(x, 0) = (\mathbf{p}_i^0, \mathbf{R}_i^0)(x) &$x \in (0, \ell_i), \, i \in \mathcal{I}$\\
\label{eq:nGEBgen_IC1}
(\partial_t \mathbf{p}_i, \mathbf{R}_i W_i)(x, 0) = (\mathbf{p}_i^1, w_i^0)(x) &$x \in (0, \ell_i), \, i \in \mathcal{I}$,
\end{subnumcases}
\end{linenomath}
with external load $f_n(t) \in \mathbb{R}^6$, Dirichlet data $(f_n^\mathbf{p}$, $f_n^\mathbf{R})(t) \in \mathbb{R}^3 \times \mathrm{SO}(3)$ and initial data $\mathbf{p}_i^0(x)$, $\mathbf{p}_i^1(x)$, $w_i^0(x) \in \mathbb{R}^3$ and $\mathbf{R}_i^0(x) \in \mathrm{SO}(3)$.
We will see in Chapter \ref{ch:wellposedness} that corresponding transmission conditions for the IGEB model may be derived from \eqref{eq:nGEBgen_cont_p}-\eqref{eq:nGEBgen_rigid}-\eqref{eq:nGEBgen_Kir}. They will take the form of a continuity of velocities and Kirchhoff conditions, involving $v_i, z_i$ and $R_i$. This will permit us to study the beam network from the point of view of the IGEB model.

\begin{subappendices}

\section{Correspondence with the notation of \cite{hodges2003geometrically}}
\label{ap:hodges}

As mentioned in Section \ref{sec:motivation} the intrinsic beam model at the center of our work (IGEB) originates from \cite{hodges2003geometrically}. As our notation differs substantially from this reference, we provide here a short section to clarify the correspondences between both notations.

Both in \cite{hodges2003geometrically} and in our work, the reference line of the undeformed beam is parametrized by its arclength, though the spatial variable is denoted $x_1$ in the former (rather than $x$). On the description of the beams' position before and after deformation, we have the following correspondences.
\begin{center}
\begin{tabular}{lcccc}\toprule
 & \multicolumn{2}{c}{in $\Omega_c$} & \multicolumn{2}{c}{in $\Omega_t$}
\\\cmidrule(lr){2-3}\cmidrule(lr){4-5}
       & in \cite{hodges2003geometrically} & here & in \cite{hodges2003geometrically} & here  \\\midrule
reference line position 
& $\mathbf{r}$ 
& $p$
& $\mathbf{R}$
& $\mathbf{p}$\\
reference line displacement 
& 
& 
& $\mathbf{u}$
& $\mathbf{p}-p$\\
body-attached basis 
& $\{b_j\}_{j=1}^3$
& $\{b^j\}_{j=1}^3$
& $\{B_j\}_{j=1}^3$
& $\{\mathbf{b}^j\}_{j=1}^3$\\
\bottomrule
\end{tabular}
\end{center}
Note that variables describing $\Omega_c$ depend on the spatial variable only, while variables describing $\Omega_t$ also depend on time.
We recall that $p, \mathbf{p}, \mathbf{R}, \{b^j\}_{j=1}^3$ and $\{\mathbf{b}^j\}_{j=1}^3$ are defined in Subsection \ref{sec:decription_beam}. 
To continue, for the undeformed beam, in \cite{hodges2003geometrically} the variable $k$ denotes the curvature expressed in the body-attached basis of $\Omega_c$, and it corresponds to $\Upsilon_c$ in our notation. 
For the deformed beam, we have the following correspondences.
\begin{center}
\begin{tabular}{lcccc}\toprule
in $\Omega_t$ & \multicolumn{2}{c}{body-attached} & \multicolumn{2}{c}{global}
\\\cmidrule(lr){2-3}\cmidrule(lr){4-5}
       & in \cite{hodges2003geometrically} & here & in \cite{hodges2003geometrically} & here  \\\midrule
linear velocity 
& $V$ & $V$ & $\mathbf{V}$ & $\partial_t \mathbf{p}$\\
angular velocity 
& $\Omega$ & $W$ & $\mathbf{\Omega}$ & $\mathrm{vec}\left(\left(\partial_t \mathbf{R}\right) \mathbf{R} \right)$\\
linear momentum 
& $P$ & $P$ & $\mathbf{P}$ & \\
angular momentum 
& $H$ & $H$ & $\mathbf{H}$ & \\
linear strains 
& $\gamma$ & $\Gamma$ & & \\
curvature (relative to $\Omega_s$)
& $K$ & $\Upsilon + \Upsilon_c$ & $\mathbf{K}$ & $\mathrm{vec}\left(\left(\partial_x \mathbf{R}\right) \mathbf{R} \right)$\\
curvature (relative to $\Omega_c$)
& $\kappa$ & $\Upsilon$ & & \\
internal forces 
& $F$ & $\Phi$ & $\mathbf{F}$ & $\phi$ \\
internal moments 
& $M$ & $\Psi$ & $\mathbf{M}$ & $\psi$ \\
external forces 
& $f$ & $\overline{\Phi}$ & $\mathbf{f}$ & $\overline{\phi}$ \\
external moments 
& $m$ & $\overline{\Psi}$ & $\mathbf{m}$ & $\overline{\psi}$ \\\bottomrule
\end{tabular}
\end{center} 
Let us recall that $P, H$ and $\Gamma, \Upsilon, \Phi, \Psi, V, W$ are defined by \eqref{eq:def_momenta} and \eqref{eq:def_VWPhiPsi} respectively together with \eqref{eq:def_v_z} here, while $\phi, \psi$ are defined by \eqref{eq:def_phi_psi}, and $\overline{\Phi}, \overline{\Psi}, \overline{\phi}, \overline{\psi}$ are introduced in Section \ref{sec:motivation}.

A final remark is that the author in \cite{hodges2003geometrically} makes use of the so-called ``direction cosine matrix'' $C$ rather than the matrices $R$ and $\mathbf{R}$. The former is also a rotation matrix, and corresponds in our case to the composition $\mathbf{R}^\intercal R$.

\section{Proof of Proposition \ref{prop:A_hyperbolic}}
\label{ap:proof_A_hyperbolic}

\begin{proof}[Proof of Proposition \ref{prop:A_hyperbolic}]
We drop the argument $x$ for clarity. Let us first explain why we introduced the matrix $\Theta$ defined in \eqref{eq:def_Theta}.
To study the eigenvalues of $A$, one may study the zeros of $\mathrm{det}(\lambda \mathbf{I}_{12} - A)$. Some computations yield that it is equal to $\mathrm{det}(\lambda^2 \mathbf{I}_6 - (\mathbf{C} \mathbf{M})^{-1})$, which reduces the problem to showing that $(\mathbf{C} \mathbf{M})^{-1}$ has are real, positive eigenvalues only. The latter matrix also writes as $(\mathbf{C} \mathbf{M})^{-1} =  \mathbf{C}^{-\sfrac{1}{2}} \Theta \mathbf{C}^{\sfrac{1}{2}}$, with $\Theta$ defined by \eqref{eq:def_Theta}, implying that it has the same eigenvalues as $\Theta$ since $\mathbf{C}$ is invertible -- moreover all are real and positive as $\Theta$ is symmetric positive definite. 
Hence, $(\mathbf{C} \mathbf{M})^{-1}$ is possibly not positive definite, but has necessarily real, positive eigenvalues.

Let us now go through the decomposition of $A$.
First, we define $L_1 = \mathrm{diag}(U, U)$ and $L_2 = \mathrm{diag}(\mathbf{C}^{-1}, \mathbf{C})^{\sfrac{1}{2}}$, that both have values in $\mathbb{R}^{12 \times 12}$. 
By the definition of $A$ and $\Theta$, in \eqref{eq:def_A_barB}-\eqref{eq:def_Theta}, one has
\begin{align*}
A = -L_2^{-1} \begin{bmatrix}
\mathbf{0} & \Theta\\
\mathbf{I}_6 & \mathbf{0}
\end{bmatrix}L_2.
\end{align*}
Then, relying on the decomposition \eqref{eq:Theta_diagonalization} of $\Theta$, we deduce that
\begin{align*}
A &= - \big(L_1 L_2 \big)^{-1} \begin{bmatrix}
\mathbf{0} & D^2\\
\mathbf{I}_6 & \mathbf{0}
\end{bmatrix} \big(L_1 L_2 \big)
\end{align*}
Introducing the matrix $L_{0}(x)$, and its inverse, defined by
\begin{equation*}
L_{0} = \begin{bmatrix}
\mathbf{I}_6 & D \\
\mathbf{I}_6 & - D
\end{bmatrix}, \qquad 
L_{0}^{-1} = \frac{1}{2} \begin{bmatrix}
\mathbf{I}_6 & \mathbf{I}_6\\
D^{-1} & - D^{-1}
\end{bmatrix},
\end{equation*}
we obtain the expression $A =  \big(L_{0}  L_1 L_2 \big)^{-1} \mathbf{D} \big(L_{0}  L_1 L_2 \big)$ for $A$.
The matrix $L$ defined in \eqref{eq:def_bfD_L} corresponds to $L = L_{0} L_1 L_2$. Its inverse $L^{-1} = L_2^{-1} L_1^{-1} L_{0}^{-1}$ takes the form \eqref{eq:inverseL}.
\end{proof}

\section{Proof of Lemma \ref{lem:dissip_barg}}
\label{ap:proof_lem_dissip_barg}


\begin{proof}[Proof of Lemma \ref{lem:dissip_barg}]
Let $x \in [0, \ell]$, and $a, b \in \mathbb{R}^{12}$ be fixed.
Below, we make use multiple times of the following identities involving the cross product
\begin{align}
\widehat{\zeta }^\intercal = - \widehat{\zeta} \qquad &\text{for any }\zeta \in \mathbb{R}^3 \label{eq:crossProd_skew} \\
\widehat{\zeta} \, z = - \widehat{z} \, \zeta \qquad &\text{for any }\zeta, z \in \mathbb{R}^3. \label{eq:crossProd_interchange}
\end{align}
Let us first consider the quantity $r^\intercal \mathcal{G}(y)$ for some $r, y \in \mathbb{R}^{12}$. Using the definition of $\mathcal{G}$, and the property  \eqref{eq:crossProd_skew} of the cross product, one obtains
\begin{align*}
(r^\intercal \mathcal{G}(y))^\intercal
&= \mathcal{G}(y)^\intercal r 
= - \begin{bmatrix}
\begin{matrix}
\widehat{y}_2 & \widehat{y}_1\\
\mathbf{0} & \widehat{y}_2
\end{matrix} &
\mathbf{0}\\
\begin{matrix}
\mathbf{0} & \widehat{y}_3 \\
\widehat{y}_3 & \widehat{y}_4
\end{matrix}
&
\begin{matrix}
\widehat{y}_2 & \mathbf{0} \\
\widehat{y}_1 & \widehat{y}_2
\end{matrix}
\end{bmatrix} \begin{bmatrix}
r_1 \\ r_2 \\ r_3 \\ r_4
\end{bmatrix}.
\end{align*}
Then, \eqref{eq:crossProd_interchange}, followed by \eqref{eq:crossProd_skew}, yield that 
\begin{align*}
(r^\intercal \mathcal{G}(y))^\intercal
= \begin{bmatrix}
\begin{matrix}
\widehat{r}_2 & \widehat{r}_1\\
\mathbf{0} & \widehat{r}_2
\end{matrix}
&
\mathbf{0}\\
\begin{matrix}
\mathbf{0} & \widehat{r}_3\\
\widehat{r}_3 & \widehat{r}_4 
\end{matrix}
&
\begin{matrix}
\widehat{r}_2 & \mathbf{0} \\
\widehat{r}_1 & \widehat{r}_2
\end{matrix}
\end{bmatrix} \begin{bmatrix}
y_1 \\ y_2 \\ y_3 \\ y_4
\end{bmatrix} = -(y^\intercal \mathcal{G}(r))^\intercal.
\end{align*}
Hence, we have deduced that \eqref{eq:id_yr} holds, and in particular, 
\begin{align}
y^\intercal \mathcal{G}(y) = \mathbf{0}_{12, 1} \qquad &\text{for any }y \in \mathbb{R}^{12}. \label{eq:id_y}
\end{align}

We turn to the second identity. By the definition of $\overline{g}$, one may rewrite
\begin{align*}
\left \langle \overline{g}(x,a) - \overline{g}(x,b), Q^\mathcal{P}(x) (a-b) \right \rangle 
&= - \left \langle a-b, \mathcal{G}(a)Q^\mathcal{P}(x)a - \mathcal{G}(b)Q^\mathcal{P}(x)b \right \rangle\\
&= - \left \langle a, \mathcal{G}(a)Q^\mathcal{P}(x)a \right \rangle + \left \langle a, \mathcal{G}(b)Q^\mathcal{P}(x)b \right \rangle \\
&\quad \, + \left \langle b, \mathcal{G}(a)Q^\mathcal{P}(x)a \right \rangle - \left \langle b, \mathcal{G}(b)Q^\mathcal{P}(x)b \right \rangle,
\end{align*}
By \eqref{eq:id_y}, $a^\intercal \mathcal{G}(a)$ and $b^\intercal \mathcal{G}(b)$ are both zero, and consequently,
\begin{align*}
\left \langle \overline{g}(x,a) - \overline{g}(x,b), Q^\mathcal{P}(x) (a-b) \right \rangle 
&= \left \langle a, \mathcal{G}(b)Q^\mathcal{P}(x)b \right \rangle + \left \langle b, \mathcal{G}(a)Q^\mathcal{P}(x)a \right \rangle.
\end{align*}
By \eqref{eq:id_yr}, the above equation also writes as
\begin{align*}
\left \langle \overline{g}(x,a) - \overline{g}(x,b), Q^\mathcal{P}(x) (a-b) \right \rangle 
&= \left \langle a, \mathcal{G}(b)Q^\mathcal{P}(x)b \right \rangle - \left \langle a, \mathcal{G}(b)Q^\mathcal{P}(x)a \right \rangle\\
&= - \left \langle a, \mathcal{G}(b)Q^\mathcal{P}(x)(a-b) \right \rangle.
\end{align*}
Finally, since $b^\intercal \mathcal{G}(b)$ is zero by \eqref{eq:id_y}, we obtain \eqref{eq:crossProd_interchange}, concluding the proof.
\end{proof}

\section{Proof of Proposition \ref{prop:transfoGov}}
\label{ap:proof_prop_tranfoGov}

\begin{proof}[Proof of Proposition \ref{prop:transfoGov}]
Here, we use the notation $y = (v^\intercal, z^\intercal)^\intercal$ as well as $v = (v_1^\intercal, v_2^\intercal)^\intercal$ and $z = (z_1^\intercal, z_2^\intercal)^\intercal$ for $v_1, v_2, z_1, z_2$ having values in $\mathbb{R}^3$.
We start from the first six equations of the IGEB governing system, which read
\begin{align*}
\mathbf{M} \partial_t v - \partial_x z - \mathbf{E} z + \begin{bmatrix}
\widehat{v}_2 & \mathbf{0}\\
\widehat{v}_1 & \widehat{v}_2
\end{bmatrix} \mathbf{M}v  + \begin{bmatrix}
 \mathbf{0} & \widehat{z}_1 \\
\widehat{z}_1 & \widehat{z}_2
\end{bmatrix} \mathbf{C}z = \begin{bmatrix}
\overline{\Phi} \\ \overline{\Psi}
\end{bmatrix}
\end{align*}
We will make use of the fact that for any $f,F \colon [0, \ell]\times[0, T] \rightarrow\mathbb{R}^3$ related by $f = \mathbf{R} F$, one has the following identities
\begin{align}
\label{eq:identity_dx_F_f}
\partial_x F &= \mathbf{R}^\intercal \partial_t f - (\widehat{\Upsilon} + \widehat{\Upsilon}_c) F, \\
\label{eq:identity_dt_F_f}
\partial_t F &= \mathbf{R}^\intercal \partial_t f - \widehat{v}_2 F,
\end{align}
that arise from the definitions of $\Upsilon, v_2$ in \eqref{eq:def_v_z},\eqref{eq:def_VWPhiPsi} and properties of the cross product.

The first identity \eqref{eq:identity_dx_F_f} applied to $\phi = \mathbf{R} z_1$ and $\psi = \mathbf{R} z_2$ (note that $z_1, z_2$ correspond to $\Phi,\Psi$), together with the definition of $\mathbf{E}$ in \eqref{eq:def_boldE_calP} and $\Gamma, \Upsilon$ in \eqref{eq:def_v_z},\eqref{eq:def_VWPhiPsi}, yields the equivalent system
\begin{align*}
\mathbf{M}\partial_t v - \begin{bmatrix}
\mathbf{R}^\intercal & \mathbf{0}\\ \mathbf{0} & \mathbf{R}^\intercal
\end{bmatrix} \partial_x \begin{bmatrix}
\phi \\ \psi
\end{bmatrix} + 
\begingroup 
\setlength\arraycolsep{3pt}
\renewcommand*{\arraystretch}{0.95}
\begin{bmatrix}
(\widehat{\Upsilon} + \widehat{\Upsilon}_c) z_1 \\
(\widehat{\Upsilon} + \widehat{\Upsilon}_c) z_2
\end{bmatrix}
\endgroup
 -  
\begingroup 
\setlength\arraycolsep{3pt}
\renewcommand*{\arraystretch}{0.95} 
 \begin{bmatrix}
\widehat{\Upsilon}_c & 0 \\ \widehat{e}_1 & \widehat{\Upsilon}_c
\end{bmatrix}
\endgroup
\begin{bmatrix}
z_1 \\ z_2
\end{bmatrix}& \\
+ \begin{bmatrix}
\widehat{v}_2 & \mathbf{0}\\
\widehat{v}_1 & \widehat{v}_2
\end{bmatrix} \mathbf{M}v  + \begin{bmatrix}
 \mathbf{0} & \widehat{z}_1 \\
\widehat{z}_1 & \widehat{z}_2
\end{bmatrix} \begin{bmatrix}
\Gamma \\ \Upsilon
\end{bmatrix}& = \begin{bmatrix}
\overline{\Phi} \\ \overline{\Psi}
\end{bmatrix} 
\end{align*}
Note that the terms involving $\Upsilon$ and $\Upsilon_c$ cancel each other.
Moreover, since $\Gamma + e_1 = \mathbf{R}^\intercal \partial_x \mathbf{p}$, one has that $(\widehat{\Gamma} + \widehat{e}_1)z_1 = \mathbf{R}^\intercal (\partial_x \widehat{\mathbf{p}}) \phi$. Hence, we obtain the equivalent system
\begin{align*}
\mathbf{M}\partial_t v - \begin{bmatrix}
\mathbf{R}^\intercal & \mathbf{0}\\ \mathbf{0} & \mathbf{R}^\intercal
\end{bmatrix} \partial_x \begin{bmatrix}
\phi \\ \psi
\end{bmatrix} + \begin{bmatrix}
\widehat{v}_2 & \mathbf{0}\\
\widehat{v}_1 & \widehat{v}_2
\end{bmatrix} \mathbf{M}v  - \begin{bmatrix}
\mathbf{0} \\ \mathbf{R}^\intercal \partial_x \widehat{\mathbf{p}} \phi
\end{bmatrix} = \begin{bmatrix}
\overline{\Phi} \\ \overline{\Psi}
\end{bmatrix} 
\end{align*}
Since the second identity \eqref{eq:identity_dt_F_f} applied to 
\begin{align*}
\begin{bmatrix}
p \\ m
\end{bmatrix}
:= \begin{bmatrix}
\mathbf{R} & \mathbf{0}\\ \mathbf{0} & \mathbf{R}
\end{bmatrix}
\mathbf{M} v
\end{align*}
yields that
\begin{align*}
\partial_t \left(\mathbf{M}v\right) = \begin{bmatrix}
\mathbf{R}^\intercal & \mathbf{0}\\ \mathbf{0} & \mathbf{R}^\intercal
\end{bmatrix} \partial_t \begin{bmatrix}
p \\ m
\end{bmatrix} - \begin{bmatrix}
\widehat{v}_2 & \mathbf{0}\\
\mathbf{0} & \widehat{v}_2
\end{bmatrix} \mathbf{M}v
\end{align*}
Injecting this expression in the system, and left-multiplying by $\mathbf{R}$, we obtain
\begin{align*}
\begingroup 
\setlength\arraycolsep{3pt}
\renewcommand*{\arraystretch}{0.9}
\partial_t \left[ \begin{bmatrix}
\mathbf{R} & \mathbf{0}\\ \mathbf{0} & \mathbf{R}
\end{bmatrix}
\mathbf{M} v \right] + \begin{bmatrix}
\mathbf{0} & \mathbf{0}\\
(\partial_t \widehat{\mathbf{p}}) \mathbf{R} & \mathbf{0}
\end{bmatrix} \mathbf{M}v =  \partial_x \begin{bmatrix}
\phi \\ \psi
\end{bmatrix} + \begin{bmatrix}
\mathbf{0} & \mathbf{0} \\ \partial_x \widehat{\mathbf{p}} & \mathbf{0}
\end{bmatrix}
\begin{bmatrix}
\phi \\ \psi
\end{bmatrix} + \begin{bmatrix}
\overline{\phi} \\ \overline{\psi}
\end{bmatrix}, 
\endgroup
\end{align*}
where we used the definition of $v_1$ to get that $\mathbf{R}\widehat{v}_1 = (\partial_t \widehat{\mathbf{p}}) \mathbf{R}$.
 
To finish, we will see that the definition of $z$ in terms of $\mathbf{p}$ and $\mathbf{R}$ (see \eqref{eq:def_v_z}) yields the last six governing equation of the IGEB system.
Indeed,
\begin{align*}
\partial_t \Gamma &= (\partial_t \mathbf{R})^{\intercal} \partial_x \mathbf{p} + \mathbf{R}^{\intercal} \partial_{x} \partial_t \mathbf{p}\\
&= -\widehat{v}_2(\Gamma  + e_1) + \partial_x(\mathbf{R}^{\intercal} \partial_t \mathbf{p}) - (\partial_x \mathbf{R})^{\intercal}\partial_t \mathbf{p}\\
&= -\widehat{v}_2(\Gamma  + e_1) + \partial_x v_1 + (\widehat{\Upsilon} + \widehat{\Upsilon}_c)v_1.
\end{align*}
and the latter satisfies
\begin{align*}
\widehat{\partial_t \Upsilon} &= (\partial_t \mathbf{R})^{\intercal} \partial_x \mathbf{R} + \mathbf{R}^{\intercal} \partial_{x}\partial_t \mathbf{R}\\
&= -\widehat{v}_2 (\widehat{\Upsilon} + \widehat{\Upsilon}_c) + \partial_x (\mathbf{R}^{\intercal} \partial_t \mathbf{R}) - (\partial_x \mathbf{R})^{\intercal} \partial_t \mathbf{R}\\
&=- \widehat{v}_2(\widehat{\Upsilon} + \widehat{\Upsilon}_c) + \widehat{\partial_x v_2} + (\widehat{\Upsilon} + \widehat{\Upsilon}_c)v_2.
\end{align*}
The identity $\widehat{(\widehat{a}b)} = \widehat{a}\widehat{b} - \widehat{b} \widehat{a}$ for the cross product implies that $\partial_t \Upsilon = \partial_x v_2 + \widehat{\Upsilon+\Upsilon_c}v_2$. Thus, we have obtained
\begin{align*}
\partial_t s = \partial_x v + 
\begingroup 
\setlength\arraycolsep{3pt}
\renewcommand*{\arraystretch}{0.9}
\begin{bmatrix}
\widehat{\Upsilon}_c & \widehat{e}_1\\
\mathbf{0} & \widehat{\Upsilon}_c
\end{bmatrix}
\endgroup
v - \begin{bmatrix}
\widehat{v}_2 & \widehat{v}_1\\
\mathbf{0} & \widehat{v}_2
\end{bmatrix}
s,
\end{align*}
and the definition of the strains \eqref{eq:def_v_z} indeed yields the last six equations of the IGEB system.
\end{proof}

\end{subappendices}

\chapter{Modelling and well-posedness}
\label{ch:wellposedness}

As alluded to in Chapter \ref{ch:intro}, our study of well-posedness will be at first focused on the IGEB model, in Section \ref{sec:wellp_igeb}. Then, we will see in Section \ref{sec:invert_transfo} how the transformation $\mathcal{T}$ introduced in Section \ref{sec:pres_transfo} may be used to deduce well-posedness for the GEB model.

\section{Main contributions and related works}

This chapter is essential for the following Chapters \ref{ch:stab} and \ref{ch:control}, not only because it treats the well-posedness problem for the systems of interest here, but also because it exposes a link between the GEB and IGEB models that will permit us to derive results for the former model from the study of the latter.
For this reason, the pace of the remainder of this exposition is set as follows.
In all subsequent studies -- well-posedness, but also stabilization and control of nodal profiles --, we first work with the intrinsic model, which has the advantage of being a first-order hyperbolic system as observed in Section \ref{sec:pres_IGEB}, and then deduce corresponding results for the matching GEB model as corollaries.

The well-posedness of the IGEB model for a single beam comes directly from known results in the literature. Of particular interest for us are local in time solutions in $C^0_t H^1_x$ for $H^1$ initial data (resp. $C^0_t H^2_x $ for $H^2$ initial data and more regular coefficients) -- such solutions are consequently $C^0_{x,t}$ (resp. $C^1_{x,t}$) --, as well as semi-global in time solutions in $C_{x,t}^1$. Semi-global in time means that for any time $T>0$, and for small enough initial and boundary data, a unique solution exists at least until time $T>0$. This solution concept originates from \cite{LiJin2001_semiglob}. The former will be of use for the stabilization study while the latter is needed for boundary control of nodal profiles, as we will see in the next two chapters. 

For the networks, well-posedness may also be derived from results on a single hyperbolic system, provided that the transmission conditions have certain properties. Therefore, the first part of our work is to derive transmission conditions for the IGEB model, compatible with the physical ones introduced in Section \ref{sec:networks}: the continuity of the centerlines \eqref{eq:nGEBgen_cont_p}, rigidity of the joints \eqref{eq:nGEBgen_rigid} and balance of forces and moments \eqref{eq:nGEBgen_Kir}.
Then, we will prove that these transmission conditions indeed are such that well-posedness is guaranteed.
The form of transmission conditions is also an essential aspect in the proof of nodal profile controllability of hyperbolic systems on networks, the object of Chapter \ref{ch:control}.

To summarise, we present the following results in this chapter:
\begin{itemize}
\item We derive transmission conditions for the IGEB network in Section \ref{sec:modelling_transmi}.
\item In Theorem \ref{th:semi_global_existence}, for general networks with nonhomogeneous Dirichlet and Neumann conditions at the nodes, we prove the existence and uniqueness of semi-global in times solutions in $C_{x, t}^1$.
\item In Theorem \ref{th:local_existence}, for (tree-shaped) networks with velocity feedback (or free, clamped) conditions at the nodes, we show the existence and uniqueness of local in time solutions in $C_t^0H_x^k$ for $k \in \{1, 2\}$.
\item For a single beam clamped at one end and free or with velocity feedback control at the other end, in Theorem \ref{th:invert_transfo_fb} we prove that the transformation $\mathcal{T}$ is bijective on some specific spaces. Then, in this setting, in Theorem \ref{thm:inv1b_igeb2geb} we prove that the existence and uniqueness of a classical solution to the IGEB model implies that of a classical solution to the GEB model.
\item For a general network with nonhomogeneous Dirichlet and Neumann conditions at the nodes, we extend both results of the previous item in Lemma \ref{lem:invert_transfo} and Theorem \ref{thm:solGEB}, respectively.
\item Then, in Corollary \ref{coro:wellposedGEB}, we give the existence and uniqueness result for the GEB model corresponding to Theorem \ref{th:semi_global_existence}.
\end{itemize}

\subsection*{Well-posedness for first-order hyperbolic systems}
%
Let us give a brief overview of well-posedness results for systems close to \eqref{eq:IGEBtyp}.
Below we refer to initial-boundary value problems (IBVP) unless stated otherwise.
%
%
Some systems may be written as abstract evolution problems that fit in the theory of semigroups. For \emph{linear} problems, this is the case for instance in \cite[Ap. A]{BC2016}, \cite{Russel_1978} and \cite{hu2019minimal} for solutions in $C_t^0 L_x^2$ (the latter uses the method of characteristics), as well as in \cite{Prieur_Winkin_2018} with some additional conditions on the coefficients (they follow \cite[Chap. 10]{tucsnak_weiss}). Using the method of characteristics and an iteration procedure to take care of the coupling induced by the linear lower order terms (such as $\overline{B}y$ here \eqref{eq:def_A_barB}), the author in \cite{friedrichs1948nonlinear} studies linear nonhomogeneous initial value problems (IVP). Furthermore, we refer to \cite[Section 3]{higdon1986initial} for a study of boundary conditions and their suitability for a given equation.

\medskip

%
\noindent For \emph{semilinear or quasilinear} systems, an iteration procedure with a contraction principle is often used to prove the existence of \emph{local in time} solutions. A semilinear system with a (locally) Lipschitz continuous semilinearity, which can be rewritten as an abstract evolution problem, may be treated as the perturbation of a linear evolution problem; see \cite[Chap. 6, Th. 1.4, 1.5, 1.7]{pazy}, and for an application see \cite{Egger_2017_semilin}. Let us note, however, that $\overline{g}(y)$ (see \eqref{eq:def_barg}) does not fit in this framework in our work. 
When such abstract results are not applicable, one may still use semigroup theory for linear systems as part of an iteration procedure, as in \cite[Ap. B]{BC2016} and \cite{bastin2017exponential, pavel_2013} in quasilinear and semilinear contexts respectively and with nonlinear boundary conditions.
For quasilinear and semilinear IVPs, local existence and uniqueness of classical solutions is proved in \cite{courant1949nonlinear} by means of the method of characteristics, and in \cite{friedrichs1948nonlinear, racke1992lectures} by iterations and using knowledge of the linearized problem. Finally, as mentionned above, results on local in time and semi-global in time existence and uniqueness of classical solutions for general \textit{quasilinear} hyperbolic systems can be found in
\cite{Li_Duke85} and \cite[Th. 3.1]{LiJin2001_semiglob}.

\medskip

%
\noindent \emph{Global in time} well-posedness for the IGEB model however -- for classical solutions or possibly other types of solutions -- is not straightforward, even when supposing small initial and boundary data. Due to the nonlinear term $\overline{g}(y)$ the solution may \textit{blow-up} in finite time (see \cite{alinhac2013blowup} for a discussion on blow-up for IVPs).
The properties of the coefficients and boundary conditions of a given system have an impact on its well-posedness. Global existence results for quasilinear systems are given in \cite[Chap. 4, Chap. 5]{Li_1994_global}, on the one hand for IVPs with dissipative properties that would translate to conditions on $\overline{B}y$ and $\overline{g}(y)$ here (see also \cite{beauchard2011large} for a relaxation of this assumption), and on the other hand for IBVPs with dissipative boundary conditions and no linear lower order term such as $\overline{B}y$ here.
In \cite{tartar1981some}, the author provides global $C_t^0 L_x^1$ solutions for semilinear IVPs with a quadratic nonlinearity fulfilling a series constraints (indeed satisfied by $\overline{g}(y)$) and no linear lower order term; see also \cite{bony1987solutions},\cite[Sec.3]{bony1994existence}. 
For global existence results assuming a monotonic behavior of certain terms or a growth restriction for the right-hand side, see \cite{myshkis2008global, turo1997mixed} for the quasilinear case (where $x \geq 0$ instead of a bounded interval), and \cite{Kmit_classical_nonlin} for classical solutions in the semilinear setting.

\section{The IGEB model}
\label{sec:wellp_igeb}

As well-posedness for a single beam is almost immediately given by the literature, and can regardless be seen as a special case of a network consisting of one edge and two simple nodes, we focus on networks in this section. Let us recall that the notation for networks has been introduced in Section \ref{sec:networks}.

The content of this sections is based on \ref{A:SICON}-\ref{A:MCRF}-\ref{A:JMPA}. We start by introducing the system which gives the dynamics of a network of beams described by the IGEB model. Then, we address questions of well-posedness: local in time $C_t^0H_x^k$ solutions in the context of boundary feedback stabilization, and semi-global in time $C_{x,t}^1$ solutions to later undertake boundary control of nodal profiles. In both cases, we will make use of the following assumption on the mass and flexibility matrices.

\begin{assumption} \label{as:mass_flex}
Let $m \in \{1, 2, \ldots\}$ be given. For all $i \in \mathcal{I}$, we suppose that
\begin{enumerate}
\item \label{eq:assump1_1} $\mathbf{C}_i, \mathbf{M}_i \in C^m([0, \ell_i]; \mathbb{S}_{++}^6)$;
\item \label{eq:assump1_2} the function $\Theta_i \in C^m([0, \ell_i]; \mathbb{S}_{++}^6)$ defined by $\Theta_i = (\mathbf{C}_i^{\sfrac{1}{2}} \mathbf{M}_i\mathbf{C}_i^{\sfrac{1}{2}})^{-1}$, is such that there exists $U_i, D_i \in C^m([0, \ell_i]; \mathbb{R}^{6 \times 6})$ for which $\Theta_i = U_i^\intercal D_i^2 U_i$ in $[0, \ell_i]$, where $D_i(x)$ is a positive definite diagonal matrix containing the square roots of the eigenvalues of $\Theta_i(x)$ as diagonal entries, while $U_i(x)$ is unitary.
\end{enumerate}
\end{assumption}

In view of the above observations on the matrix $A_i(x)$ in Section \ref{sec:pres_IGEB}, one can see that this assumption is here to ensure a certain regularity of the eigenvalues and eigenvectors of this matrix.
One may note that, in Assumption \ref{as:mass_flex}, if \ref{eq:assump1_1} holds, then \ref{eq:assump1_2} is readily verified if $\mathbf{M}_i, \mathbf{C}_i$ have values in the set of diagonal matrices, or if the eigenvalues of $\Theta_i(x)$ are distinct for all $x \in [0, \ell]$ (one may adapt \cite[Th. 2, Sec. 11.1]{evans2}). Clearly, \ref{eq:assump1_2} is also satisfied if $\mathbf{M}_i, \mathbf{C}_i$ are constant, entailing that the material and geometrical properties of the beam do not vary along its centerline.

\medskip

\noindent Finally, let us just introduce some more notation concerning the initial data $(y_i^0)_{i \in \mathcal{I}}$. For all $i \in \mathcal{I}$, if $y_i^0 \in C^1([0, \ell_i]; \mathbb{R}^{12})$, we define $y_i^1 \in C^0([0, \ell_i]; \mathbb{R}^{12})$ by 
\begin{equation} \label{eq:def_y1}
y_i^1 = -  A_i \frac{\mathrm{d}y_i^0}{\mathrm{d}x} - \overline{B}_i y_i^0 + \overline{g}_i(\cdot, y_i^0).
\end{equation}
Moreover, as for the state of the IGEB model, we denote
\begin{equation} \label{eq:notation_yvz01}
y_i^0 = \begin{bmatrix}
v_i^0 \\ z_i^0
\end{bmatrix}, \quad y_i^1 = \begin{bmatrix}
v_i^1 \\z_i^1
\end{bmatrix},
\end{equation}
where $v_i^0, z_i^0, v_i^1, z_i^1 \colon [0, \ell_i] \rightarrow \mathbb{R}^6$.
This notation will notably be of use when discussing the compatibility of $(y_i^0)_{i \in \mathcal{I}}$ with the nodal conditions.

\subsection{Modelling of transmission conditions}
\label{sec:modelling_transmi}

As in Section \ref{sec:networks}, at the multiple nodes, we assume continuity of the centerline, rigidity of the joint and balance of forces and moments.
For any $i \in \mathcal{I}$, we introduce the function $\overline{R}_{i} \colon [0, \ell_i]\rightarrow \mathbb{R}^{6 \times 6}$ defined by
\begin{equation} \label{eq:def_barRi}
\overline{R}_{i} = \mathrm{diag}(R_{i}, R_{i}),
\end{equation}
where we recall that $R_i \colon [0, \ell_i]\rightarrow \mathrm{SO}(3)$ characterizes the undeformed beam.
Then, if the beams are described by the IGEB model, for the overall network, the dynamics of the state $(y_i)_{i \in \mathcal{I}}$ are given by
\begin{subnumcases}{\label{eq:nIGEBgen}}
\label{eq:nIGEBgen_gov}
\partial_t y_i + A_i \partial_x y_i + \overline{B}_i y_i = \overline{g}_i(\cdot,y_i) &$\text{in } (0, \ell_i)\times(0, T), \, i \in \mathcal{I}$\\
\label{eq:nIGEBgen_cont}
(\overline{R}_i v_i)(\mathbf{x}_i^n, t) = (\overline{R}_{i^n} v_{i^n})(\mathbf{x}_{i^n}^n, t) &$t \in (0, T), \, i \in \mathcal{I}^n, \, n \in \mathcal{N}_M$\\
\label{eq:nIGEBgen_Kir}
\sum_{i\in\mathcal{I}^n} \tau_i^n (\overline{R}_i z_i)(\mathbf{x}_i^n, t) = q_n(t) &$t \in (0, T), \, n \in \mathcal{N}_M$\\
\label{eq:nIGEBgen_nSN}
\tau_{i^n}^n z_{i^n} (\mathbf{x}_{i^n}^n, t) = q_n(t) &$t \in (0, T), \, n \in \mathcal{N}_S^N$\\
\label{eq:nIGEBgen_nSD}
v_{i^n}(\mathbf{x}_{i^n}^n, t) = q_n(t) &$t \in (0, T), \, n \in \mathcal{N}_S^D$\\
\label{eq:nIGEBgen_IC}
y_i(x, 0) = y_i^0(x) &$x \in (0, \ell_i), \, i \in \mathcal{I}$,
\end{subnumcases}
with boundary data $q_n(t) \in \mathbb{R}^6$ and initial data $y_i^0(x) \in \mathbb{R}^{12}$. Concerning simple nodes $n \in \mathcal{N}_S$, the data $q_n$ is defined, as in \eqref{eq:def_qD}, by
\begin{align}\label{eq:def_q_nSD}
q_n &= \begin{bmatrix}
(f_n^\mathbf{R})^\intercal \frac{\mathrm{d}}{\mathrm{d}t}f_n^\mathbf{p}\\
(f_n^\mathbf{R})^\intercal \frac{\mathrm{d}}{\mathrm{d}t}f_n^\mathbf{R}
\end{bmatrix}, \quad \text{if }n \in \mathcal{N}_S^D\\
\label{eq:def_q_nSN}
q_n &= \begin{bmatrix}
\mathbf{R}_{i^n}(\mathbf{x}_{i^n}^n, \cdot)^\intercal &  \mathbf{0}\\
\mathbf{0} & \mathbf{R}_{i^n}(\mathbf{x}_{i^n}^n,\cdot)^\intercal
\end{bmatrix} f_n, \quad \text{if }n \in \mathcal{N}_S^N.
\end{align}
Note that we dropped the upperscript $D, N$ for $q_n$ for clarity.


\medskip

\noindent Let $n$ be the index of some multiple node. In \ref{A:MCRF}-\ref{A:JMPA}, to obtain the transmission conditions \eqref{eq:nIGEBgen_cont}-\eqref{eq:nIGEBgen_Kir} for the IGEB model, corresponding to that of \eqref{eq:nGEBgen}, we proceed as follows. We first differentiate the conditions \eqref{eq:nGEBgen_cont_p} on the continuity of the centerline and \eqref{eq:nGEBgen_rigid} on the rigidity of the joint with respect to time. Then, we left-multiply each of the obtained equations by $(R_j\mathbf{R}_j^\intercal)(\mathbf{x}_j^n, t)$ for the corresponding beam index $j$ (thereby also using the rigid joint assumption). Recognising the definition of the linear and angular velocities, the condition \eqref{eq:nIGEBgen_cont} on the ``continuity of the velocities'' is obtained.

The Kirchhoff condition for the IGEB model is obtained by left-multiplying each term in the right-hand side of the Kirchhoff condition \eqref{eq:nGEBgen_Kir} for the GEB model by $(R_i \mathbf{R}_i^\intercal)(\mathbf{x}_i^n, t)$ for the corresponding index $i$ (once again using the rigid joint assumption), left-multiplying $f_n$ by $(R_i \mathbf{R}_i^\intercal)(\mathbf{x}_i^n, t)$ for some $i \in \mathcal{I}_n$ (for instance as $i^n$). Then, recalling that  $\phi_i, \psi_i$ and $\Phi_i, \Psi_i$ are related by \eqref{eq:relation_phiPhi_psiPsi} one obtains the condition \eqref{eq:nIGEBgen_Kir} now in terms of the internal forces and moments $z_i$. Thus, one has the following relationship between $q_n$ and $f_n$ for $n\in \mathcal{N}_M$:
\begin{align} \label{eq:def_q_nM}
q_n &= \begin{bmatrix}
(R_{i^n}\mathbf{R}_{i^n}^\intercal)(\mathbf{x}_{i^n}^n, \cdot) &  \mathbf{0}\\
\mathbf{0} & (R_{i^n}\mathbf{R}_{i^n}^\intercal)(\mathbf{x}_{i^n}^n,\cdot)
\end{bmatrix} f_n.
\end{align}
The detailed computations are provided in \ref{A:MCRF} (Section 3.2) and \ref{A:JMPA} (Section 2.2.3).

\medskip

\noindent \textbf{Feedback at the boundary.} We are also interested in the case where a force depending on the velocities is applied at the Neumann simple nodes $n \in \mathcal{N}_S^N$ and/or at the multiple nodes $n \in \mathcal{N}_M$. More precisely, we want the applied load expressed in the body-attached basis $\mathbf{R}_{i^n}(\mathbf{x}_i^n, t)^\intercal f_n(t)$ to be proportional to the velocities $v_{i^n}(\mathbf{x}_i^n, t)$ as follows
\begin{align*}
\mathbf{R}_{i^n}(\mathbf{x}_{i^n}^n, t)^\intercal f_n(t) = - K_n v_{i^n}(\mathbf{x}_{i^n}^n, t), \quad t \in (0, T).
\end{align*}
In this case \eqref{eq:nIGEBgen_Kir}-\eqref{eq:nIGEBgen_nSN} take the form
\begin{subequations}
\begin{align*} 
\sum_{i\in\mathcal{I}^n} \tau_i^n (\overline{R}_i z_i)(\mathbf{x}_i^n, t) &= -K_n v_{i^n}(\mathbf{x}_{i^n}^n, t), \quad t \in (0, T), \, n \in \mathcal{N}_M\\
\tau_{i^n}^n z_{i^n} (\mathbf{x}_{i^n}^n, t) &= -K_n v_{i^n}(\mathbf{x}_{i^n}^n, t), \quad t \in (0, T), \, n \in \mathcal{N}_S^N,
\end{align*}
\end{subequations}
respectively.

\subsection{Semi-global in time solution in $C_{x,t}^1$}
\label{sec:sol_semi_glob}

In Section \ref{sec:sol_semi_glob}, for the initial and boundary data, we assume the following regularity
\begin{align}
\label{eq:reg_Idata_IGEB}
y_i^0 \in C^1([0, \ell_i]; \mathbb{R}^{12}), \quad i \in \mathcal{I},\\
\label{eq:reg_Ndata_IGEB}
q_n \in C^1([0, T]; \mathbb{R}^6), \quad n \in \mathcal{N},
\end{align}
while for the given data characterizing the geometry and material of the beams we assume that
\begin{equation} \label{eq:reg_beampara}
\mathbf{M}_i, \mathbf{C}_i \in C^1([0, \ell_i]; \mathbb{S}_{++}^6), \quad R_i \in C^2([0, \ell_i]; \mathrm{SO}(3)).
\end{equation}
We will need to define compatibility conditions for System \eqref{eq:nIGEBgen}. 

\begin{definition}
We say that the initial data $y_i^0 \in C^1([0, \ell_i]; \mathbb{R}^{12})$, for all $i\in \mathcal{I}$, and boundary data $q_n\in C^0([0, T]; \mathbb{R}^6)$, for all $n \in \mathcal{N}$, fulfill the first-order compatibility conditions of \eqref{eq:nIGEBgen} if
\begin{linenomath}
\begin{equation} \label{eq:compat_BC_gen} 
\begin{aligned}
&(\overline{R}_i v_i^0)(\mathbf{x}_i^n) = (\overline{R}_j v_j^0)(\mathbf{x}_j^n) \qquad && i,j \in \mathcal{I}^n, \, n \in \mathcal{N}_M\\
&{\textstyle \sum_{i\in\mathcal{I}^n}} \tau_i^n (\overline{R}_i z_i^0)(\mathbf{x}_i^n) = q_n(0) && n \in \mathcal{N}_M\\
& \tau_{i^n}^n z_{i^n}^0(\mathbf{x}_{i^n}^n) = q_n(0) && n \in \mathcal{N}_S^N\\
& v_{i^n}^0 (\mathbf{x}_{i^n}^n) = q_n(0) && n \in \mathcal{N}_S^D,
\end{aligned}
\end{equation} 
\end{linenomath}
holds (see \eqref{eq:notation_yvz01}) and  $(y_i^1)_{i\in\mathcal{I}} \subset C^0([0, \ell_i]; \mathbb{R}^{12})$ defined by \eqref{eq:def_y1} also fulfills \eqref{eq:compat_BC_gen}, where $v_i^0, z_i^0$ are replaced by $v_i^1, z_i^1$ respectively.
\end{definition}

\begin{theorem}[Th. 2.4 \ref{A:JMPA}] \label{th:semi_global_existence}
Consider a general network, suppose that $R_i$ has the regularity \eqref{eq:reg_beampara} and that Assumption \ref{as:mass_flex} is fulfilled with $m=2$.
Then, for any $T>0$, there exists $\varepsilon_0>0$ such that for all $\varepsilon \in (0, \varepsilon_0)$ and for some $\delta>0$, and all initial and boundary data $y_i^0, q_n$ of regularity \eqref{eq:reg_Idata_IGEB}-\eqref{eq:reg_Ndata_IGEB}, and satisfying $\|y_i^0\|_{C_x^1} +\|q_n\|_{C_t^1} \leq \delta$ and the first-order compatibility conditions of \eqref{eq:nIGEBgen}, there exists a unique solution $(y_i)_{i \in \mathcal{I}} \in \prod_{i=1}^N C^1([0, \ell_i] \times [0, T]; \mathbb{R}^{12})$ to \eqref{eq:nIGEBgen} and $\|y_i\|_{C_{x,t}^1}\leq \varepsilon$.
\end{theorem}

\noindent \textbf{Idea of the proof.} 
The proof of Theorem \ref{th:semi_global_existence} is based on the results on abstract one-dimensional first-order quasilinear hyperbolic systems developed by \cite{LiJin2001_semiglob, Li_Duke85}.
Such results require a certain regularity of the data and coefficients as well as a specific form of the boundary conditions for the system written in diagonal form (also called \emph{characteristic form} or \emph{Riemann invariants}).
Namely, once \eqref{eq:nIGEBgen} has been written in Riemann invariants, the new unknown state being denoted $(r_i)_{i \in \mathcal{I}}$, one should verify that the nodal conditions fulfill the following rule: at any node, the components of $r_i$ corresponding to characteristics \emph{entering} the domain $[0, \ell_i]\times [0, +\infty)$ at this node (i.e., the ``outgoing information'') is expressed explicitly as a function of the components of $r_i$ corresponding to characteristics \emph{leaving} the domain $[0, \ell_i]\times [0, +\infty)$ at this node (i.e., the ``incoming information''). More detail is given in Appendix \ref{ap:out_in_info}.
We summarise the principal steps of this proof below.
\begin{itemize}
\item \textbf{Step 1: Riemann invariants.}
We write \eqref{eq:nIGEBgen} in diagonal form by applying the change of variable 
\begin{linenomath}
\begin{align} \label{eq:change_var_Li}
r_i(x,t) = L_i(x) y_i(x,t),
\end{align}
\end{linenomath}
for all $x \in [0, \ell_i], \ t \in [0, T], \ i \in \mathcal{I}$. Recall that $L_i$ is defined by \eqref{eq:def_bfD_L} for the corresponding edge index $i$.
The first (resp. last) six components of $r_i$ correspond to the negative (resp. positive) eigenvalues of $A_i$, and therefore we denote
\begin{linenomath}
\begin{align*}
r_i = \begin{bmatrix}
r_i^-\\
r_i^+
\end{bmatrix}, \qquad r_i^-,\, r_i^+ \colon [0, \ell_i]\times [0, T] \rightarrow \mathbb{R}^6
\end{align*}
\end{linenomath}
for all $i \in \mathcal{I}$.
Then, we define the new coefficient $B_i \in C^1([0, \ell_i]; \mathbb{R}^{12\times 12})$ by $B_i(x) = L_i(x) \overline{B}_i(x) L_i(x)^{-1} + L_i(x) A_i(x) \frac{\mathrm{d}}{\mathrm{d}x}L_i^{-1}(x)$, the nonlinearity by $g_i(x,u) = L_i(x) \overline{g}_i(x,L_i(x)^{-1} u)$ for all $i \in \mathcal{I}$, $x \in [0, \ell_i]$ and $u \in \mathbb{R}^{12}$, and the initial data by $r_i^0 = L_i y_i^0$.


\item \textbf{Step 2: Outgoing/incoming information.}
Let us introduce the invertible matrix $\gamma_i^n$ and positive definite symmetric matrix $\sigma_i^n$, defined by
\begin{align} \label{eq:def_gamma_sigma_in}
\gamma_i^n &= (\overline{R}_i  \mathbf{C}_i^{\sfrac{1}{2}} U_i^\intercal)(\mathbf{x}_i^n), \quad
\sigma_i^n = (\overline{R}_i  \mathbf{C}_i^{-\sfrac{1}{2}} U_i^\intercal D_i^{-1} U_i \mathbf{C}_i^{-\sfrac{1}{2}} \overline{R}_i^\intercal)(\mathbf{x}_i^n)
\end{align}
for all $n \in \mathcal{N}$ and $i \in \mathcal{I}^n$, where we recall that $\overline{R}_i$ is defined by \eqref{eq:def_barRi} and $U_i$ is the matrix introduced in Assumption \ref{as:mass_flex}.
Having applied \eqref{eq:change_var_Li}, we obtain the following equivalent system after some computations
\begin{subnumcases}{\label{eq:nIGEBgenR}}
\partial_t r_i + \mathbf{D}_i \partial_x r_i + B_i r_i = g_i(\cdot, r_i) &$\text{in } (0, \ell_i)\times(0, T), \, i \in \mathcal{I}$\\
\label{eq:nIGEBgenR_transmi}
\frac{\mathbf{A}_n \mathbf{G}_n}{2} r^\mathrm{out}_n(t) = \frac{\mathbf{B}_n \mathbf{G}_n}{2} r^\mathrm{in}_n(t) +
\begin{bmatrix}
q_n(t)\\ \mathbf{0}
\end{bmatrix} &$t \in (0, T), \, n \in \mathcal{N}_M$\\
\label{eq:nIGEBgenR_nSN}
\frac{\overline{R}_{i^n}(\mathbf{x}_{i^n}^n)^\intercal \sigma_{i^n}^n \gamma_{i^n}^n}{2}  (r_n^\mathrm{out} - r_n^\mathrm{in})(t) = q_n(t) &$t \in (0, T), \, n \in \mathcal{N}_S^N$\\
\label{eq:nIGEBgenR_nSD}
\frac{(\mathbf{C}_{i^n}^{\sfrac{1}{2}}U_{i^n}^\intercal)(\mathbf{x}_{i^n}^n)}{2} ( r_n^\mathrm{out} + r_n^\mathrm{in})(t) = q_n(t) &$t \in (0, T), \, n \in \mathcal{N}_S^D$\\
\label{eq:nIGEBgenR_IC}
r_i(x, 0) = r_i^0(x) &$x \in (0, \ell_i), \, i \in \mathcal{I}$,
\end{subnumcases}
with $\mathbf{A}_n, \mathbf{B}_n, \mathbf{G}_n \in \mathbb{R}^{6k_n \times 6k_n}$ defined by $\mathbf{G}_n = \mathrm{diag}(\gamma_{i_1}^n, \ldots, \gamma_{i_{k_n}}^n)$ and
\begin{align*}
\mathbf{A}_n = \begin{bmatrix}
\sigma_{i_1}^n & \sigma^n_{i_2} & \ldots & \sigma^n_{i_{k_n}}\\
-\mathbf{I}_6 & \mathbf{I}_6 & & \\
\vdots & & \ddots & \\
-\mathbf{I}_6 & & & \mathbf{I}_6
\end{bmatrix}, \quad
\mathbf{B}_n =  \begin{bmatrix}
\sigma_{i_1}^n & \sigma^n_{i_2} & \ldots & \sigma^n_{i_{k_n}}\\
\mathbf{I}_6 & -\mathbf{I}_6 & & \\
\vdots & & \ddots & \\
\mathbf{I}_6 & & & -\mathbf{I}_6
\end{bmatrix},
\end{align*}
where we used the notation \eqref{eq:notation_I^n} for the set $\mathcal{I}^n$.
The outgoing and incoming information $r_n^\mathrm{out}, r_n^\mathrm{in}$ are more precisely defined in Appendix \ref{ap:out_in_info} (see \eqref{eq:r_out_in_gene}).
The key point here is that, in the diagonalised system, the transmission conditions \eqref{eq:nIGEBgen_cont}-\eqref{eq:nIGEBgen_Kir} can be rewritten in the form \eqref{eq:nIGEBgenR_transmi} where $\mathbf{A}_n \mathbf{G}_n$ is invertible.

\item \textbf{Step 3: Applying abstract results for a single system.}
Applying the change of variable $\widetilde{r}_i(\xi, t) = r_i(\ell_i \ell^{-1} \xi, t)$ for all $i \in \mathcal{I}$, $\xi \in [0, \ell]$ and $t \in [0, T]$ for some $\ell>0$, in order to make the spatial domain identical for all beams, and considering the larger $\mathbb{R}^{12N}$-valued unknown $\widetilde{r} = (\widetilde{r}_1^{\,\intercal}, \ldots, \widetilde{r}_N^{\, \intercal})^\intercal$, we obtain a single larger hyperbolic system. Due to the previous step, the boundary conditions of the obtained larger hyperbolic system are already written in such a way that the outgoing information for this system is a function of the incoming information.

The local and semi-global existence and uniqueness of $C_{x,t}^1$ solutions to general one-dimensional quasilinear hyperbolic systems have been addressed in \cite[Lem. 2.3, Th. 2.1]{wang2006exact}, which is an extension of \cite[Lem. 2.3, Th. 2.5]{li2010controllability} to nonautonomous systems. See also \cite{Zhuang2021, Zhuang2018} for a similar problem with the Saint-Venant equations on general networks. Since we are now in position to apply such results, this concludes the proof. \hfill $\meddiamond$
\end{itemize}

\subsection{Local in time solution in $C_t^0H_x^k$}

\label{sec:local_existence_igeb}

Here, we focus on \emph{tree-shaped} networks where either a velocity feedback control is applied at the nodes, or the beams are free or clamped. 

\medskip

\noindent \textbf{Tree-shaped networks.} Since we only consider the tree-shaped case, we may simplify the network notation as follows (inspired by \cite{AlabauPerrollazRosier2015}). We suppose that the set of nodes also includes the index $n=0$. Then, without loosing generality, we assume that the node $n=0$ is a simple node and is the initial point of the edge $i=1$, and we assume that for any $i \in \mathcal{I}$ the edge $i$ has for ending point the node with the same index $n=i$ (see Fig. \ref{fig:treeNet_numbering}). In other words, for all $n \in \mathcal{N}$ and all $i \in \mathcal{I}^n$, we assume that $n \in \mathcal{I}^n$ and
\begin{align*}
\mathbf{x}_i^n =
\left\{\begin{array}{ll}
\ell_n &\text{if }i=n\\
0 &\text{else.}
\end{array}\right.
\end{align*}
This allows us to let go of the notation $\mathbf{x}_i^n$.
In addition, for any multiple node $n \in \mathcal{N}_M$, we denote by $\mathcal{I}_n$ the set of indices of all edges starting at this node; we also denote the elements of $\mathcal{I}_n$ by (see Fig. \ref{fig:treeNet_multNode})
\begin{equation} \label{eq:nota_In_indices}
\mathcal{I}_n = \{i_2, i_3, \ldots, i_{k_n}\}, \qquad \text{with} \quad i_2 < i_3 < \ldots < i_{k_n}.
\end{equation}
In this case, the set $\mathcal{I}^n$ defined in Section \ref{sec:networks} is given by $\mathcal{I}^n = \{n\} \cup \mathcal{I}_n$.

\medskip

\noindent \textbf{The system.} In light of this notation, the network system, with unknown state $y = (y_i)_{i\in\mathcal{I}}$, reads
\begin{subnumcases}{\label{eq:nIGEBfb}}
\label{eq:nIGEBfb_gov}
\partial_t y_i + A_i \partial_x y_i + \overline{B}_i y_i = \overline{g}_i(\cdot,y_i) &$\text{in } (0, \ell_i)\times(0, T), \, i \in \mathcal{I}$\\
\label{eq:nIGEBfb_cont}
\overline{R}_{i}(0) v_i(0, t) = \overline{R}_{n}(\ell_n) v_n(\ell_n, t) &$t \in (0, T), \, i \in \mathcal{I}^n, \, n \in \mathcal{N}_M$\\
\nonumber
\overline{R}_{n}(\ell_n) z_n(\ell_n, t) - {\textstyle \sum_{i \in \mathcal{I}_n}} \overline{R}_{i}(0) z_i(0, t) \\
\label{eq:nIGEBfb_Kir}
\hspace{2.175cm} = - \overline{R}_n(\ell_n) K_n v_n(\ell_n, t) &$t \in (0, T), \, n \in \mathcal{N}_M$\\
\label{eq:nIGEBfb_nSell}
z_n(\ell_n, t) = - K_n v_n(\ell_n, t) &$t \in (0, T), \, n \in \mathcal{N}_S \setminus \{0\}$\\
\label{eq:nIGEBfb_nS0}
z_1(0, t) = K_0 v_1(0, t) &$t \in (0, T)$\\
\label{eq:nIGEBfb_IC}
y_i(x, 0) = y_i^0(x) &$x \in (0, \ell_i), \, i \in \mathcal{I}$.
\end{subnumcases}
In \eqref{eq:nIGEBfb_Kir}, two situations may be accounted for: either a feedback is applied at this node, in which case $K_n \in \mathbb{S}_{++}^{6}$, or $K_n = \mathbf{0}_6$ and no external load is applied at this node; the latter corresponds to the classical Kirchhoff condition. 
At simple nodes $n\in\mathcal{N}_S$, in \eqref{eq:nIGEBfb_nSell},\eqref{eq:nIGEBfb_nS0}, either the velocity feedback control is applied with $K_n \in \mathbb{S}_{++}^{6}$ or $K_n = \mathbf{0}_6$ and the beam is \emph{free}.
Instead of \eqref{eq:nIGEBfb_nSell} or \eqref{eq:nIGEBfb_nS0}, one may want to assume that the beam is clamped at a given simple node, which would amount to considering the respective homogeneous Dirichlet conditions
\begin{equation}\label{eq:diri_cond_0n}
v_n(\ell_n, t) = 0, \quad \text{or} \quad v_1(0, t) = 0, \qquad t \in (0, T).
\end{equation}

\medskip

\begin{figure}
  \begin{subfigure}{0.7\textwidth}
  \centering
    \includegraphics[scale=0.8]{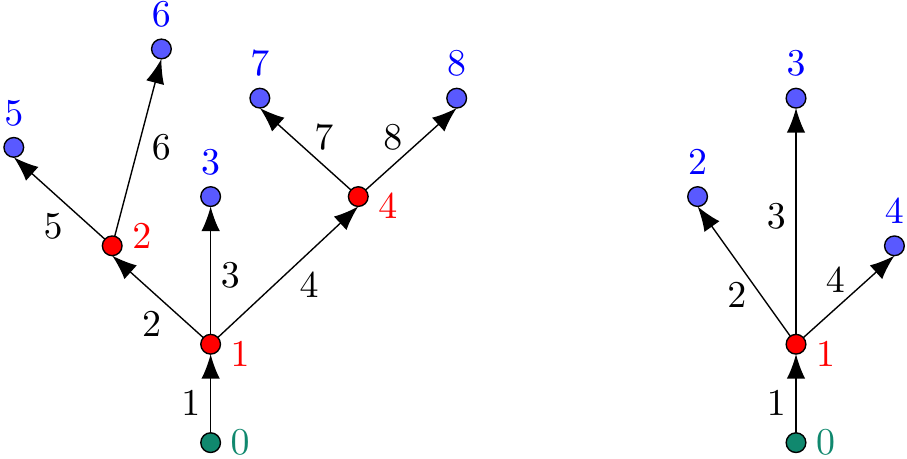}
    \caption{Left: A tree-shaped network with $N=8$ edges, $\mathcal{N}_S = \{0, 3, 5, 6, 7, 8\}$ and $\mathcal{N}_M = \{ 1, 2, 4 \}$. Right: A star-shaped network with $N=5$ edges, $\mathcal{N}_S = \{0, 2, 3, 4\}$ and $\mathcal{N}_M = \{ 1 \}$.}
    \label{fig:treeNet_numbering}
  \end{subfigure}%
    \hspace*{\fill}   
  \begin{subfigure}{0.26\textwidth}
    \centering
    \includegraphics[scale=0.85]{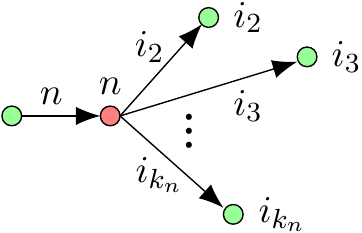}
\caption{A multiple $n$ and the incident edges and nodes.}
\label{fig:treeNet_multNode}
  \end{subfigure}

\caption{Additional notation for tree-shaped networks.}
\end{figure}

\noindent \textbf{Local in time well-posedness.} Let us now define compatibility conditions for System \eqref{eq:nIGEBfb}. 
Recall that we denote $\mathbf{H}^k_x = \prod_{i=1}^N H^k(0, \ell_i; \mathbb{R}^{12})$ for any $k \geq 1$ (see Section \ref{sec:notation}), as well as $y^0 = (y_i^0)_{i\in \mathcal{I}}$ and $y = (y_i)_{i\in \mathcal{I}}$.

\begin{definition}
\label{def:comp_cond}
We say that the initial datum $y^0 \in \mathbf{H}^1_x$ fulfills the zero-order compatibility conditions of System \eqref{eq:nIGEBfb} if
\begin{equation} \label{eq:compat_BC_fb} 
\begin{aligned}
& (\overline{R}_{i} v_i^0)(0) = (\overline{R}_{n} v_n^0)(\ell_n), &&i \in \mathcal{I}_n, \, n \in \mathcal{N}_M,\\
&(\overline{R}_{n} z_n^0)(\ell_n) - \sum_{i \in \mathcal{I}_n} (\overline{R}_{i} z_i^0)(0) = - (\overline{R}_n K_n v_n^0)(\ell_n), \quad &&n \in \mathcal{N}_M,\\
&z_n^0(\ell_n) = - K_n v_n^0(\ell_n),  &&n \in \mathcal{N}_S \setminus \{0\},\\
&z_1^0(0) = K_0 v_1^0(0), &&n=0,
\end{aligned}
\end{equation} 
holds. We say that $y^0\in \mathbf{H}^2_x$ fulfills the first-order compatibility conditions of \eqref{eq:nIGEBfb} if \eqref{eq:compat_BC_fb} holds, and if $(y_i^1)_{i \in \mathcal{I}} \in \mathbf{H}_x^1$, defined by \eqref{eq:def_y1}, also satisfies \eqref{eq:compat_BC_fb}, where $v_i^0, z_i^0$ are replaced by $v_i^1, z_i^1$ respectively. 
\end{definition}


Then, one has the following local in time well-posedness result. 
One may observe in the proof of Theorem \ref{th:local_existence} that we do not use the fact that the network is tree-shaped, and could easily adapt this to networks with loops, by using the same notation as in Section \ref{sec:sol_semi_glob}.

\begin{theorem}[Th. 2.2 \ref{A:MCRF}] \label{th:local_existence}
Let $k \in \{1, 2\}$, suppose that Assumption \ref{as:mass_flex} is fulfilled for $m=k+1$, and assume that $\mathbf{E}_i \in C^k([0, \ell_i]; \mathbb{R}^{6 \times 6})$ for all $i\in\mathcal{I}$. Then, there exists $\delta_0>0$ such that for any $y^0 \in \mathbf{H}^k_x$ satisfying $\|y^0\|_{\mathbf{H}^k_x} \leq \delta_0$ and the $(k-1)$-order compatibility conditions, there exists a unique solution $y \in C^0([0, T), \mathbf{H}^k_x)$ to \eqref{eq:nIGEBfb}, with $T\in (0, +\infty]$. Moreover, if $\|y(\cdot, t)\|_{\mathbf{H}^k_x} \leq \delta_0$ for all $t \in [0, T)$ then $T = + \infty$.
\end{theorem}

\medskip

\noindent \textbf{Idea of the proof.}
The proof of Theorem \ref{th:local_existence} is based on existing results on first-order hyperbolic systems -- the local existence and uniqueness of $C_t^0 H_x^k$ solutions to general one-dimensional semilinear ($k=1$) and quasilinear ($k=2$) hyperbolic systems, which have been addressed by Bastin and Coron in \cite[Ap. B, Rem. 6.9]{BC2016} and \cite[Th. 10.1]{bastin2017exponential}, respectively. 
Similarly to the proof of Theorem \ref{th:semi_global_existence}, these results require sufficient regularity from the data and coefficients, and for the boundary conditions to be written in such a way that the outgoing information is a function of the incoming information at each node.

For all $n \in \mathcal{N}\setminus \{0\}$ and $i \in \mathcal{I}^n$, the matrices $\gamma_i^n$ and $\sigma_i^n$ are defined as before by \eqref{eq:def_gamma_sigma_in}. For the node $n=0$, we also define the invertible matrix $\gamma_0^0 \in \mathbb{R}^{6\times 6}$ and $\sigma_0^0 \in \mathbb{S}_{++}^6$ by
\begin{align} \label{eq:def_gamma0_sigma0}
\gamma_0^0 = (\overline{R}_1\mathbf{C}_1^{\sfrac{1}{2}}U_1^\intercal)(0), \quad \sigma_0^0 = (\overline{R}_1\mathbf{C}_1^{-\sfrac{1}{2}}U_1^\intercal D_1^{-1} U_1 \mathbf{C}_1^{- \sfrac{1}{2}} \overline{R}_1^\intercal)(0).
\end{align}
Moreover, we define the symmetric matrices
\begin{equation}\label{eq:def_barKn}
\overline{K}_0 = \overline{R}_1(0) K_0 \overline{R}_1(0)^\intercal, \quad \overline{K}_n = \overline{R}_n(\ell_n) K_n \overline{R}_n(\ell_n)^\intercal
\end{equation}
for all $n \in \mathcal{N}\setminus\{0\}$, which are positive definite (resp. null) if and only if $K_n$ is positive definite (resp. $K_n = \mathbf{0}_6$).
We thus apply the change of variable \eqref{eq:change_var_Li}, just as in the first step of the proof of Theorem \ref{th:semi_global_existence}.
After some computations, one can deduce that \eqref{eq:nIGEBgen} also takes the form
\begin{subnumcases}{\label{eq:nIGEBfbR}}
\partial_t r_i + \mathbf{D}_i \partial_x r_i + B_i r_i = g_i(\cdot, r_i) &$\text{in } (0, \ell_i)\times(0, T), \, i \in \mathcal{I}$\\
\label{eq:nIGEBfbR_transmi}
\mathbf{A}_n \mathbf{G}_n r^\mathrm{out}_n(t) = \mathbf{B}_n \mathbf{G}_n r^\mathrm{in}_n(t) &$t \in (0, T), \, n \in \mathcal{N}_M$\\
\label{eq:nIGEBfbR_simple}
(\sigma_n^n + \overline{K}_n)\gamma_n^n r^\mathrm{out}_n(t) = (\sigma_n^n - \overline{K}_n)\gamma_n^n r^\mathrm{in}_n(t) &$t \in (0, T), \, n \in \mathcal{N}_S$\\
\label{eq:nIGEBfbR_IC}
r_i(x, 0) = r_i^0(x) &$x \in (0, \ell_i), \, i \in \mathcal{I}$,
\end{subnumcases}
in Riemann invariants, where $\mathbf{A}_n, \mathbf{B}_n, \mathbf{G}_n \in \mathbb{R}^{6k_n \times 6k_n}$ are defined by
\begin{align*}
\mathbf{A}_n = \begin{bmatrix}
\sigma_n^n + \overline{K}_n & \sigma^n_{i_2} & \ldots & \sigma^n_{i_{k_n}}\\
-\mathbf{I}_6 & \mathbf{I}_6 & & \\
\vdots & & \ddots & \\
-\mathbf{I}_6 & & & \mathbf{I}_6
\end{bmatrix}, \quad
\mathbf{B}_n =  \begin{bmatrix}
\sigma_n^n - \overline{K}_n & \sigma^n_{i_2} & \ldots & \sigma^n_{i_{k_n}}\\
\mathbf{I}_6 & -\mathbf{I}_6 & & \\
\vdots & & \ddots & \\
\mathbf{I}_6 & & & -\mathbf{I}_6
\end{bmatrix},
\end{align*}
and $\mathbf{G}_n = \mathrm{diag}(\gamma_n^n, \gamma_{i_2}^n, \ldots, \gamma_{i_{k_n}}^n)$.
This time we used the notation \eqref{eq:nota_In_indices} for the set $\mathcal{I}_n$, and $r_n^\mathrm{out}, r_n^\mathrm{in}$ are defined in \eqref{eq:r_out_in_mult_simp}. One can show that the matrix $\mathbf{A}_n\mathbf{G}_n$ is also invertible.
The proof is then concluded by applying \cite{BC2016, bastin2017exponential} to the larger hyperbolic system obtained analogously to the third step of the proof of Theorem \ref{th:semi_global_existence}. \hfill $\meddiamond$

\begin{remark} \label{rem:wellp_fb}
Let us make some remarks on this result. 
\begin{enumerate}
\item \label{remItem:transparent}
When at a given simple node $n \in \mathcal{N}_S$, the outgoing information $r_n^\mathrm{out}$ is identically equal to zero, one often speak about \emph{transparent boundary conditions}. 
Here, one may notice that such conditions are obtained with the specific choice of feedback:
\begin{linenomath}
\begin{equation*}
\begin{aligned}
K_n = \big(\mathbf{C}_n^{-\sfrac{1}{2}} (\mathbf{C}_n^{\sfrac{1}{2}} \mathbf{M}_n \mathbf{C}_n^{\sfrac{1}{2}})^{\sfrac{1}{2}}\mathbf{C}_n^{-\sfrac{1}{2}}\big)\big(\ell_n\big), \qquad &\text{if }n\neq 0,\\
K_0 = \big(\mathbf{C}_1^{-\sfrac{1}{2}} (\mathbf{C}_1^{\sfrac{1}{2}} \mathbf{M}_1 \mathbf{C}_1^{\sfrac{1}{2}})^{\sfrac{1}{2}}\mathbf{C}_1^{-\sfrac{1}{2}}\big)\big(0\big), \qquad &\text{if }n = 0.
\end{aligned}
\end{equation*}
\end{linenomath}
Indeed, one may compute that in this case in this case $\overline{K}_n = \sigma^n_n$.

\item \label{remItem:extra_regularity}
In Assumption \ref{as:mass_flex}, the extra-regularity required from the mass and flexibility matrices and from the decomposition of $\Theta_i$ is needed to ensure enough regularity for $\mathbf{D}_i$ and $B_i$, for the results in \cite{BC2016} are given in terms of the diagonal system. Thus, one way to be free of this assumption could be to prove well-posedness directly for the original (``physical'') system.

\item
Theorem \ref{th:local_existence} also holds if beams are clamped \eqref{eq:diri_cond_0n} at some simple nodes, provided that the compatibility conditions \eqref{eq:compat_BC_fb} are accordingly changed. 
At simple nodes, \eqref{eq:nIGEBfbR_simple} takes the form $r_n^\mathrm{out}(t) = r_n^\mathrm{in}(t)$ for a free beam, while if the beam is clamped it becomes $r_n^\mathrm{out}(t) = - r_n^\mathrm{in}(t)$.
\end{enumerate}
\end{remark}

\medskip

\noindent \textbf{Single beam}. For further reference, let us give the form of \eqref{eq:IGEBfb} in Riemann invariants. 
In the spirit of the network case, we may introduce 
$\widetilde{\sigma} = \mathbf{C}^{-\sfrac{1}{2}} U^\intercal D^{-1} U \mathbf{C}^{- \sfrac{1}{2}}$
and
$\widetilde{\gamma} = \mathbf{C}^{\sfrac{1}{2}} U^\intercal$. Then the diagonal system, with unknown state $r=Ly$, reads
\begin{align}\label{eq:IGEBfbR}
\begin{cases}
\partial_t r + \mathbf{D} \partial_x r + B r = g(\cdot, r) &\text{in } (0, \ell)\times(0, T)\\
r_0^\mathrm{out}(t) = - r_0^\mathrm{in}(t) &t \in (0, T)\\
(\widetilde{\sigma} + K)\widetilde{\gamma}\, r_\ell^\mathrm{out}(t) = (\widetilde{\sigma} - K)\widetilde{\gamma}\, r_\ell^\mathrm{in}(t) &t \in (0, T)\\
r(x, 0) = r^0(x) &x \in (0, \ell)
\end{cases}
\end{align}
with outgoing information $r_0^\mathrm{out}=r^+(0, \cdot)$ and $r_\ell^\mathrm{out}=r^-(\ell, \cdot)$, and incoming information $r_0^\mathrm{in}=r^-(0, \cdot)$ and $r_\ell^\mathrm{in}=r^+(\ell, \cdot)$.

\section{The GEB model: inverting the transformation}
\label{sec:invert_transfo}

The object of this section, is the inversion of the transformation from the GEB model to the IGEB model and the well-posedness results that we may deduce, by that means, for the former model.
The material here is based on \ref{A:SICON} and \ref{A:JMPA}. The initial idea was developed in \ref{A:SICON} and the invertibility of the transformation is thus proved in detail there. However, the context in \ref{A:SICON} is that of feedback stabilization with constant diagonal mass and flexibility matrices, and the state consists of velocities and strains. In \ref{A:JMPA}, on the other hand, the aim is to extend the previous work to the general linear-elastic constitutive law, to nonhomogeneous boundary/nodal conditions, and also to networks.

\subsection{For a single beam}
\label{sec:1b_invert_transfo}

Here, we present the proof in a simple example, that of a single beam clamped at $x=0$ and controlled via velocity feedback (or free) at $x=\ell$. Later on, in Section \ref{sec:invert_transfo_net}, we move to different boundary conditions and networks.
When such a beam is described by the GEB model, its dynamics are given by the following system:
\begin{subnumcases}{\label{eq:GEBfb}}
\label{eq:GEBfb_gov}
\begingroup 
\setlength\arraycolsep{3pt}
\renewcommand*{\arraystretch}{0.9}
\begin{bmatrix}
\partial_t & \mathbf{0}\\
(\partial_t \widehat{\mathbf{p}}) & \partial_t
\end{bmatrix} \left[ \begin{bmatrix}
\mathbf{R} & \mathbf{0}\\ \mathbf{0} & \mathbf{R}
\end{bmatrix}
\mathbf{M} v \right] = \begin{bmatrix}
\partial_x & \mathbf{0} \\ (\partial_x \widehat{\mathbf{p}}) & \partial_x
\end{bmatrix} \begin{bmatrix}
\phi \\ \psi
\end{bmatrix}
\endgroup
&$\text{in }(0, \ell)\times(0, T)$\\
\label{eq:GEBfb_BC0}
(\mathbf{p}, \mathbf{R})(0, t) = (f^\mathbf{p}, f^\mathbf{R}) &$t \in (0, T)$\\ 
\label{eq:GEBfb_BCell}
\begingroup 
\setlength\arraycolsep{3pt}
\renewcommand*{\arraystretch}{0.9}
\begin{bmatrix}
\phi \\ \psi
\end{bmatrix}(\ell, t)
= - \begin{bmatrix}
\mathbf{R}(\ell, t) & \mathbf{0}\\ \mathbf{0} & \mathbf{R}(\ell, t)
\end{bmatrix} K v(\ell, t)
\endgroup
&$t \in (0, T)$\\
\label{eq:GEBfb_IC0}
(\mathbf{p}, \mathbf{R})(x,0) = (\mathbf{p}^0, \mathbf{R}^0)(x) &$x \in (0, \ell)$\\
\label{eq:GEBfb_IC1}
(\partial_t \mathbf{p}, \mathbf{R}W)(x,0) = (\mathbf{p}^1, w^0)(x) &$x \in (0, \ell)$,
\end{subnumcases}
where we recall that $v, \phi, \psi$ are the functions of $(\mathbf{p}, \mathbf{R})$ defined by \eqref{eq:def_v_z}, \eqref{eq:def_VWPhiPsi} and \eqref{eq:def_phi_psi}, in System \eqref{eq:GEBfb}.
The corresponding IGEB system reads
\begin{subnumcases}{\label{eq:IGEBfb}}
\label{eq:IGEBfb_gov}
\partial_t y + A(x) \partial_x y + \overline{B}(x) y = \overline{g}(x, y) &$\text{in }(0, \ell)\times(0, T)$\\
\label{eq:IGEBfb_BC0}
v(0, t) = \mathbf{0} &$\text{for }t \in (0, T)$\\
\label{eq:IGEBfb_BCell}
z(\ell, t) = - K v(\ell, t) &$\text{for }t \in (0, T)$\\
\label{eq:IGEBfb_IC}
y(x, 0) = y^0(x) &$\text{for }x \in (0, \ell)$.
\end{subnumcases}
Indeed, the transformation $\mathcal{T}$, defined in \eqref{eq:transfo}, has been introduced in Section \ref{sec:pres_transfo}, and we have seen that the initial data of \eqref{eq:GEBfb} and \eqref{eq:IGEBfb} can be linked by \eqref{eq:rel_inidata}.
By the definition of $\mathcal{T}$ and and using Proposition \ref{prop:transfoGov} one can see that $\mathcal{T}\colon E_1 \rightarrow E_2$ is well defined for $E_1, E_2$ defined by
\begin{align*}
E_1 &= \left\{(\mathbf{p}, \mathbf{R}) \in C^2\left([0, \ell]\times[0, T]; \mathbb{R}^3 \times \mathrm{SO}(3)\right) \colon \eqref{eq:GEBfb_BC0}, \eqref{eq:GEBfb_IC0} \text{ hold} \right\}\\
E_2 &= \left\{y=(v^\intercal, z^\intercal)^\intercal \in C^1\left([0, \ell]\times[0, T] ; \mathbb{R}^{12}\right) \colon \text{\eqref{eq:inv1b_comp} holds for } s := \mathbf{C} z\right\}.
\end{align*}
Observe that $E_1$ includes the Dirichlet boundary conditions and initial condition of order zero, while $E_2$ involves the following ``compatibility conditions'':
\begin{subequations}\label{eq:inv1b_comp}
\begin{align}
& \tfrac{\mathrm{d}}{\mathrm{d}x}\mathbf{p}^0 = \mathbf{R}^0 (s_1(\cdot, 0) + e_1),\quad \tfrac{\mathrm{d}}{\mathrm{d}x} \mathbf{R}^0 = \mathbf{R}^0 (\widehat{s}_2(\cdot, 0) + \widehat{\Upsilon}_c), \quad \text{in } (0, \ell) 
\label{eq:inv1b_compIC}\\
&\partial_t \begin{bmatrix}
s_1 \\ s_2
\end{bmatrix} - \partial_x \begin{bmatrix}
v_{1} \\ v_{2}
\end{bmatrix} - 
\begingroup 
\setlength\arraycolsep{3pt}
\renewcommand*{\arraystretch}{0.9}
\begin{bmatrix}
\widehat{\Upsilon}_c & \widehat{e}_1 \\
\mathbf{0} & \widehat{\Upsilon}_c
\end{bmatrix}
\endgroup
\begin{bmatrix}
v_{1} \\ v_{2}
\end{bmatrix}  = 
\begin{bmatrix}
\widehat{v}_{2} & \widehat{v}_{1}\\
\mathbf{0} & \widehat{v}_{2}
\end{bmatrix} \begin{bmatrix}
s_1 \\ s_2
\end{bmatrix}, \quad \text{in }(0, \ell)\times (0, T)
\label{eq:inv1b_compGOV}
\\
& v_1(0, \cdot) = \mathbf{0}, \quad v_2(0, \cdot) = \mathbf{0}, \quad \text{in } (0, T), \label{eq:inv1b_compBC}
\end{align}
\end{subequations}
where we use the notation $v = (v_1^\intercal, v_2^\intercal)^\intercal$ and $s = (s_1^\intercal, s_2^\intercal)^\intercal$ for $v_1, v_2, s_1, s_2$ having values in $\mathbb{R}^3$.
The quantity $s$ in \eqref{eq:inv1b_comp} is nothing but the variable of strains (rather stresses), as we have just multiplied $z$ by the flexibility matrix. Then, one has the following theorem.

\begin{theorem}[Th. 1.7 \ref{A:SICON}, Lem. 5.1 \ref{A:JMPA}] \label{th:invert_transfo_fb}
If $(f^\mathbf{p}, f^\mathbf{R}) \in \mathbb{R}^3 \times \mathrm{SO}(3)$ and $(\mathbf{p}^0, \mathbf{R}^0) \in C^2([0, \ell]; \mathbb{R}^3 \times \mathrm{SO}(3))$ satisfy
\begin{align} \label{eq:inv1b_comp-1}
(f^\mathbf{p}, f^\mathbf{R}) = (\mathbf{p}^0, \mathbf{R}^0)(0),
\end{align}
then the transformation $\mathcal{T}\colon E_1 \rightarrow E_2$ defined in \eqref{eq:transfo} is bijective.
\end{theorem}

\noindent \textbf{Idea of the proof.} Our aim, given $y \in E_2$ is to show that there exists a unique $(\mathbf{p}, \mathbf{R}) \in E_1$ such that $\mathcal{T}(\mathbf{p}, \mathbf{R}) = y$. The latter rewrites easily as a system of linear PDEs consisting of twelve equations, but where there are effectively only six unknowns, if one takes into account that $\mathbf{R}(x,t)$ should be a rotation matrix. 
Then, two key steps are 1) that we use quaternions to transmute the search for a solution $\mathbf{R}$ having values in $\mathrm{SO}(3)$ to that of a $\mathbb{R}^4$-valued solution, and 2) that the last six governing equations of the IGEB model turn out to be compatibility conditions ensuring that the aforementioned system is in fact not overdetermined. More precisely, we proceed as follows.
\begin{itemize}
\item \textbf{Casting the inverse transformation as a PDE system.} Due to the first compatibility condition \eqref{eq:inv1b_compIC}, finding $(\mathbf{p}, \mathbf{R})$ that fulfill the identity $\mathcal{T}(\mathbf{p}, \mathbf{R}) = y$ and the boundary and initial conditions \eqref{eq:GEBfb_BC0} and \eqref{eq:GEBfb_IC0}, is equivalent to solving the following two coupled systems 
\begin{subnumcases}{\label{eq:ODE_for_R}}
\label{eq:ODE_for_R_v}
\partial_t \mathbf{R} = \mathbf{R} \widehat{v}_{2}&$\text{in }(0, \ell) \times (0,T)$\\
\label{eq:ODE_for_R_z}
\partial_x \mathbf{R} = \mathbf{R} (\widehat{s}_2 + \widehat{\Upsilon}_c) &$\text{in }(0, \ell) \times (0,T)$\\
\label{eq:ODE_for_R_ini}
\mathbf{R}(0, 0) = \mathbf{R}^0(0)
\end{subnumcases}
and
\begin{subnumcases}{\label{eq:ODE_for_p}}
\label{eq:ODE_for_p_v}
\partial_t \mathbf{p} = \mathbf{R} v_{1} &$\text{in }(0, \ell) \times (0,T)$\\
\label{eq:ODE_for_p_z}
\partial_x \mathbf{p} = \mathbf{R} (s_{1} + e_1) &$\text{in }(0, \ell) \times (0,T)$\\
\label{eq:ODE_for_p_ini}
\mathbf{p}(0, 0) = \mathbf{p}^0(0).
\end{subnumcases}

\item \textbf{Quaternions.} We transform the first above system by means of Lemma \ref{lem:quat_rot_ODE} below. A quaternion \cite{chou1992} is a pair of real value $q_0 \in \mathbb{R}$ and vectorial value $q \in \mathbb{R}^3$, that we denote here as the vector $\mathbf{q} =(q_0, q^\intercal)^\intercal$.
A rotation matrix $\mathbf{R} \in \mathrm{SO}(3)$ is said to be parametrized by the quaternion $\mathbf{q} \in \mathbb{R}^4$, if $|\mathbf{q}| = 1$ and 
\begin{align} \label{eq:rot_param_quat}
\mathbf{R} = (q_0^2 - \langle q \,, q \rangle)\mathbf{I}_3 + 2 q q^\intercal + 2 q_0 \widehat{q}.
\end{align}
Then, $\mathbf{R}\colon [0, \ell]\times [0, T] \rightarrow \mathrm{SO}(3)$ is parametrized by the quaternion-valued function $\mathbf{q}\colon [0, \ell]\times [0, T] \rightarrow \mathbb{R}^4$, if $|\mathbf{q}| \equiv 1$ and \eqref{eq:rot_param_quat} holds in $[0, \ell]\times [0, T]$.
Consider the linear map $\mathcal{U}$ defined by
\begin{align}\label{eq:def_calU} 
\mathcal{U}(w) = \frac{1}{2} \begin{bmatrix}
0 & - w^\intercal \\
w & -\widehat{w}
\end{bmatrix}, \qquad \text{for all }w \in \mathbb{R}^3. 
\end{align}
This function is commonly used to represent (modulo the factor $\frac{1}{2}$) the quaternion product between the vector that the resulting matrix $\mathcal{U}(w)$ acts on, and $(0, w^\intercal)^\intercal$.

\begin{lemma} \label{lem:quat_rot_ODE}
\vspace{7pt}
Let $f \in C^1([0, \ell]\times[0, T]; \mathbb{R}^3)$ and let $z$ represent either of the spatial or time variables $x, t$.
The function $\mathbf{q} \in C^1([0, \ell]\times[0, T]; \mathbb{R}^4)$ fulfills both $|\mathbf{q}| \equiv 1$ and
\begin{align} \label{eq:zODE_quat_U}
\partial_z \mathbf{q} = \mathcal{U}(f) \mathbf{q}, \quad \text{in } (0, \ell)\times (0, T),
\end{align}
if and only if the function $\mathbf{R} \in C^1([0, \ell]\times[0, T]; \mathrm{SO}(3))$ parametrized by $\mathbf{q}$ fulfills
\begin{align} \label{eq:zODE_R}
\partial_z \mathbf{R} = \mathbf{R} \widehat{f}, \quad \text{in } (0, \ell)\times(0, T).
\end{align}
\end{lemma}

Thus, the system \eqref{eq:ODE_for_R} with unknown $\mathbf{R}$ is equivalent to
\begin{subnumcases}{\label{eq:ODE_for_q}}
\label{eq:ODE_for_q_v}
\partial_t \mathbf{q} = \mathcal{U}(v_2) \mathbf{q} &$\text{in }(0, \ell) \times (0,T)$\\
\label{eq:ODE_for_q_z}
\partial_x \mathbf{q} = \mathcal{U} (s_2 + \Upsilon_c) \mathbf{q} &$\text{in }(0, \ell) \times (0,T)$\\
\label{eq:ODE_for_q_ini}
\mathbf{q}(0, 0) = \mathbf{q}_{\mathrm{in}},
\end{subnumcases}
where $\mathbf{q}_\mathrm{in}$ is the quaternion parametrizing $\mathbf{R}^0(0)$.
This type of transformation is used in practice to numerically recover rotation matrices from velocities (see \cite{Artola2019mpc} for instance). However, we include a proof of the equivalence between both PDE systems (i.e., a proof of Lemma \ref{lem:quat_rot_ODE}) in Appendix \ref{ap:quaternion} for completeness since, up to the best of our knowledge, it is not present in the literature in such a form.

\item \textbf{Seemingly overdetermined systems.} At first sight it seems that both the system for $\mathbf{R}$ and the system for $\mathbf{p}$ are overdetermined, but the first three and last three equations in \eqref{eq:inv1b_compGOV} provide compatibility conditions for $\mathbf{p}$ and for $\mathbf{R}$, respectively, thereby permitting to solve both systems. Indeed, the last three equations in \eqref{eq:inv1b_compGOV} are equivalent to
\begin{align*}
\mathcal{U}(v_2) \mathcal{U}(s_2 + \Upsilon_c) - \mathcal{U}(s_2 + \Upsilon_c)\mathcal{U}(v_2)  + \partial_x(\mathcal{U}(v_2)) - \partial_t(\mathcal{U}(s_2 + \Upsilon_c)) = 0,
\end{align*}
and we may thus solve for $\mathbf{q}$ by means of the following lemma (see Appendix \ref{ap:quaternion} for a proof using basic differential equations tools).

\begin{lemma} \label{lem:quat_overter_syst}
\vspace{7pt}
Let $A, B \in C^1([0, \ell]\times [0, T]; \mathbb{R}^{n \times n})$ be such that the compatibility condition $A B - B A + (\partial_x A) - (\partial_t B) = 0$ holds in $(0, \ell)\times (0, T)$.
Then,
\begin{align} \label{eq:overdeter_y_AB}
\begin{cases}
\partial_t y = A y& \text{in }(0, \ell)\times (0, T)\\
\partial_x y = B y & \text{in }(0, \ell)\times (0, T)\\
y(0, 0) = y_\mathrm{in}.
\end{cases}
\end{align}
admits a unique solution $y \in C^1([0, \ell]\times[0,T]; \mathbb{R}^{n})$, for any given $y_\mathrm{in} \in \mathbb{R}^n$.
\end{lemma}
Then, for $\mathbf{R}$ parametrized by the obtained unit-norm quaternion $\mathbf{q}$, we can solve for $\mathbf{p}$ using \eqref{eq:inv1b_compBC} together with \eqref{eq:inv1b_compGOV} (first three equations). \hfill $\meddiamond$
\end{itemize}

Note that the time interval could be changed to $[0, T)$ or $[0, +\infty)$, the Dirichlet boundary condition could be at $x=\ell$ rather than $x=0$, or there could be no Dirichlet boundary conditions at all. The boundary conditions could also be more general, such as in \eqref{eq:GEB_genBC_Neu}-\eqref{eq:GEB_genBC_Diri} and \eqref{eq:IGEB_genBC}. This would influence the compatibility conditions to be enforced, but not the idea of the proof.
Building on Theorem \ref{th:invert_transfo_fb}, we may then prove the following theorem. 

\begin{theorem}[Th. 1.7 \ref{A:SICON}, Th. 2.7 \ref{A:JMPA}] \label{thm:inv1b_igeb2geb}
Assume that
\begin{align*}
\mathbf{M}, \mathbf{C} \in C^1([0, \ell]; \mathbb{R}^{6\times 6}),& \quad R\in C^2([0, \ell]; \mathrm{SO}(3)),\\
(\mathbf{p}^0, \mathbf{R}^0) \in C^2([0, \ell]; \mathbb{R}^3 \times \mathrm{SO}(3)),& \quad \mathbf{p}^1, w^0 \in C^1([0, \ell]; \mathbb{R}^3)
\end{align*}
 are such that \eqref{eq:inv1b_comp-1} holds.
Let $y^0$ be the associated function defined by \eqref{eq:rel_inidata}.
Then, if there exists a unique solution $y \in C^1([0, \ell]\times [0, T]; \mathbb{R}^{12})$ to \eqref{eq:IGEBfb} with initial data $y^0$ (for some $T>0$), there exists a unique solution
$(\mathbf{p}, \mathbf{R}) \in C^2([0, \ell]\times [0, T]; \mathbb{R}^3 \times \mathrm{SO}(3))$
to \eqref{eq:GEBfb} with initial data $(\mathbf{p}^0, \mathbf{R}^0, \mathbf{p}^1, w^0)$ and boundary data $(f^\mathbf{p}, f^\mathbf{R})$, and $y = \mathcal{T}(\mathbf{p}, \mathbf{R})$.
\end{theorem}

\noindent \textbf{Idea of the proof.} A solution $y$ to \eqref{eq:IGEBfb} in fact always belongs to $E_2$ due to the last six governing equations in \eqref{eq:IGEBfb_gov}, the last six initial conditions in \eqref{eq:IGEBfb_IC} and the Dirichlet conditions \eqref{eq:IGEBfb_BC0} -- all satisfied by $y$ --, and since we maintained the link between the initial and boundary data of both models.
The previous theorem thus automatically provides $(\mathbf{p}, \mathbf{R})$, candidate to be a solution to \eqref{eq:GEBfb}. Of concern here is thus to show that this candidate fulfills the full system. 
This is indeed the case since $(\mathbf{p}, \mathbf{R})$ then fulfills the governing system of the GEB model as we already have seen in Proposition \ref{prop:transfoGov}, and we can use the rest of boundary and initial conditions of System \eqref{eq:IGEBfb} to deduce those of \eqref{eq:GEBfb}. The uniqueness results from that of the IGEB model and from the fact that $\mathcal{T}$ is bijective. \hfill $\meddiamond$

\subsection{For a network}
\label{sec:invert_transfo_net}

In Section \ref{sec:invert_transfo_net},  we work with the network systems \eqref{eq:nGEBgen} and \eqref{eq:nIGEBgen} for beams described by the GEB model and the IGEB model, respectively.
For the initial and boundary data, we assume the following regularity
\begin{linenomath}
\begin{align} \label{eq:reg_Idata_GEB}
(\mathbf{p}_i^0, \mathbf{R}_i^0) \in C^2([0, \ell_i]; \mathbb{R}^3 \times \mathrm{SO}(3)), \quad \mathbf{p}_i^1, w_i^0 \in C^1([0, \ell_i]; \mathbb{R}^3), \quad &i \in\mathcal{I},\\
\label{eq:reg_Ndata_GEB_N}
f_n \in C^1([0, T]; \mathbb{R}^6), \quad &n \in \mathcal{N}_M \cup \mathcal{N}_S^N,\\
\label{eq:reg_Ndata_GEB_D}
(f_n^\mathbf{p}, f_n^\mathbf{R}) \in C^2([0, T]; \mathbb{R}^3\times \mathrm{SO}(3)), \quad &n \in\mathcal{N}_S^D.
\end{align}
\end{linenomath}
The transformation $\mathcal{T}_\mathrm{net}$ for a network simply consists in applying the previous transformation $\mathcal{T}$ (for a single beam) to each of the beams $i \in \mathcal{I}$. In other words,
\begin{align} \label{eq:transfo_net} 
\mathcal{T}_\mathrm{net}((\mathbf{p}_i, \mathbf{R}_i)_{i\in \mathcal{I}}) = (\mathcal{T}_i(\mathbf{p}_i, \mathbf{R}_i))_{i \in \mathcal{I}}, \quad \text{where} \quad \mathcal{T}_i(\mathbf{p}_i, \mathbf{R}_i) = \begin{bmatrix}
v_i \\ z_i
\end{bmatrix}, 
\end{align}
where $v_i, z_i$ are the functions of $\mathbf{p}_i, \mathbf{R}_i$ defined by \eqref{eq:def_v_z}.
In Lemma \ref{lem:invert_transfo} below we will invert $\mathcal{T}_\mathrm{net}$ on the following spaces
\begin{linenomath}
\begin{align*}
\overline{E}_1 &= \left\{(\mathbf{p}_i, \mathbf{R}_i)_{i \in \mathcal{I}} \in {\textstyle \prod_{i=1}^N} C^2\left([0, \ell_i]\times[0, T]; \mathbb{R}^3 \times \mathrm{SO}(3)\right) \colon \eqref{eq:nGEBgen_nSD}, \eqref{eq:nGEBgen_IC0} \text{ hold} \right\}\\
\overline{E}_2 &= \Big\{\Big(y_i= (v_i^\intercal, z_i^\intercal)^\intercal\Big)_{i \in \mathcal{I}} \in {\textstyle \prod_{i=1}^N} C^1\left([0, \ell_i]\times[0, T] ; \mathbb{R}^{12}\right) \colon \text{\eqref{eq:invNet_comp} holds for all }s_i := \mathbf{C}_i z_i \Big\},
\end{align*}
\end{linenomath}
where, using the notation $v_{i} = (v_{i, 1}^\intercal, v_{i, 2}^\intercal)^\intercal$ and $s_{i} = (s_{i, 1}^\intercal, s_{i, 2}^\intercal)^\intercal$ for $v_{i, 1}, v_{i, 2}, s_{i, 1}, s_{i, 2}$ having values in $\mathbb{R}^3$, the condition \eqref{eq:invNet_comp} consists of 
\begin{subequations}\label{eq:invNet_comp}
\begin{align}
\label{eq:compat_last6eq}
\begingroup 
\begin{aligned}
& \partial_t \begin{bmatrix}
s_{i,1} \\ s_{i,2}
\end{bmatrix} - \partial_x \begin{bmatrix}
v_{i,1} \\ v_{i,2}
\end{bmatrix} - 
\begingroup 
\renewcommand*{\arraystretch}{0.8}
\begin{bmatrix}
\widehat{\Upsilon}_c^i & \widehat{e}_1 \\
\mathbf{0} & \widehat{\Upsilon}_c^i
\end{bmatrix}
\endgroup
\begin{bmatrix}
v_{i,1} \\ v_{i,2}
\end{bmatrix} \\
& \hspace{3.25cm}= 
\begin{bmatrix}
\widehat{v}_{i,2} & \widehat{v}_{i,1}\\
\mathbf{0} & \widehat{v}_{i,2}
\end{bmatrix} \begin{bmatrix}
s_{i,1} \\ s_{i,2}
\end{bmatrix},
\end{aligned}
\endgroup
\qquad
&\begin{aligned}
&\text{in }(0, \ell_i)\times (0, T),\\
&\text{for all }i \in \mathcal{I},
\end{aligned}
\\
\label{eq:compat_ini}
\left\{\begin{array}{l}
\tfrac{\mathrm{d}}{\mathrm{d}x}\mathbf{p}_i^0(\cdot) = \mathbf{R}_i^0(\cdot) (s_{i,1}(\cdot, 0) + e_1),\\
\tfrac{\mathrm{d}}{\mathrm{d}x} \mathbf{R}_i^0(\cdot) = \mathbf{R}_i^0(\cdot)(\widehat{s}_{i,2}(\cdot, 0) + \widehat{\Upsilon}_c^i(\cdot)),
\end{array}\right.
\qquad 
&\text{in } (0, \ell_i), \, \text{for all }i \in \mathcal{I},
\\
\label{eq:compat_nod}
\left\{\begin{array}{l}
\tfrac{\mathrm{d}}{\mathrm{d}t} f_n^\mathbf{p} (\cdot) = f_n^\mathbf{R}(\cdot) \widehat{v}_{{i^n},1}(\mathbf{x}_{i^n}^n, \cdot),\\
\tfrac{\mathrm{d}}{\mathrm{d}t} f_n^\mathbf{R} (\cdot) = f_n^\mathbf{R}(\cdot) \widehat{v}_{{i^n},2}(\mathbf{x}_{i^n}^n, \cdot),
\end{array}\right.
\qquad
&\text{in } (0, T), \, \text{for all } n \in \mathcal{N}_S^D.
\end{align}
\end{subequations}
Then, we have the following lemma. 
The essential ideas of the proof have been presented for Theorem \ref{th:invert_transfo_fb}, and even though the setting is slightly more general here, it is still to use quaternions to parametrize $\mathbf{R}_i$ and use the conditions \eqref{eq:invNet_comp} at different stages of the proof in order to solve the -- apparently overdetermined -- system of linear PDEs that characterizes $\mathcal{T}_\mathrm{net}$.


\begin{lemma}[Lem. 5.1 \ref{A:JMPA}] \label{lem:invert_transfo}
The transformation $\mathcal{T}_\mathrm{net}\colon \overline{E}_1 \rightarrow \overline{E}_2$ defined in \eqref{eq:transfo_net} is bijective if $(\mathbf{p}_i^0, \mathbf{R}_i^0, f_n^\mathbf{p}, f_n^\mathbf{R})$ are of regularity \eqref{eq:reg_Idata_GEB} and \eqref{eq:reg_Ndata_GEB_D}, and fulfill 
\begin{align}
\label{eq:compat_GEB_-1}
&(f_n^\mathbf{p}, f_n^\mathbf{R})(0) = (\mathbf{p}_{i^n}^0, \mathbf{R}_{i^n}^0)(\mathbf{x}_{i^n}^n), \quad n\in \mathcal{N}_S^D.
\end{align}
\end{lemma}

We may now, with Theorem \ref{thm:solGEB} below, make the link between the existence of a unique classical solution to \eqref{eq:nIGEBgen} and the same property for \eqref{eq:nGEBgen}.
As explained at the beginning of Section \ref{sec:invert_transfo_net}, we focus on the network described by \eqref{eq:nIGEBgen}, even though such results could also be developed for the system \eqref{eq:nIGEBfb} with boundary feedback -- from Section \ref{sec:local_existence_igeb}.

One constraint that might have been overlooked when considering the latter system, is that in order to switch from the GEB point of view to the IGEB point of view and use the existence and uniqueness results of the latter to deduce such results for the former, is that the Neumann data $f_n$ should have the form
\begin{align}\label{eq:def_fn} 
f_n
= \left\{ \begin{array}{ll}
\begingroup 
\setlength\arraycolsep{3pt}
\renewcommand*{\arraystretch}{0.9}
\begin{bmatrix}
\mathbf{R}_{i^n} &  \mathbf{0}\\
\mathbf{0} & \mathbf{R}_{i^n}
\end{bmatrix}
\endgroup
(\mathbf{x}_{i^n}^n, \cdot)\, q_n &n \in \mathcal{N}_S^N\\
\begingroup 
\setlength\arraycolsep{3pt}
\renewcommand*{\arraystretch}{0.8}
\begin{bmatrix}
\mathbf{R}_{i^n} R_{i^n}^\intercal &  \mathbf{0}\\
\mathbf{0} & \mathbf{R}_{i^n} R_{i^n}^\intercal
\end{bmatrix}
\endgroup
(\mathbf{x}_{i^n}^n, \cdot) \, q_n &n \in \mathcal{N}_M
\end{array}\right.
\end{align}
(equivalent to \eqref{eq:def_q_nSN} and \eqref{eq:def_q_nM}). In other words, its form in the body-attached basis must be available.



\begin{theorem}[Th. 2.7 \ref{A:JMPA}] \label{thm:solGEB}
Consider a general network, and assume that:
\begin{enumerate}[label=(\roman*)]
\item \label{thm:solGEB_c1} the beam parameters $(\mathbf{M}_i, \mathbf{C}_i, R_i)$ and initial data $(\mathbf{p}_i^0, \mathbf{R}_i^0, \mathbf{p}_i^1, w_i^0)$ have the regularity \eqref{eq:reg_beampara} and \eqref{eq:reg_Idata_GEB}, and $y_i^0$ is the associated function defined by \eqref{eq:rel_inidata},

\item \label{thm:solGEB_c2} the Neumann data $f_n$ have the form \eqref{eq:def_fn}, for given $q_n$ of regularity \eqref{eq:reg_Ndata_IGEB}, 

\item \label{thm:solGEB_c3} the Dirichlet data $(f_n^\mathbf{p}, f_n^\mathbf{R})$  are of regularity \eqref{eq:reg_Ndata_GEB_D}, and $q_n$ are the associated functions defined as in \eqref{eq:def_q_nSD},

\item \label{thm:solGEB_c4} the compatibility conditions \eqref{eq:compat_GEB_-1} and the following hold:
\begin{align}
\label{eq:compat_GEB_transmi}
&\mathbf{p}_i^0(\mathbf{x}_i^n)  = \mathbf{p}_{i^n}^0(\mathbf{x}_{i^n}^n), \quad (\mathbf{R}_i^0 R_i^\intercal)(\mathbf{x}_i^n)  = (\mathbf{R}_{i^n}^0 R_{i^n}^\intercal)(\mathbf{x}_{i^n}^n), \quad i\in \mathcal{I}^n, \, n \in \mathcal{N}_M.
\end{align}
\end{enumerate}
\vspace{-0.25cm}
Then, if there exists a unique solution $(y_i)_{i \in \mathcal{I}} \in \prod_{i=1}^N C^1([0, \ell_i]\times [0, T]; \mathbb{R}^{12})$ to \eqref{eq:nGEBgen} with initial and nodal data $y_i^0$ and $q_n$ (for some $T>0$), there exists a unique solution
$(\mathbf{p}_i, \mathbf{R}_i)_{i \in \mathcal{I}} \in  \prod_{i=1}^N C^2([0, \ell_i]\times [0, T]; \mathbb{R}^3 \times \mathrm{SO}(3))$
to \eqref{eq:nGEBgen} with initial data $(\mathbf{p}_i^0, \mathbf{R}_i^0, \mathbf{p}_i^1, w_i^0)$ and nodal data $f_n$, $(f_n^\mathbf{p}, f_n^\mathbf{R})$, and $(y_i)_{i\in \mathcal{I}} = \mathcal{T}_\mathrm{net}((\mathbf{p}_i, \mathbf{R}_i)_{i\in\mathcal{I}})$.
\end{theorem}

\noindent \textbf{Idea of the proof.}
The main difference between the network case and that of a single beam is that one has to take care of the rigid joint condition in order to be able to recover the other transmission conditions of the GEB model.
As for a single beam, one can check that the solution $(y_i)_{i \in \mathcal{I}}$ to \eqref{eq:nIGEBgen} belongs to $\overline{E}_2$ -- one sees, once again, the role of the last six governing equations of the IGEB model and the importance of the link between the initial and boundary data of both models -- and thus obtain a candidate $(\mathbf{p}_i, \mathbf{R}_i)_{i \in \mathcal{I}}$ by Lemma \ref{lem:invert_transfo}. One soon gets that this candidate fulfills the governing system \eqref{eq:nGEBgen_gov} due to Proposition \ref{prop:transfoGov}, as well as the remaining initial conditions \eqref{eq:nGEBgen_IC1} and Neumann conditions \eqref{eq:nGEBgen_nSN} at simple nodes.
The last steps of the proof are the following.
\begin{itemize}

\item \textbf{At Dirichlet nodes.} Here, we make use of the compatibility condition \eqref{eq:compat_GEB_-1}.
Indeed, for $n \in \mathcal{N}_S^D$, one can show that $(\mathbf{p}_{i^n},\mathbf{R}_{i^n})(\mathbf{x}_{i^n}^n, \cdot) \equiv (f_n^\mathbf{p}, f_n^\mathbf{R})$ by proving that both sides solve the following system with state $(\alpha, \beta)$
\begin{align} \label{eq:nodPDE}
\left\{ \begin{aligned}
&\frac{\mathrm{d}\beta}{\mathrm{d}t}(t) = \beta(t)  \widehat{q_n^2}(t), \ \ \frac{\mathrm{d}\alpha}{\mathrm{d}t}(t) = \beta(t) q_n^1(t) && \ \text{ in }(0, T)\\
&(\alpha, \beta)(0) = (\mathbf{p}_{i^n}^0, \mathbf{R}_{i^n}^0)(\mathbf{x}_{i^n}^n),
\end{aligned} \right.
\end{align}
where we denoted by $q_n^1$ and $q_n^2$ the first and last three components of $q_n$, respectively.
That System \eqref{eq:nodPDE} admits a unique solution in $C^2([0, T]; \mathbb{R}^{3}\times \mathrm{SO}(3))$ may be seen from classical ODE theory after having parametrized rotation matrices by quaternions as in Theorem \ref{th:invert_transfo_fb}.

\item \textbf{Rigid joint condition.}
This is where the second equation in \eqref{eq:compat_GEB_transmi} comes into play. Together with the last three equations in the continuity condition \eqref{eq:nIGEBgen_cont} of the IGEB model, it yields that for any $n \in \mathcal{N}_M$ and all $i \in \mathcal{I}^n$ the time dependent function $\Lambda_i := (R_i \mathbf{R}_i^\intercal)(\mathbf{x}_i^n, \cdot)$ solves 
\begin{align*}
\begin{dcases}
\frac{\mathrm{d}}{\mathrm{d}t} \Lambda_i(t) = F_n(t) \Lambda_i(t) & \text{for all }t \in (0, T)\\
\Lambda_i(0) = a_n,
\end{dcases}
\end{align*}
where $F_n := \big(\frac{\mathrm{d}}{\mathrm{d}t} \Lambda_{i^n} \big)\Lambda_{i^n}^\intercal$ and $a_n := \big(R_{i^n}{\mathbf{R}_{i^n}^0}^\intercal \big)(\mathbf{x}_{i^n}^n)$, which admits a unique solution in $C^1([0, T]; \mathbb{R}^{3 \times 3})$ (see \cite[Sec. 2.1 and Th. 4.1.1 or Coro. 2.4.4]{vrabie2004}, for instance).

\item \textbf{Continuity of the displacement.}
Then, the first equation in \eqref{eq:compat_GEB_transmi} is used. Together with the rigid joint condition, we can use this time the first three equations in \eqref{eq:nIGEBgen_cont} to see that for any $n \in \mathcal{N}_M$ and all $i \in \mathcal{I}^n$, the function $\mathbf{p}_i(\mathbf{x}_i^n, \cdot)$ solves
\begin{linenomath}
\begin{align*} 
\begin{dcases}
\partial_t \mathbf{p}_i(\mathbf{x}_i^n, t) = h_n(t) & \text{for all }t \in (0, T)\\
\mathbf{p}_i(\mathbf{x}_i^n, 0) = \alpha_n,
\end{dcases}
\end{align*}
\end{linenomath}
where $h_n := \partial_t \mathbf{p}_{i^n}(\mathbf{x}_{i^n}^n, \cdot)$ and $\alpha_n := \mathbf{p}_{i^n}^0(\mathbf{x}_{i^n}^n)$, which admits a unique $C^1([0, T];\mathbb{R}^3)$ solution.
\end{itemize}
Thanks to the rigid joint condition, the Kirchhoff condition of the IGEB network is recovered similarly to the Neumann nodes. The uniqueness question is treated as in Theroem \ref{th:invert_transfo_fb}. \hfill  $\meddiamond$

\medskip

Then, Corollary \ref{coro:wellposedGEB} below follows from Theorems \ref{th:semi_global_existence} and \ref{thm:solGEB}.

\begin{corollary}[Coro. 2.9 \ref{A:JMPA}]
\label{coro:wellposedGEB}
Consider a general network and suppose that \ref{thm:solGEB_c1}-\ref{thm:solGEB_c2}-\ref{thm:solGEB_c3}-\ref{thm:solGEB_c4} of Theorem \ref{thm:solGEB} are fulfilled, suppose that the beam parameters $(\mathbf{M}_i, \mathbf{C}_i)$ satisfy Assumption \ref{as:mass_flex} with $m=2$, and that $(y_i^0)_{i\in \mathcal{I}}$ fulfills the first-order compatibility conditions of \eqref{eq:nIGEBgen}.
Then, for any $T>0$, there exists $\varepsilon_0>0$ such that for all $\varepsilon \in (0, \varepsilon_0)$, and for some $\delta>0$, if moreover $\|y_i^0\|_{C_x^1}+ \|q_n\|_{C_t^1}\leq \delta$, then there exists a unique solution $(\mathbf{p}_i, \mathbf{R}_i)_{i \in \mathcal{I}} \in \prod_{i=1}^N C^2([0, \ell_i]\times[0, T]; \mathbb{R}^3 \times \mathrm{SO}(3))$ to \eqref{eq:nGEBgen} with initial data $(\mathbf{p}_i^0, \mathbf{R}_i^0, \mathbf{p}_i^1, w_i^0)$ and nodal data $f_n$, $(f_n^\mathbf{p}, f_n^\mathbf{R})$.
\end{corollary}

%
%
%
%
%
%
%

\begin{subappendices}

\section{Incoming and outgoing information}
\label{ap:out_in_info}

In this Appendix we give precise definitions of the outgoing information $r_n^\mathrm{out}$ and incoming information $r_n^\mathrm{in}$ at the nodes $n \in \mathcal{N}$, which play an important role in the well-posedness of the IGEB model on networks.

\medskip

\noindent \textbf{A single system.}
Consider the IGEB model for a single beam -- for instance \eqref{eq:IGEBtyp}. We have seen in Section \ref{sec:wellp_igeb} that this system may be diagonalized and that the new variable $r$ is such that $r^-(x,t) \in \mathbb{R}^6$ (resp. $r^+(x,t) \in \mathbb{R}^6$) are the components of $r$ corresponding to the negative (resp. positive) diagonal entries of $\mathbf{D}$.

Then, the \emph{outgoing information} consists of the components of $r$ corresponding to characteristics which are outgoing at the boundaries $x=0$ and $x=\ell$ (in other words, entering the domain $[0, \ell]\times [0, T]$): these are $r^{+}(0, t)$ and $r^{-}(\ell, t)$ respectively. 
Likewise, we mean by \emph{incoming information}, the components of $r$ corresponding to characteristics which are incoming at the boundaries $x=\ell$ and $x=0$ (leaving the domain $[0, \ell]\times[0, T]$): these are $r^{-}(0, t)$ and $r^{+}(\ell, t)$ respectively. 
We refer to Fig. \ref{fig:charac} for visualization.

\begin{figure}[h]
  \centering
    \includegraphics[scale=0.7]{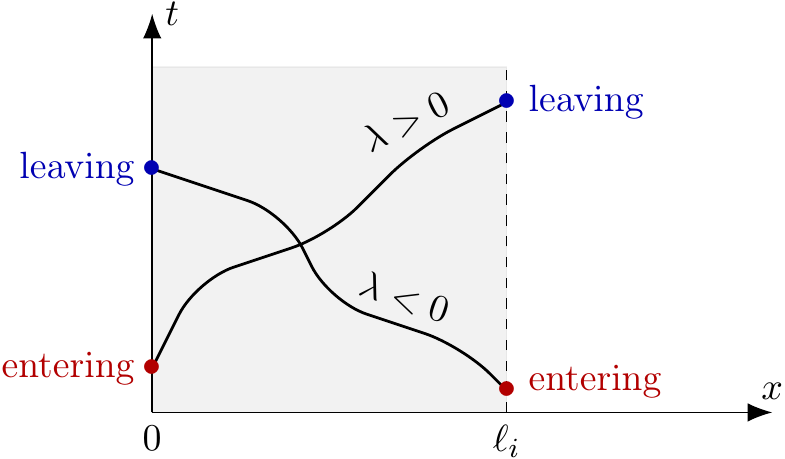}
\caption{Characteristic curves $(\mathbf{x}(t), t)$ with $\frac{\mathrm{d}\mathbf{x}}{\mathrm{d}t}(t) = \lambda(\mathbf{x}(t))$, where either $\lambda(s)>0$ or $\lambda(s)<0$ for all $s \in [0, \ell_i]$.}
\label{fig:charac}
\end{figure}

\medskip

\noindent \textbf{General networks.} For hyperbolic systems on networks, for any beam $i\in \mathcal{I}$, the outgoing information is then $r_i^- (\ell, t)$ and $r_i^+ (0, t)$, while the incoming information is $r_i^+ (\ell, t)$ and $r_i^- (0, t)$. 

We may also specify what is the outgoing and incoming information at each node. Let us use the notation introduced in \eqref{eq:notation_I^n} for the set $\mathcal{I}^n$.
Then, the outgoing information at the node $n\in \mathcal{N}$ is $r_{i_\alpha}^-(\ell_{i_\alpha}, t)$ for all $\alpha \in \{1, \ldots, s_n\}$, and $r_{i_k}^+(0, t)$ for all $k \in \{s_n+1, \ldots, k_n\}$. On the other hand, the incoming information at the node $n$ is $r_{i_\alpha}^+(\ell_{i_\alpha}, t)$ for all $\alpha \in \{1, \ldots, s_n\}$, and $r_{i_k}^-(0, t)$ for all $k \in \{s_n+1, \ldots, k_n\}$.
We then define the functions $r_n^\mathrm{out}, r_n^\mathrm{in} \colon [0, T] \rightarrow \mathbb{R}^{6k_n}$ by
\begingroup 
\setlength\arraycolsep{3pt}
\renewcommand*{\arraystretch}{0.9}
\begin{linenomath}
\begin{align} \label{eq:r_out_in_gene}
r_n^\mathrm{out}(t) = \begin{bmatrix}
r_{i_1}^-(\ell_{i_1}, t)\\
\vdots\\
r_{i_{s_n}}^-(\ell_{i_{s_n}}, t)\\
r_{i_{s_n+1}}^+(0, t) \\
\vdots\\
r_{i_{k_n}}^+(0, t)
\end{bmatrix}, \qquad r_n^\mathrm{in}(t) = \begin{bmatrix}
r_{i_1}^+(\ell_{i_1}, t)\\
\vdots\\
r_{i_{s_n}}^+(\ell_{i_{s_n}}, t)\\
r_{i_{s_n+1}}^-(0, t) \\
\vdots\\
r_{i_{k_n}}^-(0, t)
\end{bmatrix}.
\end{align}
\end{linenomath}
\endgroup

\medskip

\noindent \textbf{Tree-shaped networks.} Let us do the same for tree-shaped networks, this time using the notation introduced in \eqref{eq:nota_In_indices} for the set $\mathcal{I}_n$. For any node $n \in \mathcal{N}$, the outgoing and incoming information $r^\mathrm{out}_n, r^\mathrm{in}_n \colon [0, T] \rightarrow \mathbb{R}^{6k_n}$ are then given by
\begin{equation} \label{eq:r_out_in_mult_simp}
r^\mathrm{out}_n =
\begin{cases}
\left[\begin{smallmatrix}
r_n^- (\ell_n, \cdot)\\
r_{i_2}^+ (0, \cdot)\\
\vdots\\
r_{i_{k_n}}^+ (0, \cdot)
\end{smallmatrix}\right], & n \in \mathcal{N}_M\\
r_n^- (\ell_n, \cdot), & n \in \mathcal{N}_S \setminus \{0\}\\
r_1^+(0, \cdot), & n = 0
\end{cases}, \quad
r^\mathrm{in}_n =
\begin{cases}
\left[\begin{smallmatrix}
r_n^+(\ell_n, \cdot)\\
r_{i_2}^- (0, \cdot)\\
\vdots\\
r_{i_{k_n}}^- (0, \cdot)
\end{smallmatrix}\right], & n \in \mathcal{N}_M\\
r_n^+(\ell_n, \cdot), &n \in  \mathcal{N}_S \setminus \{0\}\\
r_1^-(0, \cdot), & n = 0.
\end{cases}
\end{equation}

\section{Quaternions and differential equations}

\label{ap:quaternion}



Consider the following quaternion $\mathbf{q}\in \mathbb{R}^4$ with real value $q_0 \in \mathbb{R}$ and vectorial value $q \in \mathbb{R}^3$:
\begin{align} \label{eq:notation_quaternion}
\mathbf{q} = \begin{bmatrix}
q_0 \\ q
\end{bmatrix}.
\end{align}
Its \emph{conjugate} $\mathbf{q}^*$ and the \emph{quaternion product} $\circ$ are defined by
\begin{align*}
\mathbf{q}^* = \begin{bmatrix}
q_0 \\ -q
\end{bmatrix}, \qquad \mathbf{q} \circ \mathbf{w}= \begin{bmatrix}
w_0q_0 -  \langle w \,, q \rangle \\
w_0 q + q_0 w + \widehat{q}\, w
\end{bmatrix} =\begin{bmatrix}
q_0 & - q^\intercal \\ q & q_0 \mathbf{I}_3 + \widehat{q}
\end{bmatrix} \mathbf{w},
\end{align*}
for all $ \mathbf{q}, \mathbf{w} \in \mathbb{R}^4$ where we use an analogous notation to \eqref{eq:notation_quaternion} for $\mathbf{w}$.
Note that $\mathbf{q} \circ \mathbf{q}^* = (|\mathbf{q}|^2, 0, 0, 0)^\intercal$.
As already explained in Section \ref{sec:1b_invert_transfo}, 
%
$\mathbf{R}\colon [0, \ell]\times [0, T] \rightarrow \mathrm{SO}(3)$ is parametrized by the quaternion-valued function $\mathbf{q}\colon [0, \ell]\times [0, T] \rightarrow \mathbb{R}^4$ if $|\mathbf{q}| \equiv 1$ and \eqref{eq:rot_param_quat} is fulfilled for all $(x,t) \in [0, \ell]\times [0, T]$.
Below, we make use the following identities:
\begin{align} \label{eq:cross_prod_identities}
\widehat{ ( \widehat{u} v ) } &= \widehat{u}\widehat{v} - \widehat{v}\widehat{u}, \qquad 
\widehat{u}\widehat{v} = vu^\intercal + \langle u \,, v \rangle \mathbf{I}_3,
\end{align}
which hold for all $u, v \in \mathbb{R}^3$.

\medskip

\noindent Lemma \ref{lem:quat_rot_ODE} allows us to rewrite a linear PDE system whose unknown $\mathbf{R}$ has values in $\mathrm{SO}(3)$, as another linear PDE system whose unknown state is the quaternion-valued map $\mathbf{q}$ which parametrizes $\mathbf{R}$.
Recall that the function $\mathcal{U}$ is defined by \eqref{eq:def_calU}. In fact, by definition of the quaternion product, $\partial_z \mathbf{q} = \mathcal{U}(f) \mathbf{q}$, in \eqref{eq:zODE_quat_U}, is just an equivalent way of writing the following equivalent equations
\begin{align} \label{eq:ODE_quat_circ}
\partial_z \mathbf{q} = \frac{1}{2} \mathbf{q} \circ \begin{bmatrix}
0 \\ f
\end{bmatrix} \qquad \text{and} \qquad \begin{bmatrix}
0 \\ f
\end{bmatrix} = \frac{2}{|\mathbf{q}|^2} \mathbf{q}^* \circ \partial_z q.
\end{align}
Note also that, in itself, \eqref{eq:zODE_quat_U} implies that $\partial_z (|\mathbf{q}|^2) \equiv 0$, since it yields $\partial_z (|\mathbf{q}|^2) = 2 \langle \mathbf{q} \,, \mathcal{U}(f) \mathbf{q} \rangle$ and straightforward computations yield that the right-hand side is null.
Then, the proof unfolds as follows.

\begin{proof}[Proof of Lemma \ref{lem:quat_rot_ODE}]
We will use the fact that a rotation matrix $\mathbf{R}$ parametrized by a quaternion $\mathbf{q}$ also writes as $\mathbf{R} = EG^\intercal$ where $E, G \colon [0, \ell]\times[0, +\infty) \rightarrow \mathbb{R}^{3\times 4}$ are defined by
\begin{align*}
E = \begin{bmatrix}
- q & q_0 \mathbf{I}_3 + \widehat{q}
\end{bmatrix} ,\qquad G = \begin{bmatrix}
- q & q_0 \mathbf{I}_3 - \widehat{q}
\end{bmatrix}.
\end{align*}
Assume that $|\mathbf{q}| \equiv 1$ and \eqref{eq:zODE_quat_U} hold, and let $\mathbf{R}$ be the rotation parametrized by $\mathbf{q}$. For any quaternion $\mathbf{w} \in \mathbb{R}^4$, by definition of the quaternion product and of $G$, one has
\begin{align} \label{eq:qstar_circ_qdot_circ}
\mathbf{q}^* \circ (\partial_z \mathbf{q}) \circ \mathbf{w}
&= \begin{bmatrix}
\mathbf{q}^\intercal \\ G
\end{bmatrix}\begin{bmatrix}
\partial_z \mathbf{q} & \partial_z G^\intercal
\end{bmatrix}\mathbf{w}
= \begin{bmatrix}
\mathbf{q}^\intercal (\partial_z \mathbf{q}) & \mathbf{q}^\intercal (\partial_z G)^\intercal \\
G (\partial_z \mathbf{q}) & G (\partial_z G)^\intercal
\end{bmatrix}\mathbf{w}.
\end{align}
Since \eqref{eq:ODE_quat_circ} and \eqref{eq:qstar_circ_qdot_circ} hold, one has for all $v \in \mathbb{R}^3$
\begin{align*}
\begin{bmatrix}
0\\f
\end{bmatrix} \circ \begin{bmatrix}
0 \\v 
\end{bmatrix} = 2\begin{bmatrix}
\mathbf{q}^\intercal (\partial_z \mathbf{q}) & \mathbf{q}^\intercal (\partial_z G)^\intercal \\
G (\partial_z \mathbf{q}) & G (\partial_z G)^\intercal
\end{bmatrix}  \begin{bmatrix}
0 \\ v
\end{bmatrix}.
\end{align*}
In particular, the last three equations of this system read $\widehat{f}v = 2 G (\partial_z G)^\intercal v$. Hence,
\begin{align*}
\mathbf{R} \widehat{f} = 2 E G^\intercal G (\partial_z G)^\intercal,
\end{align*}
where we used that $\mathbf{R}  =EG^\intercal$.
The definition of $E,G$, the second identity in \eqref{eq:cross_prod_identities} and the fact that $|\mathbf{q}| \equiv 1$, lead via straightforward computations to the identities $G^\intercal G = (\mathbf{I}_4 - \mathbf{q}\mathbf{q}^\intercal)$, $E \mathbf{q} = 0$ and $\partial_z \mathbf{R} = 2 E (\partial_z G)^\intercal$. Using this, we deduce that
\begin{align*}
\mathbf{R} \widehat{f} &= 2 \big(E (\partial_z G)^\intercal - E\mathbf{q}\mathbf{q}^\intercal (\partial_z G)^\intercal\big) \\
&= 2E (\partial_z G)^\intercal \\
&= \partial_z \mathbf{R},
\end{align*}
which concludes the first part of the proof.

Let us now prove the converse. For $\mathbf{R}$ fulfilling \eqref{eq:zODE_R}, let $\mathbf{q}$ be the quaternion of unit norm $|\mathbf{q}| \equiv 1$ that parametrizes $\mathbf{R}$. We know, from the computations done in the first part of the proof, that $\partial_z \mathbf{R} = \mathbf{R} \widehat{f}$ is equivalent to $\widehat{f} = 2 G (\partial_z G)^\intercal$. Using the definition of $G$, we first rewrite $G (\partial_zG)^\intercal$ as
\begin{align*}
G (\partial_zG)^\intercal
&= q(\partial_z q)^\intercal + q_0 (\partial_z q_0) \mathbf{I}_3 +q_0 \widehat{(\partial_z q)} - (\partial_zq_0) \widehat{q} - \widehat{q} \widehat{(\partial_zq)}.
\end{align*}
By \eqref{eq:cross_prod_identities}, one has the identity $-\widehat{ {\widehat{q}\, \partial_z q}} + \langle q \,, \partial_z q \rangle \mathbf{I}_3=-\widehat{q} \widehat{(\partial_zq)} + q(\partial_z q)^\intercal$. Then, we obtain
\begin{align*}
G (\partial_zG)^\intercal
&= -\widehat{\widehat{q}\, \partial_z q} + \langle \mathbf{q}\,, \partial_z \mathbf{q}\rangle \mathbf{I}_3 +q_0 \widehat{(\partial_z q)} - (\partial_zq_0) \widehat{q},
\end{align*}
which is equivalent to $\mathrm{vec}(G (\partial_zG)^\intercal) = -\widehat{q}(\partial_z q) + q_0 (\partial_z q) - (\partial_zq_0) q$, since $\langle \mathbf{q}\,, \partial_z \mathbf{q}\rangle \equiv \frac{1}{2} \partial_z(|\mathbf{q}|^2)$ is identically equal to zero. We have obtained
\begin{align} \label{eq:zODE_R_to_q}
\frac{1}{2} f = -\widehat{q}(\partial_z q) + q_0 (\partial_z q) - (\partial_zq_0) q.
\end{align}
However, by definition of the quaternion product, the quantity $\mathbf{q}^* \circ (\partial_z\mathbf{q})$ also reads
\begin{align*}
\mathbf{q}^*\circ (\partial_z\mathbf{q})
= \begin{bmatrix}
0 \\
-(\partial_z q_0)q + q_0(\partial_zq) - \widehat{q}(\partial_z q)
\end{bmatrix}.
\end{align*}
Using \eqref{eq:zODE_R_to_q} and the fact that $\mathbf{q}$ has constant unit norm, we obtain \eqref{eq:ODE_quat_circ}.
\end{proof}

Lemma \ref{lem:quat_overter_syst} yields the existence of a unique solution to a seemingly overdetermined system of first-order linear PDEs. Its proof unfolds as follows.


\begin{proof}[Proof of Lemma \ref{lem:quat_overter_syst}]
There exists a unique solution $w \in C^1([0, \ell]; \mathbb{R}^n)$ to 
\begin{align} \label{eq:ODE_w}
\left\{ \begin{aligned}
&\frac{\mathrm{d}}{\mathrm{d}x} w(x) = B(x, 0) w(x) &&\text{in } (0, \ell)\\
&w(0) = y_\mathrm{in}
\end{aligned}\right.
\end{align}
(see for example \cite[Sec. 2.1 and Th. 4.1.1]{vrabie2004}).
Furthermore, there exists a unique solution $y \in C^1([0, \ell] \times [0, T]; \mathbb{R}^n)$ to
\begin{align} \label{ODE_y_B}
\left\{\begin{aligned}
&\partial_t y(x,t) = A(x,t) y(x,t) &&\text{in }(0, \ell)\times (0, T)\\
&y(x, 0) = w(x) &&x \in (0, \ell).
\end{aligned}\right.
\end{align}
This follows, for instance, from a fixed point argument and the use of Gronwall's inequality, for the map $y \longmapsto w + \int_{0}^t (Ay)(\cdot, \tau) d \tau$ defined on $C^0([0, t]; C^1([0, \ell]; \mathbb{R}^n))$.
Since $w$ is the solution to \eqref{eq:ODE_w}, we know that $y$ fulfills
\begin{align*}
\partial_x y(x,0) = B(x,0) y(x, 0), \qquad \text{for all }x \in [0, \ell].
\end{align*}
Let us now show that $y$ also fulfills
\begin{align} \label{eq:exist_overdeter}
(\partial_t - A)(\partial_x y - By) = 0, \qquad \text{in }[0, \ell]\times [0, T].
\end{align}
We have $(\partial_t - A)(\partial_x y - By)
= (\partial_x \partial_t y - (\partial_t B)y - B (\partial_t y) - A \partial_x y + AB y)$.
Using that $y$ is solution to \eqref{ODE_y_B}, we may replace $\partial_t y$ by $Ay$ in the above equation, thus obtaining
\begin{align*}
(\partial_t - A)(\partial_x y - By)
&=  (\partial_x A) y + A (\partial_x y) - (\partial_t B)y - BA - A \partial_x y + AB y\\
&= (AB -BA + (\partial_x A)   - (\partial_t B))y,
\end{align*}
which, making use of the assumption $A B - B A + (\partial_x A) - (\partial_t B) = 0$, yields \eqref{eq:exist_overdeter}. Then, the function $u \colon [0, \ell]\times [0, T] \rightarrow \mathbb{R}^4$ defined by $u = \partial_x y - By$ is solution to \eqref{ODE_y_B} with initial condition $u(x, 0) = 0$.
Hence, $u \equiv 0$ in $[0, \ell]\times [0,T]$, implying that $y$ is solution to \eqref{eq:overdeter_y_AB}.
\end{proof}

\end{subappendices}

\chapter{Exponential stabilization}
\label{ch:stab}

\section{Main contributions and related works}


By means of \emph{quadratic Lyapunov functionals}, we investigate exponential stabilization problems, for a single beam and networks of beams without loops.
As explained in Chapter \ref{ch:wellposedness}, at first we consider the (network) system for beams described by the IGEB model. Indeed, our approach relies on the fact that the latter is a one-dimensional first-order hyperbolic system.
For such systems, in the case of a single system, Bastin and Coron \cite{BC2016, bastin2017exponential} have systematized the search of quadratic Lyapunov functionals, giving sufficient criteria for their existence on the system written in diagonal form. 

More precisely, it is sufficient to find a certain matrix-valued function -- which will be part of the definition of the Lyapunov functional -- fulfilling matrix inequalities that involve both the coefficients appearing in the equations and the boundary conditions (and thus, the feedback controls).
In view of this, after having chosen appropriate \emph{velocity feedback boundary controls} so as to make the energy of the beam decreasing, we studied the model's coefficients, the form of the beam energy, in order to prove the existence of such a functional.
Up to the best of our knowledge, this method of \cite{BC2016, bastin2017exponential}, has been applied to chemotaxis models or the Saint-Venant equations for instance, but not in the context of precise mechanical models for beams such as the IGEB model.

Then, we turn to \emph{tree-shaped} networks of beams for which one may extend the definition of the previous Lyapunov functional. Using a similar approach to \cite{BC2016, bastin2017exponential}, we are also faced with the task of finding a matrix-valued function fulfilling a series of matrix inequalities. Among these inequalities, in comparison with the single beam case, differences appear in those that concern the nodal conditions, due to the transmission conditions at multiple nodes.

We present the stabilization study both from the point of view of \eqref{eq:IGEBfb} or \eqref{eq:nIGEBfb} -- that is, the ``\emph{physical system}'' -- and that of the \emph{diagonalized system} \eqref{eq:GEBfb} or \eqref{eq:nIGEBfbR}, respectively. 
One may equivalently look to prove exponential stability from both perspectives, and each has its own advantages and drawbacks.
%
The latter is the one adopted in \ref{A:SICON} for a single beam with constant and diagonal mass and flexibility matrices, in which case the change of variable to the diagonal system is more explicit (see Remark \ref{rem:diag} \ref{remItem:diag_moreExplicit}). 
For a general linear-elastic material law, the former point of view conveniently spares us some extensive computations, this is why it is rather used at first in \ref{A:MCRF}.
However, from the diagonal point of view, the conditions on the boundary terms are more explicit (see Section \ref{sec:pov_diag_syst}).

More precisely, in the subsequent sections, we present the following.
\begin{itemize}

\item We start with the case of a single beam, proving in Theorem \ref{thm:1b_stabilization} the local exponential stability of the zero steady state for the $H^1$ and $H^2$ norms, when the beam is clamped at one end and the other end is under feedback control. 
This theorem is a combination of \ref{A:SICON} where the proof was initially developed for a beam with \emph{constant diagonal} mass and flexibility matrices $\mathbf{M}, \mathbf{C}$ and where the form of the velocity feedback was constrained to the case $K = \mathrm{diag}(\mu_1 \mathbf{I}_3, \mu_2 \mathbf{I}_3)$ for some feedback parameters $\mu_1, \mu_2>2$ to be chosen, and of \ref{A:MCRF} where (in particular) these results were extended to general linear constitutive laws and velocity feedback controls. 

\item In Theorem \ref{thm:stabilization}, for a \emph{star-shaped} network, we show that if velocity feedback controls are applied at \emph{all} external nodes, then the zero steady state is locally exponentially stable for the $H^1$ and $H^2$ norms.
This theorem, and the considerations in the following item, are based on \ref{A:MCRF}.

\item A natural ensuing question is the following. Does the stabilization result for the star-shaped network also holds for after removal of one of the controls (for instance, clamping one of the simple nodes)?
For the type of Lyapunov functional found here, the difficulty in removing one control lies in the estimation of the boundary terms (at the simple and multiple nodes) that appear when looking to obtain the exponential decay of the Lyapunov functional.
Though no positive conclusion is reached, Lemma \ref{lem:class_barQi_Riem} and Proposition \ref{prop:rel_Lyap_Q_barQ} render more explicit the conditions for achieving exponential stability (notably at multiple nodes), and thus also where the difficulty in removing one control comes from in our proof.

\item Then, as for the well-posedness study in Section \ref{sec:invert_transfo}, we make use of the transformation $\mathcal{T}$. For the single beam under feedback control, in Corollary \ref{coro:fromIGEBtoGEB}, we deduce the existence of a unique global in time solution to the corresponding GEB model, and properties of this solution (in terms of velocities and strains) as time goes to infinity. In fact, though in \ref{A:MCRF} we did not use the transformation to also deduce the corresponding result for the star-shaped network under feedback control, let us mention that such a corollary may also be derived in a similar manner -- by means of a result analogous to Theorem \ref{thm:solGEB} for a network with feedback controls applied at the nodes.
\end{itemize}

\subsection*{Stabilization of beams, beam networks and first-order hyperbolic systems}
%

Stabilization of beam equations by means of feedback boundary controls goes back to \cite{quinnrussel78} for the string, \cite{kimrenardy87} for the Timoshenko beam; see also \cite{do18, hegarty12, morgul91, xu2005} and the references therein for other linear and nonlinear beam models. 
As mentionned above, we focus on the Lyapunov approach to prove stability, and adopt at first the perspective of the IGEB model.

For one-dimensional first-order hyperbolic systems, such as \eqref{eq:IGEBfb}, several results of stabilization under boundary control are shown by means of quadratic Lyapunov functionals in \cite{BC2016} and the references therein. There, when the system does not have any lower order term such as $\overline{B}y$ and $\overline{g}(y)$ here (one then speaks of systems of \emph{conservation laws}), the exponential stability result may rely on the dissipativity of the boundary conditions alone.
However, when lower order terms are present (systems of \emph{balance laws}) the equations must also be taken into consideration. Some systems of nonlinear balance laws with a uniform steady state may be seen as systems of nonlinear conservation laws perturbed by the lower order terms: if the perturbation is small enough then the $C^1$-\,exponential stability is preserved, see \cite[Th. 6.1]{BC2016}.
See also \cite{gugat2018} for two by two quasilinear systems with small lower order terms. 
System \eqref{eq:IGEBfb} may have dissipative boundary conditions due to the presence of the feedback control, however as we have seen in Section \ref{sec:pres_IGEB} the perturbation ($\overline{B}y + \overline{g}(y)$ here) is not small in general because of the term $\overline{B}y$. Indeed, the linearized system is no homogeneous.
In that case, for general linear, semilinear and quasilinear systems, assumptions on both the boundary conditions and the system's coefficients are required in \cite[Pr. 5.1]{BC2016}, \cite[Th. 10.2]{bastin2017exponential}, \cite{hayat2018exponential} and \cite[Th. 6.10]{BC2016} for $L^2, H^1$, $C^1$ and $H^2$ exponential stability respectively.

\medskip

\noindent Numerous works have been carried out on the stabilization of tree-shaped networks of d'Alembert wave equations \cite{DagerZuazuaBook}, by means of velocity feedback controls applied at some nodes. 
In \cite{ValeinZuazua2009, Zuazua2012}, the control is located at a single simple node and stability properties (e.g. polynomial) are proved by making use of suitable observability inequalities. 
In \cite{NicaiseValein2007} the exponential stabilization is obtained by applying velocity feedback controls with delay at the multiple nodes.
In \cite{AlabauPerrollazRosier2015}, the authors apply transparent boundary conditions (see Remark \ref{rem:wellp_fb} \ref{remItem:transparent}) at all simple nodes in addition to velocity feedback controls at the multiple nodes, in order to obtain finite time stabilization. For the exponential stabilization of star-shaped networks using spectral methods, we refer to \cite{GuoXu2011} where the controls are applied at the multiple nodes and all simple nodes but one, and \cite{ZhangXu2013} where the controls are applied at all simple nodes but one. In \cite{GugatSigalotti2019}, the authors showed that, for star-shaped networks, finite time stability is achieved by applying velocity feedback controls at all simple nodes, and that exponential stability is still realized if one of the controls is removed from time to time.
Applying the controls at all simple nodes but one, \cite{LLS} also studies the exponential stabilization of tree-shaped networks of strings, as well as that of beams.

The stabilization of beam networks has also been considered by \cite{HanXu2011}, who applied time-delay controls at all the simple nodes of a star-shaped network of Timoshenko beams, to obtain exponential stability via spectral methods.
In \cite{ZhangXuMastorakis2009}, by means of semigroup theory and spectral analysis, the exponential stability of a tree-shaped network of Euler-Bernoulli beams is proved, when all simple nodes are clamped while velocity feedback controls are applied at the interior nodes. 
Also using spectral methods to study exponential stabilization, \cite{XuHanYung2007} considered serially connected Timoshenko beams, applying velocity feedback controls at all nodes except one simple node, while \cite{HanXu2010} considered a specific star-shaped network of Timoshenko beams, where velocity feedback controls are applied all simple nodes but one.

Stabilization problems for networks of first order hyperbolic systems have also been extensively studied, in particular for the Saint-Venant equations -- e.g. \cite{AlabauPerrollazRosier2015}, as well as \cite{coron2003} and \cite{LeugeringSchmidt2002} which both make use of the Li-Greenberg Theorem \cite[Chap. 5, Th. 1.3]{Li_blue_book} to obtain the exponential decay result.
Since the tree-shaped network system may be rewritten as a single hyperbolic system -- as done for example in \cite{Coron2007} where exponential stabilization is then proved by means of a Lyapunov functional -- the aforementioned literature on such systems \cite{BC2016,bastin2017exponential, gugat2018, hayat2018exponential, HertyYu2018} is also of interest here.

\section{Stabilization and energy of the beam}
\label{sec:stab_energy}

Below, we use the notation $\mathbf{H}_x^k$ which reduces to $H_x^k$ when a single beam is considered (see Section \ref{sec:notation}), as well as $y=(y_i)_{i\in \mathcal{I}}$ and $y^0=(y_i^0)_{i\in \mathcal{I}}$ which reduce to simply $y$ and $y^0$ for a single beam.
Both for a single beam and for networks, our interest is with the following notion of stabilization.

\begin{definition}[Local exponential stability]
\label{def:stability}
Let $k \in \{1, 2\}$ and let $(\mathrm{S})$ represent either \eqref{eq:IGEBfb} or \eqref{eq:nIGEBfb}. The \emph{steady state $y\equiv 0$ of $(\mathrm{S})$ is locally $\mathbf{H}^k_x$ exponentially stable} if there exist $\varepsilon>0$, $\beta>0$ and $\eta \geq 1$ such that the following holds. For any $y^0 \in \mathbf{H}^k_x$ such that $\|y^0\|_{\mathbf{H}^k_x}\leq \varepsilon$ and the $(k-1)$-order compatibility conditions hold, there exists a unique global in time solution $y \in C^0([0, +\infty); \mathbf{H}^k_x)$ to $(\mathrm{S})$, and
\begin{linenomath}
\begin{equation*}
\|y(\cdot, t)\|_{\mathbf{H}^k_x} \leq \eta  e^{- \beta t } \|y^0\|_{\mathbf{H}^k_x}, \qquad \text{for all } \, t \in [0, +\infty).
\end{equation*}
\end{linenomath}
\end{definition}

This notion is indeed \emph{local} since the initial datum has to be small enough for the resulting solution to have such properties as time goes to infinity. As suggested at the beginning of this chapter, our choice of feedback control has to do with the \emph{energy of the beam}.

\medskip

\noindent \textbf{In the physical system.} The energy of a freely vibrating beam, is by definition the sum of the \emph{kinetic} and \emph{elastic} energies, and -- for the kind of beam considered here -- it takes the form 
\begin{equation} \label{eq:1b_def_calEP}
\mathcal{E}^\mathcal{P}(t) = \int_0^{\ell} \langle y(x,t) \,, Q^\mathcal{P}(x) y(x,t) \rangle dx,
\end{equation}
for $y$ solution to \eqref{eq:IGEBfb}, and where $Q^\mathcal{P} = \mathrm{diag}(\mathbf{M}, \mathbf{C})$ was in fact already presented in \eqref{eq:def_boldE_calP}. 
The velocity feedback controls have been introduced in System \eqref{eq:IGEBfb} in such a way that the energy $\mathcal{E}^\mathcal{P}$ is dissipated. Indeed, one can check that $\frac{\mathrm{d}}{\mathrm{d}t}\mathcal{E}^\mathcal{P}(t) \leq 0$ for any $K \in \mathbb{S}_+^6$.

Here, the notation $\mathcal{P}$ stresses that we refer to System \eqref{eq:IGEBfb}, whose unknown state is the ``\emph{physical variable}'' $y$, as opposed to the ``\emph{diagonal variable}'' $r$ for the system written in Riemann invariants.
For networks similar considerations hold, one can check that for any positive semi-definite matrices $K_n$ ($n \in \mathcal{N}$) the energy $\mathcal{E}_\mathrm{net}^\mathcal{P}$ of the beam network, defined by
\begin{equation*} 
\mathcal{E}_\mathrm{net}^\mathcal{P} = \sum_{i \in \mathcal{I}} \mathcal{E}_i^\mathcal{P}, \quad \text{with} \quad \mathcal{E}_i^\mathcal{P} = \int_0^{\ell_i} \langle y_i \,, Q_i^\mathcal{P} y_i \rangle dx
\end{equation*}
for $y = (y_i)_{i \in \mathcal{I}}$ solution to \eqref{eq:nIGEBfb} and $Q_i^\mathcal{P} = \mathrm{diag}(\mathbf{M}_i, \mathbf{C}_i)$, satisfies $\frac{\mathrm{d}}{\mathrm{d}t}\mathcal{E}_\mathrm{net}^\mathcal{P}(t) \leq 0$.

\medskip

\noindent \textbf{In the diagonal system.} 
The physical and diagonal systems are related by the change of variable $r = Ly$ (with $L$ defined by \eqref{eq:def_bfD_L}), or \eqref{eq:change_var_Li} for the network. Hence, the energy $\mathcal{E}^\mathcal{D}$ of a beam described by the diagonalized system takes the form  
\begin{equation} \label{eq:1b_def_calED}
\mathcal{E}^\mathcal{D}(t) = \int_0^{\ell} \langle
r(x,t) \,, Q^\mathcal{D}(x) r(x,t) \rangle dx,
\end{equation}
for $r = (r_i)_{i \in \mathcal{I}}$ solution to \eqref{eq:IGEBfbR}, and where $Q^\mathcal{D} = (L^{-1})^\intercal Q^\mathcal{P} L^{-1}$; one may compute that $Q^\mathcal{D} = \frac{1}{2}\mathrm{diag}\left(D^{-2}, D^{-2} \right)$.
Just as $\mathcal{P}$ refers to the physical system, here the subscript $\mathcal{D}$ refers to the diagonal system. For networks, it will take the form
\begin{align*} 
\mathcal{E}_\mathrm{net}^\mathcal{D} = \sum_{i \in \mathcal{I}} \mathcal{E}_i^\mathcal{D}, \quad \text{with} \quad \mathcal{E}_i^\mathcal{D} = \int_0^{\ell} \langle
r_i \,, Q_i^\mathcal{D} r_i \rangle dx,
\end{align*}
for $r_i$ solution to \eqref{eq:nIGEBfbR}, where $Q_i^\mathcal{D} = \frac{1}{2}\mathrm{diag}\left(D_i^{-2}, D_i^{-2} \right)$. Here also, the energy is nonincreasing if the matrices $K$ or $K_n$ are positive semi-definite.

\section{For a single beam}
\label{sec:stab_1b}

As in the previous chapter, in order to divide the difficulty in two parts, we first present the result for a single beam, Theorem \ref{thm:1b_stabilization} below, and afterwards for the network (see Section \ref{sec:stab_net}). 

\begin{theorem}[Th. 1.5 \ref{A:SICON}, Th. 2.4 \ref{A:MCRF}] \label{thm:1b_stabilization}
Let $k \in \{1, 2\}$, suppose that the beam parameters $(\mathbf{M}, \mathbf{C})$ satisfy Assumption \ref{as:mass_flex} with $m=k+1$, and that $\mathbf{E} \in C^k([0, \ell]; \mathbb{R}^{6 \times 6})$. If $K \in \mathbb{S}_{++}^6$, then the steady state $y \equiv 0$ of \eqref{eq:IGEBfb} is locally $H^k_x$ exponentially stable.
\end{theorem}

Let $k \in \{1, 2\}$.
To prove Theorem \ref{thm:1b_stabilization}, in the spirit of Bastin and Coron \cite{BC2016, bastin2017exponential}, we look for a so-called \emph{quadratic $H^k_x$ Lyapunov functional}. Namely, a functional $\overline{\mathcal{L}}\colon [0, T] \rightarrow [0, +\infty)$ of the form
\begin{equation} \label{eq:1b_def_calLP}
\overline{\mathcal{L}} = \sum_{\alpha=0}^k \overline{\mathcal{L}}_{\alpha}, \quad \text{with} \ \ \overline{\mathcal{L}}_{\alpha} = \int_0^{\ell} \left\langle \partial_t^\alpha y \,, \overline{Q} \partial_t^\alpha y \right\rangle dx,
\end{equation}
where $\overline{Q} \in C^1([0, \ell_i]; \mathbb{S}^{12})$ and $y \in C^0([0, T]; H_x^k)$ is solution to \eqref{eq:IGEBfb}, such that \emph{when $y$ is in some ball of $C_t^0 C_x^{k-1}$, then $\overline{\mathcal{L}}(t)$ is equivalent to the squared $H_x^k$ norm of $y(\cdot, t)$ and has an exponential decay with respect to time}. In other words, the Lyapunov functional should fulfill the assumptions of the following proposition due to \cite{BC2016, bastin2017exponential}.

\begin{proposition} \label{prop:1b_existence_Lyap}
Let $k \in \{1, 2\}$. Assume that for any fixed $T>0$, there exists $\delta>0$, $\eta \geq 1$ and $\beta >0$ such that for any $y \in C^0([0, T]; {H}^k_x)$ solution to \eqref{eq:IGEBfb} and fulfilling $\|y(\cdot, t)\|_{{C}^{k-1}_{x}} \leq \delta$ for all $t \in [0, T]$,
\begin{subequations}
\begin{align}
\eta^{-1} \|y(\cdot, t)\|_{{H}^k_x}^2 &\leq \overline{\mathcal{L}}(t) \leq \eta \|y(\cdot, t)\|_{{H}^k_x}^2,  &&\text{for all } t \in [0, T] \label{eq:lyap_equiv}\\
\overline{\mathcal{L}}(t) &\leq e^{-2 \beta t} \overline{\mathcal{L}}(0), &&\text{for all }t \in [0, T] \label{eq:lyap_decay}
\end{align}
\end{subequations}
both hold. Then, the steady state $y\equiv 0$ of \eqref{eq:IGEBfb} is locally ${H}^k_x$ exponentially stable.
\end{proposition}

The same proposition holds if instead of taking the point of view of the physical system \eqref{eq:IGEBfb}, we rather take that of the diagonal system \eqref{eq:IGEBfbR}. One just has to replace, in Proposition \ref{prop:1b_existence_Lyap}, the system \eqref{eq:IGEBfb} by \eqref{eq:IGEBfbR}, the state $y$ by $r$, and the functional $\overline{\mathcal{L}}$ by
\begin{align} \label{eq:1d_def_calLD}
{\mathcal{L}} = \sum_{\alpha=0}^k {\mathcal{L}}_{\alpha}, \quad \text{with} \ \ {\mathcal{L}}_{\alpha} = \int_0^{\ell} \left\langle \partial_t^\alpha r \,, {Q} \partial_t^\alpha r \right\rangle dx,
\end{align}
where $r$ is solution to \eqref{eq:IGEBfbR}. As for the energy, if $Q$ and $\overline{Q}$ are such that
\begin{equation} \label{eq:rel_Q_barQ}
Q = (L^{-1})^\intercal \overline{Q} L^{-1},
\end{equation}
then $\mathcal{L}$ and $\overline{\mathcal{L}}$ are equivalent expressions.
One can make more explicit the task of finding the Lyapunov functional via the following lemma.

\begin{lemma}[Lem. 5.1 \ref{A:MCRF}] \label{lem:class_barQ}
If there exists $\overline{Q} \in C^1([0, \ell]; \mathbb{R}^{12 \times 12})$ fulfilling
\begin{enumerate}[label=(\roman*)]
\item \label{pty:barQ_symm_posDef}
$\overline{Q}(x) \in \mathbb{S}_{++}^{12}$ for all $x \in [0, \ell]$,
\item \label{pty:barQA_symm}
$(\overline{Q}A)(x) \in \mathbb{S}^{12}$ for all $x \in [0, \ell]$,
\item \label{pty:barS_negDef}
$\overline{S}(x) \in \mathbb{S}_{--}^{12}$ for all $x \in [0, \ell]$, where $\overline{S} := \frac{\mathrm{d}}{\mathrm{d}x}( \overline{Q} A ) - \overline{Q}\,\overline{B} - \overline{B}^\intercal\overline{Q}$,
\item \label{pty:bt0ell_nonPos}
for all $y \in C^0([0, T]; C^0_x)$ satisfying \eqref{eq:IGEBfb_BC0}-\eqref{eq:IGEBfb_BCell}, and all $t \in [0, T]$,
\begin{equation} \label{eq:1b_boundaryTerms_barQ}
 \left\langle y \,, \overline{Q} A y \right\rangle (0, t) - \left\langle y \,, \overline{Q} A y\right\rangle (\ell, t)\leq 0,
\end{equation}
\end{enumerate}
\vspace{-0.2cm}
then the associated $\overline{\mathcal{L}}$ $($see \eqref{eq:1b_def_calLP}$)$ fulfills the assumptions of Proposition \ref{prop:1b_existence_Lyap}.
\end{lemma}

\noindent \textbf{Idea of the proof.} The arguments of this proof are following that of \cite{BC2016, bastin2017exponential} where the authors work with systems in diagonal form. We provide them here for the sake of completeness and to illustrate where each assumption -- in particular \ref{pty:barQA_symm} -- comes from (see \ref{A:MCRF} for more detail). 
\begin{itemize}
\item \textbf{Equivalence of the norms.} The assumption \ref{pty:barQ_symm_posDef} is of use to obtain the equivalence between $\overline{\mathcal{L}}_\alpha(t)$ and the $L_x^2$ norm of $\partial_t^\alpha y(\cdot, t)$, modulo a constant depending on $\overline{Q}$. Complemented with the fact that $\sum_{\alpha=0}^k|\partial_t^\alpha y(\cdot, t)|^2$ is itself equivalent to $\sum_{\alpha=0}^k|\partial_x^\alpha y(\cdot, t)|^2$ modulo a constant depending on the beam parameters and $\|y\|_{C_{x,t}^0}$, this yields \eqref{eq:lyap_equiv}.

\item \textbf{Exponential decay.} The second assumption \ref{pty:barQA_symm} is instrumental in obtaining the following expression of the derivative of the Lyapunov functional
\begin{align} \label{eq:def_calLP}
\begin{aligned}
\frac{\mathrm{d}}{\mathrm{d}t}\overline{\mathcal{L}} 
&=  \sum_{\alpha=0}^k \int_0^{\ell} \Big\langle \partial_t^\alpha y \,, \overline{S} \partial_t^\alpha y \Big\rangle dx + \sum_{\alpha=0}^k \int_0^{\ell} \Big\langle \partial_t^\alpha y \,,2\overline{Q} \partial_t^\alpha \big(\overline{g}(\cdot, y)\big) \Big\rangle dx \\
&\quad +\sum_{\alpha=0}^k  \Big[- \left\langle \partial_t^\alpha y \,, \overline{Q} A \partial_t^\alpha y \right\rangle \Big]_0^{\ell}.
\end{aligned}
\end{align}
We may neglect the boundary terms in \eqref{eq:def_calLP} by means of \ref{pty:bt0ell_nonPos}.
Due to the form of $\overline{g}$, the second term in the right-hand side of \eqref{eq:def_calLP} is bounded above by $(\delta + \delta^2)\|\sum_{\alpha=0}^k \partial_t^\alpha y(\cdot, t)\|_{L_x^2}^2$, modulo a constant depending on $\overline{Q}$ and the beam parameters. 
The third assumption \ref{pty:barS_negDef} then comes into play as we can balance between the largest (negative) eigenvalue of $\overline{S}$ and the size of $\delta$ to deduce an inequality of the form $\frac{\mathrm{d}}{\mathrm{d}t}\overline{\mathcal{L}} \leq -2 \beta \overline{\mathcal{L}}$, thus yielding \eqref{eq:lyap_decay}. \hfill $\meddiamond$
\end{itemize}

Such arguments applied in the context of the diagonal system yield the following lemma.

\begin{lemma}[Lem. 5.3 \ref{A:MCRF}] \label{lem:class_Q}
If there exists $Q \in C^1([0, \ell]; \mathbb{R}^{12 \times 12})$ fulfilling
\begin{enumerate}[label=(\roman*)]
\item 
$Q(x) \in \mathbb{D}_{++}^{12}$ for all $x \in [0, \ell]$,
\item 
$\overline{S}(x) \in \mathbb{S}_{--}^{12}$ for all $x \in [0, \ell]$, where $\overline{S} := \frac{\mathrm{d}}{\mathrm{d}x}( Q\mathbf{D} ) - QB - B^\intercal Q$,
\item \label{pty:bound_terms} 
denoting $Q = \mathrm{diag}(Q^-, \ Q^+)$ with $Q^-, Q^+ \in C^1([0, \ell]; \mathbb{D}^6)$, the following matrices are negative semi-definite
\begin{equation*}
Q^+(0)D(0)-Q^-(0), \qquad Q^-(\ell) D(\ell) - Q^+(\ell), 
\end{equation*}
\end{enumerate}
then steady state $r\equiv 0$ of \eqref{eq:IGEBfbR} is locally ${H}^k_x$ exponentially stable.
\end{lemma}

The main difference with Lemma \ref{lem:class_barQ} is that the constraints on the boundary terms (\ref{pty:bound_terms} in Lemma \ref{lem:class_Q}) are here more explicit, as they do not involve the state anymore. We may now present the key steps of the proof of the stabilization result Theorem \ref{thm:1b_stabilization}.

\medskip



\noindent \textbf{Idea of the proof of Theorem \ref{thm:1b_stabilization}.} 
From the perspective of the physical system, it is sufficient to look for a map $\overline{Q}$ fulfilling all assumptions of Lemma \ref{lem:class_barQ}.
\begin{itemize}
\item \textbf{Step 1: Ansatz for $\overline{Q}$.}
We use the ``\emph{energy matrix}'' $Q^\mathcal{P}$ that characterizes $\mathcal{E}^\mathcal{P}$ in \eqref{eq:1b_def_calEP}, multiplying this matrix by a constant $\rho \in \mathbb{R}$, and adding extradiagonal terms $\mathbf{W} \in C^1([0, \ell]; \mathbb{R}^{6 \times 6})$, themselves multiplied by a \emph{weight function} $w \in C^1([0, \ell])$:
\begin{equation*} 
\overline{Q} = \rho Q^\mathcal{P} + w \begin{bmatrix}
\mathbf{0} & \mathbf{W} \\ \mathbf{W}^\intercal & \mathbf{0}
\end{bmatrix}.
\end{equation*}
This allows us to make use of the fact that $Q^\mathcal{P} A$ is constant and $Q^\mathcal{P} \overline{B}$ skew-symmetric, so that $\overline{S}$ reduces to the sum of two matrices (the latter possibly indefinite)
\begin{linenomath}
\begin{equation*}
\overline{S} = - \frac{\mathrm{d}w}{\mathrm{d}x} \Lambda + |w| \Xi
\end{equation*}
\end{linenomath}
where 
\begin{align} \label{eq:def_Lambda_Xi}
\begin{aligned}
\Lambda &:= \begin{bmatrix}
\Lambda^\mathrm{I} & \mathbf{0}\\
\mathbf{0}& \Lambda^\mathrm{II}
\end{bmatrix}, \qquad \text{with} \quad \Lambda^\mathrm{I} := \mathbf{W} \mathbf{C}^{-1}, \quad \Lambda^\mathrm{II} := \mathbf{W}^\intercal \mathbf{M}^{-1},\\
\Xi &:= \mathrm{sign}(w) 
\begin{bmatrix}
\Lambda^\mathrm{I}\mathbf{E}^\intercal + (\Lambda^\mathrm{I}\mathbf{E}^\intercal)^\intercal - \tfrac{\mathrm{d}}{\mathrm{d}x} \Lambda^\mathrm{I}  & \mathbf{0}\\
\mathbf{0} & - \Lambda^\mathrm{II} \mathbf{E} - (\Lambda^\mathrm{II} \mathbf{E})^\intercal  - \tfrac{\mathrm{d}}{\mathrm{d}x}\Lambda^\mathrm{II}
\end{bmatrix}.
\end{aligned}
\end{align}

\item \textbf{Step 2: Constraints on $\mathbf{W}$.}
Exemples extradiagonal terms such that not only \ref{pty:barQA_symm} is fulfilled, but also $\Lambda(x) \in \mathbb{S}^{12}_{++}$ for all $x \in [0, \ell]$, are
\begin{subequations} \label{eq:boldW}
\begin{align}
\mathbf{W} &= \mathbf{I}_6, \label{eq:boldW_I}\\
\mathbf{W} &= \mathbf{M} \mathbf{C}, \label{eq:boldW_MC}\\
\label{eq:boldW_MCfrac}
\mathbf{W} &= \mathbf{C}^{-\sfrac{1}{2}}(\mathbf{C}^{\sfrac{1}{2}}\mathbf{M}\mathbf{C}^{\sfrac{1}{2}})^{\sfrac{1}{2}} \mathbf{C}^{\sfrac{1}{2}},
\end{align}
\end{subequations}
where the latter \eqref{eq:boldW_MCfrac} also writes as $\mathbf{W} = \mathbf{M}^{\sfrac{1}{2}}\mathbf{C}^{\sfrac{1}{2}}$ for commuting $\mathbf{M}, \mathbf{C}$.

\begin{figure}\centering
\includegraphics[scale=0.9]{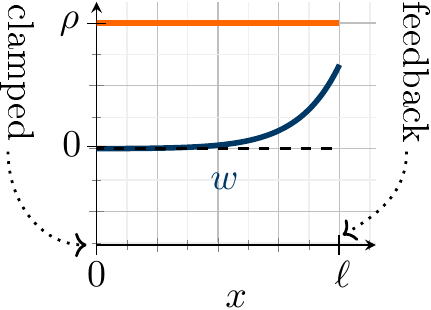}\hspace{1cm}
\includegraphics[scale=0.9]{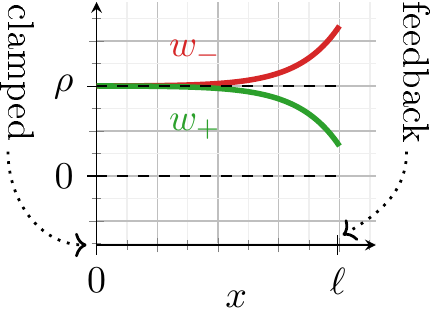}
\caption{Plots of $\rho$ and $w=q-q(0)$ obtained via \eqref{eq:def_q_pos}, and corresponding $w_- = \rho+w$ and $w_+ = \rho-w$. We used the constants $a=0$, $b=1$, $\eta=5$, $\rho=1.5$ and $\ell = 1$.}
\label{fig:1b_weights_tikz}
\end{figure}

\item \textbf{Step 3: Constraints on $w$ and $\rho$.}
We require $w$ to be increasing to render $-(\frac{\mathrm{d}w}{\mathrm{d}x} \Lambda)(x)$ negative definite for all $x \in [0, \ell]$. Then, denoting by $C_{\Lambda} >0$ and $C_{\Xi} \geq 0$ the maximum over $[0, \ell]$ of the smallest eigenvalue of $\Lambda(x)$ and largest eigenvalue of $\Xi(x)$, respectively, we require that
\begin{equation} \label{eq:cond_barSi_neg}
\frac{\mathrm{d}w}{\mathrm{d}x}> C_{\Xi} C_{\Lambda}^{-1} |w|,
\end{equation}
for \ref{pty:barSi_negDef} to hold \cite[Th. 4.3.1, Coro. 4.3.15]{horn2012matrix}.
Finally, the use of the Schur complement of $\rho \mathbf{C}$ yields that \ref{pty:barQi_symm_posDef} holds if and only if
\begin{equation*}
\rho > |w(x)| \sqrt{\lambda_{\theta(x)}}, \qquad \text{for all } x \in [0, \ell],
\end{equation*}
where $\lambda_{\theta(x)}$ denotes the largest eigenvalue of the matrix $\theta(x) = (\mathbf{M}^{-\sfrac{1}{2}} \mathbf{W} \mathbf{C}^{-1} \mathbf{W}^\intercal \mathbf{M}^{-\sfrac{1}{2}})(x)$.
Note that if $\mathbf{W}$ is defined by \eqref{eq:boldW_MCfrac} then $\theta(x) = \mathbf{I}_6$ and $\lambda_{\theta}(x) = 1$ for all $x \in [0, \ell]$, while in general a sufficient condition for the above to hold is that $|w| < \rho \sqrt{ C_{\theta}}$ in $[0 ,\ell]$ for $C_{\theta} := \max_{x \in [0, \ell]} \lambda_{\theta(x)}$.

%
%

\item \textbf{Step 4: Estimating the boundary terms.}
While the extradiagonal (block) terms are of help to make $\overline{S}(x)$ negative definite, they are also present in the boundary terms, thereby constraining further the choice of $w$ and $\mathbf{W}$. 
By definition $\langle y \,, \overline{Q} A y \rangle = -2\rho \langle v \,, z \rangle - w \langle v\,, \Lambda^\mathrm{I} v \rangle - w \langle z \,, \Lambda^\mathrm{II} z\rangle $. Thus, the boundary conditions, yield that \eqref{eq:1b_boundaryTerms_barQ} is equal to
\begin{align} \label{eq:boundary_terms_physical}
- \left \langle z \,, w \Lambda^\mathrm{II} z\right \rangle (0, t) +  \left \langle v \,, (- 2 \rho K + w \Lambda^\mathrm{I} + w K \Lambda^\mathrm{II} K) v\right \rangle (\ell, t).
\end{align}
The first term may be removed by assuming that $w(0)\geq 0$. Then, one can notice the importance of supposing that $K \in \mathbb{S}_{++}^6$. Indeed, the second term now writes as $\left\langle v, \mu(\ell, K) v \right \rangle (\ell, t)$ where $\mu \colon [0, \ell] \times \mathbb{S}_{++}^6 \rightarrow \mathbb{R}^{6\times 6}$ is defined by
\begin{linenomath}
\begin{equation} \label{eq:def_mu(x,K)}
\mu(x,K) = (-2\rho K  + |w| \Lambda^\mathrm{I} + |w| K \Lambda^\mathrm{II} K)(x),
\end{equation}
\end{linenomath}
and is in fact negative semi-definite if $\rho$ is large enough in comparison to $|w(x)|$. In more precise terms, it is sufficient that $|w(\ell)| \leq \rho C_{\mu(\ell,K)}^{-1}$ where $C_{\mu(x,K)}$ denotes the largest eigenvalue of $(K^{-\sfrac{1}{2}} \Lambda^\mathrm{I} K^{-\sfrac{1}{2}} + K^{\sfrac{1}{2}} \Lambda^\mathrm{II} K^{\sfrac{1}{2}})(x)$.

\item \textbf{Step 5: Existence of the weights.}
Overall, for $\mathbf{W}$ as in \eqref{eq:boldW}, one should find an increasing and nonnegative function $w\in C^1([0, \ell])$ and $\rho>0$ fulfilling \eqref{eq:cond_barSi_neg} and
\begin{equation*} 
w(\ell) < \chi \rho , \qquad \text{with } \chi = \min \left\{C_{\theta}^{-\sfrac{1}{2}}, C_{\mu(\ell,K)}^{-1}\right\}.
\end{equation*}
This is achieved by means of the following lemma.

\begin{lemma}[Lem. 5.2 \ref{A:MCRF}] \label{lem:exist_g}
\vspace{7pt}
Let $\eta>0$, $\ell>0$ be fixed. 
\begin{enumerate}[label=\alph*)]
\item \label{item:exist_g_neg}
For any choice of constants $a < b\leq 0$, there exists $q\in C^\infty([0, \ell])$ such that $q(0) = a$, $q(\ell) = b$ and $\frac{\mathrm{d}}{\mathrm{d}x}q(x) > (q(\ell)-q(x)) \eta$, for all $x \in [0, \ell]$.
\item \label{item:exist_g_pos}
For any choice of constants $0 \leq a < b$, there exists $q\in C^\infty([0, \ell])$ such that $q(0) = a$, $q(\ell) = b$ and $\frac{\mathrm{d}}{\mathrm{d}x}q(x) > (q(x) - q(0))\eta$, for all $x \in [0, \ell]$.
\end{enumerate}
\end{lemma}
\vspace{-7pt}
Using the item \ref{item:exist_g_pos}, for any $\rho>0$, with $\eta = C_{{\Xi}} C_{\Lambda}^{-1}$ and with $a,b$ such that $0\leq a<b<(a+ \chi \rho )$, one sets $w = q - q(0)$ (see Fig. \eqref{fig:1b_weights_tikz} (Left)). \hfill $\meddiamond$
\end{itemize}

One may use similar ideas to prove Theorem \ref{thm:1b_stabilization} from the point of view of the diagonal system, but with a different Ansatz for $Q$. Indeed, some computations will yield that the function $\overline{Q}$ found in the above proof takes the following form in the diagonal system (see \eqref{eq:rel_Q_barQ}):
\begin{align} \label{eq:corresponding_Q}
Q = \begin{cases}
\mathrm{diag}\big(\rho \mathbf{I}_6 + w D \,, \rho \mathbf{I}_6 - w D \big) Q^\mathcal{D} &\text{if }\eqref{eq:boldW_I}\\
\mathrm{diag}\big(\rho \mathbf{I}_6 + w D^{-1} \,, \rho \mathbf{I}_6 - w D^{-1} \big) Q^\mathcal{D} &\text{if }\eqref{eq:boldW_MC}\\
\mathrm{diag} \big ((\rho + w)\mathbf{I}_6 \,, (\rho - w) \mathbf{I}_6 \big )  Q^\mathcal{D} &\text{if }\eqref{eq:boldW_MCfrac},
\end{cases}
\end{align}
where $Q^\mathcal{D}$ is the matrix characterizing the energy \eqref{eq:1b_def_calED} in the diagonal system.
In \ref{A:SICON}, where we adopt the perspective of the diagonal diagonal system, we use the Ansatz $Q(x) = \mathrm{diag} \big (w_-\mathbf{I}_6 \,, w_+ \mathbf{I}_6 \big ) Q^\mathcal{D}$, 
where the ``energy matrix'' $Q^\mathcal{D}$ is multiplied by some weights $w_-, w_+ \in C^1([0, \ell])$. 
Comparing with the third expression in \eqref{eq:corresponding_Q}, the following link between the weights appears (see Fig. \ref{fig:1b_weights_tikz} (Right)):
\begin{align*}
w_- = \rho + w, \qquad  w_+ = \rho - w.
\end{align*}

\begin{figure}\centering
\includegraphics[height = 4.15cm]{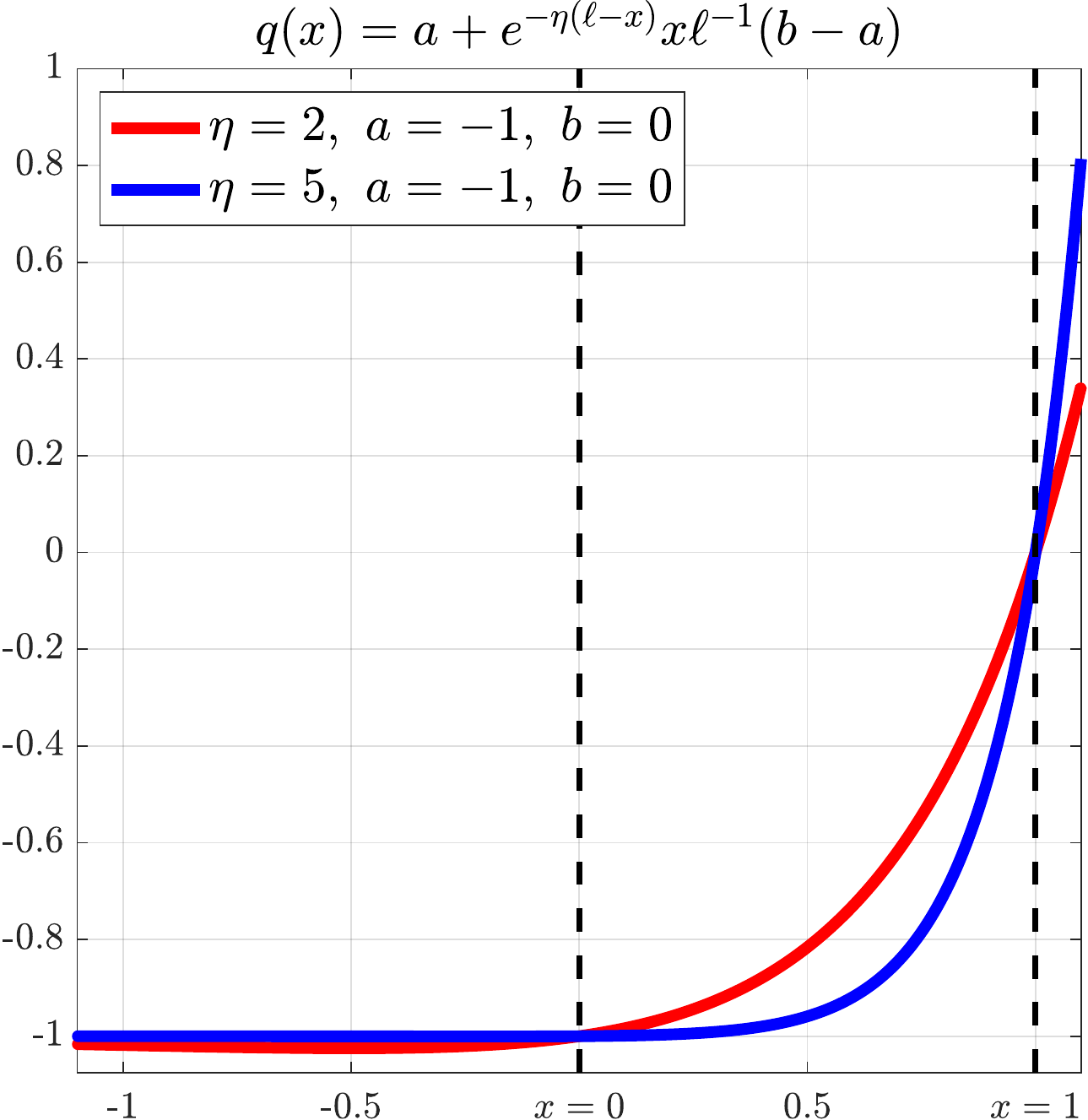}\ 
\includegraphics[height = 4.15cm]{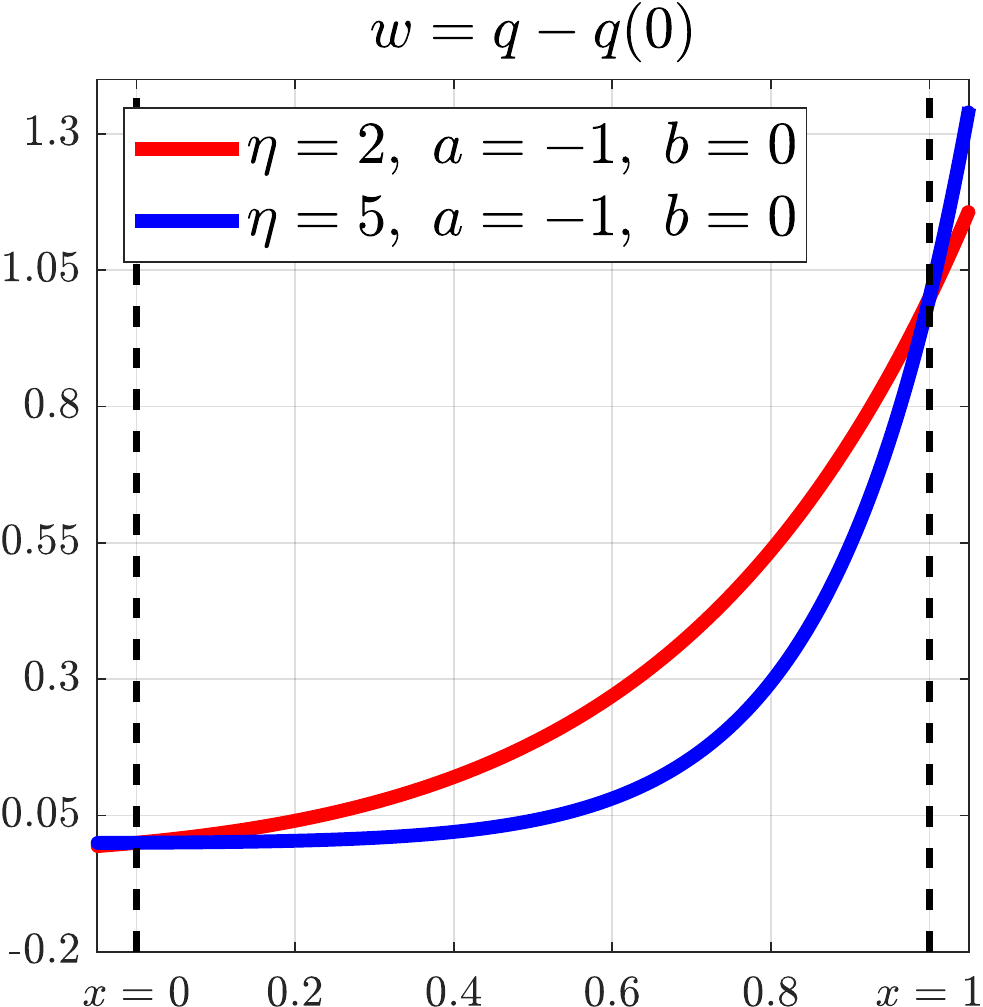}\ 
\includegraphics[height = 4.15cm]{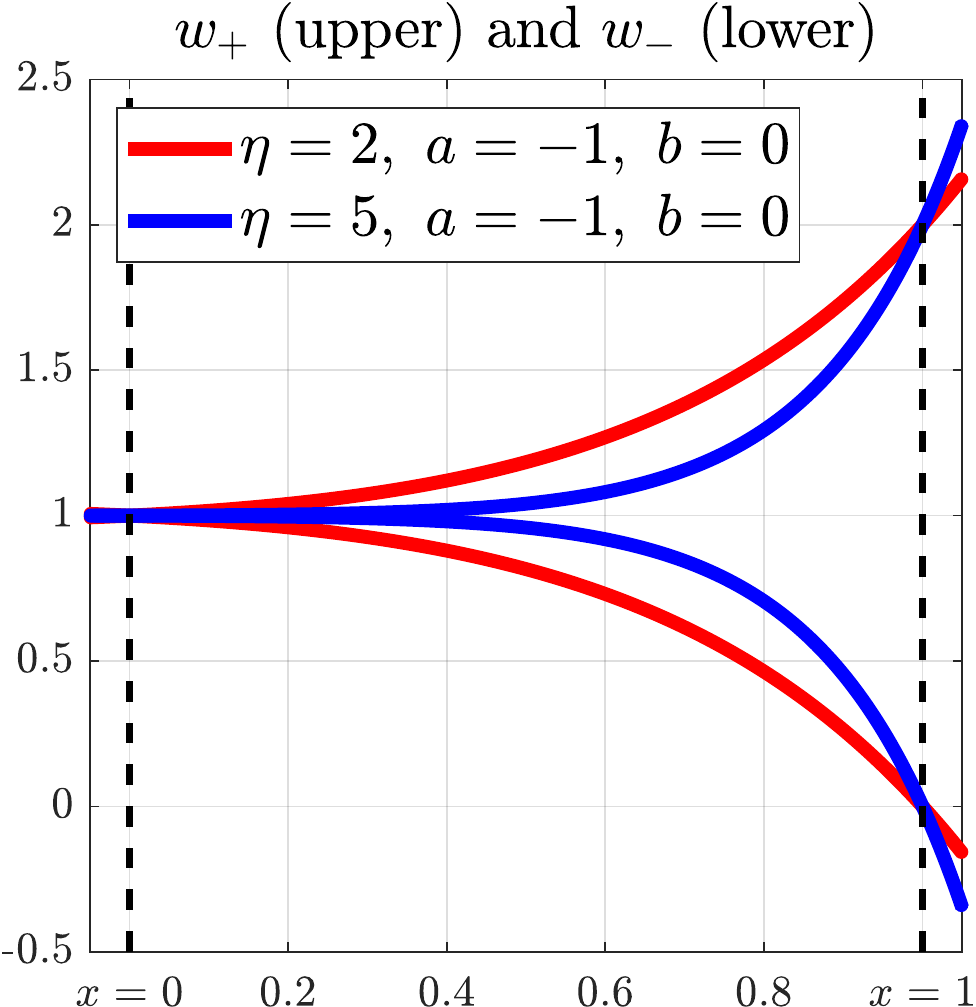}

\caption{The weights provided by the functions $q$ defined in \eqref{eq:def_q_pos} (with $\ell = 1$), used here for a beam controlled at $x=\ell$.}
\label{fig:weights_qPos}
\end{figure}

The proof of Lemma \ref{lem:exist_g} given in \ref{A:MCRF} yields functions of the form
\begin{subequations}
\begin{align} \label{eq:def_q_neg}
q(x) &= -e^{-\eta x}(b-a)(1 - x \ell^{-1}) + b,\\
\label{eq:def_q_pos}
q(x) &= a + e^{-\eta (\ell-x)} x \ell^{-1}(b - a),
\end{align}
\end{subequations}
for \ref{item:exist_g_neg} and \ref{item:exist_g_pos}, respectively. Examples of weight functions $w$ defined in terms of \eqref{eq:def_q_pos} can be seen in Fig. \ref{fig:weights_qPos}.

\begin{remark}[Other weights]
However, one can also look for other types of weights, such as polynomials. For instance, we see in the following lemma (proved in Appendix \ref{ap:poly_weights}), that $q$ may be replaced by polynomials $p_n^-$ and $p_n^+$ of degree $n \in \{1, 2, 3, \ldots\}$, provided that $n$ is large enough (depending on ${\eta}$ and $\ell$) for \eqref{eq:cond1_n} and $\chi \rho > \frac{1}{2^n}$ to be satisfied (see \eqref{eq:cond_ab_pn}).  
We then refer to Fig. \ref{fig:weights_polyPos} for visualization.
\end{remark}

\begin{lemma}\label{lem:polyWeights}
For any given $\ell, \eta >0$, if $n \in \{1, 2, \ldots\}$ is large enough to satisfy
\begin{align} \label{eq:cond1_n}
2{\eta}\ell < n,
\end{align}
then the polynomials $p_n^-$ and $p_n^+$, of degree $n$, defined by 
\begin{align} 
\label{eq:def_pn}
p_n^-(x) = - \left(\frac{1}{2} - \frac{{\eta}}{n}x \right)^{n}, \quad p_n^+(x) =  \frac{1}{2^n} +   \left(\frac{1}{2} + \frac{{\eta}}{n}(x-\ell) \right)^{n}
\end{align}
fulfill 
\begin{subequations}
\begin{align} \label{eq:cond_der_pn}
\frac{\mathrm{d}}{\mathrm{d}x}p_n^-(x)>(p_n^-(\ell) - p_n^-(x))\eta, \quad &\frac{\mathrm{d}}{\mathrm{d}x}p_n^+(x) > (p_n^+(x) - p_n^+(0))\eta,\\
\label{eq:cond_ab_pn}
p_n^-(\ell) - p_n^-(0) \leq \frac{1}{2^n}, \quad &p_n^+(\ell) - p_n^+(0) \leq \frac{1}{2^n}
\end{align}
\end{subequations}
for all $x \in (0, \ell)$.
\end{lemma}

Finally, the perspective of the GEB model is given in Corollary \ref{coro:fromIGEBtoGEB}, which follows from Theorems \ref{thm:inv1b_igeb2geb} and \ref{thm:1b_stabilization}.

\begin{corollary}[Th. 1.7 \ref{A:SICON}] \label{coro:fromIGEBtoGEB}
Let $K \in \mathbb{S}_{++}^6$. There exists $\varepsilon>0$, $C_1>0$ and $C_2>0$ such that the following holds. Assume that $\mathbf{p}^1, w^0 \in C^2([0, \ell];\mathbb{R}^3)$, $R, \mathbf{R}^0 \in H^3(0, \ell; \mathrm{SO}(3))$, $f^\mathbf{R}\in \mathrm{SO}(3)$, $\mathbf{p}^0 \in H^3(0, \ell; \mathbb{R}^3)$ and $f^\mathbf{p} \in \mathbb{R}^3$, with $f^\mathbf{p} = \mathbf{p}^0(0)$ and $f^\mathbf{R} = \mathbf{R}^0(0)$. Assume that the function $y^0 \in H^2(0, \ell; \mathbb{R}^{12})$ defined by \eqref{eq:rel_inidata} fulfills the first-order compatibility conditions of \eqref{eq:IGEBfb}, as well as $\|y^0\|_{H_x^2}\leq \varepsilon$. 
Then, there exists a unique solution $(\mathbf{p}, \mathbf{R}) \in C^2( [0, \ell]\times[0, +\infty); \mathbb{R}^3 \times \mathrm{SO}(3))$ to \eqref{eq:GEBfb}. Furthermore, for all $(x, t) \in [0, \ell]\times[0, +\infty)$,
\begin{align*}
|\partial_t \mathbf{p}(x,t)| + \|\partial_t \mathbf{R}(x,t)\| + |\Gamma(x,t)| +  |\Upsilon(x,t)| \leq C_1 e^{-C_2 t}.
\end{align*}
\end{corollary}

\begin{figure} \centering
\includegraphics[height = 4.15cm]{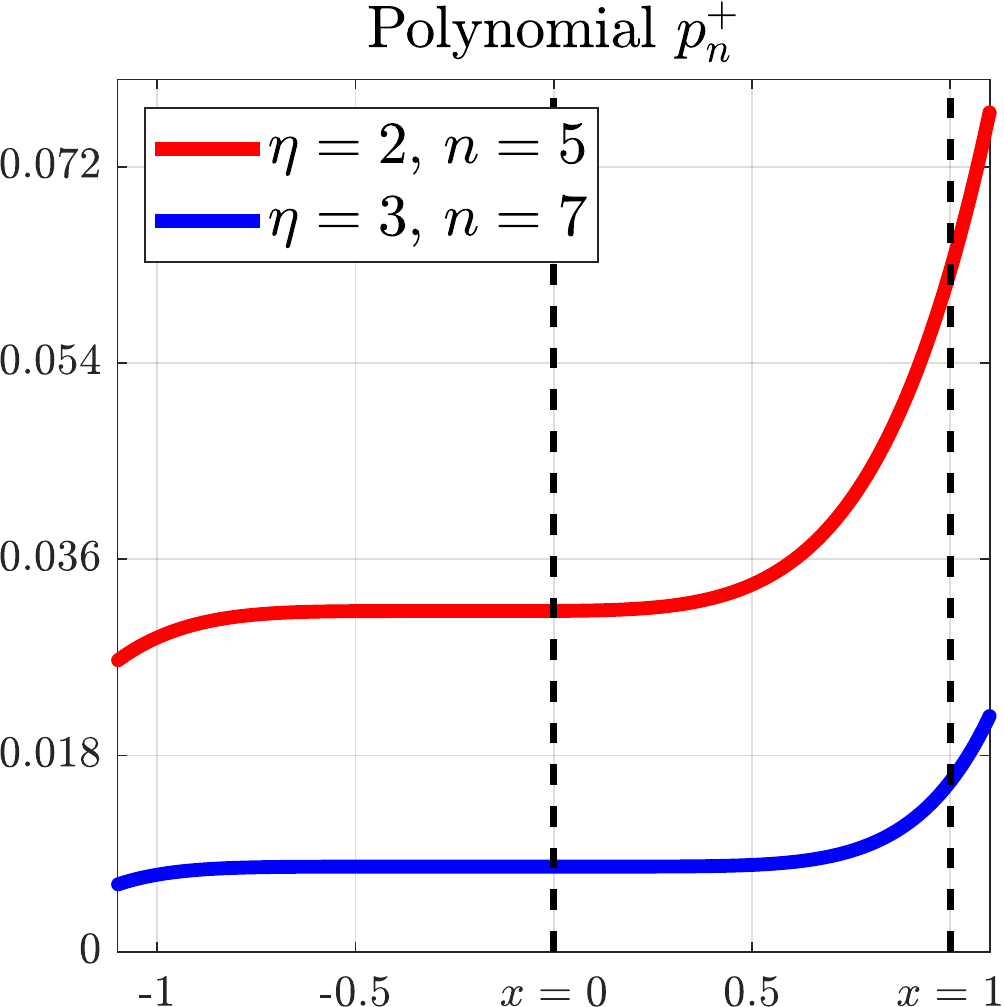}\ 
\includegraphics[height = 4.15cm]{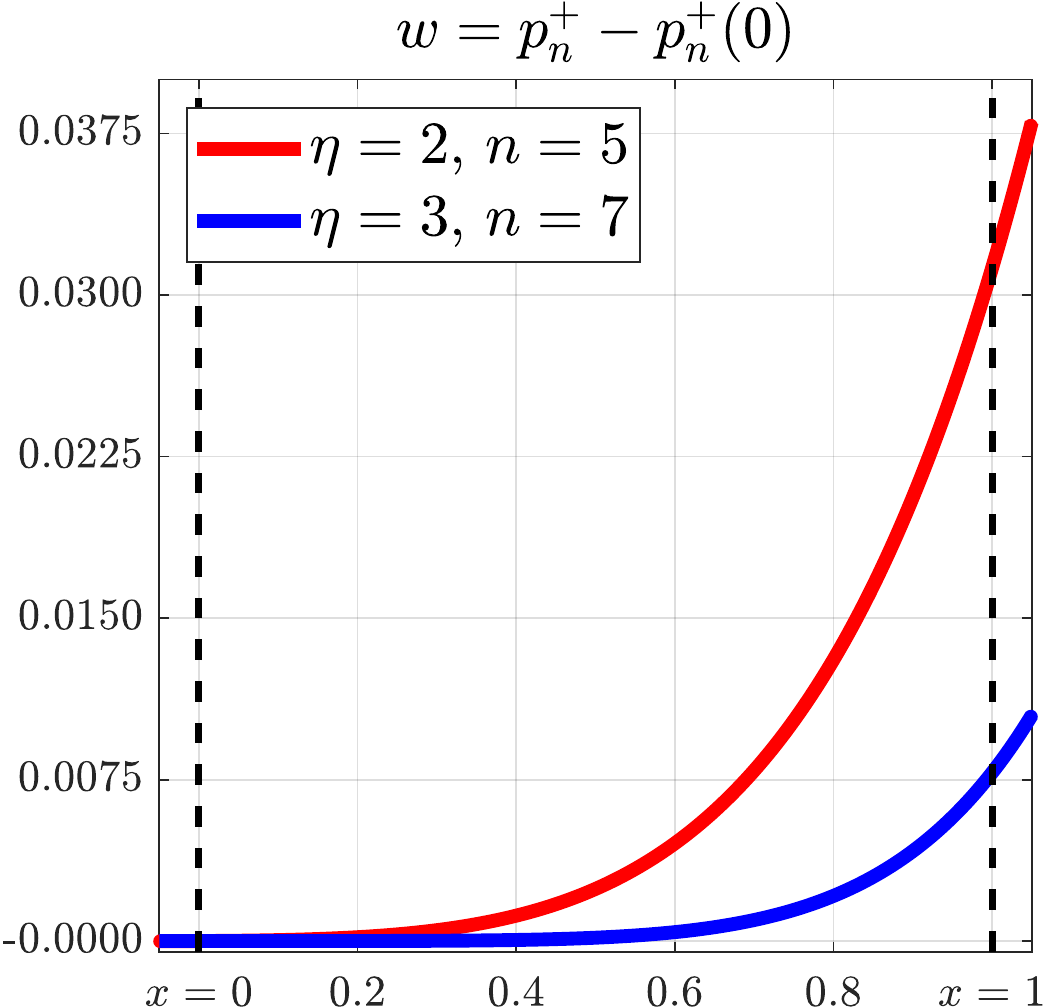}\ 
\includegraphics[height = 4.15cm]{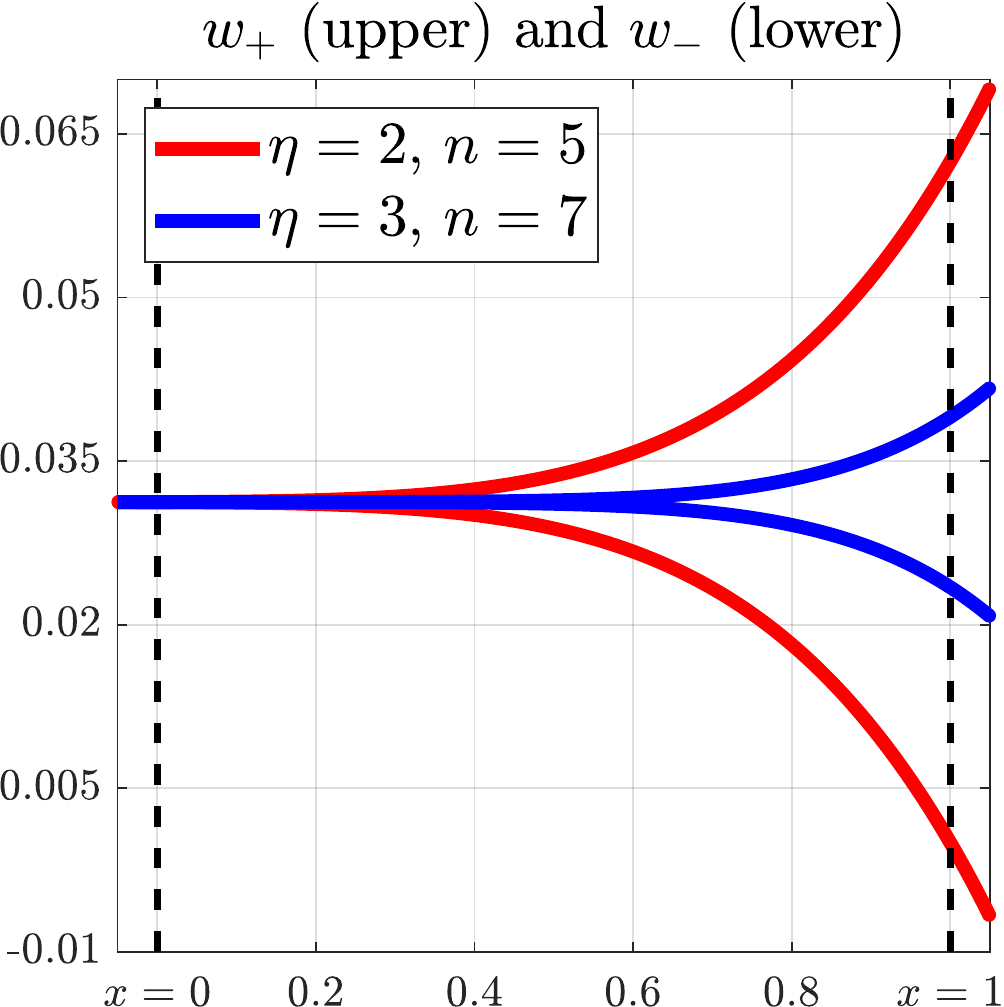}

\caption{The weights provided by the polynomial $p_n^+$ defined in \eqref{eq:def_pn} (with $\ell = 1$), used here for a beam controlled at $x=\ell$.}
\label{fig:weights_polyPos}
\end{figure}

\section{For a network}
\label{sec:stab_net}

We investigate exponential stabilization for \emph{tree-shaped} networks by means of quadratic Lyapunov functionals. As for the energy, we can extend the definition of $\overline{\mathcal{L}}$ in \eqref{eq:1b_def_calLP}, by $\overline{\mathcal{L}}_\mathrm{net} \colon [0, T] \rightarrow [0, +\infty)$ of the form
\begin{equation*} 
\overline{\mathcal{L}}_\mathrm{net} = \sum_{i \in \mathcal{I}} \sum_{\alpha=0}^k \overline{\mathcal{L}}_{\alpha i}, \quad \text{with} \ \ \overline{\mathcal{L}}_{\alpha i} := \int_0^{\ell_i} \left\langle \partial_t^\alpha y_i \,, \overline{Q}_i \partial_t^\alpha y_i \right\rangle dx,
\end{equation*}
where $\overline{Q}_i \in C^1([0, \ell_i]; \mathbb{S}^{12})$ and $y = (y_i)_{i\in \mathcal{I}} \in C^0([0, T]; \mathbf{H}_x^k)$ is solution to \eqref{eq:nIGEBfb}, and that of $\mathcal{L}$ in\eqref{eq:1b_def_calED}, by 
${\mathcal{L}}_\mathrm{net} \colon [0, T] \rightarrow [0, +\infty)$ of the form
\begin{equation*} 
{\mathcal{L}}_\mathrm{net} = \sum_{i \in \mathcal{I}} \sum_{\alpha=0}^k {\mathcal{L}}_{\alpha i}, \quad \text{with} \ \ {\mathcal{L}}_{\alpha i} := \int_0^{\ell_i} \left\langle \partial_t^\alpha r_i \,, {Q}_i \partial_t^\alpha r_i \right\rangle dx,
\end{equation*}
where ${Q}_i \in C^1([0, \ell_i]; \mathbb{S}^{12})$ and $r = (r_i)_{i\in \mathcal{I}} \in C^0([0, T]; \mathbf{H}_x^k)$ is solution to \eqref{eq:nIGEBfbR}. An analogous result to Proposition \ref{prop:1b_existence_Lyap} also holds for these two functionals, replacing $H_x^k$ with $\mathbf{H}_x^k$. Here, we have in mind to build maps $\overline{Q}_i$ (or $Q_i$) via the same procedure as the single beam case.
First, similar arguments of proofs to that of Lemma \ref{lem:class_barQ}, yield the following lemma for  the tree-shaped network.

\begin{lemma}[Lem. 5.1 \ref{A:MCRF}] \label{lem:class_barQi}
Assume that there exists $\overline{Q}_i \in C^1([0, \ell_i]; \mathbb{R}^{12 \times 12})$ $(i \in \mathcal{I})$ fulfilling:
\begin{enumerate}[label=(\roman*)]
\item \label{pty:barQi_symm_posDef}
$\overline{Q}_i(x) \in \mathbb{S}_{++}^{12}$ for all $x \in [0, \ell_i]$ and $i \in \mathcal{I}$;
\item \label{pty:barQiA_i_symm}
$(\overline{Q}_iA_i)(x) \in \mathbb{S}^{12}$ for all $x \in [0, \ell_i]$ and $i \in \mathcal{I}$;
\item \label{pty:barSi_negDef}
$\overline{S}_i(x) \in \mathbb{S}_{--}^{12}$ for all $x \in [0, \ell_i]$ and $i \in \mathcal{I}$, where $\overline{S}_i := \frac{\mathrm{d}}{\mathrm{d}x}( \overline{Q}_i A_i ) - \overline{Q}_i\overline{B}_i - \overline{B}_i^\intercal\overline{Q}_i$;
\item \label{pty:barcalR_nonPos}
for all $y \in C^0([0, T]; \mathbf{C}^0_x)$ satisfying \eqref{eq:nIGEBfb_cont}-\eqref{eq:nIGEBfb_Kir}-\eqref{eq:nIGEBfb_nSell}-\eqref{eq:nIGEBfb_nS0}, and all $t \in [0, T]$, $\overline{\mathcal{R}}(y,t)\leq 0$, where $\overline{\mathcal{R}}$ is defined by
\begin{align*} 
&
\begin{aligned}
\overline{\mathcal{R}} = \overline{\mathcal{R}}_0 + \overline{\mathcal{R}}_M + \overline{\mathcal{R}}_S,
\end{aligned}\\
\nonumber
\text{with}\qquad
&\overline{\mathcal{R}}_0(y,t) = \left\langle y_1 \,, \overline{Q}_1 A_1 y_1 \right\rangle (0, t), \\
\nonumber
&\overline{\mathcal{R}}_S(y,t) = - {\textstyle \sum_{n \in \mathcal{N}_S \setminus \{0\}}} \left\langle y_n \,, \overline{Q}_n A_n y_n\right\rangle (\ell_n, t),\\
\nonumber
&\overline{\mathcal{R}}_M(y,t) = {\textstyle \sum_{n \in \mathcal{N}_M}} \big[- \left\langle y_n \,, \overline{Q}_n A_n y_n \right \rangle (\ell_n, t) + {\textstyle \sum_{i \in \mathcal{I}_n} } \left\langle y_i \,, \overline{Q}_i A_i y_i \right \rangle (0, t)\big].
\end{align*}
\end{enumerate}
\vspace{-0.2cm}
Then, the steady state $y\equiv 0$ of \eqref{eq:nIGEBfb} is locally $\mathbf{H}^k_x$ exponentially stable.
\end{lemma}

It turns out that the main difficulty in following this idea is the treatment of the boundary terms stored in $\overline{\mathcal{R}}$ appearing during the estimation of the derivative of the Lyapunov functional. Note that $\overline{\mathcal{R}}$ is just another expression for $\sum_{i\in\mathcal{I}} \sum_{\alpha=0}^k \big[\left\langle \partial_t^\alpha y_i \,, \overline{Q}_i A_i \partial_t^\alpha y_i \right\rangle\big]_0^{\ell_i}$.

For a \emph{star-shaped} network (i.e. $\mathcal{N}_M = \{1\}$) such that the multiple node is free (i.e. $K_1 = \mathbf{0}_6$) while velocity feedback controls are applied at \emph{all} simple nodes (i.e. $K_n$ is symmetric positive definite for all $n \in \mathcal{N}_S$), we obtain the following stabilization result.

\begin{theorem}[Th. 2.4 \ref{A:MCRF}] \label{thm:stabilization}
Let $k \in \{1, 2\}$, suppose that Assumption \ref{as:mass_flex} is fulfilled for $m=k+1$, and that $\mathbf{E}_i \in C^k([0, \ell_i]; \mathbb{R}^{6 \times 6})$ for all $i\in\mathcal{I}$. If $\mathcal{N}_M = \{1\}$, $K_1= \mathbf{0}_6$, and $K_n$ is symmetric positive definite for all $n \in \mathcal{N}_S$, then the steady state $y \equiv 0$ of \eqref{eq:nIGEBfb} is locally $\mathbf{H}^k_x$ exponentially stable.
\end{theorem}

\noindent \textbf{Idea of the proof.} As we just said, if we want to extend the Lyapunov functional found in the proof of Theorem \ref{thm:1b_stabilization}, differences will appear during the estimation of the boundary terms. One may apply the steps 1, 2 and 3 of the proof of Theorem \ref{thm:1b_stabilization} for each $i \in \mathcal{I}$. Then, each $\overline{Q}_i$ has the form
\begin{align} \label{anzats_barQi}
\overline{Q}_i = \rho_i Q_i^\mathcal{P} + w_i \begin{bmatrix}
\mathbf{0} & \mathbf{W}_i\\
\mathbf{W}_i & \mathbf{0}
\end{bmatrix},
\end{align}
with $\mathbf{W}_i$ as in \eqref{eq:boldW}, and $w_i \in C^1([0, \ell_i])$ increasing and satisfying \eqref{eq:cond_barSi_neg}.
Up until now, we only assumed that the network is tree-shaped. Continuing with the tree-shaped point of view, we use the nodal conditions -- for the transmission conditions one needs $\rho_i = \rho$ to hold for all $i \in \mathcal{I}$ -- to obtain
\begin{align*}
\overline{\mathcal{R}}_0(y,t) &= \left \langle v_1 \,, (-2\rho K_0 - w_1 \Lambda_1^\mathrm{I} - w_1 K_0 \Lambda_1^\mathrm{II} K_0) v_1 \right\rangle (0, t),\\
\overline{\mathcal{R}}_M(y,t) &= {\textstyle \sum_{n \in \mathcal{N}_M}} \big[ \left(\left \langle v_n \,, w_n \Lambda_n^\mathrm{I} v_n\right \rangle  + \left \langle z_n \,, w_n \Lambda_n^\mathrm{II} z_n \right \rangle \right) (\ell_n, t)  \\
&\quad -2 \rho \left \langle v_n \,, \overline{K}_n v_n \right \rangle (\ell_n, t) - {\textstyle \sum_{i \in \mathcal{I}_n}}  \left( \left \langle v_i \,, w_i \Lambda_i^\mathrm{I} v_i\right \rangle + \left \langle z_i \,, w_i \Lambda_i^\mathrm{II} z_i\right \rangle \right) (0, t) \big],\\
\overline{\mathcal{R}}_S(y,t) &= {\textstyle \sum_{n \in \mathcal{N}_S \setminus \{0\}} } \left \langle v_n \,, (- 2 \rho K_n + w_n \Lambda_n^\mathrm{I} + w_n K_n \Lambda_n^\mathrm{II} K_n) v_n\right \rangle (\ell_n, t),
\end{align*}
where $\Lambda_i^\mathrm{I}$ and $\Lambda_i^\mathrm{II}$ are defined as in \eqref{eq:def_Lambda_Xi}.
On the other hand, if the network is star-shaped then the boundary terms take the form
\begin{linenomath}
\begin{equation*}
\begin{aligned}
\overline{\mathcal{R}}(y,t) &= 
\left \langle v_1 \,, (-2\rho K_0 - w_1 \Lambda_1^\mathrm{I} -w_1 K_0 \Lambda_1^\mathrm{II} K_0) v_1 \right \rangle (0, t) + \left \langle v_1 \,, w_1 \Lambda_1^\mathrm{I} v_1\right \rangle (\ell_1, t) \\
&\quad   + \left \langle z_1 \,, w_1 \Lambda_1^\mathrm{II} z_1 \right \rangle (\ell_1, t) + {\textstyle\sum_{i=2}^N} \big[ \left \langle v_i \,, (- 2 \rho K_i + w_i \Lambda_i^\mathrm{I} + w_i K_i \Lambda_i^\mathrm{II} K_i ) v_i\right \rangle (\ell_i, t)\\
&\quad - \left \langle v_i \,, w_i \Lambda_i^\mathrm{I} v_i \right \rangle (0, t) - \left \langle z_i \,, w_i \Lambda_i^\mathrm{II} z_i \right \rangle (0, t) \big].
\end{aligned}
\end{equation*}
\end{linenomath}
The last steps of the proof unfold as follows.
\begin{itemize}
\item \textbf{Estimating the boundary terms.}
To render nonpositive the scalar products containing neither $-2\rho K_0$ nor $-2\rho K_i$, we assume that $w_1(\ell_1)$ is nonpostitive and $w_i(0)$ for all $i \geq 2$ is nonnegative.
This yields $\overline{\mathcal{R}}(y,t) \leq \left \langle v_1 \,, \mu_1(0, K_0) v_1 \right \rangle (0, t) + \sum_{i=2}^N \left \langle v_i \,, \mu_i(\ell_i, K_i) v_i \right \rangle (\ell_i, t)$, defining the matrix $\mu_i(x,K)$ as in \eqref{eq:def_mu(x,K)}. Once again $\mu_i(\ell_i, K_i)$ is negative semi-definite if $\rho$ is large enough in comparison to $|w_i(x)|$, and this enables us to make the above sign assumption on the weights without disagreeing with the fact that they all have to be increasing.
To summarise, all weights should be increasing, $w_1$ nonpositive and $w_i$ $(i\geq 2)$ nonnegative, and should satisfy 
\begin{equation*} 
\begin{aligned}
&w_1(0) > - \rho \chi_1, \quad \text{with }\chi_1= \min \left\{C_{\theta_1}^{-\sfrac{1}{2}}, C_{\mu_1(0,K_0)}^{-1}\right\},\\
&w_i(\ell_i) < \rho \chi_i, \quad \ \ \, \text{with } \chi_i = \min \left\{C_{\theta_i}^{-\sfrac{1}{2}}, C_{\mu_i(\ell_i,K_i)}^{-1}\right\}, \quad \text{for all }i \in \{2, \ldots, N\}
\end{aligned}
\end{equation*}
for $C_{\theta_i}$ and $C_{\mu_i(x,K)}$ defined as in the step 3 and 4 of the proof of Theorem \ref{thm:1b_stabilization}. 

\item \textbf{Existence of the weights.} 
Such weights are provided by Lemma \ref{lem:exist_g}, using this time \ref{item:exist_g_neg} for the first beam (with $b-\rho\chi_1 < a < b \leq 0$) and setting $w_1 = q - q(\ell)$, while \ref{item:exist_g_pos} is again used for the other beams with $w_i$ chosen similarly to the single beam case in Theorem \ref{thm:1b_stabilization} (see Fig. \eqref{fig:star_weights}). \hfill $\meddiamond$
\end{itemize}

\begin{figure}
\centering
\includegraphics[scale=0.9]{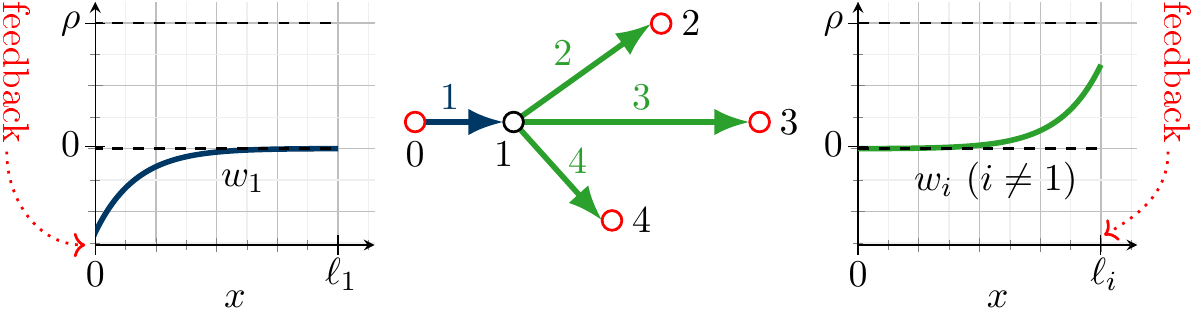}
\caption{Plot of $\rho$ and the weight functions $w_1 = q - q(\ell)$ with $q$ defined by \eqref{eq:def_q_neg} ($a=-1$, $b=0$), and $w_i = q-q(0)$ with $q$ defined by \eqref{eq:def_q_pos} for $i\geq 2$ ($a=0$, $b=1$). 
We used the constants $\eta=5$, $\rho=1.5$ and $\ell = 1$.}
\label{fig:star_weights}
\end{figure}


\medskip

Out of the setting of star-shaped networks controlled at all simple nodes, it is not clear how one may obtain the property \ref{pty:barcalR_nonPos} of Lemma \ref{lem:class_barQi} without contradicting the monotonicity assumption on $w_i$ ($i \in \mathcal{I}$). 
One may wonder what happens when one of the controls is removed, the beam of index $i=1$ being clamped or free at the node $n=0$ for instance, while we may apply a feedback at the multiple node -- meaning that $K_1 \in \mathbb{S}^6_{++}$. In that case, the boundary terms stored in $\overline{\mathcal{R}}(y,t)$ become
\begin{linenomath}
\begin{equation*}
\begin{aligned}
\overline{\mathcal{R}}&(y,t) = 
- \left \langle z_1 \,, w_1 \Lambda_1^\mathrm{II} z_1\right \rangle (0, t) + \left \langle v_1 \,, (- 2\rho\overline{K}_1 + w_1 \Lambda_1^\mathrm{I}) v_1\right \rangle (\ell_1, t)  \\
&+ \left \langle z_1 \,, w_1 \Lambda_1^\mathrm{II} z_1 \right \rangle (\ell_1, t) + { \textstyle \sum_{i=2}^N} \big[ \left \langle v_i \,, (- 2 \rho K_i + w_i \Lambda_i^\mathrm{I} + w_i K_i \Lambda_i^\mathrm{II} K_i ) v_i\right \rangle (\ell_i, t) \\
&- \left \langle v_i \,, w_i \Lambda_i^\mathrm{I} v_i \right \rangle (0, t) + \left \langle z_i \,, w_i \Lambda_i^\mathrm{II} z_i \right \rangle (0, t) \big].
\end{aligned}
\end{equation*}
\end{linenomath}
Here, one cannot both assume that $w_1(0) \geq 0$ -- in order to estimate the first term in the above expression -- and that $w_1(\ell_1) \leq 0$ -- in order to estimate the term $\left \langle z_1 \,, w_1 \Lambda_1^\mathrm{II} z_1 \right \rangle (\ell_1, t)$ -- without contradicting the fact that $w_1$ should be increasing.

\subsection*{Point of view of the diagonal system}
\label{sec:pov_diag_syst}

%
We now take the perspective of the system written in Riemann invariants.
We introduce some additional notation for multiple nodes $n \in \mathcal{N}_M$. Recall that $k_n$ is the degree of the $n$ and that the matrices $\gamma_0^0, \sigma_0^0$, $\gamma_i^n, \sigma_i^n$ and $\overline{K}_n$ are defined in \eqref{eq:def_gamma0_sigma0}, \eqref{eq:def_gamma_sigma_in} and \eqref{eq:def_barKn}. We define $\sigma^n \in \mathbb{S}_{++}^{6k_n}$ and the square block matrix $\textswab{I}_n \in \mathbb{R}^{6k_n \times 6k_n}$ by
\begin{equation*}
\sigma^n = \sigma_n^n + {\textstyle \sum_{j=2}^{k_n}} \sigma_{i_j}^n, \qquad \textswab{I}_n = \begin{bmatrix}
\mathbf{I}_6 &\ldots & \mathbf{I}_6\\
\vdots & \ddots  & \vdots \\
\mathbf{I}_6 & \ldots & \mathbf{I}_6
\end{bmatrix},
\end{equation*}
and the block diagonal matrices $\mathrm{diag}(\sigma^n + \overline{K}_n)\in \mathbb{R}^{6k_n \times 6k_n}$ and $\mathrm{diag}(\sigma_i^n) \in \mathbb{R}^{6k_n \times 6k_n}$ by
\begin{equation*}
\begin{aligned}
\mathrm{diag}(\sigma_i^n) &=\mathrm{diag}\big(\sigma_n^n, \sigma_{i_2}^n, \ldots, \sigma_{i_{k_n}}^n \big),\\
 \mathrm{diag}(\sigma^n+\overline{K}_n) &= \mathrm{diag}\big(\sigma^n + \overline{K}_n , \ldots , \sigma^n + \overline{K}_n\big).
\end{aligned}
\end{equation*}
Then for the network system in Riemann invariants, one may rewrite the nodal conditions as in the following lemma.

\begin{lemma}[Lem. 4.4 \ref{A:MCRF}] 
The nodal conditions of \eqref{eq:nIGEBfbR} are equivalent to 
\begin{equation} \label{eq:nodal_cond_calBn}
r^{\mathrm{out}}_n(t) = \mathcal{B}_n r^{\mathrm{in}}_n(t), \qquad t \in [0, T], \text{ for all } n \in \mathcal{N},
\end{equation}
where $\mathcal{B}_n \in \mathbb{R}^{6 k_n \times 6 k_n}$ is defined by
\begin{equation*} 
\mathcal{B}_n =
\begin{dcases}
((\sigma_n^n + \overline{K}_n)\gamma_n^n)^{-1}(\sigma_n^n - \overline{K}_n)\gamma_n^n, & \text{if }n \in \mathcal{N}_S,\\
2 \mathbf{G}_n^{-1}\mathrm{diag}(\sigma^n+\overline{K}_n)^{-1}  
\textswab{I}_n \mathrm{diag}(\sigma_i^n) \mathbf{G}_n -  \mathbf{I}_{6k_n}, & \text{if } n \in \mathcal{N}_M.
\end{dcases}
\end{equation*}
\end{lemma}


On the other hand, if the beam is rather free (or clamped) at the node $n \in \mathcal{N}_S$, then $\mathcal{B}_n$ is given by $\mathcal{B}_n = \mathbf{I}_6$ (resp. $\mathcal{B}_n = - \mathbf{I}_6$).
Henceforth, for functions $Q_i$ ($i \in \mathcal{I}$) having values in $\mathbb{D}^{12}$, we denote $Q_i = \mathrm{diag}(Q_{i}^-, \ Q_{i}^+)$ where $Q_{i}^-, Q_{i}^+$ have values in $\mathbb{D}^6$. Moreover, for all $n \in \mathcal{N}$, we introduce the matrices $Q_n^\mathrm{out}, Q_n^\mathrm{in}, \bar{D}_n \in \mathbb{R}^{6k_n \times 6 k_n}$ defined as follows (see \eqref{eq:nota_In_indices}):
\begin{align*}
Q_n^\mathrm{out} &=
\begin{cases}
Q_1^+(0) & n=0\\
Q_{n}^-(\ell_n) & n \in \mathcal{N}_S \setminus \{0\}\\
\mathrm{diag}(Q_{n}^-(\ell_n), \ Q_{i_2}^+(0), \ Q_{i_3}^+(0) , \ \ldots, \ Q_{i_{k_n}}^+(0) ) & n \in \mathcal{N}_M,
\end{cases}\\
Q_n^\mathrm{in} &= 
\begin{cases}
Q_1^-(0) & n=0\\
Q_n^+(\ell_n) & n\in\mathcal{N}_S \setminus\{0\}\\
\mathrm{diag}(Q_{n}^+(\ell_n), \ Q_{i_2}^-(0), \ Q_{i_3}^-(0) , \ \ldots, \ Q_{i_{k_n}}^-(0) ) &n\in \mathcal{N}_M,
\end{cases}\\
\bar{D}_n &= 
\begin{cases}
D_1(0) & n=0\\
D_n(\ell_n) & n \in \mathcal{N}_S\setminus \{0\}\\
\mathrm{diag} ( D_n(\ell_n), \ D_{ i_2}(0), \ D_{i_3}(0) , \ \ldots, \ D_{i_{k_n}}(0) ) & n\in\mathcal{N}_M.
\end{cases}
\end{align*}
Then, one obtains the following lemma, analogous to Lemma \ref{lem:class_barQi} but where the boundary terms $\sum_{i\in\mathcal{I}} \sum_{\alpha=0}^k \big[\left\langle \partial_t^\alpha r_i \,, Q_i \mathbf{D}_i \partial_t^\alpha r_i \right\rangle\big]_0^{\ell_i}$ have been further processed by means of \eqref{eq:nodal_cond_calBn}.

\begin{lemma}[Lem. 5.3 \ref{A:MCRF}] \label{lem:class_barQi_Riem}
Assume that there exists $Q_i  \in C^1([0, \ell_i]; \mathbb{R}^{12\times 12})$ $(i \in \mathcal{I})$, fulfilling
\begin{enumerate}[label=(\roman*)]
\item \label{pty:Qi_diag_posDef}
$Q_i(x) \in \mathbb{D}_{++}^{12}$ for all $x \in [0, \ell_i]$ and $i\in \mathcal{I}$;
\item \label{pty:Si_negDef}
for all $x \in [0, \ell_i]$ and $i\in \mathcal{I}$, $S_i(x) \in \mathbb{S}_{--}^{12}$, where $S_i := \frac{\mathrm{d}}{\mathrm{d}x}( Q_i \mathbf{D}_i ) - Q_i B_i - B_i^\intercal Q_i$;
\item \label{pty:calMn_negSemiDef}
$\mathcal{M}_n \in \mathbb{S}_-^{6k_n}$ for all $n \in \mathcal{N}$, where $\mathcal{M}_n := \mathcal{B}_n^\intercal Q_n^\mathrm{out} \bar{D}_n \mathcal{B}_n - Q_n^\mathrm{in} \bar{D}_n$.
\end{enumerate}
Then, the steady state $r\equiv 0$ of \eqref{eq:nIGEBfbR} is locally $\mathbf{H}^k_x$ exponentially stable.
\end{lemma}
%
%

Note that for any functions $\{Q_i\}_{i\in \mathcal{I}}$ defined by \eqref{eq:rel_Q_barQ} with $\{\overline{Q}_i\}_{i \in \mathcal{I}}$ satisfying the assumptions of Lemma \ref{lem:class_barQi}, the assumptions of Lemma \ref{lem:class_barQi_Riem} hold. Indeed, basic computations yield the following proposition.

\begin{proposition} \label{prop:rel_Lyap_Q_barQ}
Let $i \in \mathcal{I}$ and $\overline{Q}_i, Q_i \in C^1([0, \ell_i]; \mathbb{S}^{12})$ be such that \eqref{eq:rel_Q_barQ}. 
\begin{itemize}
\item For any $x \in [0, \ell_i]$, if $(\overline{Q}_iA_i)(x) \in \mathbb{S}^{12}$, then $\overline{S}_i(x) \in \mathbb{S}_{--}^{12}$ if and only if $S_i(x) \in \mathbb{S}_{--}^{12}$.
\item  For any $x \in [0, \ell_i]$, $\overline{Q}_i(x)$ is positive definite if and only if $Q_i(x)$ is positive definite. 
\item For any $y\in C^0([0, T]; \mathbf{C}^0_x)$ fulfilling the nodal conditions \eqref{eq:nIGEBfb_cont}-\eqref{eq:nIGEBfb_Kir}-\eqref{eq:nIGEBfb_nSell}-\eqref{eq:nIGEBfb_nS0} of the physical system and any $t \in [0, T]$, $\overline{\mathcal{R}}(y,t)\leq 0$ if and only if $\{\mathcal{M}_n\}_{n \in \mathcal{N}} \subset \mathbb{S}_-^{6k_n}$.
\end{itemize}
\end{proposition}

Consider $\overline{Q}_i$ ($i \in \mathcal{I}$) defined by \eqref{anzats_barQi} with $\mathbf{W}_i$ given by \eqref{eq:boldW_MCfrac}, where each $w_i\in C^1([0, \ell_i])$ is increasing with its derivative fulfilling \eqref{eq:cond_barSi_neg} for all $x \in [0, \ell_i]$, and where $\rho_i = \rho >0$ for all $i \in \mathcal{I}$. We have seen that associated $Q_i$ defined by \eqref{eq:rel_Q_barQ} takes the form
\begin{align}\label{eq:Qi_multEnergy}
Q_i = \mathrm{diag} \left( w_-^i \mathbf{I}_6, w_+^i \mathbf{I}_6 \right) Q_i^\mathcal{D}, \qquad \text{with} \quad w_\mp^i = \rho \pm w_i
\end{align}
and we know from Proposition \ref{prop:rel_Lyap_Q_barQ} that  \ref{pty:Qi_diag_posDef}-\ref{pty:Si_negDef} of Lemma \ref{lem:class_barQi_Riem} are satisfied by these maps $\{Q_i\}_{i \in \mathcal{I}}$. 

Examples of weights provided by Lemma \ref{lem:exist_g} \ref{item:exist_g_pos} and Lemma \ref{lem:polyWeights} (with $p_n^+$), which can be used for a beam controlled at $x = \ell$, have been presented in Fig. \ref{fig:weights_qPos} and \ref{fig:weights_polyPos}. This is also representative of the weights used for the star shaped network, here, for the beams of indexes $i \in \{2, \ldots, N\}$. For a beam controlled at $x=0$, such as the first beam of the network, one rather uses the function $q$ from Lemma \ref{lem:exist_g} \ref{item:exist_g_neg} (or the polynomial $p_n^-$ from Lemma \ref{lem:polyWeights}). The latter are depicted in Fig. \ref{fig:weights_Neg}, together with the corresponding weights in the physical and diagonal systems.

\medskip

\begin{figure}\centering
\includegraphics[height = 4.15cm]{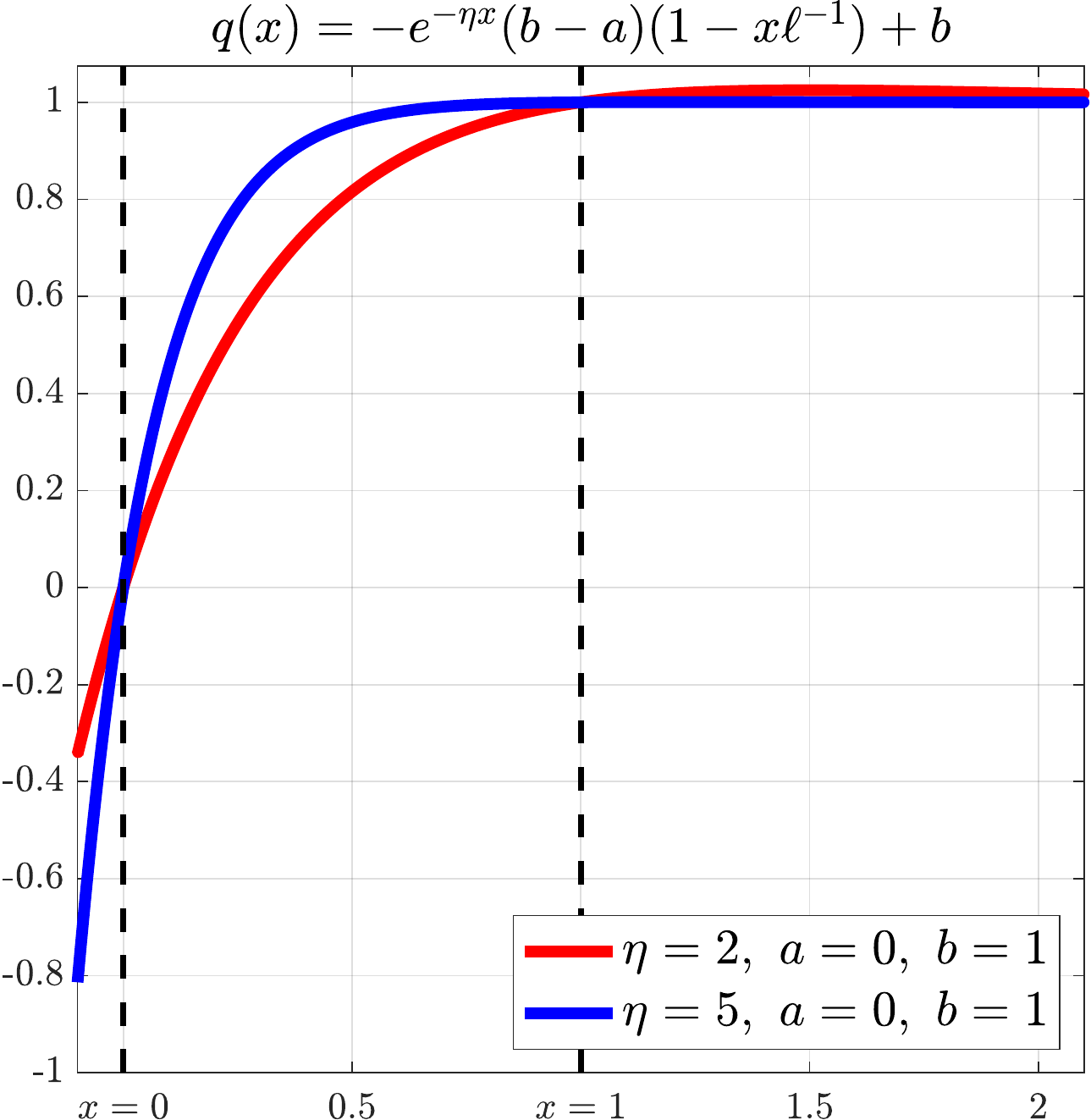}\ 
\includegraphics[height = 4.15cm]{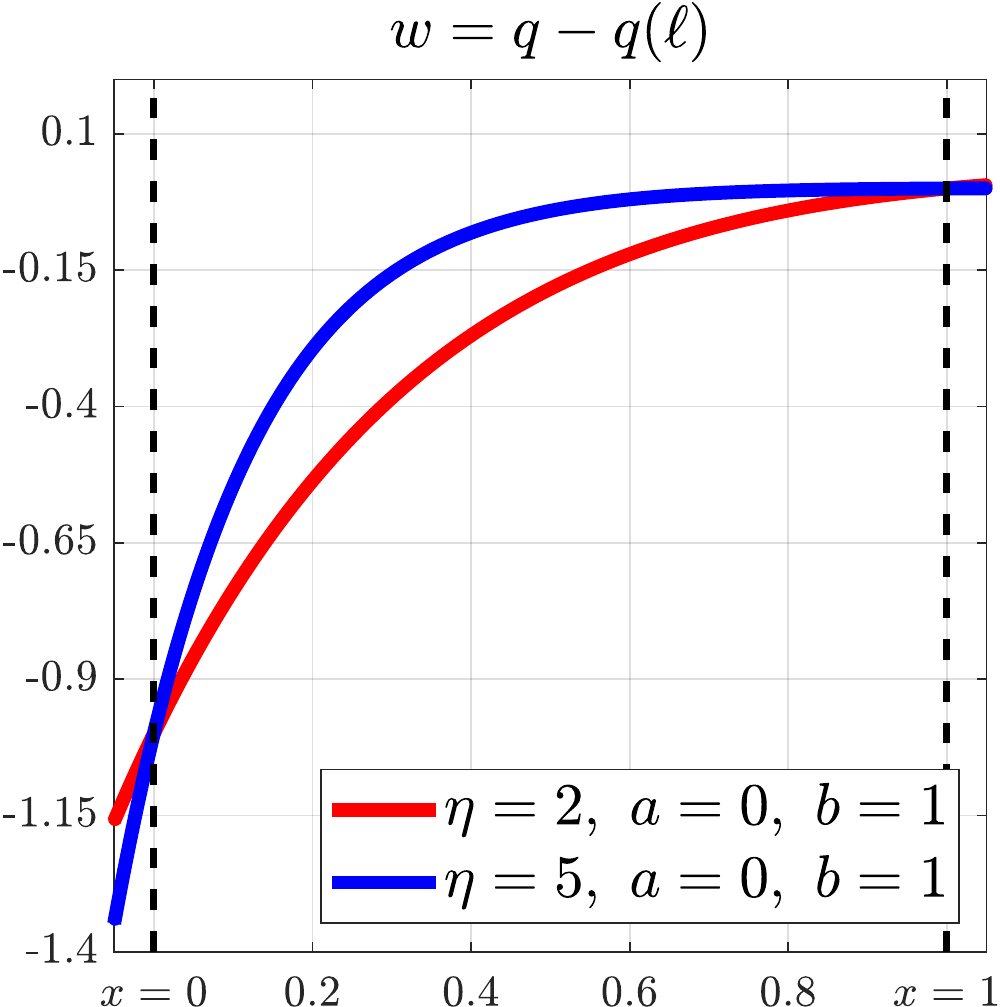}\ 
\includegraphics[height = 4.15cm]{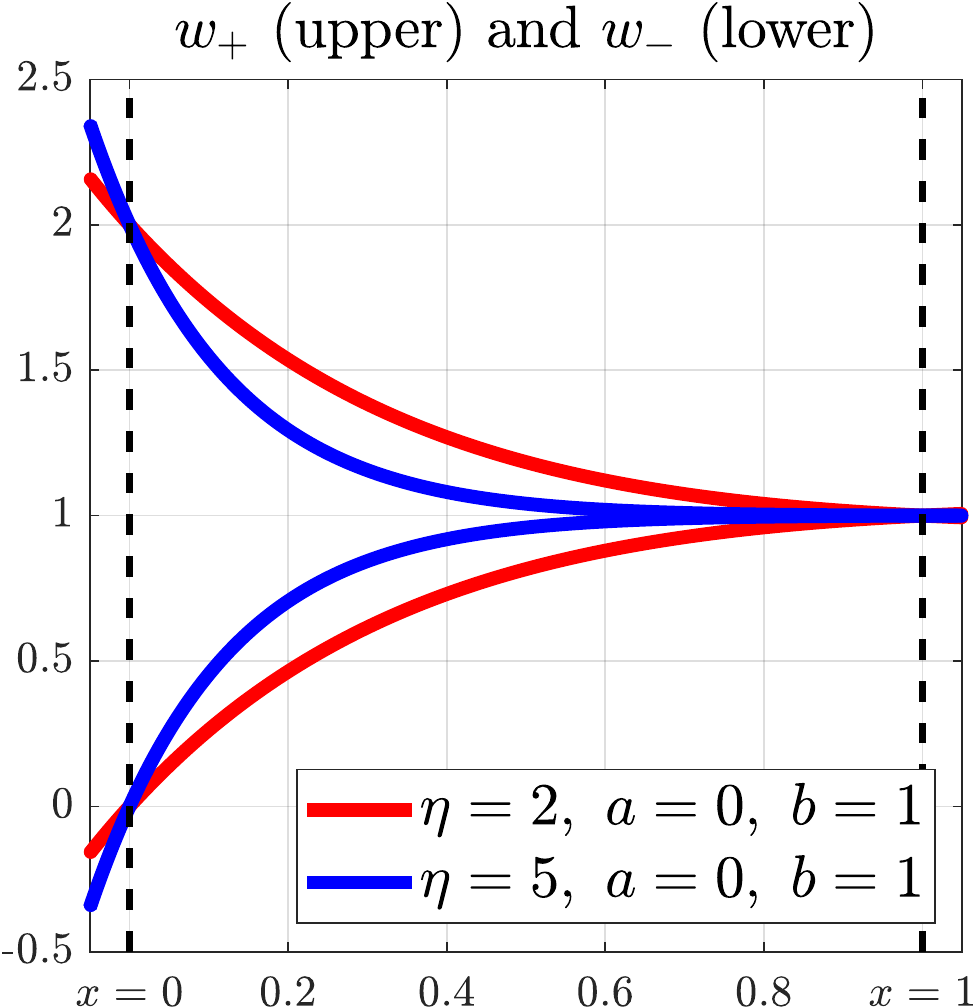}\\
\vspace{0.25cm}
\includegraphics[height = 4.15cm]{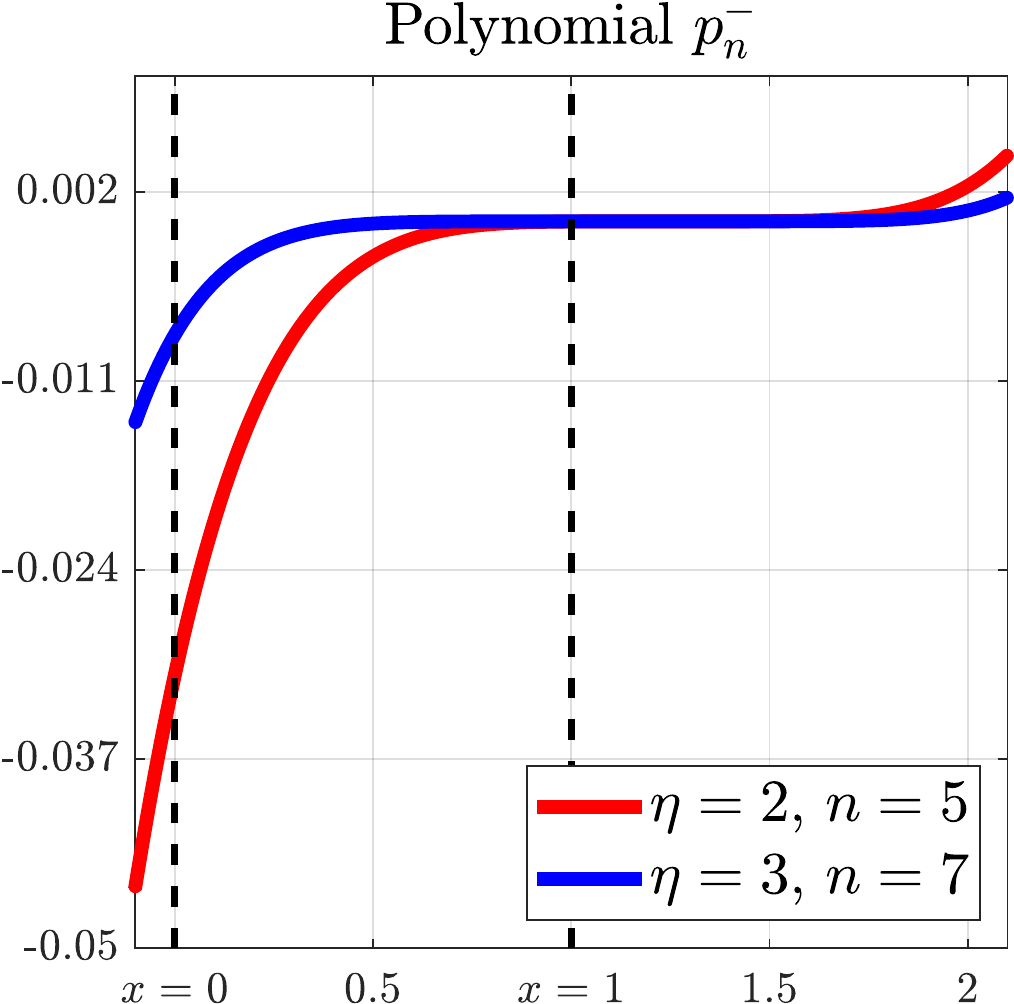}\ 
\includegraphics[height = 4.15cm]{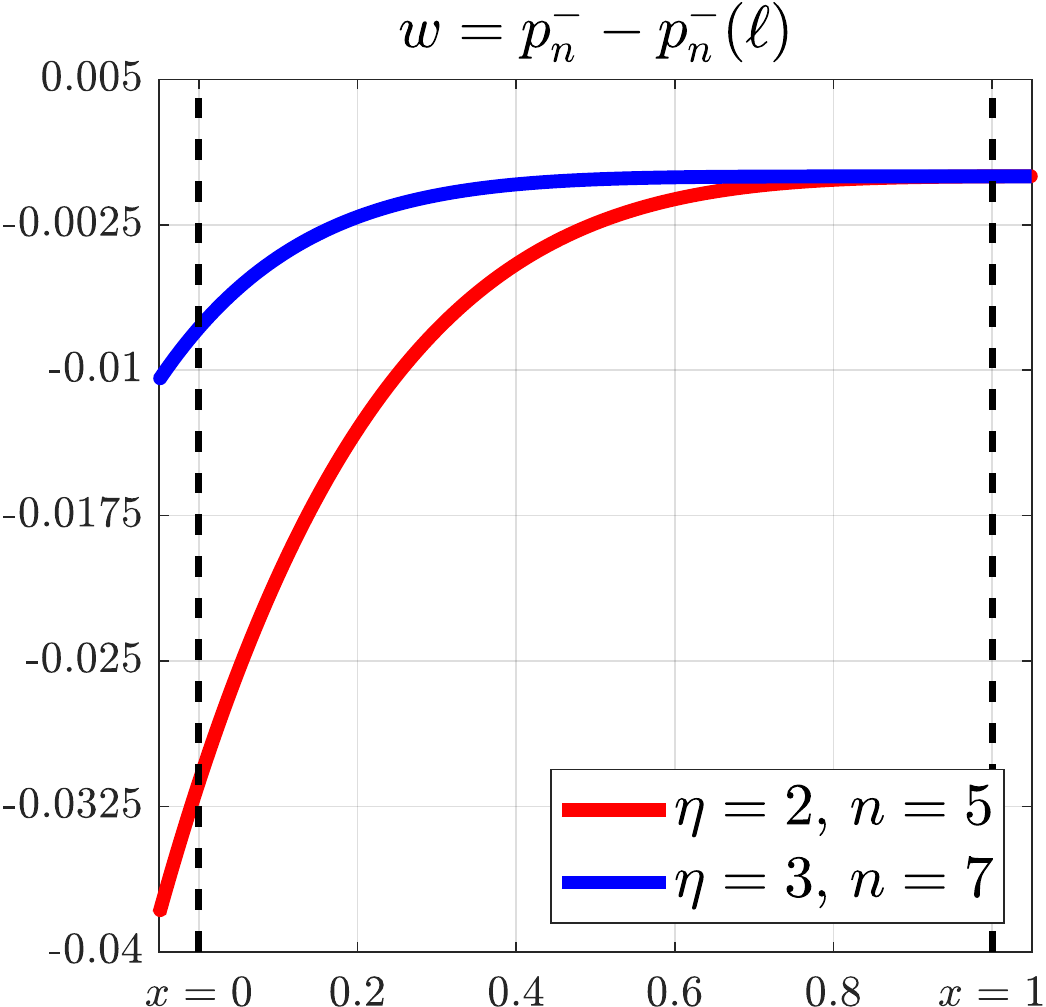}\ 
\includegraphics[height = 4.15cm]{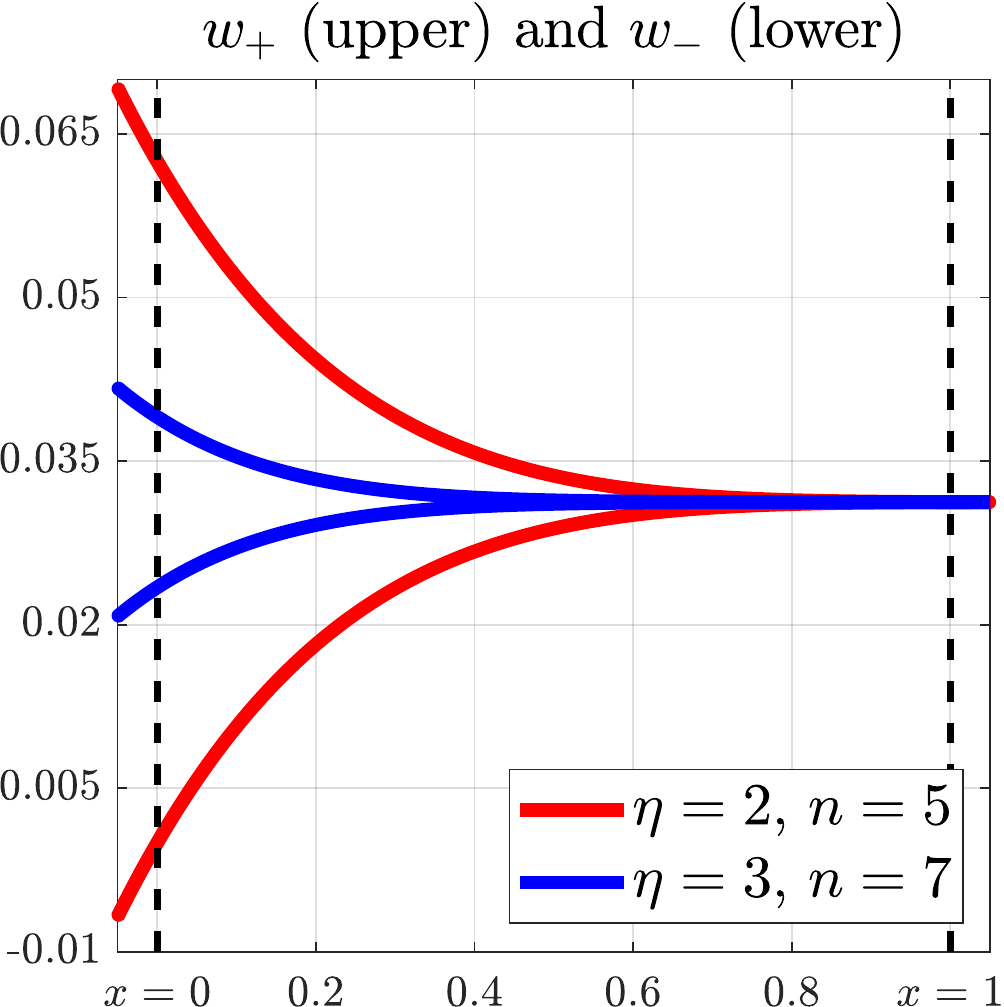}

\caption{Top: the weights provided by the functions $q$ defined in \eqref{eq:def_q_neg}. Bottom: the weights provided by the polynomial $p_n^-$ defined in \eqref{eq:def_pn}. Here $\ell = 1$. These weights are used here for a beam controlled at $x=0$.}

\label{fig:weights_Neg}
\end{figure}

\noindent By means of the following proposition, we look further into the form of the matrices $\mathcal{M}_n$, observing that they are congruent\footnote{For any $d \in \{1, 2, \ldots\}$, two matrices $A, B \in \mathbb{R}^{d \times d}$ are called congruent if there exists $P \in \mathbb{R}^{d\times d}$ invertible such that $A = P^\intercal B P$. In particular, $A$ is negative semi-definite if and only if $B$ is also negative semi-definite.} to some matrices which are less difficult to analyse.
Beforehand, for any $n \in \mathcal{N}$, let us define the matrix $\overline{w}_n \in \mathbb{R}^{6k_n \times 6k_n}$ by
\begin{equation*}
\begin{aligned}
\overline{w}_n &= \begin{cases}
-w_1(0)\mathbf{I}_6, & \text{if } n=0\\
w_n(\ell_n)\mathbf{I}_6, & \text{if } n \in \mathcal{N}_S\setminus \{0\}\\
\mathrm{diag}(w_n(\ell_n) \mathbf{I}_6\, , \, -w_{i_2}(0)\mathbf{I}_6\, , \, \ldots\, , \, -w_{i_{k_n}}(0)\mathbf{I}_6), & \text{if } n\in\mathcal{N}_M.
\end{cases}
\end{aligned}
\end{equation*}

\begin{proposition}[Prop. 4 \ref{A:MCRF}] \label{prop:calMn_form}
Let $Q_i$ $(i \in \mathcal{I})$ be given by \eqref{eq:Qi_multEnergy} and let $\{\mathcal{M}_n\}_{n \in \mathcal{N}}$ be the matrices defined in Lemma \ref{lem:class_barQi_Riem} \ref{pty:calMn_negSemiDef}.
\begin{enumerate}[label=\arabic*), noitemsep]
\item 
Let $m \in \mathcal{N}_M$. If $K_m \in \mathbb{S}_{++}^6$ or $K_m = \mathbf{0}_6$, then $\mathcal{M}_m$ is congruent to
\begin{equation*} 
\begin{aligned}
\widetilde{\mathcal{M}}_m &= - 2 \rho k_m^{-1} \textswab{I}_{m} \mathrm{diag}(\overline{K}_m)  \textswab{I}_m + 2
\textswab{I}_m \overline{w}_m \mathrm{diag}(\sigma_i^m) \textswab{I}_m - \textswab{I}_m \overline{w}_m \mathrm{diag}(\sigma^m + \overline{K}_m)  \\
& -  \mathrm{diag}(\sigma^m + \overline{K}_m)\overline{w}_m \textswab{I}_m + \overline{w}_m \mathrm{diag}(\sigma^m+\overline{K}_m)\mathrm{diag}(\sigma_i^m)^{-1}\mathrm{diag}(\sigma^m+\overline{K}_m).
\end{aligned}
\end{equation*}
\item 
Let $n\in \mathcal{N}_S$. If $K_n \in \mathbb{S}_{++}^6$, then $\mathcal{M}_n$ is congruent to the matrix
\begin{equation*} 
\widetilde{\mathcal{M}}_n = \rho  \big[ (\mathbf{I} + \Upsilon^n)^{-2}(\mathbf{I} - \Upsilon^n)^2 - \mathbf{I} \big] + \overline{w}_n \big[
(\mathbf{I} + \Upsilon^n)^{-2}(\mathbf{I} - \Upsilon^n)^2 + \mathbf{I} \big],
\end{equation*}
where $\Upsilon^n \in \mathbb{D}_{++}^6$ has the eigenvalues of $\bar{D}_n^{\sfrac{1}{2}} (\gamma_n^n)^\intercal \overline{K}_n \gamma_n^n \bar{D}_n^{\sfrac{1}{2}}$ as diagonal entries.\\
If $K_n = \mathbf{0}_6$ or if the beam clamped, then $\mathcal{M}_n = \overline{w}_n \bar{D}_n^{-1}$.
\end{enumerate}
\end{proposition}

Let us now make some observations by using Proposition \ref{prop:calMn_form}.
For any \emph{multiple node} $m \in \mathcal{N}_M$, at which $K_m \in \mathbb{S}^6_{++}$ or $K_m = \mathbf{0}_6$, if $w_m(\ell_m)\leq 0$ and $w_i(0)\geq 0$ for all $i \in \mathcal{I}_m$, then $\mathcal{M}_m$ is necessarily negative semi-definite.
For any \emph{controlled simple node} $n \in \mathcal{N}_S$ with $K_n \in \mathbb{S}_{++}^6$, one sees that $\mathcal{M}_n \in \mathbb{S}_-^{6 k_n}$ if and only if: the largest diagonal entry of $\widetilde{\mathcal{M}}_n$ (namely, $\rho C_{K_n} - w_1(0)$ if $n=0$ or $\rho C_{K_n} + w_n(\ell_n)$ if $n\neq 0$) is nonpositive, where the \emph{negative} constant
$C_{K_n}<0$ is defined by
\begin{linenomath}
\begin{equation*}
C_{K_n} = \max_{1\leq j \leq 6} \frac{(1- \Upsilon^n_j)^2(1+\Upsilon^n_j)^{-2} - 1 }{(1- \Upsilon^n_j)^2(1+\Upsilon^n_j)^{-2} + 1},
\end{equation*}
\end{linenomath}
$\{\Upsilon^n_j\}_{j=1}^6$ denoting the diagonal entries of $\Upsilon^n$.
Note that such inequalities always hold provided that $\rho>0$ is large enough. 
For any \emph{free or clamped simple node} $n \in \mathcal{N}_S$, one sees that $\mathcal{M}_n$ is negative semi-definite if and only if: $w_1(0)\geq 0$ if $n=0$, and $w_n(\ell_n) \leq 0$ if $n \neq 0$.
We thus recover the stabilization result of Theorem \ref{thm:stabilization}. 

\medskip 

\noindent \textbf{Removing one control.} Suppose that the node $n=0$ is clamped or free rather than controlled. Then, $w_1$ being increasing, one cannot both assume that $w_1(0)\geq 0$ (an equivalent characterization of $\mathcal{M}_0 \in \mathbb{S}_-^6$) and $w_1(\ell_1) \leq 0$ (which, together with the controls applied at the other simple nodes, would be sufficient to make $\mathcal{M}_1$ negative semi-definite).
Moreover, even though choosing $K_1 \in \mathbb{S}_{++}^6$ would lead to the presence of the matrix $- 2 \rho k_1^{-1} \textswab{I}_1 \mathrm{diag}(\overline{K}_1)  \textswab{I}_1$ in the expression of $\widetilde{\mathcal{M}}_1$. This matrix is only negative semi-definite (and not negative definite), which does not seem to be enough to estimate the matrix $\overline{w}_1 \mathrm{diag}(\sigma^1+\overline{K}_1)\mathrm{diag}(\sigma_i^1)^{-1}\mathrm{diag}(\sigma^1+\overline{K}_1)$ which contains the term $w_1(\ell_1)>0$.

\medskip

\noindent \textbf{Several aligned beams.} A single beam having been divided into several shorter beams by placing nodes at different locations of its spatial domain amounts to a network of beams which are serially connected (i.e. $k_n=2$) at $n \in \mathcal{N}_M$, without angle (i.e. $R_n(\ell_n) = R_{i_2}(0)$), with $K_n = \mathbf{0}_6$ and having the same material and geometrical properties at the multiple nodes -- i.e. the mass and flexibility matrices coincide at any $n \in \mathcal{N}_M$. Then for $w_n(\ell_n) = w_{i_2}(0)$, one can compute that $\widetilde{\mathcal{M}}_n$ (and consequently $\mathcal{M}_n$) is equal to the null matrix. Hence, one can stabilize these beams by applying a feedback control at one of the two simple nodes of the overall network (see Fig. \ref{fig:serial_beams}). This is in agreement with the stabilization result Theorem \ref{thm:1b_stabilization} for a single beam.

\begin{figure}\centering
\includegraphics[scale=0.8]{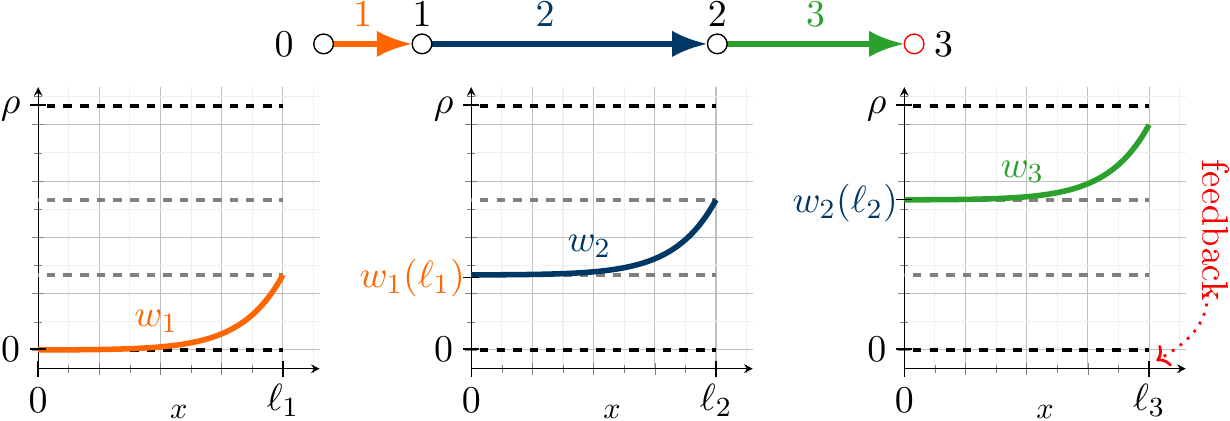}
\caption{Recovering the single beam case: example of choice of $\rho$ and weight functions.}
\label{fig:serial_beams}
\end{figure}

\begin{subappendices}
\section{Proof of Lemma \ref{lem:polyWeights}}
\label{ap:poly_weights}



%

\begin{proof}[Proof of Lemma \ref{lem:polyWeights}]
Let $n \in \{1, 2, \ldots\}$ be such that \eqref{eq:cond1_n} holds. Then,
\begin{align}
\label{eq:cond1_n_1}
&\left(\ell-\frac{n}{2\eta}\right) < 0 \leq x \leq \ell < \frac{n}{2{\eta}}.
\end{align}
For $p_n^-$, $p_n^+$ defined by \eqref{eq:def_pn}, one has
\begin{align*}
&p_n^-(0) = - \frac{1}{2^n}, \quad p_n^-(\ell) = - \left( \frac{1}{2} - \frac{{\eta}}{n}\ell \right)^{n}, \quad p_n^-\left(\frac{n}{2{\eta}}\right) = 0,\\
&p_n^+\left(\ell -\frac{n}{2{\eta}}\right) = \frac{1}{2^n}, \quad p_n^+(0) = \left(\frac{1}{2}\right)^{n} + \left( \frac{1}{2} - \frac{\eta}{n}\ell \right)^n, \quad p_n^+(\ell) = 2 \frac{1}{2^n},
\end{align*}
while the derivatives take the form
\begin{align*}
\frac{\mathrm{d}}{\mathrm{d}x} p_n^-(x) = {\eta}\left(\frac{1}{2} - \frac{{\eta}}{n}x \right)^{n-1}, \quad \frac{\mathrm{d}}{\mathrm{d}x} p_n^+(x) = {\eta}\left(\frac{1}{2} + \frac{{\eta}}{n}(x-\ell) \right)^{n-1}.
\end{align*}
Notice that due to \eqref{eq:cond1_n}, 
\begin{align} \label{eq:cond1_n_2}
0<\left(\frac{1}{2} - \frac{{\eta}}{n}x\right)\leq \frac{1}{2}, \quad \text{and} \quad 0< \left( \frac{1}{2} + \frac{{\eta}}{n}(x-\ell) \right) \leq \frac{1}{2},
\end{align}
for all $x \in [0, \ell]$, so that both polynomials are strictly increasing (in fact $2{\eta}\ell \leq n$ would be enough for this property).
On another hand, due to \eqref{eq:cond1_n_2}, one has
\begin{align*}
\frac{\mathrm{d}}{\mathrm{d}x} p_n^-(x)
&> {\eta} \left[\left(\frac{1}{2} - \frac{{\eta}}{n}x \right)^{n} - \left(\frac{1}{2} - \frac{{\eta}}{n}\ell \right)^{n}\right]\\
&= {\eta}(p_n^-(\ell) - p_n^-(x)),
\end{align*}
while
\begin{align*}
\frac{\mathrm{d}}{\mathrm{d}x} p_n^+(x)
&> {\eta} \left[\left(\frac{1}{2} + \frac{{\eta}}{n}(x-\ell) \right)^{n} - \left( \frac{1}{2} - \frac{{\eta}}{n}\ell \right)^{n}\right]\\
&= {\eta}(p_n^+(x) - p_n^+(0)).
\end{align*}
Thus, \eqref{eq:cond_der_pn} is fulfilled.
Furthermore, due to \eqref{eq:cond1_n_1}, one has
\begin{align*}
&\left(p_n^-(\ell) - p_n^-(0)\right) \leq \left(p_n^-\left(\frac{n}{2{\eta}}\right) - p_n^-(0) \right) = \frac{1}{2^n},\\
&\left(p_n^+(\ell) - p_n^+(0)\right) \leq \left(p_n^+(\ell) - p_n^+\left(\ell-\frac{n}{2\eta}\right)\right)= \frac{1}{2^n}. \qedhere
\end{align*}
\end{proof}

\end{subappendices}
\chapter{Exact controllability of nodal profiles}
\label{ch:control}


\section{Main contributions and related works.}

The problem of \emph{nodal profile controllability} of partial differential equations on networks refers to the task of steering the solution thereof to fit prescribed profiles on specific nodes. Formally speaking, this amounts to saying that said solution should be controlled to given time-dependent functions (called \emph{nodal profiles}) over certain time intervals by means of controls actuating at one or several other nodes.
This is in contrast to the classical question of exact controllability, wherein one seeks to steer the state, at a certain time, to a given final state on the entire network.
The nodes with prescribed profiles are then called \emph{charged nodes} \cite{YWang2019partialNP} (or \emph{object-nodes} \cite{Zhuang2021}), while the nodes at which the controls are applied are the \emph{controlled nodes} \cite{YWang2019partialNP} (or \emph{control nodes} \cite{Zhuang2021}).

One of the specificities of this notion of controllability is that it is not hindered by the presence of loops in the network, a situation which is encountered in many practical applications. This has notably been exhibited for the Saint-Venant equations. Whilst the exact controllability of these equations on networks with loops is not true in general \cite{LLS, li2010no}, for certain networks with cycles, the exact nodal profile controllability can be shown by means of a so-called \emph{cut-off} method \cite{Zhuang2021, Zhuang2018}.

\medskip

\noindent We have seen in Theorem \ref{th:semi_global_existence} that any network of beams governed by the IGEB model (System \eqref{eq:nIGEBgen}) admits a unique \emph{semi-global in time} $C_{x,t}^1$ solution, in the sense that the solution may exist on arbitrarily large time intervals provided that the initial and boundary/nodal data are small enough.
Via Theorem \ref{thm:solGEB}, we also made the link between \eqref{eq:nIGEBgen} and the corresponding system \eqref{eq:nGEBgen} in which the beams dynamics are given by the GEB model. 
More precisely, we showed that the existence of a unique $C^1_{x,t}$ solution to \eqref{eq:nIGEBgen} implies that of a unique $C^2_{x,t}$ solution to \eqref{eq:nGEBgen}, provided that the data of both systems fulfill some compatibility conditions.

\begin{figure} \centering
\includegraphics[scale=0.75]{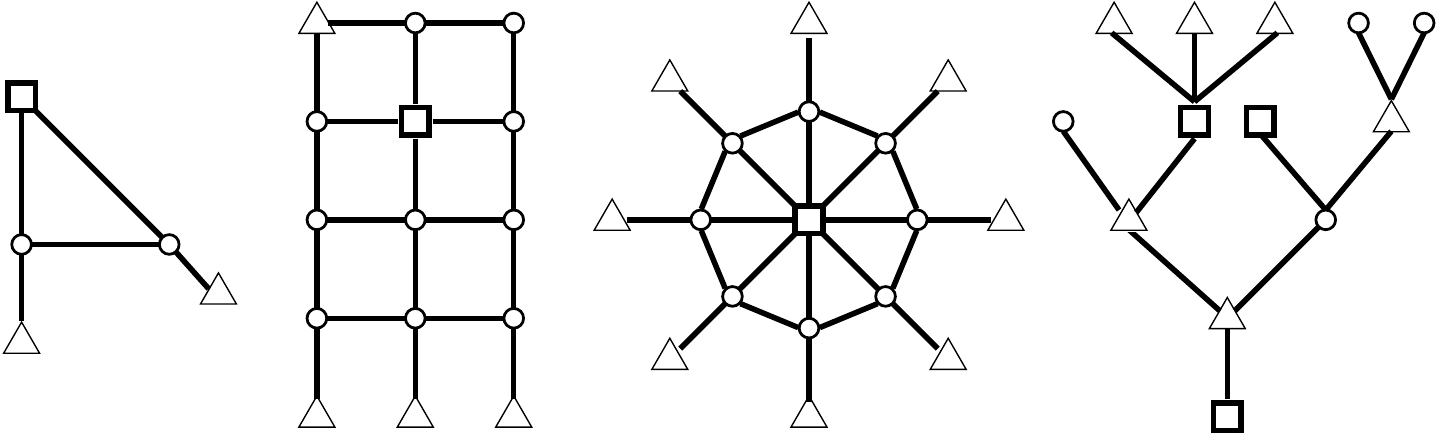}
\caption{Examples of nodal profile control problem on networks. Squares represent charged nodes, while triangles represent controlled nodes.}
\label{fig:pres_nodal_control}
\end{figure}

\medskip

\noindent In this chapter we consider the problem of \emph{local exact nodal profile controllability} in the context of a specific network of geometrically exact beams containing one cycle: the A-shaped network represented in Fig. \ref{fig:pres_nodal_control} (left).
Our main results will be given on the IGEB network (Theorem \ref{th:controllability}) and the GEB network \eqref{eq:nGEBgen} (Corollary \ref{coro:controlGEB}) as follows.
In Theorem \ref{th:controllability}, in line with \cite{Zhuang2021, Zhuang2018}, we drive the solution of \eqref{eq:nIGEBgen} to satisfy given profiles at one of the multiple nodes -- the ``tip'' of the A-shaped network -- by controlling the internal forces and moments at the two simple nodes. There, the semi-global existence and uniqueness result Theorem \ref{th:semi_global_existence} plays an important role.
Then, Theorem \ref{thm:solGEB} permits to translate this to a corresponding result in terms of the GEB network, which is Corollary \ref{coro:controlGEB}.

Afterwards, the case of other networks, possibly containing several cycles, is discussed in Section \ref{sec:more_gene_net}: we give a few typical examples together with a brief  algorithm (Algorithm \ref{algo:control}) to realize nodal profile controllability under some requirements.

\subsection*{Nodal profile control.}

The notion of exact boundary controllability of nodal profiles was, to our knowledge, first introduced by Gugat, Herty and Schleper in \cite{gugat10}, motivated by applications in the context of gas transport through pipelines networks. 
Therein, consumers are located at the endpoints of the network and the nodal profiles represent the consumer satisfaction, and the former are sought to be attained by the flow which is controlled by means of a number of compressors actuating at several nodes.

Motivated by the abundant practical relevance of such control problems, Tatsien Li and coauthors generalized the aforementioned results to one-dimensional first-order quasilinear hyperbolic systems with nonlinear boundary conditions \cite{gu2011, li2010nodal, li2016book}. 
Results on the wave equation on a tree-shaped networks with a general topology or the unsteady flow in open canals, may be found in \cite{kw2011, kw2014, YWang2019partialNP} and \cite{gu2013}, respectively. As mentionned above, for the Saint-Venant equations on networks with loops, one is referred to \cite{Zhuang2021, Zhuang2018}.


\medskip

\noindent
The method used by Li et al. to prove nodal profile controllability is \emph{constructive} in nature, in the sense that it relies on solving the equations forward in time and sidewise, to build a specific solution which achieves the desired goal, before evaluating the trace of this solution to obtain the controls. 
All this is done in the context of regular semi-global in time $C^1_{x,t}$ solutions for first-order systems.
This notion of solution is used in \cite{LiRao2002_cam, LiRao2003_sicon} for proving local exact boundary controllability of one-dimensional quasilinear hyperbolic systems. In these works, a general framework for a constructive method is proposed, from which all subsequent constructive methods derive.
The cornerstone of Li's method is thus the proof of semi-global existence and uniqueness, and, in the case of networks, a thorough study of the transmission conditions at multiple nodes. As solving a sidewise problem entails exchanging the role of the spatial and time variables, 
this method fundamentally exploits the one-dimensional nature of the system.

Very recently, in the context of the one-dimensional linear wave equation, the controllability of nodal profiles has also been studied in the context of less regular states and controls spaces, by using the duality between controllability and observability and showing an observability inequality. For star-shaped networks, one may see \cite{YWang2021_NP_HUM} where the sidewise D’Alembert Formula is used, and for a single string one may see \cite{Sarac2021} which relies on sidewise energy estimates.

\section{The A-shaped network}

\begin{figure}[b]\centering
\includegraphics[scale=0.65]{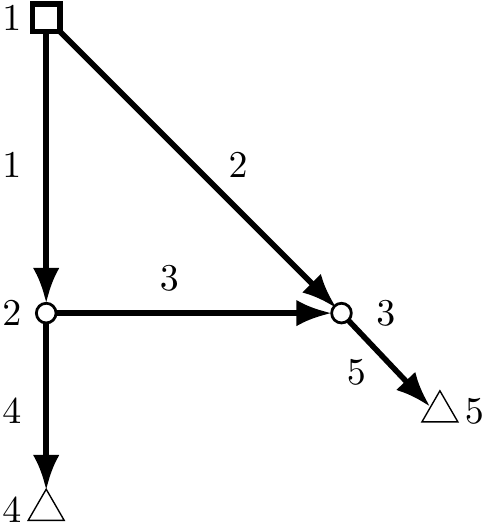}
\caption{The A-shaped network.}
\label{fig:Ashaped_control}
\end{figure}

Consider the A-shaped network illustrated in Fig. \ref{fig:Ashaped_control}, a network consisting of five nodes, five edges and containing one cycle. In other words, we set
\begin{linenomath}
\begin{align} \label{eq:A_netw}
\begin{aligned}
&\mathcal{N}_S = \mathcal{N}_S^N = \{4, 5\}, \ \mathcal{N}_M = \{1, 2, 3\}, \ \mathcal{I} = \{1, \ldots, 5\}\\
&\mathbf{x}_1^1 = 0, \ \ \mathbf{x}_2^1 = 0, \ \ \mathbf{x}_3^2 = 0, \ \ \mathbf{x}_4^2 = 0, \ \ \mathbf{x}_5^3 = 0\\
&\mathbf{x}_1^2 = \ell_1, \ \mathbf{x}_2^3 = \ell_2, \ \mathbf{x}_3^3 = \ell_3, \ \mathbf{x}_4^4 = \ell_4, \ \mathbf{x}_5^5 = \ell_5.
\end{aligned}
\end{align}
\end{linenomath}
Due to Assumption \ref{as:mass_flex}, the eigenvalues $\{\lambda_i^k\}_{k=1}^{12}$ of $A_i$ belong to $C^2([0, \ell_i])$ and we may define, for any $i\in\mathcal{I}$, the function $\Lambda_i \in C^0([0, \ell_i]; (0, +\infty))$ and the time $T_i>0$ by
\begin{linenomath}
\begin{align} \label{eq:def_Lambdai_Ti}
    \Lambda_i(x) = \left( \min_{k \in \{1, \ldots, 6\}} \left| \lambda_i^k(x) \right| \right)^{-1} \quad \text{and} \quad T_i = \int_0^{\ell_i} \Lambda_i(x) dx;
\end{align}
\end{linenomath}
note that the minimum ranges over the \textit{negative} eigenvalues of $A_i(x)$ (see \eqref{eq:sign_eigval}). 
The latter, $T_i$, corresponds to the transmission (or travelling) time from one end of the beam $i$ to its other end.

Then, one has the following local nodal profile controllability result, Theorem \ref{th:controllability}, where one sees that the \emph{controllability time} $T^*$ (from which one can prescribe nodal profiles) has to be large enough, depending on the lengths of the beams and eigenvalues of $(A_i)_{i \in \mathcal{I}}$ -- and so, on the geometrical and material properties of the beam. More precisely, it should be larger than the transmission time $\overline{T}$ from the controlled nodes to the charged node.

\begin{theorem}[Th. 2.5 \ref{A:JMPA}] \label{th:controllability}
Consider the A-shaped network defined by \eqref{eq:A_netw}.
Suppose that $R_i$ has the regularity \eqref{eq:reg_beampara} and that Assumption \ref{as:mass_flex} is fulfilled.
Let $\overline{T}>0$ be defined by (see \eqref{eq:def_Lambdai_Ti})
\begin{linenomath}
\begin{align} \label{eq:minT}
\overline{T} = \max \left\{T_1, T_2 \right\} + \max \left\{T_4, T_5 \right\}. 
\end{align}
\end{linenomath}
Then, for any $T> T^*>\overline{T}$, there exists $\varepsilon_0>0$ such that for all $\varepsilon \in (0, \varepsilon_0)$, for some $\delta, \gamma>0$, and
\begin{enumerate}[label=(\roman*)]
\item for all initial data $(y_i^0)_{i \in \mathcal{I}}$ and boundary data $(q_n)_{n \in \{1, 2, 3\}}$ of regularity \eqref{eq:reg_Idata_IGEB}-\eqref{eq:reg_Ndata_IGEB}, satisfying $\|y_i^0\|_{C_x^1} + \|q_n\|_{C_t^1} \leq \delta$ and the first-order compatibility conditions of \eqref{eq:nIGEBgen}, and
\item for all nodal profiles $\overline{y}_1, \overline{y}_2 \in C^1([T^*, T]; \mathbb{R}^{12})$, satisfying $\|\overline{y}_i\|_{C_t^1} \leq \gamma$ and the transmission conditions \eqref{eq:nIGEBgen_cont}-\eqref{eq:nIGEBgen_Kir} at the node $n=1$,
\end{enumerate}
there exist controls $q_4, q_5 \in C^1([0, T]; \mathbb{R}^6)$ with $\|q_i\|_{C_t^1}\leq \varepsilon$, such that \eqref{eq:nIGEBgen} admits a unique solution $(y_i)_{i \in \mathcal{I}} \in \prod_{i=1}^N C^1([0, \ell_i] \times [0, T]; \mathbb{R}^{12})$, which fulfills $\|y_i\|_{C_x^1} \leq \varepsilon$ and
\begin{linenomath}
\begin{align}\label{eq:aim}
y_i(0, t) = \overline{y}_i(t) \quad \text{for all }i \in \{1, 2\}, \, t \in [T^*, T].
\end{align}
\end{linenomath}
\end{theorem}

\begin{figure} 
\includegraphics[scale=0.74]{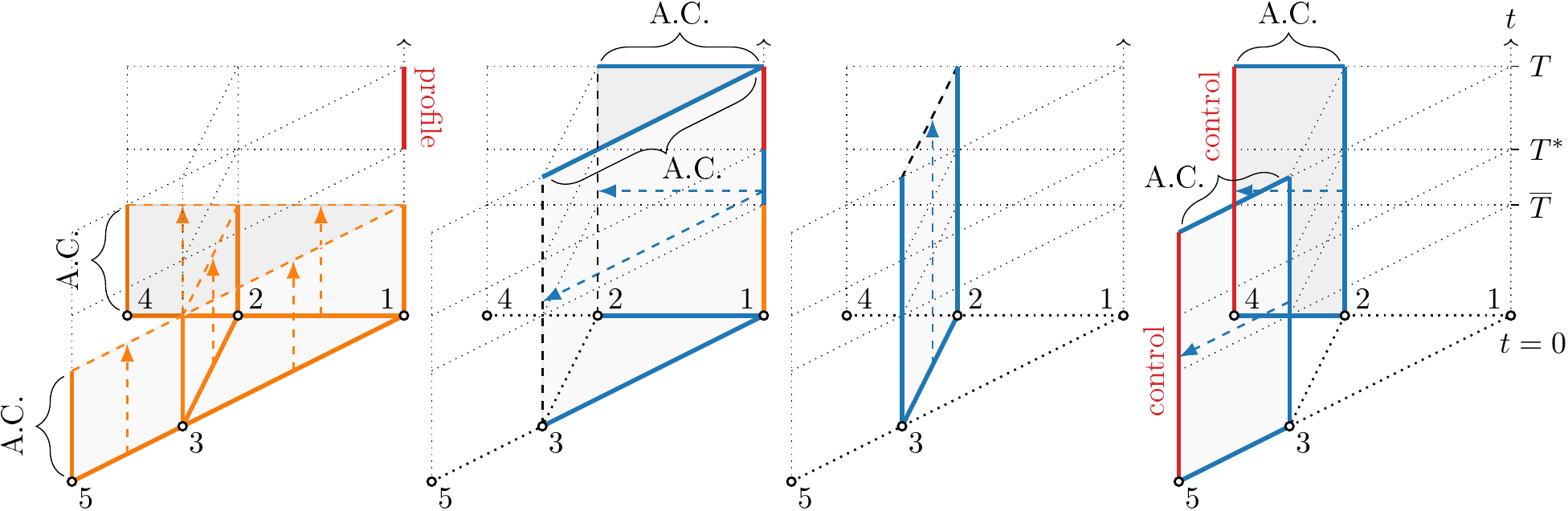}
\caption{Construction of the solution (left to right), where ``A.C.'' stands for ``artificial conditions''.}
\label{fig:A}
\end{figure}

%

\noindent \textbf{Idea of the proof.} The proof of Theorem \ref{th:controllability} relies upon the existence and uniqueness theory of \emph{semi-global} classical solutions to the network problem (here, Theorem \ref{th:semi_global_existence}), the form of the transmission condition of the network, and on the \emph{constructive method} of Li and al.. The idea of the proof is to build a solution $(y_i)_{i\in \mathcal{I}}$ to \eqref{eq:nIGEBgen} -- following the steps illustrated in Fig. \ref{fig:A} --, such that it satisfies the initial conditions, the nodal conditions, and the given nodal profiles. Substituting this solution into the conditions at the nodes $n \in \{4, 5\}$, one then obtains the desired controls $q_4, q_5$. \textcolor{black}{Our proof follows the lines of \cite{Zhuang2018}, where the authors develop a methodology for proving the nodal profile controllability for A-shaped networks of canals governed by the Saint-Venant equations.}
%
%
%
During the construction of the solution, we make sure to keep the solution and data small in order to be able to go on with the construction in the subsequent construction step.
\begin{itemize}
\item[1.] \textbf{Forward problem for the network.}
This is the step at which we make use of Theorem \ref{th:semi_global_existence}, solving the forward problem for the entire network until time $\overline{T}$, where some ``artificial conditions'' are arbitrarily set at the simple nodes $n \in \{4,5\}$. This is one of the reasons why the controls obtained at the end of the proof are in fact not unique. We call the obtained solution $(y_i^f)_{i\in \mathcal{I}}$.

Another reason why the obtained controls are not unique is that, now at the multiple node $n=1$ where the solution has to match the profiles in $[T^*, T]$, we complete the gap between the $y_1^f, y_2^f$ and $\overline{y}_1, \overline{y}_2$ for $t \in [\overline{T}, T^*]$ (using for instance Hermite splines), in order to obtain continuously differentiable ``data'' $\overline{\overline{y}}_1, \overline{\overline{y}}_2$ on the entire time interval, that coincide with the profiles and forward solution and also fulfill the transmission conditions.

\end{itemize}
\noindent In the next steps, we construct a solution of the network system, though not at once for the whole network but rather edge by edge (thus relying on \cite{li2010controllability, wang2006exact} for well-posedness, as previously for Theorem \ref{th:semi_global_existence}), not always forward in time but sometime rather sidewise.
Solving a sidewise problem  for \eqref{eq:nIGEBgen_gov} entails changing the role of $x$ and $t$, considering a governing system of the form 
\begin{linenomath}
\begin{align*}
\partial_x y_i + A_i^{-1}\partial_t y_i + A_i^{-1}\overline{B}_i y_i = A_i^{-1}\overline{g}_i(\cdot, y_i)
\end{align*}
\end{linenomath}
and providing ``boundary conditions'' at $t=0$ and $t = T$, and ``initial conditions'' at $x=0$ (rightward problem) or $x = \ell_i$ (leftward problem). It is consequently important here that $A_i$ does not have any zero eigenvalue.
\begin{itemize}

\item[2.] \textbf{Sidewise problem for the edges $1$ and $2$.}
Now that we have $\overline{\overline{y}}_i$, we solve a rightward problem on $[0, \ell_i] \times [0, T]$ for each of the edges $i\in \{1, 2\}$, the ``initial data'' at $x=0$ being $\overline{\overline{y}}_i$, while for the ``boundary data'' at $t=0$ we use the initial velocities $v_i^0$ and set artificial boundary conditions on the stresses $z_i$ at $t = T$. 


\item[3.] \textbf{Forward problem for the edge $3$.}
Then, from these two solutions $y_i$ obtained in the preceding step, we use the velocities $v_i$ at $x = \ell_i$  to build boundary data for solving a forward in time problem on the edge $i=3$. Here we must again pay attention to the transmission conditions, making sure that the data for velocities are chosen so that the continuity of velocities are respected for the overall network.


\item[4.] \textbf{Sidewise problem for the edges $4$ and $5$.}
Finally, we now solve a rightward problem on $[0, \ell_i]\times[0, T]$ for the edges $i \in \{4, 5\}$.
In addition to the values $y_1(\ell_1, \cdot)$, $y_2(\ell_2, \cdot)$, we use the values of the solution $y_3$ obtained in the preceding step at each end of its spatial interval. 
We have just enough freedom to choose ``initial data'' $\widetilde{y}_i$ ($i \in \{4, 5\}$) at the nodes $2$ and $3$ in such a way that the transmission conditions of the network are respected.

\end{itemize}
\begin{figure}\centering
\includegraphics[scale=0.8]{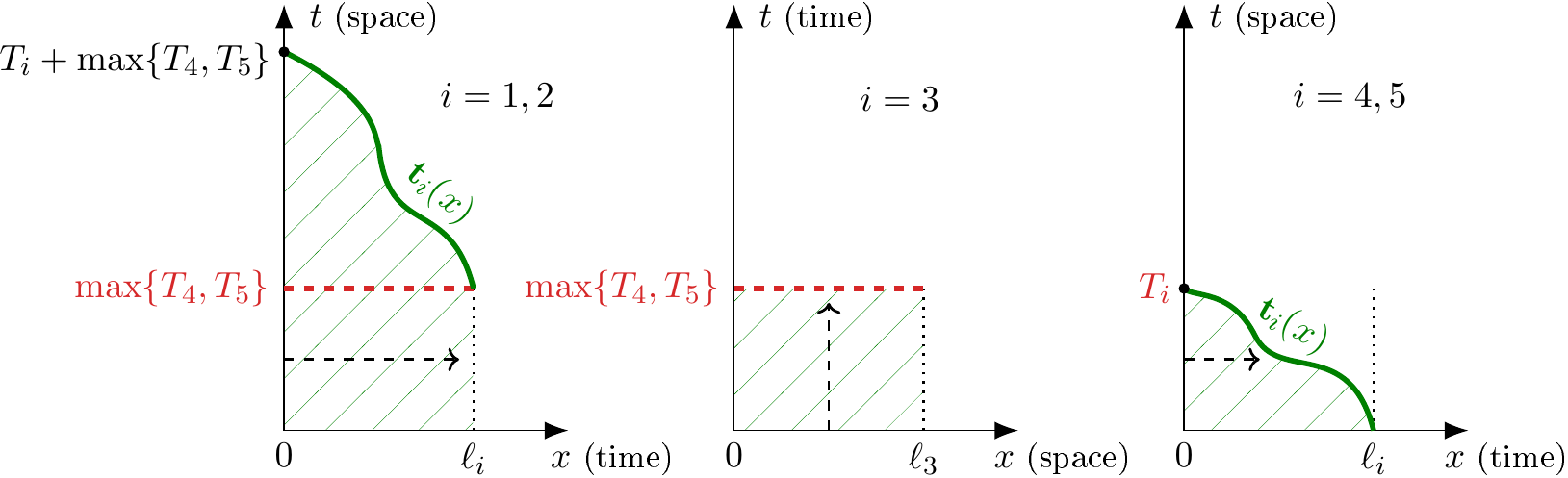}
\caption{Recovering the initial conditions: \textcolor{black}{meaning of the controllability time}.}
\label{fig:A_ini}
\end{figure}

%
\noindent
The desired controls are then chosen as $q_i = z_i(\ell_i, \cdot)$ for $i \in \{4, 5\}$, 
and the associated solution $(y_i)_{i\in\mathcal{I}}$ fulfills the nodal profiles and all nodal conditions. The remaining question is: why does this solution also fulfill the initial conditions of \eqref{eq:nIGEBgen}? We see below that this is due to the choice of $\overline{T}$ which permits to show that $(y_i)_{i\in\mathcal{I}}$ coincides with $y_i^f$ on some domain including $[0, \ell_i]\times \{0\}$ for all $i \in \mathcal{I}$ (depicted with hashed lines in Fig. \ref{fig:A_ini}).
\begin{itemize}
\item[5.] \textbf{Coincidence on the edges $1$ and $2$.}
For $i \in \{1, 2\}$, both $y_i$ and $y_i^f$ solve the one-sided rightward problem with ``initial data'' $\overline{\overline{y}}_i$ at $x=0$ and boundary data $v_i^0$ at $t=0$. Now, by \cite[Section 1.7]{li2016book} the solution to this problem is unique in the domain $\{(x,t)\colon 0 \leq x \leq \ell_i, \ 0 \leq t \leq \mathbf{t}_i(x)\}$ delimited by\footnote{This is related to the fact that any characteristic curve of the system passing by a point $(x,t)$ of this domain (i.e. the curves specified by the functions $\mathbf{t}_i^k$ for $k \in \{1, \ldots, 12\}$ with derivative $\frac{\mathrm{d}}{\mathrm{d}s} \mathbf{t}_i^k(\xi) = \lambda_i^k(\xi)^{-1}$ and such that $\mathbf{t}_i^k(x) = t$) is necessarily entering the domain at $\{0\} \times [0, T_i+\max \{T_4, T_5\}]$ or at $[0, \ell_i] \times \{0\}$.}
\begin{align*}
\mathbf{t}_i(x) = T_i + \max \{T_4, T_5\} + \int_0^x \underbrace{\min_{1\leq k \leq 12} \frac{1}{\lambda_i^k(\xi)}}_{ =-\Lambda_i(\xi)} d\xi,
\end{align*}
which implies in particular -- due to the definition of $T_1, T_2$ and $\Lambda_1, \Lambda_2$ -- that both functions coincide in $[0, \ell_i] \times [0, \max\{T_4, T_5\}]$.

\item[6.] \textbf{Coincidence on the edge $3$.}
Since $y_1 \equiv y_1^f$ and $y_2 \equiv y_2^f$ for all $t \in [0, \max\{T_4, T_5\}]$, the function $y_3^f$ solves the same forward problem as $y_3$ in $(0, \ell_3)\times (0, \max\{T_4, T_5\})$ and both are consequently equal on this domain.


\item[7.] \textbf{Coincidence on the edges $4$ and $5$.}
Due to the fact $y_1^f, y_2^f, y_3^f$ coincide with $y_1, y_2, y_3$ at the nodes $2$ and $3$, we deduce that for all $i\in\{4, 5\}$ the function $y_i^f$ solves the same the one-sided rightward problem as $y_i$, with ``initial data'' $\widetilde{y}_i$ at $x=0$ and boundary data $v_i^0$ at $t=0$. As above, since the solution to this problem is unique in the domain $\{(x,t)\colon 0 \leq x \leq \ell_i, \ 0 \leq t \leq \mathbf{t}_i(x)\}$ delimited by 
\begin{align*}
\mathbf{t}_i(x) = T_i - \int_0^x \Lambda_i(\xi)d\xi,
\end{align*}
we deduce -- by definition of $T_4,T_5$ and $\Lambda_4, \Lambda_5$ -- that $y_i \equiv y_i^f$ on $(0, \ell_i)\times\{0\}$. \hfill $\meddiamond$

\end{itemize}

Finally, from Theorems \ref{th:controllability} and \ref{thm:solGEB}, one obtains the corresponding controllability result from the perspective of the GEB model in Corollary \ref{coro:controlGEB} below. There, the profiles prescribed at the node $n=1$ affect the intrinsic variables $\{\mathcal{T}_i(\mathbf{p}_i, \mathbf{R}_i)\}_{i \in \{1, 2\}}$.

\begin{corollary}[Coro. 2.11 \ref{A:JMPA}] \label{coro:controlGEB}
Consider the A-shaped network defined by \eqref{eq:A_netw}, and assume that

\begin{enumerate}[label=(\roman*)]
\item the beam parameters $(\mathbf{M}_i, \mathbf{C}_i, R_i)$ and initial data $(\mathbf{p}_i^0, \mathbf{R}_i^0, \mathbf{p}_i^1, w_i^0)$ have the regularity \eqref{eq:reg_beampara} and \eqref{eq:reg_Idata_GEB}, satisfy Assumption \ref{as:mass_flex} and \eqref{eq:compat_GEB_transmi}, and $y_i^0$ is the associated function defined by \eqref{eq:rel_inidata},

\item the Neumann data $f_n = f_n(t, \mathbf{R}_{i^n})$, for $n \in \{1, 2, 3\}$ are of the form \eqref{eq:def_fn}, for given functions $q_n$ of regularity \eqref{eq:reg_Ndata_IGEB},

\item $(y_i^0)_{i \in \mathcal{I}}$ satisfies the first-order compatibility conditions of \eqref{eq:nIGEBgen}.
\end{enumerate}
Let $\overline{T}>0$ be defined by \eqref{eq:minT}. Then, for any $T>T^*>\overline{T}$, there exists $\varepsilon_0>0$ such that for all $\varepsilon\in (0, \varepsilon_0)$, for some $\delta, \gamma>0$, and for any nodal profiles $\overline{y}_1, \overline{y}_2 \in C^1([T^*, T]; \mathbb{R}^{12})$ satisfying $\|\overline{y}_i\|_{C_t^1}\leq \gamma$ and the  transmission conditions \eqref{eq:nIGEBgen_cont}-\eqref{eq:nIGEBgen_Kir} at the node $n=1$, if additionally $\|y_i^0\|_{C_x^1} + \|f_n\|_{C_t^1} \leq \delta$ ($i \in \mathcal{I}, \ n \in \{1, 2, 3\}$), then there exist controls $f_4, f_5 \in C^1([0, T]; \mathbb{R}^6)$ with $\|f_n\|_{C_t^1}\leq \varepsilon$ such that System \eqref{eq:nGEBgen} with initial data $(\mathbf{p}_i^0, \mathbf{R}_i^0, \mathbf{p}_i^1, w_i^0)$ and boundary data $(f_n)_{n \in \{1, 2, 3\}}$, admits a unique solution $(\mathbf{p}_i, \mathbf{R}_i)_{i \in \mathcal{I}} \in \prod_{i=1}^N C^2([0, \ell_i]\times[0, T]; \mathbb{R}^3 \times \mathrm{SO}(3))$, and $(y_i)_{i\in\mathcal{I}} := \mathcal{T}((\mathbf{p}_i, \mathbf{R}_i)_{i\in\mathcal{I}})$ fulfills $\|y_i\|_{C_{x,t}^1}\leq \varepsilon$ and the nodal profiles \eqref{eq:aim}.
\end{corollary}

%

\section{Toward more general networks}
\label{sec:more_gene_net}

\begin{figure}[b]
    \centering
    \includegraphics[scale = 0.775]{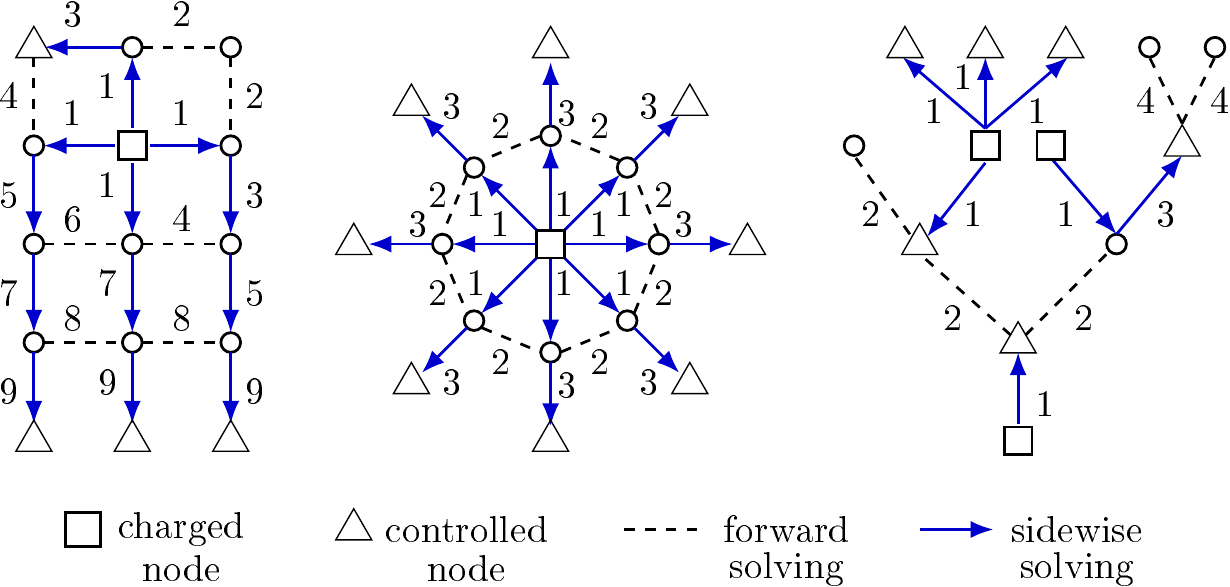}
    \caption{Other networks for which local exact controllability of nodal profiles is achievable by following Algorithm \ref{algo:control} (the numbers refer to the variable \textsf{step}).}
    \label{fig:otherControllableNet}
\end{figure}

The A-shaped network gives an illustrative example of a network with cycles where the controllability of nodal profiles is achievable, and similar arguments can be applied to various networks, and with controls at different locations.
On tree-shaped networks, some conditions were proved to be sufficient for the exact controllability of nodal profiles to be achieved \cite{gu2011, gu2013, li2016book, kw2011, kw2014, YWang2019partialNP}.
The nodal profile controllability has also been established for the Saint-Venant system \cite{Zhuang2021, Zhuang2018} for numerous networks with cycles, of various shapes and with several charged nodes. 
 

\begin{algorithm}
\DontPrintSemicolon

\Input{%
$\mathcal{I}, \, \mathcal{N}, \, \mathcal{N}_S, \, k_n$\tcp*{edges, nodes, simple nodes, degrees}
\hspace{1.2cm}$\mathcal{I}^n$ for all $n\in \mathcal{N}$\tcp*{edges incident to the node $n$}
\hspace{1.2cm}$\mathcal{N}^i$ for all $i \in \mathcal{I}$\tcp*{nodes at the tips of the edge $i$}
\hspace{1.2cm}$\mathcal{P}$, $\mathcal{C}$\tcp*{charged nodes, controlled nodes}
\hspace{1.2cm}$\mathcal{S}$\tcp*{edges on control paths}
}

\setcounter{AlgoLine}{0}
\ShowLn
$J$ $\leftarrow$ $[\, 0, \quad $for $n = 1 \ldots \#\mathcal{N}\, ]$; \ \lFor{all $n \in \mathcal{P}$}{($J(n)$ $\leftarrow$ $k_n - 1$);} 
\tcp*{amount $J(n)$ of data available at $n$ to solve sidewise}

\ShowLn
$\mathcal{F}$ $\leftarrow$ $\emptyset$;\tcp*{solved edges}
 
\ShowLn
\textsf{step} $\leftarrow$ $1$; \ \textsf{moved} $\leftarrow$ \textsf{false};\tcp*{to count the steps}

\ShowLn
$\mathcal{M}$  $\leftarrow$ $\mathcal{P}\cup \{n \in \mathcal{N} \colon\text{$n$ not incident with any edge in $\mathcal{S}$}\};$\;
 
\ShowLn
\While(\hfill \tcp*[h]{while entire network not solved}){$\mathcal{F} \neq \mathcal{I}$}{
 
\ShowLn
\For(\hfill \tcp*[h]{Principle 1}){$m =\#\mathcal{M}, \ldots, 3, 2, 1$}{

\ShowLn
\For{all ${\mathcal{N}}^\dagger \subseteq \mathcal{M}$ such that $\#\mathcal{N}^\dagger = m$}{

\ShowLn
\If{there exists a connected subgraph with nodes ${\mathcal{N}}^\dagger$ and edges ${\mathcal{I}}^\dagger$, such that ${\mathcal{I}}^\dagger \cap ( \mathcal{S}\cup \mathcal{F}) = \emptyset$}{

\ShowLn
  solve forward problem for the network $({\mathcal{I}^\dagger}, {\mathcal{N}}^\dagger)$;\;
  
\ShowLn
\lFor{all $n \in {\mathcal{N}}^\dagger$}{
  ($J(n)$ $\leftarrow$ $J(n) + 1$);}
  
\ShowLn
$\mathcal{M}$ $\leftarrow$ $\mathcal{M} \cup {\mathcal{N}}^\dagger$; \, $\mathcal{F}$ $\leftarrow$ $\mathcal{F} \cup {\mathcal{I}}^\dagger$; \ \textsf{moved} $\leftarrow$ \textsf{true};\;
}}}

\ShowLn
\lIf{\upshape \textsf{moved} $=$ \textsf{true} }{(\textsf{step} $\leftarrow$ \textsf{step} $+ 1$; \ \textsf{moved} $\leftarrow$ \textsf{false});}  
   
\ShowLn
\For(\hfill \tcp*[h]{Principle 2}){all $n \in \mathcal{M}$}{

\ShowLn
\If(\hfill \tcp*[h]{if enough data at $n$}){$J(n) = k_n - 1$}{

\ShowLn
\For{all $i \in \mathcal{I}^n \cap (\mathcal{S} \setminus \mathcal{F}$)}{

\ShowLn
solve sidewise problem for the edge $i$ with ``initial \;
  conditions'' at the node $n$;\;

\ShowLn
$\mathcal{M}$ $\leftarrow$ $\mathcal{M} \cup \mathcal{N}^i$; \, $\mathcal{F}$ $\leftarrow$ $\mathcal{F}\cup \{i\}$;\;

\ShowLn
\lFor{all $m \in \mathcal{N}^i$}{
  ($J(m)$ $\leftarrow$ $J(m) + 1$);}
  } 
  
\ShowLn 
\textsf{moved} $\leftarrow$ \textsf{true};\;
}}

\ShowLn
\lIf{\upshape \textsf{moved} $=$ \textsf{true}}{(\textsf{step} $\leftarrow$ \textsf{step} $+ 1$; \ \textsf{moved} $\leftarrow$ \textsf{false});} 
}

\ShowLn  
Compute controls $q_n$ by evaluating the trace at nodes $n\in\mathcal{C}$;\;
 
\caption{Steps of controllability proof for other networks}
\label{algo:control}
\end{algorithm}

\medskip

\noindent It arises, from these works, a series of conditions on the number and location of the charged nodes, which are sufficient to achieve the respective controllability goals; see \cite[Theorem 5.1]{YWang2019partialNP}, and \cite[Sections 7 and 8]{Zhuang2021}.
%
Let us denote by $\mathcal{P}$ and $\mathcal{C}$ the set of indexes of the \emph{charged nodes} and \emph{controlled nodes}, respectively.
Given a charged node $n\in\mathcal{P}$ and a controlled node $m\in\mathcal{C}$, a \emph{control path} between $n$ and $m$ \cite{YWang2019partialNP}, is any connected subgraph (of that which represents the beam network) forming a \emph{path graph}\footnote{Namely, an oriented graph without a loop such that two of its nodes are of degree $1$, and all others have a degree $2$.} with $n$ and $m$ as nodes of degree $1$. Examples of control paths are highlighted by arrows in Fig. \ref{fig:otherControllableNet}.
In the case of the beam networks considered in this article, these conditions become:
\begin{enumerate}
\item The total number of controlled nodes $\#\mathcal{C}$ is equal to $\displaystyle \sum_{n\in \mathcal{P}} k_n$ (recall that $k_n$ is the degree of $n$).
\item For any $n \in \mathcal{P}$, there are $k_n$ controlled nodes connected with $n$ through control paths. The sole common node of these control paths is the charged node $n$. 
\item Control paths corresponding to different charged nodes do not have any common node.
\end{enumerate}
We stress that at any $n \in\mathcal{P}$, profiles are prescribed for all incident beams and the entire state.
Then, one may use the constructive method by following the steps instructed by Algorithm \ref{algo:control} for networks such as those depicted in Fig. \ref{fig:otherControllableNet}. However, note that there isn't any proof that Algorithm \ref{algo:control} yields a controllability result for \emph{any} network fulfilling the three above conditions.
In this algorithm, edges belonging to control paths are solved according to the \emph{Principle 2} -- solving a sidewise problem as for the edges $1, 2, 4, 5$ -- while the other edges are solved according to the \emph{Principle 1} -- solving a forward problem as for the edge $3$.

\chapter{Summary and outlook}
\label{ch:conclusion}



We have studied freely vibrating beams described by the GEB and IGEB models, the former being expressed in terms of positions $\mathbf{p}$ and rotation matrices $\mathbf{R}$ (or $\mathbf{p}_i$ and $\mathbf{R}_i$ for networks) expressed in a fixed coordinate system, whereas the latter is rather expressed in terms of the velocities $v$ and internal forces and moments $z$ (resp. $v_i$ and $z_i$) expressed in a moving coordinate system attached to the beam. For such beams, the strains are small but the beams may undergo large displacements of the centerline and rotations of the cross sections. The constitutive law includes possibly composite anisotropic materials. 
We took advantage of the fact that when the beam is described by the IGEB model, the governing system is of first order in space and time, hyperbolic and only semilinear, while for the GEB model the governing system is of second order and quasilinear.

For the IGEB model, for a single beam, tree-shaped networks, or networks with loops, one obtains existence and uniqueness of local in time $C_t^0H_x^k$ solution and semi-global in time $C_{x,t}^1$ solutions. While well-posedness for a single beam is almost immediately provided by the existing literature, for networks one should pay attention to the transmission conditions in the \emph{diagonal system} (the system written in Riemann invariants) and verify that at each node the components of the state corresponding to characteristics entering the domain are expressed in terms of those leaving the domain. From the classical rigid joint and Kirchhoff conditions given in terms of $\mathbf{p}_i$ and $\mathbf{R}_i$, we derive the corresponding transmission conditions in terms of $v_i$ and $z_i$, and prove that the diagonal system has the desired properties.

Looking at the definition of the unknowns of the GEB model, one can see that it is connected to the IGEB model by the nonlinear transformation
\begin{align*}
\mathcal{T} \colon (\mathbf{p}, \mathbf{R}) \longmapsto \begin{bmatrix}
\mathbf{I}_6 & \mathbf{0} \\
\mathbf{0} & \mathbf{C}^{-1}
\end{bmatrix}
\begin{bmatrix}
\mathbf{R}^\intercal \partial_t \mathbf{p}\\ \mathrm{vec}\left( \mathbf{R}^\intercal \partial_t \mathbf{R} \right) \\
\mathbf{R}^\intercal \partial_x \mathbf{p}  - e_1 \\ 
\mathrm{vec}\left(\mathbf{R}^\intercal \partial_x \mathbf{R} \right) - \Upsilon_c
\end{bmatrix}
\end{align*}
(or $\mathcal{T}_\mathrm{net}$ for networks, defined in \eqref{eq:transfo_net}).
It turns out that inverting this transformation can be seen as solving a linear system of PDEs which at first sight seems overdetermined, but not if the last six governing equations of the IGEB model are used as compatibility conditions. Then, if the initial and boundary data of both systems are compatible, one may see that the existence and uniqueness of classical solutions to the IGEB system, implies that of classical solutions to the corresponding GEB system. 

The utility of this result is, notably, that one has the possibility to choose data for the GEB model in such a way that well-posedness, exponential stabilization or nodal profile controllability is achieved for the corresponding IGEB model, and as a result obtain a solution to the GEB model with some properties of the states -- as time goes to infinity or after some controllability time for instance -- in terms of velocities and internal forces and moments.

For a single beam and for a star-shaped network we show that the zero steady state is locally exponentially stable for the $H^1$ and $H^2$ norms. We use quadratic Lyapunov functionals characterized by functions $\overline{Q}$ (or $\overline{Q}_i$ for networks) of the form
\begin{align*}
\overline{Q} = \rho \, Q^\mathcal{P} + w \begin{bmatrix}
\mathbf{0} & \mathbf{C}^{-\sfrac{1}{2}}(\mathbf{C}^{\sfrac{1}{2}}\mathbf{M}\mathbf{C}^{\sfrac{1}{2}})^{\sfrac{1}{2}} \mathbf{C}^{\sfrac{1}{2}}\\
\mathbf{C}^{-\sfrac{1}{2}}(\mathbf{C}^{\sfrac{1}{2}}\mathbf{M}\mathbf{C}^{\sfrac{1}{2}})^{\sfrac{1}{2}} \mathbf{C}^{\sfrac{1}{2}} & \mathbf{0}
\end{bmatrix}
\end{align*}
in the \emph{physical system}, or $Q$ (resp. $Q_i$) in the \emph{diagonal system}
\begin{align*}
Q = \begin{bmatrix}
(\rho + w)\mathbf{I}_6 & \mathbf{0}\\
\mathbf{0} & (\rho - w)\mathbf{I}_6
\end{bmatrix} Q^\mathcal{D},
\end{align*}
where $\rho>0$ and $w$ is increasing -- with other properties depending on the boundary or nodal conditions -- and where $Q^\mathcal{P}$ and $Q^\mathcal{D}$ are the matrices characterizing the energy of the beam from each perspective. We also discuss the obstacles encountered if one of the external nodes of the star-shaped network is clamped or free, rather than controlled.

Finally, for a network with a loop -- more precisely an A-shaped network -- we prove local exact nodal profile controllability of the multiple node at the ``tip'' of the A when internal forces and moments are controlled at the two simple nodes. Just as the exponential stability proof rests on the existence and uniqueness of local in time solutions in $C_t^0H_x^k$ -- extending this local solution to the whole time interval by means of the Lyapunov functionnal \cite{BC2016} --, this controllability proof relies on the existence and uniqueness of semi-global in times solutions in $C_{x,t}^1$. The constructive method of \cite{LiRao2002_cam, LiRao2003_sicon, Zhuang2021, Zhuang2018} makes use of this well-posedness result to build a solution matching the desired profiles. We conclude by discussing the case of different networks of geometrically exact beams.

\medskip

Throughout this work, we encountered several new questions, obstacles, and paths to possible extensions. Let us now present some of them.

\section{Modelling, analysis}

\noindent \textbf{More general data and external forces and moments.}
One may observe that the proofs here make use of the fact that the given data -- initial and boundary data, external forces and moments -- are expressed in a way that is suitable to the study of the IGEB model.
As mentioned in Remark \ref{rem:equiv_geb_igeb}, one may want to assume that external forces, such as gravity \cite[eq. (4)]{Artola2019mpc} or aerodynamic forces \cite[eq. (12)]{Palacios2011intrinsic}, which may be functions of $x$ or $(\mathbf{p}, \mathbf{R})$, are applied on the beam, or to set boundary conditions whose expression in the body-attached basis is unavailable.
It would thus be interesting to see if studying simultaneously the IGEB system and the transformation $\mathcal{T}$ (or $\mathcal{T}_\mathrm{net}$ for networks) could help in this direction. This is, in some way, what is being done on the computational side, for instance in \cite{Artola2019mpc} to implement the force exerted by gravity in simulations of intrinsic beams.

\medskip


\begin{figure}[b]\centering
\includegraphics[scale=0.6]{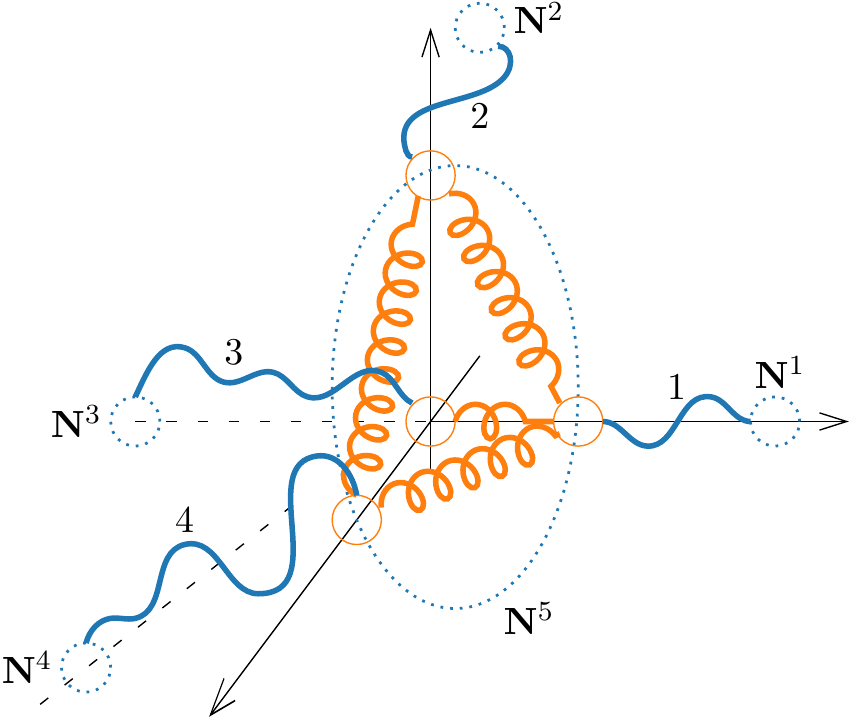}
\caption{Nonlinear strings coupled to an elastic body in $\mathbb{R}^3$. The resulting network has four simple nodes $\{\mathbf{N}_i\}_{i=1}^4$ and one multiple nodes $\mathbf{N}_5$, itself expended into a network of springs.}
\label{fig:ssm}
\end{figure}

\noindent \textbf{Different junctions at multiple nodes.} In \ref{P:ssm}, to represent a network of nonlinear strings coupled to elastic bodies in $\mathbb{R}^3$, we consider strings with masses attached to their ends, and replace each multiple node by a subnetwork of springs (see Fig. \ref{fig:ssm}). Also by means of a constructive method (thus relying on the existence and uniqueness of semi-global in time classical solutions), we then prove local exact controllability of a star-shaped network by means of controls applied at all simple nodes but one, which is clamped. It would be of interest to see how such a coupling at the multiple nodes may be modelled in the case of geometrically exact beams, and what are the resulting controllability properties. 

\medskip

\noindent \textbf{Kelvin-Voigt damping.}
One may want to account for structural damping in the beam model -- coming for instance from friction
between the particles that constitute the object --, and see the influence of this on well-posedness or stabilization results. 
As discussed in Section \ref{sec:mass_flex_materialLaw}, in \cite{Artola2021damping} the authors introduce Kelvin-Voigt damping in the IGEB equations and  obtain a system with internal damping appearing in the first six governing equations. It would be of interest to see if one can then prove well-posedness or stabilization results without any smallness assumption on the initial and boundary data, or a moderate ones depending on the damping parameters.


\medskip

\noindent \textbf{Well-posedness.}
Since we consider a very specific hyperbolic system -- the IGEB model --, one could establish a well-posedness result tailored to this model and keep track of the bounds on the initial and boundary data to obtain more quantitative information.
Indeed, here we rely either on \cite{BC2016} or \cite{LiJin2001_semiglob} which have been developed for abstract systems. Moreover, a well-posedness result devised directly for the physical system might eliminate the additional regularity assumptions made on the eigenvalues and eigenvectors of $A$ (and $A_i$ for networks); see also Remark \ref{rem:wellp_fb} \ref{remItem:extra_regularity}.

\section{Stabilization, control}

\noindent \textbf{Exponential decay.} Looking more closely at the exponential decay of the Lyapunov functional $\overline{\mathcal{L}}$, and thus of the solution, one can see that there is a competition between the largest eigenvalue of $\overline{S}$ (which is negative) and the constant $\delta$ (in Proposition \ref{prop:1b_existence_Lyap}) which notably constrains the size of the initial data ($\varepsilon$ in Definition \ref{def:stability}). This eigenvalue notably depends on how large the derivative of the weight $w$ is. 
We have seen in Chapter \ref{ch:stab} that a clamped of free end boundary condition adds a constraint on the sign of $w$, while the choice of the feedback matrix $K$ influences the value of $w$ at the other end\footnote{
However, one can see in \eqref{eq:boundary_terms_physical}, or in Proposition \ref{prop:calMn_form}, that this constraint disappears in the case of transparent boundary conditions, when $\mathbf{W}$ is defined by \eqref{eq:boldW_MCfrac}.
} and thereby its derivative over the whole interval. 

In \ref{A:SICON}, where $K$ is constrained to be of the form $K = \mathrm{diag}(\mu_1 \mathbf{I}_3, \mu_2 \mathbf{I}_3)$ and the mass and flexibility matrices are constant and diagonal, we see that the feedback parameters $\mu_1, \mu_2>0$ yielding the largest freedom for the choice of $w$ are of the form
\begin{align} \label{eq:K_closeTransp}
\mu_1 = \sqrt{\min_{1\leq i \leq 3}b_i\max_{1\leq j \leq 3} b_j }, \qquad \mu_2 = \sqrt{\min_{4\leq i \leq 6} b_i \max_{4\leq j \leq 6} b_j},
\end{align}
where $\{b_i\}_{i=1}^6$ are the diagonal entries of $\mathbf{M}^{\sfrac{1}{2}}\mathbf{C}^{- \sfrac{1}{2}}$.
Overall, given a certain setting, many factors can be taken into account to choose the ``best'' feedback matrix $K$, weight $w$, constraint $\delta$ on the initial data. One may also be interested in the impact of the length of the beam and that of the other beam parameters.
For instance, in \ref{P:CDC}, we look for quadratic Lyapunov functionals with polynomial weights that maximize some of the region of attraction estimates (i.e., the constraints on the size of the initial data $\varepsilon$).

\medskip

\noindent \textbf{Removing one feedback control.} As explained in Section \ref{sec:stab_net}, removing one control and retaining the stabilization result for the star-shaped network is not straightforward. If the boundary terms coming from the transmission condition cannot be analysed further than in Section \ref{sec:pov_diag_syst}, the difficulty encountered here might be technical only and one could look to find a different Ansatz for the Lyapunov functional (as in \ref{P:CDC} for example), or a different method altogether to deduce exponential stability.

\medskip

\noindent \textbf{Nodal profile control.} There is a lot that can be done in view of generalising the nodal profile control result. The Algorithm \ref{algo:control} in itself does not constitute a proof that the conditions 1, 2 and 3 given in Section \ref{sec:more_gene_net} are sufficient to guarantee controllability for any network that fulfills them. 
Naturally, as in the aforementioned works on nodal profile control, the conditions given to obtain controllability of nodal profiles are only \emph{sufficient} to ensure the controllability result, even for a specific network, and the search for necessary and sufficient conditions is open.

Finally, as in \cite{YWang2019partialNP}, one may also want to allow profiles to be prescribed for only some (rather than all) of the incident edges at a given charged node $n \in \mathcal{P}$, and to prescribe profiles for only part of the state (for instance, in the case of \eqref{eq:nIGEBgen}, prescribing the velocities or internal forces and moments only). This type of problem is then called \emph{partial} nodal profile controllability.

\section{Numerics}

To gain better intuition and understanding of the stabilization problem studied here, we would like to explore the influence of feedback control in various contexts through numerical simulations. Notably, we are interested in comparing the effects of different initial data and beam parameters, and that of a positive semi-definite (rather than positive definite) matrix $K$. Moreover, if the numerical simulations may be extended to networks then we are interested in testing different settings such as the star-shaped network with all simple nodes but one controlled, and more general network shapes (starting with tree-shaped networks).

On the other hand, we want not to simulate directly the GEB model, but rather the IGEB model and afterwards use the transformation $\mathcal{T}$ (defined by \eqref{eq:transfo}) to recover the position of the beam. What we present here is of course not general and does not involve profound numerical analysis; further studies would have to be realized to validate the obtained finite dimensional model. We solely give results of numerical simulations in a specific situation where the beam parameters are those which are provided in \cite[Fig. 2]{hesse2012}, more precisely the mass and flexibility matrices are as follows
\begin{align*}
\mathbf{M} &= \mathrm{diag} (1, 1, 1, 20, 10, 10), \quad \mathbf{C} = \mathrm{diag} (10^4, 10^4, 10^4, 500, 500, 500)^{-1}.
\end{align*}
In addition, we choose a specific initial datum for which computing the initial strains analytically is possible, as we will see below. We consider a single beam, clamped at $x=0$ and controlled at $x = \ell$, which is thus described by System \eqref{eq:GEBfb} or its intrinsic counterpart \eqref{eq:IGEBfb}. As for the feedback matrix $K$, we choose $K = \mathrm{diag}(\mu_1 \mathbf{I}_3, \mu_2 \mathbf{I}_3)$ where $\mu_1, \mu_2 > 0$ are defined by \eqref{eq:K_closeTransp}, 
so that $K$ is close, but not equal, to the matrix that yields transparent boundary conditions.

In this section we present the approximation methodology, and stress some links with the study realized in Section \ref{sec:1b_invert_transfo} on the inversion of the transformation $\mathcal{T}$. 
Even though, in this setting, the mass and flexibility matrices are constant and diagonal, the following presentation is given for general $\mathbf{M}, \mathbf{C} \in \mathbb{S}_{++}^6$.
Our \texttt{Matlab} code is available at \href{https://github.com/chrdz/GEB-Feedback}{\texttt{https://github.com/chrdz/GEB-Feedback}}.
We are very grateful to Marc Artola (Imperial College London) and Dani\"el Veldman (FAU Erlangen-N\"urnberg) for their guidance and advice.
 
\medskip



First, we semi-discretize in space the IGEB model by means of the finite element method.
Beforehand, it is convenient to write the governing system in the equivalent form
\begin{align*}
Q^\mathcal{P}(x) \partial_t y + \mathbf{A} \partial_x y + \mathbf{B}(x) y + \mathbf{G}(x, y)y = 0,
\end{align*}
with $Q^\mathcal{P}$ defined in \eqref{eq:def_boldE_calP} and $\mathbf{A}, \mathbf{B}(x), \mathbf{G}(x, y) \in \mathbb{R}^{12 \times 12}$ defined by
\begin{align*}
\mathbf{A} = - \begin{bmatrix}
\mathbf{0} & \mathbf{I}_6\\
\mathbf{I}_6 & \mathbf{0}
\end{bmatrix}, \quad \mathbf{B}(x) = \begin{bmatrix}
\mathbf{0} & -\mathbf{E}(x)\\
\mathbf{E}(x)^\intercal & \mathbf{0}
\end{bmatrix} , \quad \mathbf{G}(x, y) = - \mathcal{G}(y)Q^\mathcal{P}(x),
\end{align*}
where we recall that the matrix $\mathcal{G}(y) \in \mathbb{R}^{12 \times 12}$ is defined in \eqref{eq:def_calG}.
As in \cite{Artola2019mpc}, it is also convenient to introduce the linear maps $L_1, L_2$ defined below, so that $\mathbf{G}$ also reads as
\begin{align*}
\mathbf{G}( \cdot ,y ) = \begin{bmatrix}
L_1(v) \mathbf{M} & L_2(z) \mathbf{C}\\
\mathbf{0} & - L_1 (v)^\intercal \mathbf{C}
\end{bmatrix}, \quad \text{with}  \quad L_1(w) := \begin{bmatrix}
\widehat{w}_2 & \mathbf{0}\\
\widehat{w}_1 & \widehat{w}_2
\end{bmatrix}, \quad L_2(w) := \begin{bmatrix}
\mathbf{0} & \widehat{w}_1 \\
 \widehat{w}_1  & \widehat{w}_2
\end{bmatrix},
\end{align*}
where $w_1$ and $w_2$ just denote the first three and last three components of any vector $w\in \mathbb{R}^{12}$. Then, for $\mathbf{G}_\dagger$ defined by
\begin{align*}
\mathbf{G}_\dagger (\cdot,y) = \begin{bmatrix}
-L_2(\mathbf{M}v) & - L_1(\mathbf{C}z)\\
L_1(\mathbf{C}z)^\intercal & \mathbf{0}
\end{bmatrix},
\end{align*}
one has the identity $\mathbf{G}(x, y)\overline{y} = \mathbf{G}_\dagger (x, \overline{y})y$ for all $y, \overline{y} \in \mathbb{R}^{12}$ (see \cite{Artola2019mpc} and the proof of the Lemma \ref{lem:dissip_barg}).
Henceforth, we denote by $\{\widetilde{\mathbf{e}}_i\}_{i=1}^6$, $\{\overline{\mathbf{e}}_i\}_{i=1}^\Ni$ and $\{\mathbf{e}_i\}_{i=1}^\Ntot$ the standard basis of $\mathbb{R}^6$, $\mathbb{R}^\Ni$ and $\mathbb{R}^\Ntot$, respectively, where $\Ni := 12$ is just the number of governing equations of the IGEB model, and $\Ntot$ will be defined later on.

\medskip

\noindent \textbf{Weak formulation. \;}
Let $\Pi^v := \sum_{i=1}^6 \widetilde{\mathbf{e}}_i \overline{\mathbf{e}}_i^\intercal$ and $\Pi^z := \sum_{i=7}^{12} \widetilde{\mathbf{e}}_{i-6} \overline{\mathbf{e}}_i^\intercal$ be the projections on the first six and last six components of any vector in $\mathbb{R}^{12}$, respectively, and let $V$ be the functional space defined by $V := \left\{ \psi \in H^1(0, \ell; \mathbb{R}^{12}) \colon (\Pi^v\psi)(0) = 0  \right\}$. We choose the weak formulation as follows: find $y \in C^0(0, T;V)$ such that
\begin{align*}
&\frac{\mathrm{d}}{\mathrm{d}t} \left(\int_{0}^\ell  \left \langle \psi \,, Q^\mathcal{P} y \right \rangle \right) - \int_0^\ell \left \langle \frac{\mathrm{d}\psi}{\mathrm{d}x} \,, \mathbf{A}  y \right \rangle dx
+\left\langle \Pi^v\psi(\ell), K \Pi^v y(\ell, t) \right\rangle\\
&\hspace{1.25cm} -\left\langle \Pi^z\psi(\ell), \Pi^v y(\ell, t) \right\rangle + \int_0^\ell \Big( \left \langle \psi \,, \mathbf{B} y \right \rangle +  \left \langle \psi \,, \mathbf{G}(x,y)y \right \rangle \Big) dx = 0
\end{align*}
holds for all $\psi \in V$.

\begin{figure}\centering
\includegraphics[scale=0.5]{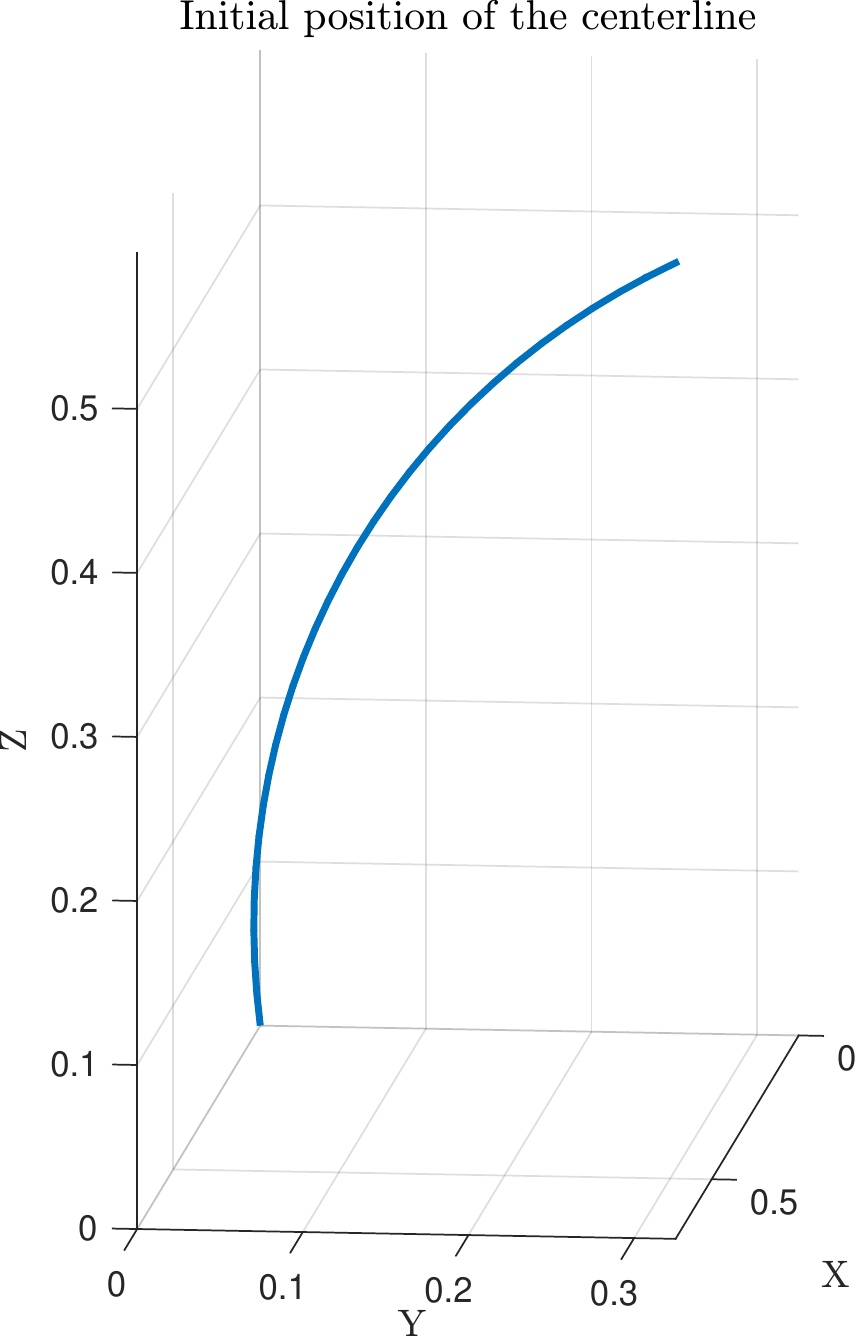}\quad
\includegraphics[scale=0.5]{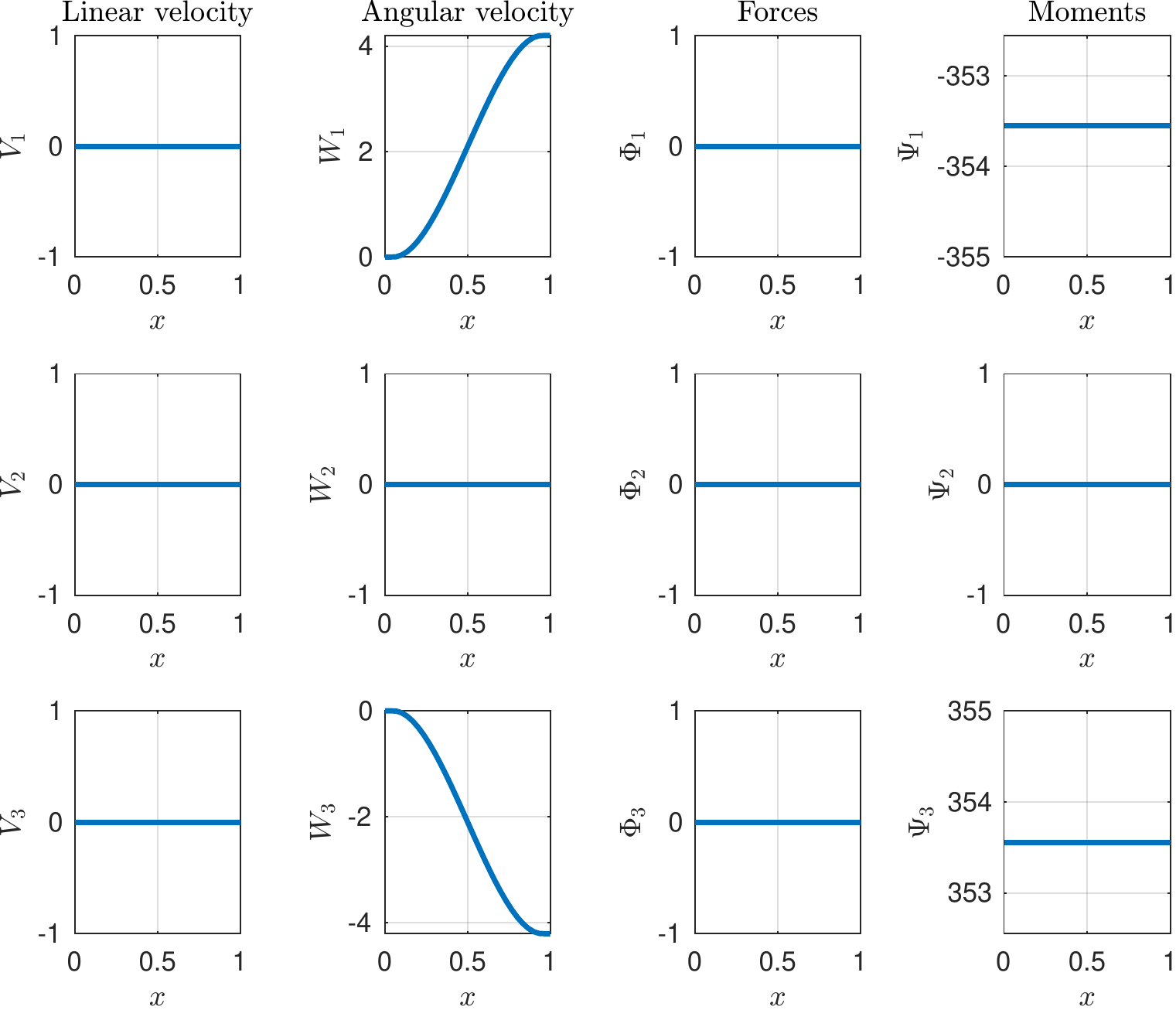}
\caption{The initial data in terms of positions and rotations (left) and in terms of velocities and internal forces and moments (right). Here, $V, W, \Phi, \Psi$ refer to the notation introduced in \eqref{eq:def_VWPhiPsi}.}
\label{fig:ini_data}
\end{figure}

\medskip

\noindent \textbf{Initial data. \;}
We choose an initial datum without shear (i.e., the cross sections are perpendicular to the centerline). We specify the initial position of the centerline $\mathbf{p}^0$, which is here parametrized by its arclength (see Figure \ref{fig:ini_data}), while the columns for $\mathbf{R}^0$ are given by the Frenet-Serret frame:
\begin{align*}
\mathbf{p}^0(x) := \frac{1}{\sqrt{2}} \begin{bmatrix}
x \\ 1-\cos(x)\\
\sin(x)
\end{bmatrix}, \quad \mathbf{R}^0(x)
:= \frac{1}{\sqrt{2}} \begin{bmatrix}
1 & 0 & - 1\\
\sin(x)&  \sqrt{2}\cos(x) & \sin(x)\\
\cos(x) & -\sqrt{2}\sin(x) & \cos(x)
\end{bmatrix}.
\end{align*}
Making use of \eqref{eq:rel_inidata}, one sees that the initial internal forces and moments $z^0 = ((z_1^0)^\intercal, (z_2^0)^\intercal)^\intercal$ take the form
\begin{align*}
z_1^0 = \mathbf{0}, \qquad z_2^0 = \frac{1}{\sqrt{2}} 
\begingroup 
\setlength\arraycolsep{2.5pt}
\renewcommand*{\arraystretch}{0.9}
\begin{bmatrix}
-1 \\ 0 \\ 1
\end{bmatrix}
\endgroup
 - \Upsilon_c.
\end{align*}
One may choose the initial velocities in such a way that the zero-order compatibility conditions of System \eqref{eq:IGEBfb} are satisfied (see Definition \ref{def:comp_cond}).
For instance, we just choose curves going from $0$ at $x=0$ to the components of the vector $- K^{-1} z^0(\ell)$ at $x=\ell$, and with zero slope at these two endpoints.
We set $\Upsilon_c := \mathbf{0}$, meaning that at rest the beam is not curved.

\medskip

\begin{figure}\centering
\includegraphics[scale=0.75]{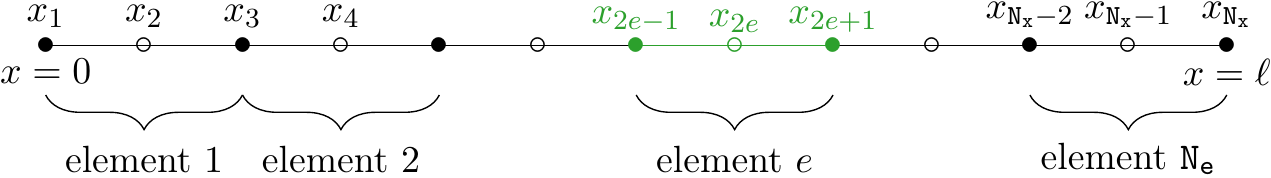}
\caption{Discretization of the spatial interval $[0, \ell]$.}
\label{fig:discretization0l}
\end{figure}

\begin{figure}\centering
\includegraphics[width=4.5cm]{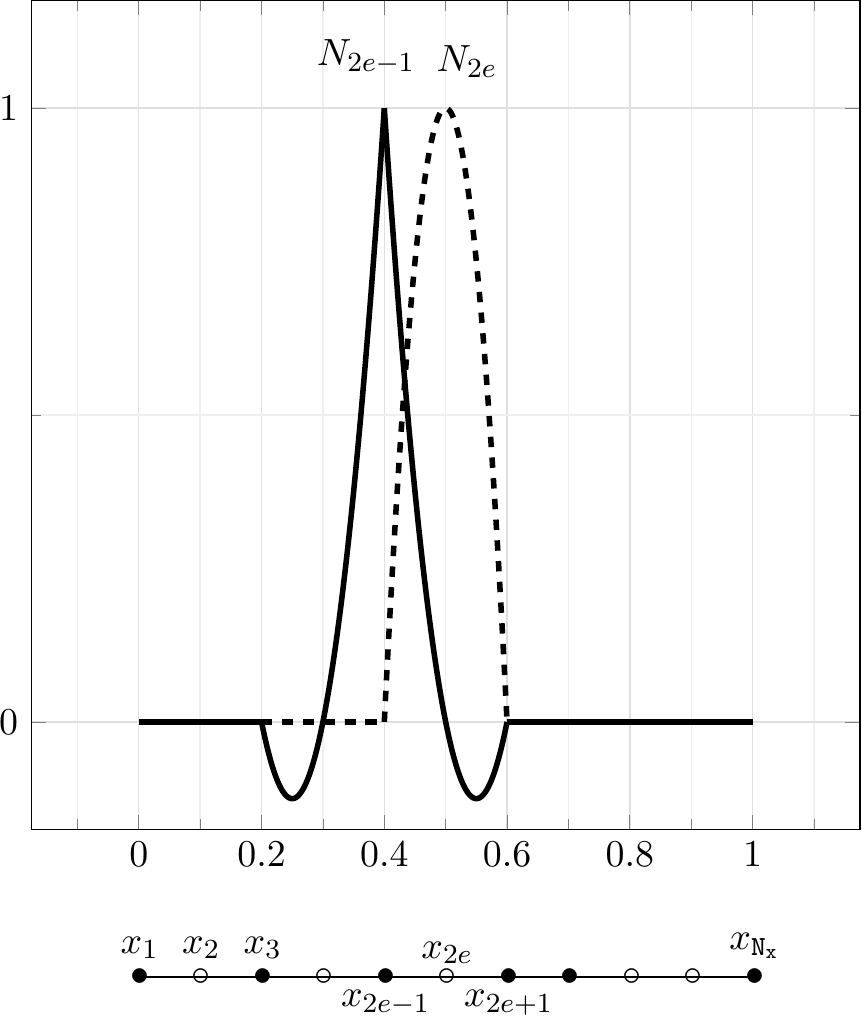}%
\quad \includegraphics[width=4.5cm]{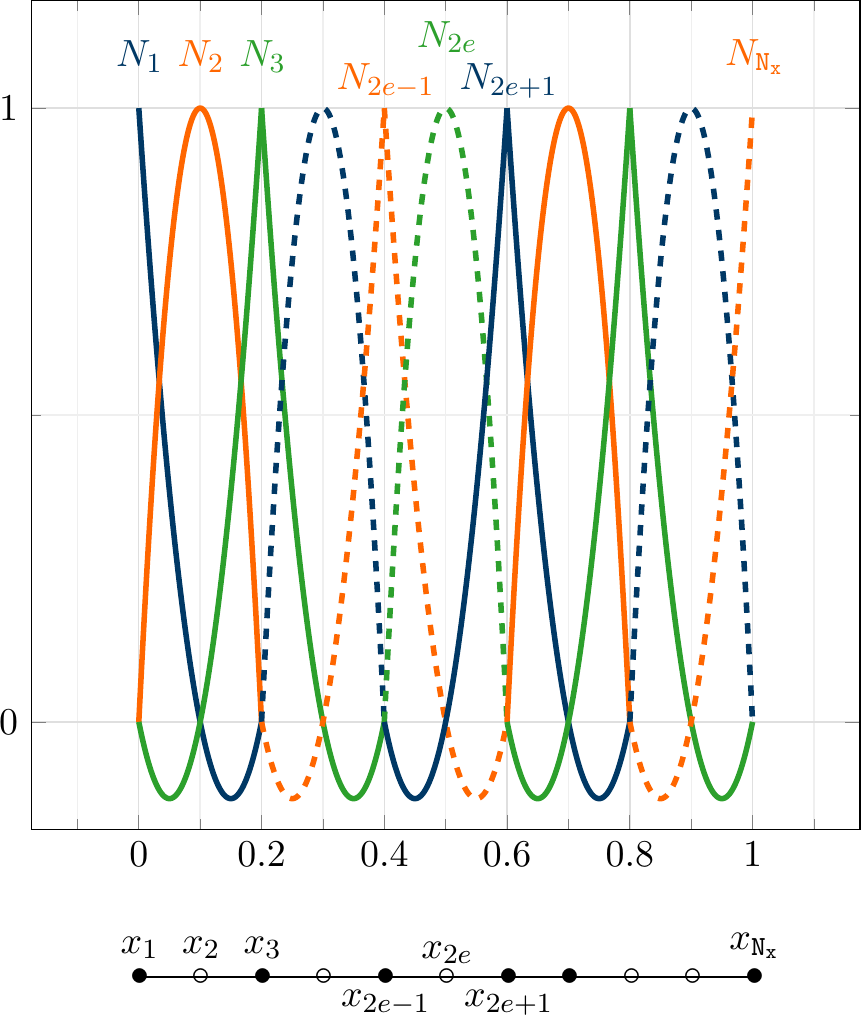}
\caption{Two ``kinds'' of shape functions (left), and all shape functions on $[0, 1]$ (right).}
\label{fig:shape_function}
\end{figure}

\noindent \textbf{Spatial interval and shape functions. \;}
As shown in Fig. \ref{fig:discretization0l}, we place $\Nx$ points $\{x_k\}_{k=1}^\Nx$ on the interval $[0, \ell]$, such that $x_1 = 0$ and $x_\Nx = \ell$. Each interval $\omega^e := [x_{2e-1}, x_{2e+1}]$ for $e \in \{1, 2, \ldots, \Ne\}$ constitutes an element, which contains the points $x_{2e-1}, x_{2e}$ and $x_{2e+1}$, and has length $\he = x_{2e+1} - x_{2e-1}$ (constant over the elements here). Note that $\Nx = 2\Ne+1$.

We semi-discretize in space $V$ by using $\mathbb{P}_2$ (quadratic) elements\footnote{$\mathbb{P}_2$ is the space of polynomials of degree two (i.e., quadratic polynomials).}, as follows:
\begin{align*}
\mathbf{V} := \big\{ \psi \in C^0([0, \ell]; \mathbb{R}^\Ni) \colon &\psi \big|_{\omega_e} \in (\mathbb{P}_2)^\Ni \ \text{for all }e\in \{1, \ldots, \Ne\}, \ \Pi^v \psi(0) = 0 \big\}.
\end{align*}
Let us introduce some maps that will be of help to link double or triple indexes with corresponding indexes for the components of the state obtained after the semi-discretization:
\begin{align}\label{eq:index_eta_mu}
\eta(i,k)&:= \Ni(k-1)+i, \quad \mu(i, e, p):= \eta(i, 2e-2 + p).
\end{align}
Without taking into account the Dirichlet boundary conditions, the total number of unknowns is $\Ntot := \Ni \Nx$.
Let $\mathbf{N}\colon [0, \ell] \rightarrow \mathbb{R}^{\Ni \times \Ntot}$ be defined by
\begin{align*}
\mathbf{N} :=  \sum_{i=1}^\Ni \sum_{k=1}^\Nx  N_k \overline{\mathbf{e}}_i \mathbf{e}_{\eta(i, k)}^\intercal,
\end{align*}
or in other words $\mathbf{N} = \begin{bmatrix}
\mathbf{I}_\Ni \otimes N_1 & \mathbf{I}_\Ni \otimes N_2 & \ldots & \mathbf{I}_\Ni \otimes N_\Nx
\end{bmatrix}$ with $\otimes$ denoting the Kronecker product. On any element $\omega^e$, only three shape functions are nonzero and are given by (see Fig. \ref{fig:shape_function} and Fig. \ref{fig:support_shape_functions})
\begin{align*}
\left[ N_{2e-1} , N_{2e}, N_{2e + 1} \right] = \widetilde{\mathbf{N}}\left(\frac{x-x_{2e-1}}{x_{2e+1} - x_{2e-1}}\right),
\end{align*}
where the reference shape function $\widetilde{\mathbf{N}} \colon [0, 1] \rightarrow \mathbb{R}^{1 \times 3}$ is defined by (see Fig. \eqref{fig:reference_shape_function})
\begin{align*}
\widetilde{\mathbf{N}}(\xi) := \Big[ \underbrace{(1-\xi)(1-2\xi)}_{=:\widetilde{N}_1(\xi)},  \underbrace{4\xi(1-\xi)}_{=:\widetilde{N}_2(\xi)},  \underbrace{\xi(2\xi-1)}_{=:\widetilde{N}_3(\xi)} \Big].
\end{align*}

\begin{figure}\centering
\includegraphics[width=3.5cm]{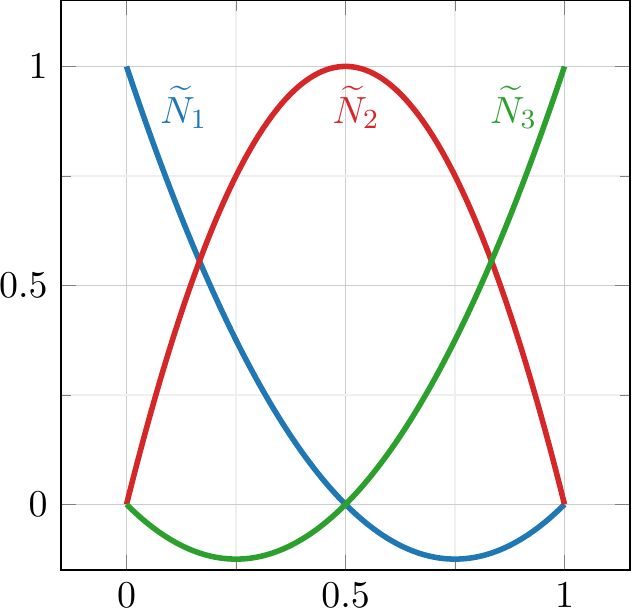}
\caption{Reference element.}
\label{fig:reference_shape_function}
\end{figure}

\medskip

\noindent \textbf{Semi-discretization of the equations. \;}
In a first instance, we do not take into account the homogeneous Dirichlet boundary condition at $x=0$.
We use the approximation
\begin{align*}
y(x, t) \approx \sum_{i=1}^\Ni \sum_{k =1}^\Nx N_k(x) \mathbf{y}_{\eta(i, k)}(t) \bar{\mathbf{e}}_i,
\end{align*}
or in other words $y(x, t) \approx \mathbf{N}(x) \mathbf{y}(t)$, and similarly we write $\psi(x) \approx \mathbf{N}(x) \bpsi(t)$, with $\mathbf{y}(t), \bpsi(t) \in \mathbb{R}^{\Ntot}$. 
Note that the $\eta(i, k)$-th component of $\mathbf{y}$ corresponds to $y_i(x_k, \cdot)$ (see \eqref{eq:index_eta_mu}).
We inject these approximations of $y$ and $\psi$ into the weak formulation to obtain the Ordinary Differential Equation (ODE)
\begin{align} \label{eq:nonlin_ODE}
\mathcal{M} \frac{\mathrm{d}}{\mathrm{d}t} \mathbf{y}(t) + \mathcal{K} \mathbf{y}(t) + \mathcal{Q}(\mathbf{y}(t)) \mathbf{y}(t)= 0,
\end{align}
for $\mathcal{M}, \mathcal{K}, \mathcal{Q}(\mathbf{y}) \in \mathbb{R}^{\Ntot \times \Ntot}$ defined by
\begin{align*}
&\qquad \quad \mathcal{M} := \int_0^\ell \mathbf{N}^\intercal Q^\mathcal{P} \mathbf{N} dx, \qquad \mathcal{Q}(\mathbf{y}) := \int_0^\ell \mathbf{N}^\intercal \mathbf{G}(x,\mathbf{N}\mathbf{y}) \mathbf{N} dx, \\
\mathcal{K} := &\underbrace{- \int_0^\ell \frac{\mathrm{d}\mathbf{N}}{\mathrm{d}x}^\intercal \mathbf{A} \mathbf{N} dx}_{=: \mathcal{K}_1} + \underbrace{\int_0^\ell \mathbf{N}^\intercal \mathbf{B} \mathbf{N}dx}_{=: \mathcal{K}_2}  + \underbrace{\left\langle \Pi^v\mathbf{N}(\ell), K \Pi^v \mathbf{N}(\ell) \right\rangle -\left\langle \Pi^z \mathbf{N}(\ell), \Pi^v \mathbf{N}(\ell) \right\rangle}_{=: \mathcal{K}_3}.
\end{align*}
Now, we want to deduce an algorithm to build theses matrices.
Let us start with the matrix $\mathcal{M}$. Replacing the integral over $(0, \ell)$ with the sum of integrals over all elements, an then using the definition of $\mathbf{N}$, we obtain the expression
\begin{align*}
\mathcal{M}
&= \sum_{e=1}^\Ne \sum_{i=1}^\Ni \sum_{j=1}^\Ni \sum_{k=1}^\Nx \sum_{m=1}^\Nx \left(\int_{\omega^e} Q^\mathcal{P}_{ij} N_k N_m dx\right) \mathbf{e}_{\eta(i, k)}\mathbf{e}_{\eta(j,m)}^\intercal.
\end{align*}
We approximate $Q^\mathcal{P}$ over the element $\omega^e$ by $Q^\mathcal{P}(x_{2e})$ so that this term may leave the integral. Then,
\begin{align*}
\mathcal{M}
&= \sum_{e=1}^\Ne \sum_{i=1}^\Ni \sum_{j=1}^\Ni Q^\mathcal{P}_{ij}(x_{2e}) \sum_{k=1}^\Nx \sum_{m=1}^\Nx  \left(\int_{\omega^e} N_k N_m dx\right) \mathbf{e}_{\eta(i, k)}\mathbf{e}_{\eta(j,m)}^\intercal.
\end{align*}
Now we may use the fact that on $\omega_e$ only the shape functions with indexes in $\{2e-1, 2e, 2e+1\}$ are nonzero (see Fig. \ref{fig:support_shape_functions}), and then apply
the change of variable to the reference element to deduce
\begin{align*}
\mathcal{M}
&= \sum_{e=1}^\Ne \sum_{i=1}^\Ni \sum_{j=1}^\Ni Q^\mathcal{P}_{ij}(x_{2e}) \he \sum_{k=1}^3 \sum_{m=1}^3  \left(\int_0^1 \widetilde{N}_k \widetilde{N}_m d\xi \right) \mathbf{e}_{\mu(i, e, k)}\mathbf{e}_{\mu(j,e, m)}^\intercal.
\end{align*}
Similar considerations hold for $\mathcal{K}_2$, and a similar procedure (only the change of variable to the reference element is different) applied to $\mathcal{K}_1$ yields that
\begin{align*}
\mathcal{K}_1
&= - \sum_{e=1}^\Ne \sum_{i=1}^\Ni \sum_{j=1}^\Ni \mathbf{A}_{ij} \sum_{k=1}^3 \sum_{m=1}^3 \left(\int_0^1 \frac{\mathrm{d}\widetilde{N}_k}{\mathrm{d}x} \widetilde{N}_m d\xi \right) \mathbf{e}_{\mu(i, e, k)}\mathbf{e}_{\mu(j,e, m)}^\intercal,
\end{align*}
These three matrices may thus be constructed as follows, where the ``element-mass and element-stiffness matrices'' $\mathcal{M}^e, \mathcal{K}^e  \in \mathbb{R}^{3\times 3}$ are given by
\begin{align*}
\mathcal{M}^e := \int_0^1 \widetilde{\mathbf{N}}^\intercal \widetilde{\mathbf{N}} d\xi = \frac{1}{30} \begin{bmatrix}
4 & 2 & -1\\
2 & 16 & 2\\
-1 & 2 & 4
\end{bmatrix}, \quad \mathcal{K}^e := \int_0^1 \frac{\mathrm{d}\widetilde{\mathbf{N}}}{\mathrm{d}\xi}^\intercal \widetilde{\mathbf{N}} d\xi = \frac{1}{6} \begin{bmatrix}
-3& -4 & 1\\
4 & 0 & -4\\
-1 & 4 & 3
\end{bmatrix}.
\end{align*}

\begin{figure}\centering
\includegraphics[width=4.5cm]{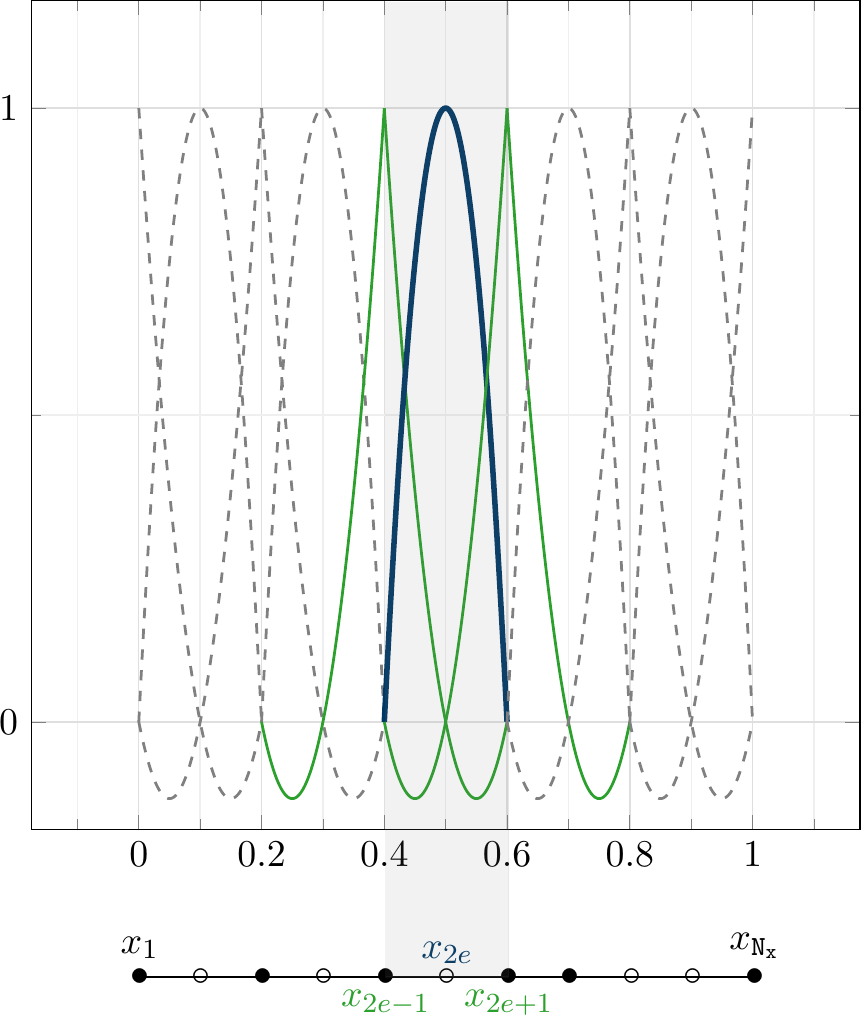}%
\quad \includegraphics[width=4.5cm]{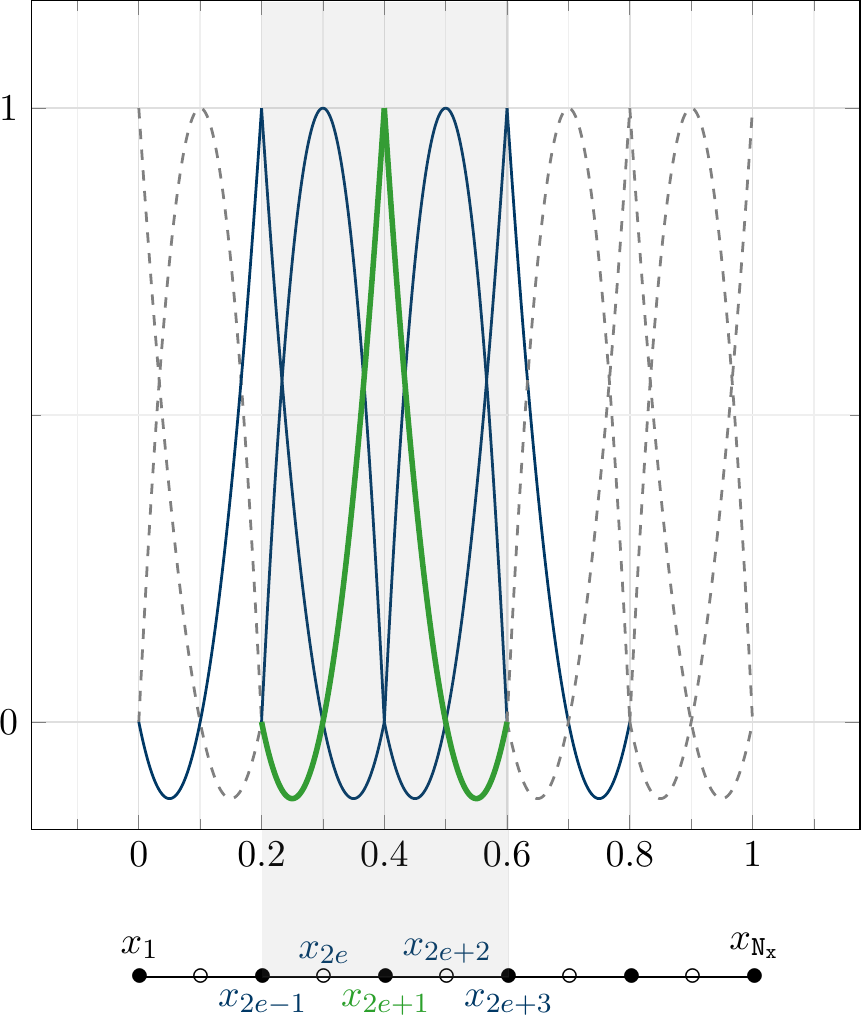}
\caption{The shape functions whose support intersects that of $N_{2e}$(left), and the shape functions whose support intersect that of $N_{2e+1}$ (right).}
\label{fig:support_shape_functions}
\end{figure}

\begin{algorithm}[H]
\DontPrintSemicolon

\ShowLn\For{$i = 1 \ldots \Ni$}{
\ShowLn\For{$j = 1 \ldots \Ni$}{
\ShowLn\For{$e = 1, \ldots, \Ne$}{
\ShowLn
$\mathtt{idxRow} = [\mu(i, e, 1), \mu(i, e, 2), \mu(i, e, 3)]$

\ShowLn
$\mathtt{idxCol} = [\mu(j, e, 1), \mu(j, e, 2), \mu(j, e, 3)]$

\ShowLn
$\mathcal{M}(\mathtt{idxRow}, \mathtt{idxCol}) = \mathcal{M}(\mathtt{idxRow}, \mathtt{idxCol})+ Q^\mathcal{P}_{ij}(x_{2e}) \he \mathcal{M}^e$

\ShowLn
$\mathcal{K}_2(\mathtt{idxRow}, \mathtt{idxCol}) = \mathcal{K}_2(\mathtt{idxRow}, \mathtt{idxCol})+ \mathbf{B}_{ij}(x_{2e}) \he \mathcal{M}^e$

\ShowLn
$\mathcal{K}_1(\mathtt{idxRow}, \mathtt{idxCol}) = \mathcal{K}_1(\mathtt{idxRow}, \mathtt{idxCol}) - \mathbf{A}_{ij} \mathcal{K}^e$
}}}
 
\caption{Building the matrices $\mathcal{M}, \mathcal{K}_1$ and $\mathcal{K}_2$.}
\end{algorithm}

For the matrix $\mathcal{K}_3$, by definition of $\Pi^v$ and $\Pi^z$ one has
\begin{align*}
\mathcal{K}_3
&= \sum_{i=1}^6 \sum_{j=1}^6 \mathbf{e}_{\eta(i, \Nx)} \widetilde{\mathbf{e}}_i^\intercal K \widetilde{\mathbf{e}}_j \mathbf{e}_{\eta(j, \Nx)}^\intercal - \sum_{i=7}^\Ni \sum_{j=1}^6 \mathbf{e}_{\eta(i, \Nx)} \widetilde{\mathbf{e}}_{i-6}^\intercal \widetilde{\mathbf{e}}_j \mathbf{e}_{\eta(j, \Nx)}^\intercal \\
&= \bigg[ \sum_{i=1}^6 \sum_{j=1}^6  K_{ij} \mathbf{e}_{\eta(i, \Nx)} \mathbf{e}_{\eta(j, \Nx)}^\intercal\bigg] - \bigg[\sum_{i=1}^6 \mathbf{e}_{\eta(i+6, \Nx)} \mathbf{e}_{\eta(i, \Nx)}^\intercal\bigg].
\end{align*}

\begin{algorithm}[H]
\DontPrintSemicolon

\ShowLn\For{$i=1 \ldots 6$}{
\ShowLn\For{$j = 1 \ldots 6$}{
\ShowLn
$\mathcal{K}_3(\eta(i, \Nx), \eta(j, \Nx)) = \mathcal{K}_3(\eta(i, \Nx), \eta(j, \Nx)) + K_{ij}$
}
\ShowLn
$\mathcal{K}_3(\eta(i+6, \Nx), \eta(i, \Nx)) = \mathcal{K}_3(\eta(i+6, \Nx), \eta(i, \Nx)) -1$
}
 
\caption{Building the matrix $\mathcal{K}_3$.}
\end{algorithm}

\noindent It remains to build the nonlinearity. First, notice that for any $i \in \{1, \ldots, \Ni\}$ one may write
\begin{align*}
\sum_{k=1}^\Nx N_k \mathbf{e}_{\eta(i, k)} = \left[\sum_{e=1}^{\Ne+1} N_{2e-1} \mathbf{e}_{\eta(i, 2e-1)}^\intercal \right] + \left[ \sum_{e=1}^\Ne N_{2e} \mathbf{e}_{\eta(i, 2e)}^\intercal \right].
\end{align*}
Consequently, 
\begin{align*}
\mathbf{N}\mathbf{y}
&= \sum_{p=1}^\Ni \overline{\mathbf{e}}_p \left(\left[\sum_{e=1}^{\Ne+1} N_{2e-1} \mathbf{y}_{\eta(i, 2e-1)} \right] + \left[ \sum_{e=1}^\Ne N_{2e} \mathbf{y}_{\eta(i, 2e)} \right]\right).
\end{align*}
Injecting this formula into the definition of $\mathcal{Q}(\mathbf{y})$, we obtain
\begin{align}
\label{eq:Qbuild-1}
\begin{aligned}
\mathcal{Q}(\mathbf{y}) = \sum_{p=1}^\Ni \bigg\{ &\bigg[
\sum_{e=1}^{\Ne+1}
\mathbf{y}_{\eta(p, 2e-1)}  \int_0^\ell N_{2e-1} \mathbf{N}^\intercal \mathbf{G}(x, \overline{\mathbf{e}}_p) \mathbf{N} dx  \bigg]\\
& + \bigg[\sum_{e=1}^\Ne \mathbf{y}_{\eta(p, 2e)}  \int_0^\ell N_{2e} \mathbf{N}^\intercal \mathbf{G}(x, \overline{\mathbf{e}}_p) \mathbf{N}dx  \bigg]\bigg\}.
\end{aligned}
\end{align}
Taking into account the support of the shape functions $N_{2e-1},N_{2e}$ (see Fig. \ref{fig:support_shape_functions}), we deduce that
\begin{align} \label{eq:Qbuild-2}
\mathcal{Q}(\mathbf{y}) = \sum_{p=1}^\Ni \sum_{e=1}^\Ne \left(   \mathbf{y}_{\eta(p, 2e-1)} P_1^{p, e} + \mathbf{y}_{\eta(p, 2e)} P_{2}^{p, e} + \mathbf{y}_{\eta(p, 2e+1)} P_3^{p, e} \right),
\end{align}
where the matrices $P_n^{p, e} \in \mathbb{R}^{\Ntot \times \Ntot}$ for $n \in \{1, 2, 3\}$, are defined by
\begin{align} \label{eq:Qbuild-3}
&P_n^{p, e} = \int_{\omega^{e}} N_{2e-2+n} \mathbf{N}^\intercal \mathbf{G}(x, \overline{\mathbf{e}}_p) \mathbf{N} dx.
\end{align}
In a similar manner to that used for the matrix $\mathcal{M}$, we obtain the equivalent expression 
\begin{align} \label{eq:Qbuild-4}
P_n^{p, e}
&=  \sum_{i = 1}^\Ni \sum_{j=1}^\Ni \mathbf{G}_{ij}(x_{2e}, \overline{\mathbf{e}}_p) \he \sum_{k=1}^3 \sum_{m=1}^3  \left( \int_0^1 \widetilde{N}_n \widetilde{N}_{k} \widetilde{N}_{m} d\xi \right) \mathbf{e}_{\mu(i, e, k)}\mathbf{e}_{\mu(j, e, m)}^\intercal.
\end{align}
Furthermore, one can see that $\mathcal{Q}(\mathbf{y}) \overline{\mathbf{y}} = \mathcal{Q}_\dagger (\overline{\mathbf{y}})\mathbf{y}$, where
\begin{align*}
\mathcal{Q}_\dagger(\mathbf{y}) &:= \int_0^\ell \mathbf{N}^\intercal \mathbf{G}_\dagger(x,\mathbf{N}\mathbf{y}) \mathbf{N} dx
\end{align*}
can be rewritten just as $\mathcal{Q}(\mathbf{y})$, replacing $\mathbf{G}_{ij}$ and $P_{n}^{p, e}$ with $\mathbf{G}_{\dagger ij}$ and $P_{\dagger n}^{p, e}$, respectively, in \eqref{eq:Qbuild-1}-\eqref{eq:Qbuild-2}-\eqref{eq:Qbuild-3}-\eqref{eq:Qbuild-4}. Then, to build the matrices $P_{n}^{p, e}$ and $P_{\dagger n}^{p, e}$, we use the following procedure where the element matrices $\mathcal{P}_n^e$ are given by
\begin{align*}
\mathcal{P}_n^e &= \int_0^1 \widetilde{N}_n \widetilde{\mathbf{N}}^\intercal \widetilde{\mathbf{N}} d\xi, \\
\mathcal{P}_1^e
= \frac{1}{420} \begin{bmatrix}
39 & 20 & -3 \\
20 & 16 & -8 \\
-3 & -8 & -3
\end{bmatrix},
\quad
\mathcal{P}_2^e
&= \frac{1}{105} \begin{bmatrix}
5 & 4 & -2 \\
4 & 48 & 4 \\
-2 & 4 & 5
\end{bmatrix},
\quad
\mathcal{P}_3^e
= \frac{1}{420} \begin{bmatrix}
-3 & -8 & -3 \\
-8 & 16 & 20 \\
-3 & 20 & 39
\end{bmatrix}.
\end{align*}

\begin{algorithm}[H]
\DontPrintSemicolon

\ShowLn\For{$p=1\ldots \Ni$}{
\ShowLn\For{$e=1\ldots \Ne$}{
\ShowLn\For{$i=1\ldots \Ni$}{
\ShowLn\For{$j=1\ldots \Ni$}{

\ShowLn $\mathtt{idxRow} = [\mu(i,e,1), \mu(i,e,2), \mu(i,e,3)]$

\ShowLn $\mathtt{idxCol} = [\mu(j,e,1), \mu(j,e,2), \mu(j,e,3)]$ 

\ShowLn\For{$n = 1, 2, 3$}{
\ShowLn $P_n^{p, e}(\mathtt{idxRow}, \mathtt{idxCol}) = P_n^{p, e}(\mathtt{idxRow}, \mathtt{idxCol}) + \mathbf{G}_{ij}(x_{2e}, \overline{\mathbf{e}}_p) \he \mathcal{P}_n^e$
 
\ShowLn $P_{\dagger n}^{p, e}(\mathtt{idxRow}, \mathtt{idxCol}) = P_{\dagger n}^{p, e}(\mathtt{idxRow}, \mathtt{idxCol}) + \mathbf{G}_{\dagger ij}(x_{2e}, \overline{\mathbf{e}}_p) \he \mathcal{P}_{n}^e$

}}}}}
 
\caption{Building the matrices ${P}_n^{p, e}$ and ${P}_{\dagger n}^{p, e}$.}
\end{algorithm}


\medskip

\noindent \textbf{Dirichlet boundary conditions. \;}
To take into account the homogeneous Dirichlet boundary conditions, we introduce
\begin{align*}
\overline{\eta}(i,k)&:= \Ni(k-1)+i - 6, \quad \Nf := \Ni \Nx - 6.
\end{align*} 
We construct the smaller matrices $\overline{\mathcal{M}}, \overline{\mathcal{K}}, \overline{P}_n^{p, e},\overline{P}_{\dagger n}^{p, e} \in \mathbb{R}^{\Nf \times \Nf}$ as follows, by removing some rows and columns of the matrices constructed before.

\begin{algorithm}[H]
\DontPrintSemicolon

\ShowLn 
$\mathtt{dof} = 7:\Ntot$

\ShowLn $\overline{\mathcal{M}} = \mathcal{M}(\mathtt{dof}, \mathtt{dof})$

\ShowLn $\overline{\mathcal{K}} = \mathcal{K}(\mathtt{dof}, \mathtt{dof})$

\ShowLn\For{$p=1\ldots \Ni$}{
\ShowLn\For{$e=1\ldots \Ne$}{
\ShowLn\For{$n = 1, 2, 3$}{
\ShowLn $\overline{P}_n^{p, e} = P_n^{p, e}(\mathtt{dof}, \mathtt{dof})$ 
 
\ShowLn $\overline{P}_{\dagger n}^{p, e} = P_{\dagger n}^{p, e}(\mathtt{dof}, \mathtt{dof})$
}}}
 
\caption{Building the matrices $\overline{\mathcal{M}}, \overline{\mathcal{K}}$ and $\overline{P}_n^{p, e}, \overline{P}_{\dagger n}^{p, e}$.}
\end{algorithm}

\noindent Furthermore, we define $\overline{\mathcal{Q}} \colon \mathbb{R}^\Nf \rightarrow \mathbb{R}^{\Nf \times \Nf}$ by
\begin{align*}
\overline{\mathcal{Q}}(\mathbf{y}) &= \sum_{p=1}^6\left[ \mathbf{y}_{\overline{\eta}(p, 2)} \overline{P}_{2}^{p, 1} + \mathbf{y}_{\overline{\eta}(p, 3)} \overline{P}_3^{p, 1}+ \sum_{e=2}^\Ne \left(   \mathbf{y}_{\overline{\eta}(p, 2e-1)} \overline{P}_1^{p, e} + \mathbf{y}_{\overline{\eta}(p, 2e)} \overline{P}_{2}^{p, e} + \mathbf{y}_{\overline{\eta}(p, 2e+1)} \overline{P}_3^{p, e} \right)\right]\\
&+ \sum_{p=7}^\Ni \sum_{e=1}^\Ne \left(\mathbf{y}_{\overline{\eta}(p, 2e-1)} \overline{P}_1^{p, e} + \mathbf{y}_{\overline{\eta}(p, 2e)} \overline{P}_{2}^{p, e} + \mathbf{y}_{\overline{\eta}(p, 2e+1)} \overline{P}_3^{p, e} \right),
\end{align*}
and $\overline{\mathcal{Q}}_\dagger \colon \mathbb{R}^\Nf \rightarrow \mathbb{R}^{\Nf \times \Nf}$ analogously (replacing $\overline{P}_n^{p, e}$ with $\overline{P}_{\dagger n}^{p, e}$).
The ODE to be solved becomes $\overline{\mathcal{M}} \frac{\mathrm{d}}{\mathrm{d}t} \mathbf{y} + \overline{\mathcal{K}} \mathbf{y} + \overline{\mathcal{Q}}(\mathbf{y}) \mathbf{y}= 0$ with a smaller unknown $\mathbf{y}$ having values in $\mathbb{R}^\Nf$. However, in what follows, we drop the ``bar'' notation for clarity. 

\medskip

\begin{figure}
\hspace{-0.8cm}\includegraphics[scale=0.485]{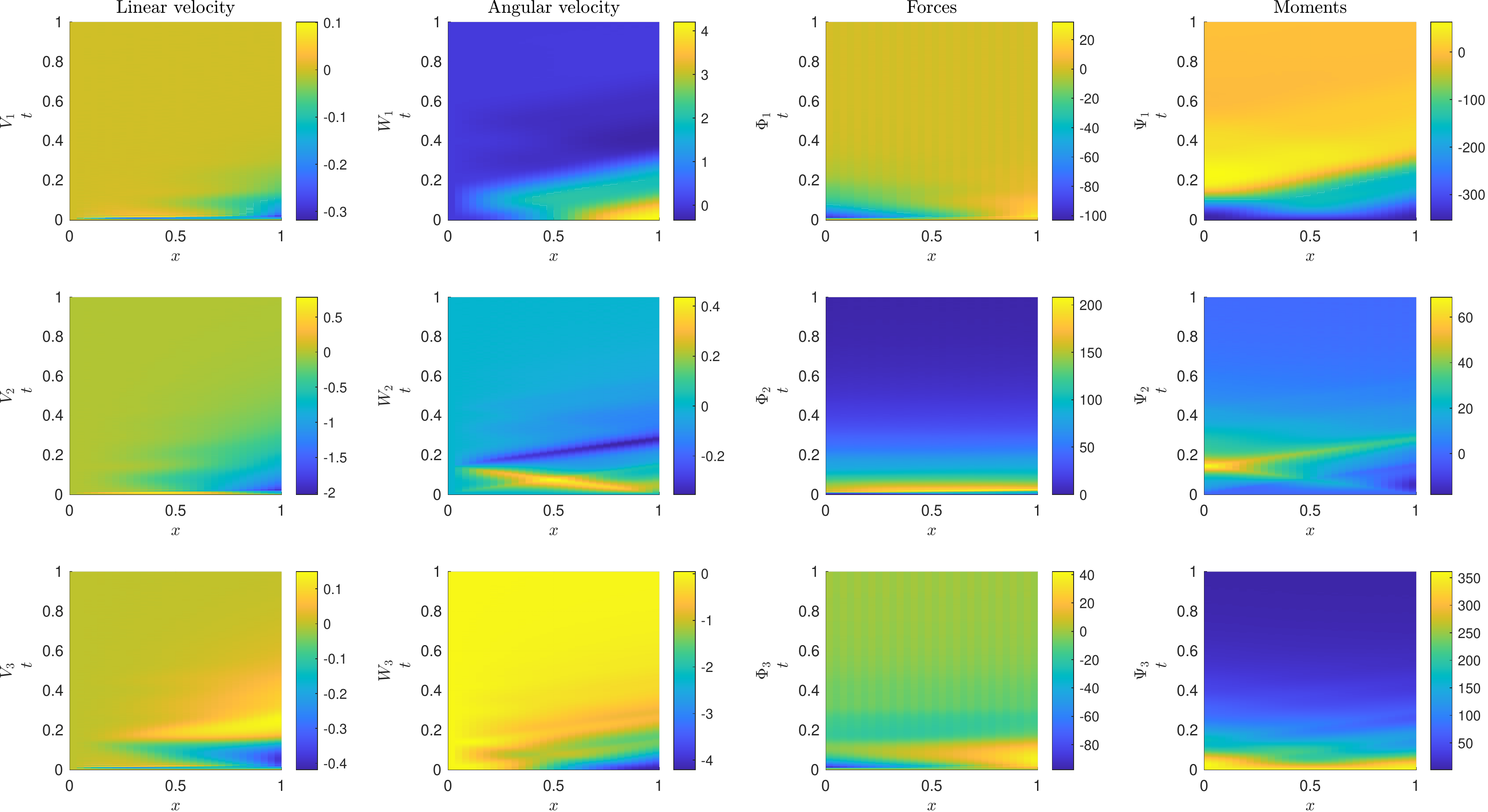}
\caption{Solution $y$ with $y^0$ fulfilling the zero-order compatibility conditions.}
\label{fig:sol_y}
\end{figure}

\noindent \textbf{Time discretization of the equations. \;}
Let the time interval be divided into $\Nt$ points $\{t_k\}_{k=1}^{\Nt}$ with $t_1 = 0$ and $t_{\Nt} = T$, and let $\htt = \frac{T}{\Nt-1}$ be the time step.
We have obtained the nonlinear ODE \eqref{eq:nonlin_ODE} and now want to discretize it in time by means of an implicit midpoint rule. The scheme reads 
\begin{align} \label{eq:implicit_midpoint_rule}
\mathcal{M}\mathbf{y}^{k+1} = \mathcal{M}\mathbf{y}^k - \htt \mathcal{K}\frac{\mathbf{y}^k+\mathbf{y}^{k+1}}{2} - \htt \mathcal{Q}\left(\frac{\mathbf{y}^k+\mathbf{y}^{k+1}}{2}\right)\frac{\mathbf{y}^k+\mathbf{y}^{k+1}}{2}.
\end{align}
Defining the function $F_k \colon \mathbb{R}^\Nf \rightarrow \mathbb{R}^\Nf$ by
\begin{align*}
F_k(\zeta) = \left(\mathcal{M}+\frac{\htt}{2} \mathcal{K} \right)\zeta - \left(\mathcal{M}-\frac{\htt}{2}\mathcal{K} \right)\mathbf{y}^k + \frac{\htt}{4} \left( \mathcal{Q}( \mathbf{y}^k)\mathbf{y}^k + \mathcal{Q}( \mathbf{y}^k )\zeta + \mathcal{Q}(\zeta)\mathbf{y}^k + \mathcal{Q}(\zeta)\zeta\right),
\end{align*}
the equation \eqref{eq:implicit_midpoint_rule} also writes as $F_k(\mathbf{y}^{k+1}) = 0$. Thus, at each time step $t_k$, given $\mathbf{y}^k$, we want to compute $\mathbf{y}^{k+1}$ in order to define $F_{k+1}$ and move to the next iteration. Let $k$ be fixed. We use the Newton--Raphson method which provides an approximation of a zero of $F_k$, i.e., $\zeta$ such that $F_k(\zeta) = 0$, by means of the scheme:
\begin{align*}
\zeta_{m+1} = \zeta_m - (\mathrm{Jac} F_k(\zeta_m))^{-1} F_k(\zeta_m).
\end{align*}
After some computations, and using the property $\mathcal{Q}(\mathbf{y}) \overline{\mathbf{y}} = \mathcal{Q}_\dagger (\overline{\mathbf{y}})\mathbf{y}$, one obtains that the Jacobian matrix of $F_k$ is equal to
\begin{align*}
\mathrm{Jac} F_k(\zeta_\circ)
&= \mathcal{M}+\frac{\htt}{2} \mathcal{K} + \frac{\htt}{4} (\mathcal{Q}(\mathbf{y}^k)+\mathcal{Q}_\dagger(\mathbf{y}^k)) + \frac{\htt}{4} (\mathcal{Q}(\zeta_\circ) +  \mathcal{Q}_\dagger (\zeta_\circ)).
\end{align*}
Then, the following algorithm yields an approximation of the solution to \eqref{eq:nonlin_ODE}. For the initial data described above, the result of this scheme is illustrated in Fig. \ref{fig:sol_y}.

\begin{algorithm}[H]
\DontPrintSemicolon

\ShowLn Choose maximal number of iterations $\mathtt{N_{iter}}$

\ShowLn\For{$k = 1 \ldots \Nt-1$}{

\ShowLn $\zeta_1 = \mathbf{y}^k$

\ShowLn\For{$m = 1 \ldots \mathtt{N_{iter}}$}{

\ShowLn $\zeta_{m+1} = \zeta_m - (\mathrm{Jac}F_k(\zeta_m))^{-1} F_k(\zeta_m)$
 
\ShowLn If converges than stop the $m$ loop.

}

\ShowLn $\mathbf{y}^{k+1} = \zeta_{m+1}$
}
 
\caption{Solve the obtained ODE for $\mathbf{y}$.}
\end{algorithm}

\medskip

\noindent \textbf{Recovering the beam's position. \;}
We have the initial data $(\mathbf{p}^0, \mathbf{R}^0) \colon [0, \ell]\rightarrow \mathbb{R}^3 \times \mathrm{SO}(3)$ for the GEB model. We have obtained an approximation of the solution $y(x, t) \approx \mathbf{N}(x)\mathbf{y}(t)$. Recall that we use the notation $y = (v^\intercal, z^\intercal)^\intercal$ as well as $v = (v_1^\intercal, v_2^\intercal)^\intercal$ and $z = (z_1^\intercal, z_2^\intercal)^\intercal$
for $v_1, v_2, z_1, z_2$ having values in $\mathbb{R}^3$.
Now, we just follow the proof of Theorem \ref{th:invert_transfo_fb} on the inversion of the transformation $\mathcal{T}$. As we mentioned before, recovering the position of the beam from velocities (or strains) is done in practice; see for instance \cite{Artola2019mpc}.

We integrate the differential equation \eqref{eq:ODE_for_q_v}-\eqref{eq:ODE_for_q_ini} whose unknown $\mathbf{q} \colon [0, \ell]\times[0, T]\rightarrow \mathbb{R}^4$ is the quaternion that parametrizes the function $\mathbf{R}\colon [0, \ell]\times [0, T] \rightarrow \mathrm{SO}(3)$ of interest to us.
Built-in functions of \texttt{Matlab} permit to go easily from quaternion to rotation matrices and the other way around.
Now, let us fix $x_\alpha$ for some $\alpha\in \{1, \ldots, \Nx\}$, and denote $\mathbf{q}_\alpha^k := \mathbf{q}(x_\alpha, t_k)$, with a similar notation for $\mathbf{R}$. We use a midpoint rule to for the time discretization
\begin{align*}
\mathbf{q}_\alpha^{k+1} = \mathbf{q}_\alpha^k + \htt \, \mathcal{U}\left( v_2\left(x_\alpha, \frac{t_k + t_{k+1}}{2}\right) \right) \frac{\mathbf{q}_\alpha^{k} + \mathbf{q}_\alpha^{k+1} }{2}.
\end{align*}
Recall that the linear map $\mathcal{U}$ is defined by
\eqref{eq:def_calU}. We approximate $v_2$ in the middle of the interval $[t_k, t_{k+1}]$ by the mean of its values at the endpoints. Then, the scheme takes the form
\begin{align*}
\mathbf{q}_\alpha^{k+1} = \left( \mathbf{I}_4 - \frac{\htt}{2} U^\mathrm{mid}_{\alpha, k} \right)^{-1}  \left( \mathbf{I}_4 + \frac{\htt}{2} U^\mathrm{mid}_{\alpha, k} \right) \mathbf{q}_\alpha^k \quad \text{with} \ \ U^\mathrm{mid}_{\alpha, k} := \mathcal{U}\left( \frac{v_2(x_\alpha, t_k) + v_2(x_\alpha, t_{k+1})}{2} \right).
\end{align*}
Each $\mathbf{q}_\alpha^k$ can be converted to the corresponding rotation matrix $\mathbf{R}_\alpha^k$, and we inject the latter in the differential equation \eqref{eq:ODE_for_p_v}-\eqref{eq:ODE_for_p_ini} whose solution is the position of the centerline $\mathbf{p} \colon [0, \ell]\times[0, T] \rightarrow \mathbb{R}^3$. One may then use the function \texttt{trapz} of \texttt{Matlab} to obtain an approximation of $\mathbf{p}$. Also for the example considered here, snapshots of the resulting centerline's position are displayed in Fig. \ref{fig:position}.

\begin{algorithm}[H]
\DontPrintSemicolon

\ShowLn\For{$\alpha  = 1 \ldots \Nx$}{

\ShowLn $\mathbf{q}_{{\alpha }}^1$ is the quaternion that parametrizes $\mathbf{R}_{{\alpha }}^1 := \mathbf{R}^0(x_{\alpha })$

\ShowLn\For{$k = 1 \ldots \Nt-1$}{

\ShowLn $U_{\alpha ,k}^\mathrm{mid} = \frac{1}{2} \mathcal{U}\left(v_2(x_\alpha, t_k) + v_2(x_\alpha, t_{k+1}) \right)$

\ShowLn $\mathbf{q}_{{\alpha }}^{k+1} = \left( \mathbf{I}_4 - \frac{\htt}{2} U^\mathrm{mid}_{{\alpha }, k} \right)^{-1}  \left( \mathbf{I}_4 + \frac{\htt}{2} U^\mathrm{mid}_{{\alpha }, k} \right) \mathbf{q}_{{\alpha }}^k$

\ShowLn Calculate $\mathbf{R}_{{\alpha }}^{k+1}$ the rotation matrix parametrized by $\mathbf{q}_{{\alpha }}^{k+1}$

\ShowLn Compute $\mathbf{p}_{\alpha }^{k+1}$ using the function \texttt{trapz}

}}
 
\caption{Recovering the position of the beam from the velocities $v$.}
\end{algorithm}

\begin{figure} \centering
\includegraphics[scale=0.5]{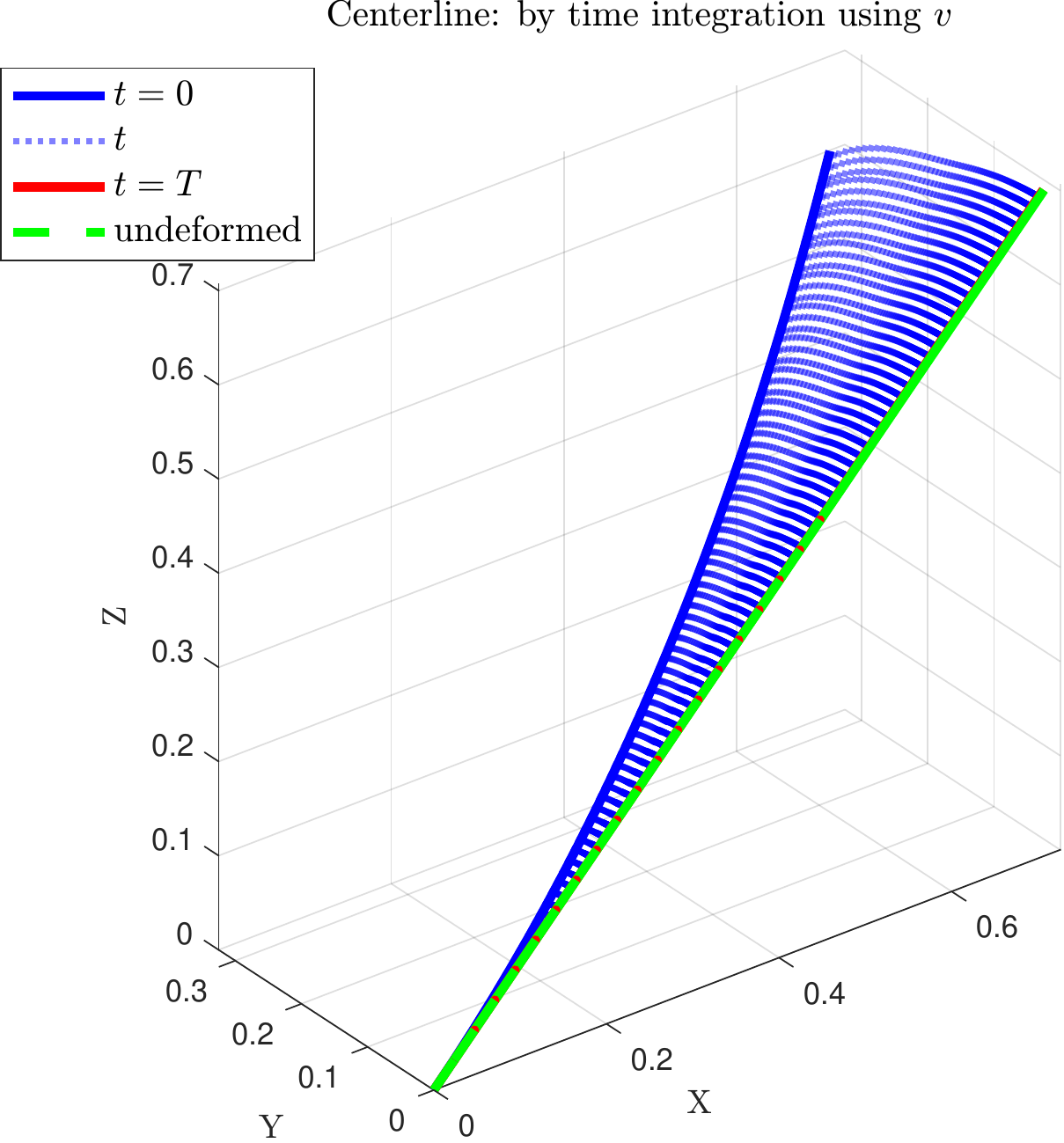}\qquad
\includegraphics[scale=0.5]{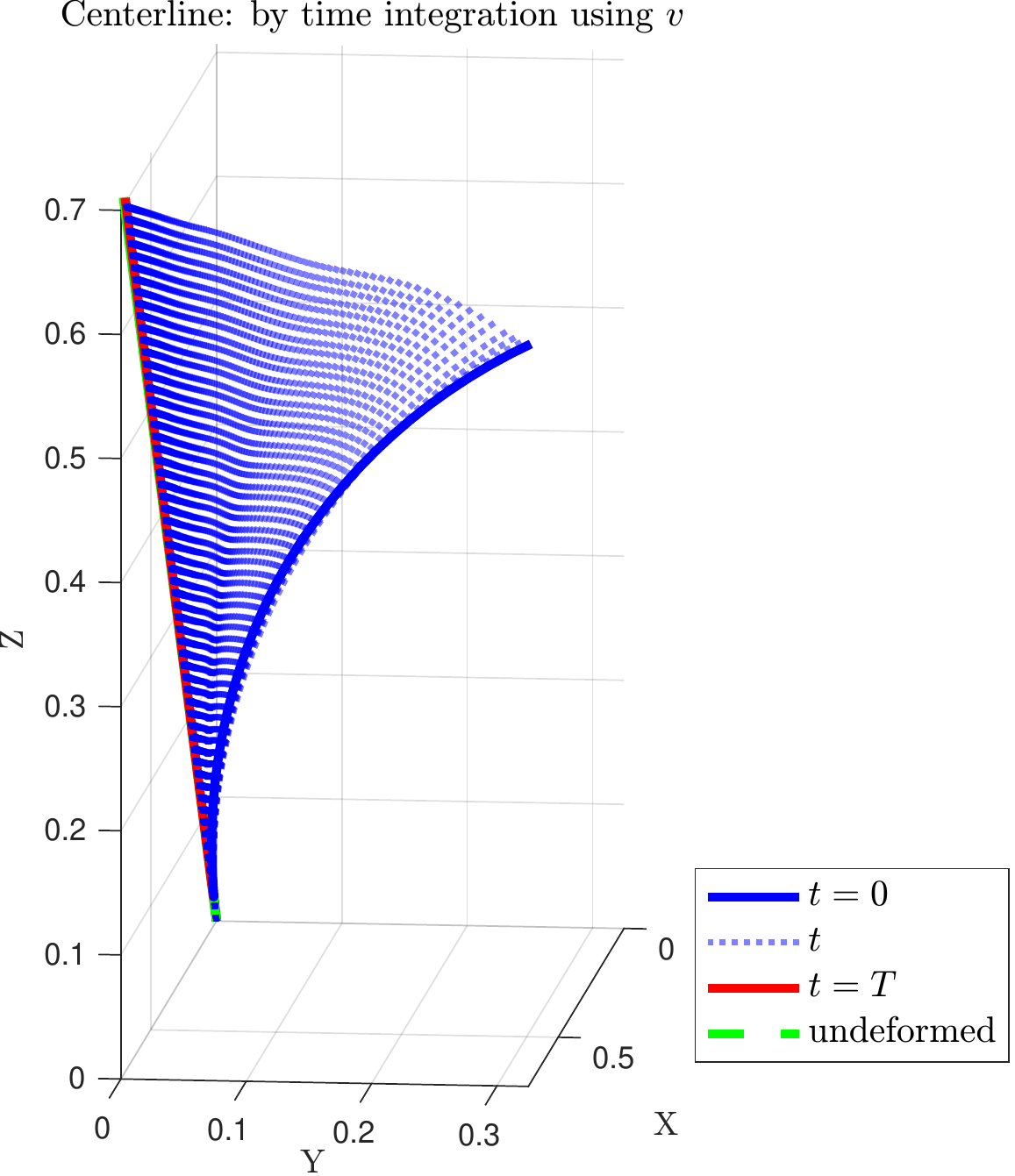}
\caption{Centerline's position through time, with $y^0$ fulfilling the zero-order compatibility conditions.}
\label{fig:position}
\end{figure}

\noindent When Dirichlet boundary data is available -- which is the case here at $x=0$ -- , one may rather use the strain components of $y$ to recover the position of the beam. It suffices to fix $t$ (rather than $x)$ and solve the differential equations  \eqref{eq:ODE_for_q_z}-\eqref{eq:ODE_for_q_ini}
and then \eqref{eq:ODE_for_p_z}-\eqref{eq:ODE_for_p_ini}, with a similar method.

\medskip

\noindent \textbf{Controlled versus free beam.} Finally, let us compare the case of a controlled beam to that of a free beam ($K = \mathbf{0}_6$). To do so with our choice of initial datum $(\mathbf{p}^0, \mathbf{R}^0)$, we do not choose the initial velocities $v^0$ in such a way that $y^0$ fulfills the zero-order compatibility conditions of the IGEB model, since we may not invert $K$ in the free beam case and since we want the same initial data for both simulations. The initial data is thus just as in Fig. \ref{fig:ini_data} except that all velocities are set to zero. We may then observe the difference between the evolution of the free beam in Fig. 
\ref{fig:no-zero-ord-free} which is not stabilized to the straight rest state, and that of the controlled beam in Fig. \ref{fig:no-zero-ord-controlled} which reaches the rest state.

\begin{figure}\centering
\begin{subfigure}{\textwidth} 
\hspace{-0.8cm}\includegraphics[scale=0.485]{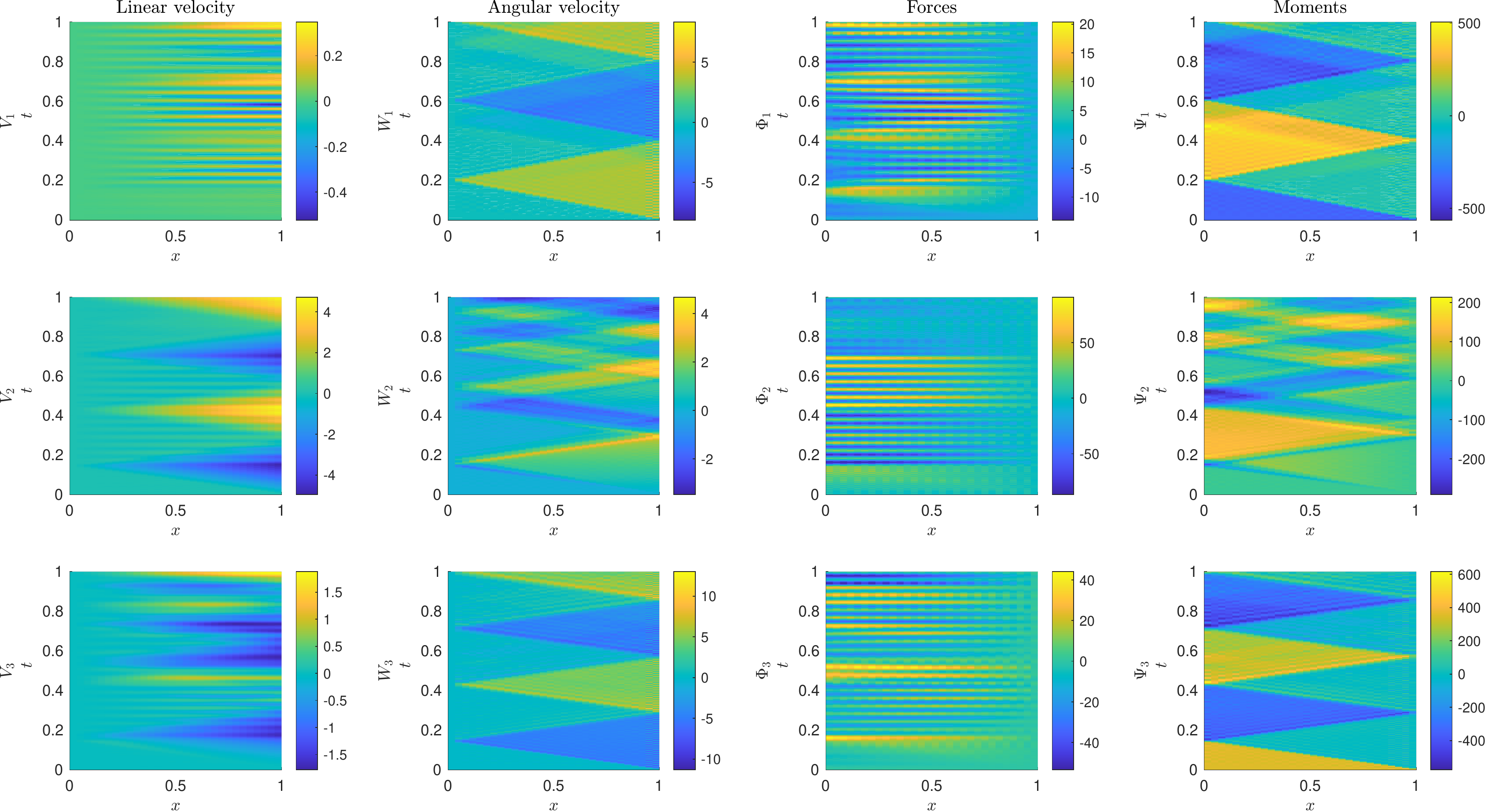}
\caption{Solution $y$.}
\end{subfigure}%

\vspace{0.5cm}

\begin{subfigure}{\textwidth}
\centering
\includegraphics[scale=0.5]{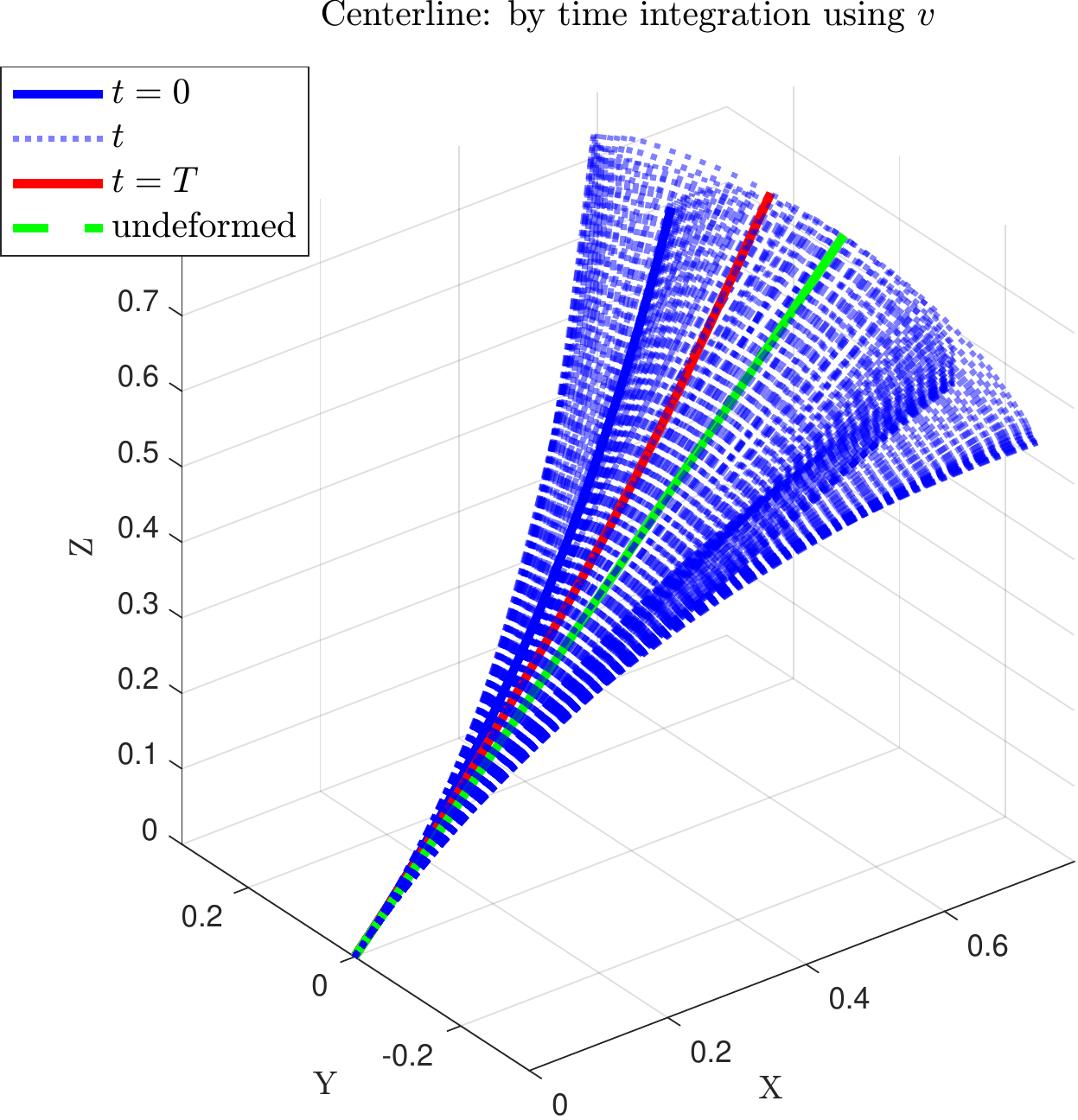}\qquad
\includegraphics[scale=0.5]{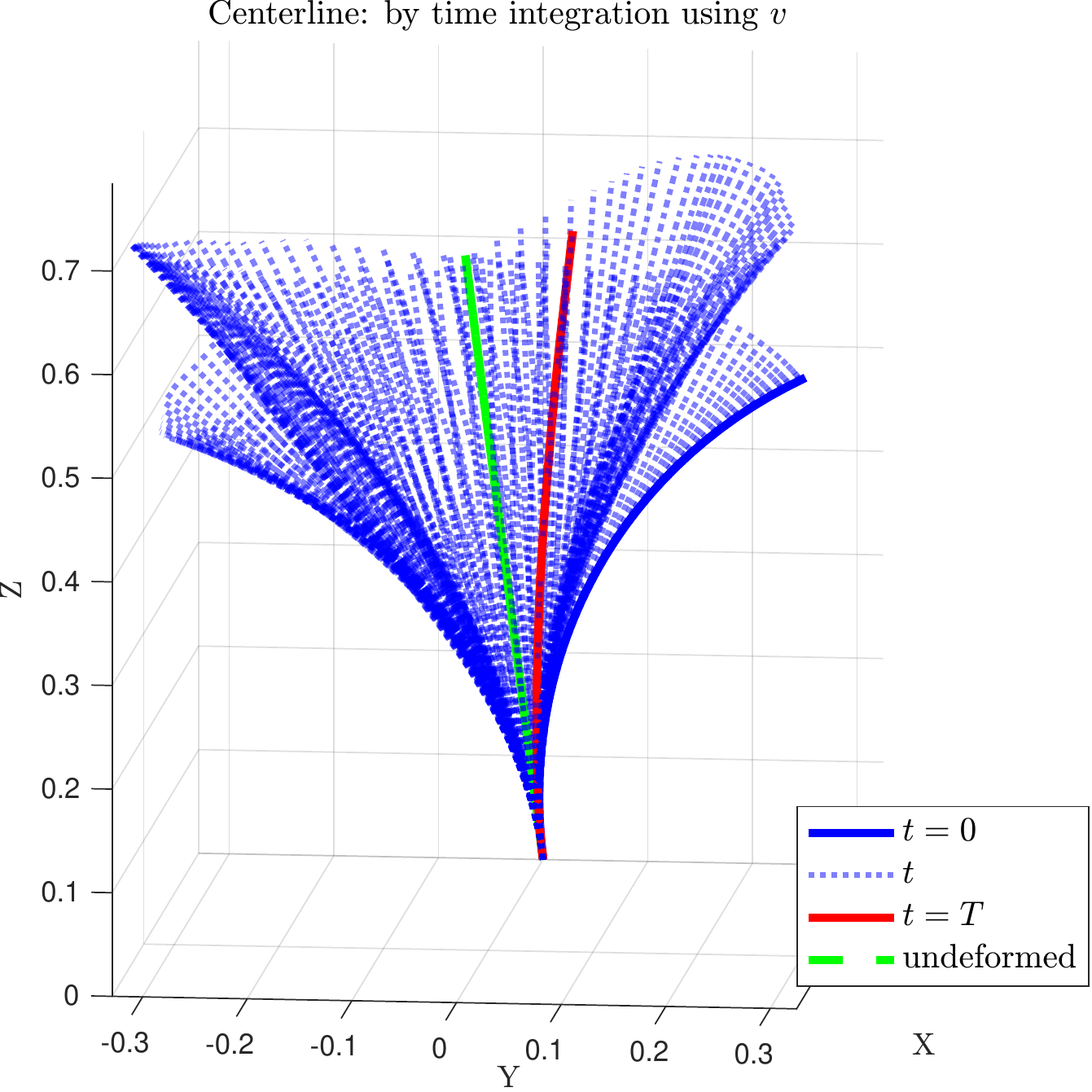}
\caption{Position of the centerline through time.}
\end{subfigure}%

\caption{Free beam with zero initial velocities.}
\label{fig:no-zero-ord-free}
\end{figure}



\begin{figure}\centering
\begin{subfigure}{\textwidth} 
\hspace{-0.8cm}\includegraphics[scale=0.485]{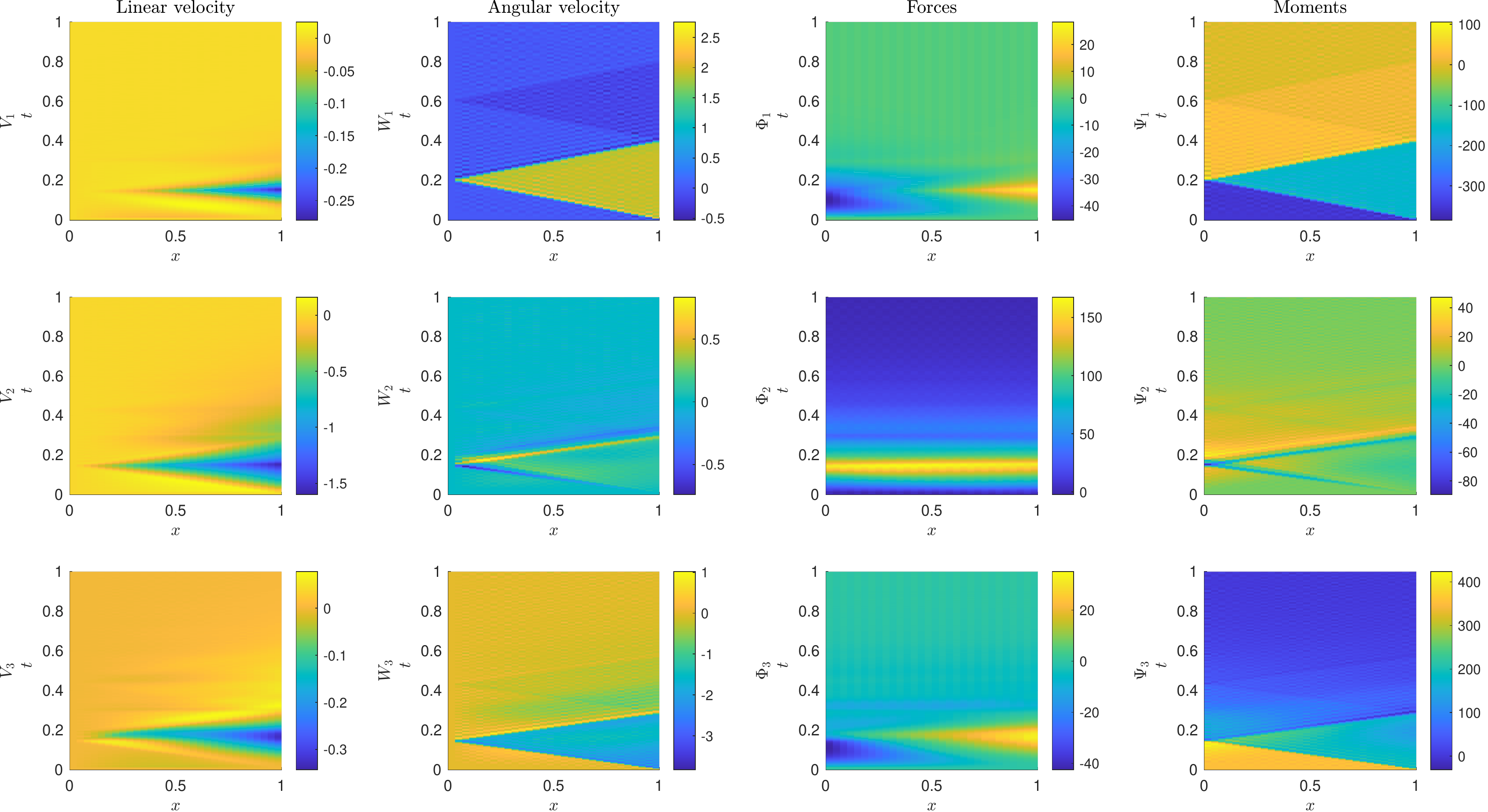}
\caption{Solution $y$.}
\end{subfigure}%

\vspace{0.5cm}

\begin{subfigure}{\textwidth}
\centering
\includegraphics[scale=0.5]{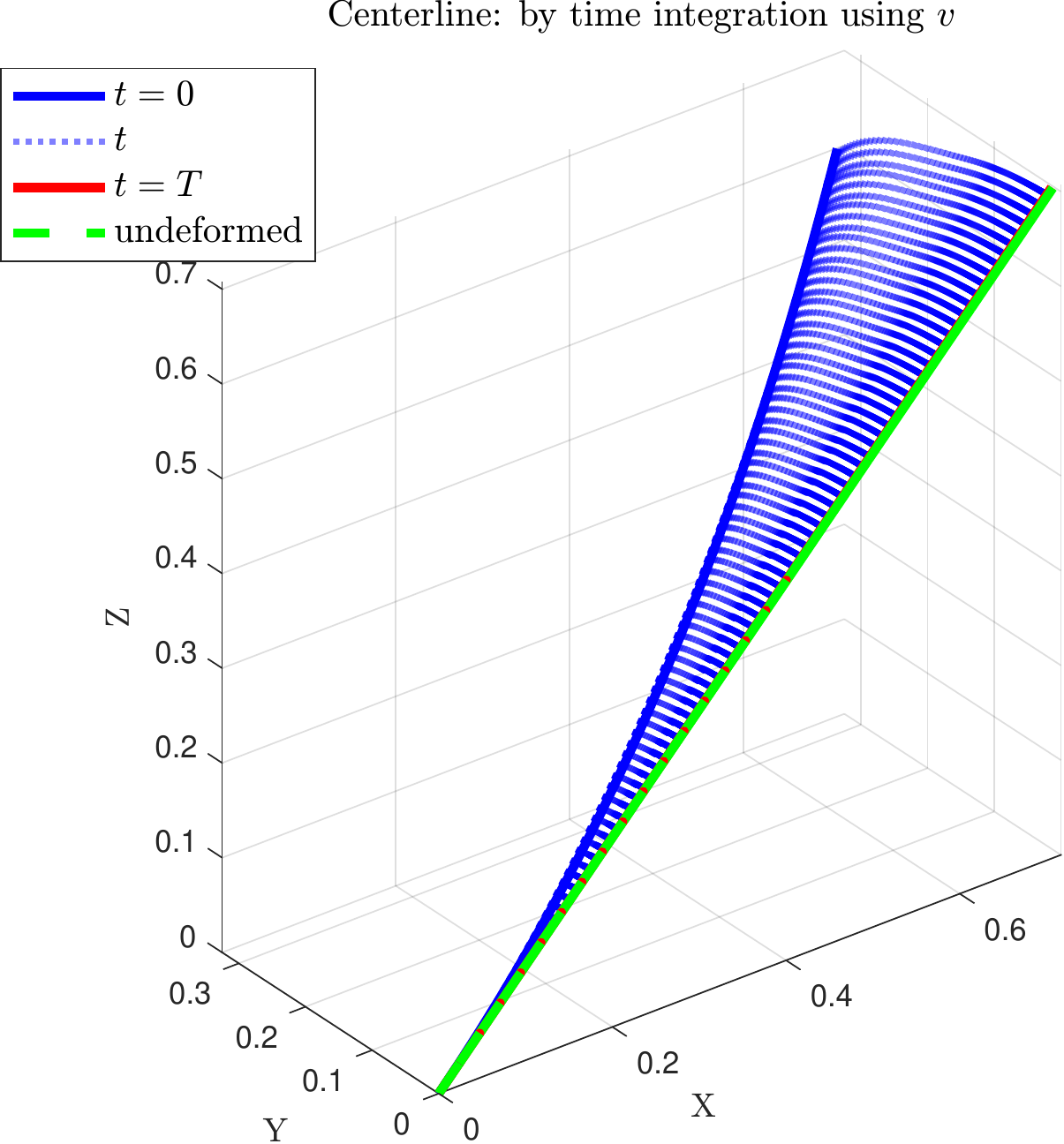}\qquad
\includegraphics[scale=0.5]{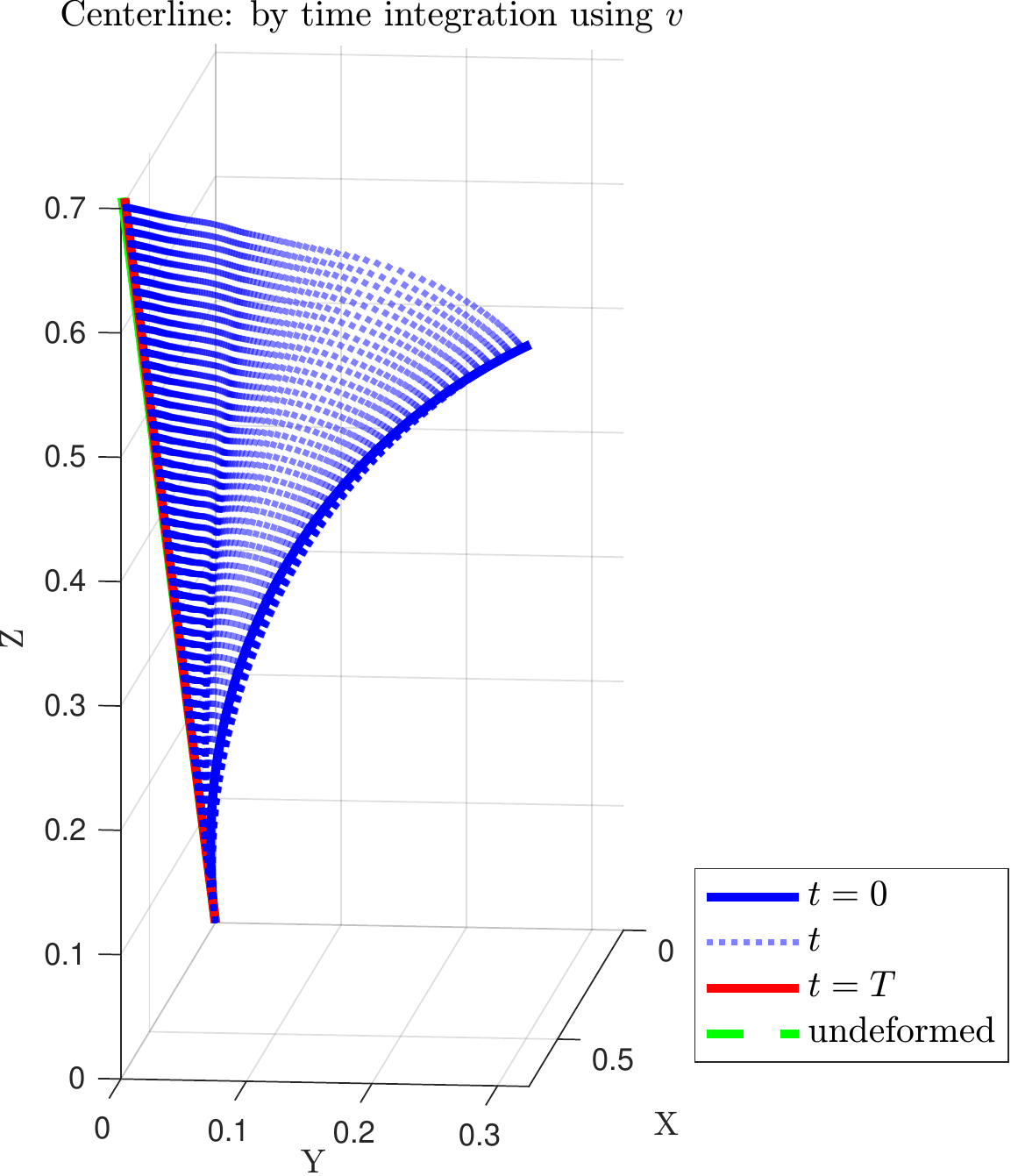}
\caption{Position of the centerline through time.}
\end{subfigure}%

\caption{Controlled beam with zero initial velocities.}
\label{fig:no-zero-ord-controlled}
\end{figure}

%


\bibliographystyle{acm}
\bibliography{mybibfile}

\begin{thebibliography}{10}

\bibitem{AlabauPerrollazRosier2015}
{\sc Alabau-Boussouira, F., Perrollaz, V., and Rosier, L.}
\newblock Finite-time stabilization of a network of strings.
\newblock {\em Math. Control Relat. Fields 5}, 4 (2015), 721--742.

\bibitem{alinhac2013blowup}
{\sc Alinhac, S.}
\newblock {\em Blowup for nonlinear hyperbolic equations}, vol.~17 of {\em
  Progr. Nonlinear Differential Equations Appl.}
\newblock Birkh\"auser Boston, 1995.

\bibitem{Artola2020aero}
{\sc Artola, M., Goizueta, N., Wynn, A., and Palacios, R.}
\newblock Modal-based nonlinear estimation and control for highly flexible
  aeroelastic systems.
\newblock In {\em AIAA Scitech Forum\/} (2020).

\bibitem{Artola2019mpc}
{\sc Artola, M., Wynn, A., and Palacios, R.}
\newblock A nonlinear modal-based framework for low computational cost optimal
  control of 3{D} very flexible structures.
\newblock In {\em 18th European Control Conference\/} (2019), pp.~3836--3841.

\bibitem{Artola2021damping}
{\sc Artola, M., Wynn, A., and Palacios, R.}
\newblock Generalized {K}elvin-{V}oigt damping for geometrically nonlinear
  beams.
\newblock {\em AIAA Journal 59}, 1 (2021), 356--365.

\bibitem{BC2016}
{\sc Bastin, G., and Coron, J.-M.}
\newblock {\em {S}tability and {B}oundary {S}tabilization of 1-{D} {H}yperbolic
  {S}ystems}, vol.~88 of {\em Progr. Nonlinear Differential Equations Appl.}
\newblock Birkh\"{a}user/Springer, [Cham], 2016.

\bibitem{bastin2017exponential}
{\sc Bastin, G., and Coron, J.-M.}
\newblock Exponential stability of semi-linear one-dimensional balance laws.
\newblock In {\em Feedback stabilization of controlled dynamical systems},
  vol.~473 of {\em Lect. Notes Control Inf. Sci.} Springer, Cham, 2017,
  pp.~265--278.

\bibitem{Coron2007}
{\sc Bastin, G., Haut, B., Coron, J.-M., and D'andr\'{e}a-Novel, B.}
\newblock Lyapunov stability analysis of networks of scalar conservation laws.
\newblock {\em Netw. Heterog. Media 2}, 4 (2007), 751--759.

\bibitem{beauchard2011large}
{\sc Beauchard, K., and Zuazua, E.}
\newblock Large time asymptotics for partially dissipative hyperbolic systems.
\newblock {\em Arch. Rational Mech. Anal. 199}, 1 (2011), 177--227.

\bibitem{bony1987solutions}
{\sc Bony, J.-M.}
\newblock Solutions globales born{\'e}es pour les mod{\`e}les discrets de
  l'{\'e}quation de {Boltzmann}, en dimension 1 d'espace.
\newblock In {\em Journ{\'e}es {\'E}quations aux deriv{\'e}es partielles\/}
  (1987).

\bibitem{bony1994existence}
{\sc Bony, J.-M.}
\newblock Existence globale et diffusion pour les mod{\`e}les discrets de la
  cin{\'e}tique des gaz.
\newblock In {\em First European Congress of Mathematics\/} (1994), vol.~119 of
  {\em Progr. Math.}, Birkh\"auser, pp.~391--410.

\bibitem{chen_serial_EBbeams}
{\sc Chen, G., Delfour, M.~C., Krall, A.~M., and Payre, G.}
\newblock Modeling, stabilization and control of serially connected beams.
\newblock {\em SIAM J. Control Optim. 25}, 3 (1987), 526--546.

\bibitem{chou1992}
{\sc Chou, J. C.~K.}
\newblock Quaternion kinematic and dynamic differential equations.
\newblock {\em IEEE Trans. Robot. Autom. 8}, 1 (1992), 53--64.

\bibitem{courant1949nonlinear}
{\sc Courant, R., and Lax, P.}
\newblock On nonlinear partial differential equations with two independent
  variables.
\newblock {\em Comm. Pure Appl. Math. 2}, 2-3 (1949), 255--273.

\bibitem{DagerZuazuaBook}
{\sc D\'{a}ger, R., and Zuazua, E.}
\newblock {\em Wave propagation, observation and control in {$1\text{-}d$}
  flexible multi-structures}, vol.~50 of {\em Math\'{e}matiques \& Applications
  (Berlin) [Mathematics \& Applications]}.
\newblock Springer-Verlag, Berlin, 2006.

\bibitem{coron2003}
{\sc de~Halleux, J., Prieur, C., Coron, J.-M., d'Andr\'{e}a Novel, B., and
  Bastin, G.}
\newblock Boundary feedback control in networks of open channels.
\newblock {\em Automatica J. IFAC 39}, 8 (2003), 1365--1376.

\bibitem{do18}
{\sc Do, K.}
\newblock Stabilization of exact nonlinear {T}imoshenko beams in space by
  boundary feedback.
\newblock {\em J. Sound Vib. 422\/} (2018), 278 -- 299.

\bibitem{Macchelli2009book}
{\sc Duindam, V., Macchelli, A., Stramigioli, S., and Bruyninckx, H.}, Eds.
\newblock {\em {M}odeling and {C}ontrol of {C}omplex {P}hysical {S}ystems.
  {T}he port-{H}amiltonian approach}.
\newblock Springer-Verlag, Berlin, 2009.

\bibitem{Egger_2017_semilin}
{\sc Egger, H., Kugler, T., and Strogies, N.}
\newblock Parameter identification in a semilinear hyperbolic system.
\newblock {\em Inverse Problems 33}, 5 (2017).

\bibitem{evans2}
{\sc Evans, L.~C.}
\newblock {\em {P}artial {D}ifferential {E}quations}, second~ed., vol.~19 of
  {\em Graduate Studies in Mathematics}.
\newblock American Mathematical Society, Providence, RI, 2010.

\bibitem{friedrichs1948nonlinear}
{\sc Friedrichs, K.~O.}
\newblock Nonlinear hyperbolic differential equations for functions of two
  independent variables.
\newblock {\em Amer. J. Math. 70}, 3 (1948), 555--589.

\bibitem{grazioso2018robot}
{\sc Grazioso, S., Di~Gironimo, G., and Siciliano, B.}
\newblock A geometrically exact model for soft continuum robots: The finite
  element deformation space formulation.
\newblock {\em Soft robotics 6}, 6 (2019), 790--811.

\bibitem{gu2011}
{\sc Gu, Q., and Li, T.}
\newblock Exact boundary controllability of nodal profile for quasilinear
  hyperbolic systems in a tree-like network.
\newblock {\em Math. Methods Appl. Sci. 34\/} (2011), 911--928.

\bibitem{gu2013}
{\sc Gu, Q., and Li, T.}
\newblock Exact boundary controllability of nodal profile for unsteady flows on
  a tree-like network of open canals.
\newblock {\em J. Math. Pures Appl. 99}, 1 (2013), 86--105.

\bibitem{gugat10}
{\sc Gugat, M., Herty, M., and Schleper, V.}
\newblock Flow control in gas networks: exact controllability to a given
  demand.
\newblock {\em Math. Methods Appl. Sci. 34}, 7 (2011), 745--757.

\bibitem{gugat2018}
{\sc Gugat, M., Perrollaz, V., and Rosier, L.}
\newblock Boundary stabilization of quasilinear hyperbolic systems of balance
  laws: exponential decay for small source terms.
\newblock {\em J. Evol. Equ. 18}, 3 (2018), 1471--1500.

\bibitem{GugatSigalotti2019}
{\sc Gugat, M., and Sigalotti, M.}
\newblock Stars of vibrating strings: switching boundary feedback
  stabilization.
\newblock {\em Netw. Heterog. Media 5}, 2 (2010), 299--314.

\bibitem{GuoXu2011}
{\sc Guo, Y.~N., and Xu, G.~Q.}
\newblock Exponential stabilisation of a tree-shaped network of strings with
  variable coefficients.
\newblock {\em Glasg. Math. J. 53}, 3 (2011), 481--499.

\bibitem{HanXu2010}
{\sc Han, Z.~J., and Xu, G.~Q.}
\newblock Exponential stabilisation of a simple tree-shaped network of
  {T}imoshenko beams system.
\newblock {\em Internat. J. Control 83}, 7 (2010), 1485--1503.

\bibitem{HanXu2011}
{\sc Han, Z.~J., and Xu, G.~Q.}
\newblock Dynamical behavior of networks of non-uniform {T}imoshenko beams
  system with boundary time-delay inputs.
\newblock {\em Netw. Heterog. Media 6}, 2 (2011), 297--327.

\bibitem{hayat2018exponential}
{\sc Hayat, A.}
\newblock {Exponential stability of general 1-D quasilinear systems with source
  terms for the $C^1$ norm under boundary conditions}.
\newblock {\em SIAM J. Control Optim.\/} (2019).
\newblock In press.

\bibitem{hegarty12}
{\sc Hegarty, G., and Taylor, S.}
\newblock Classical solutions of nonlinear beam equations: existence and
  stabilization.
\newblock {\em SIAM J. Control Optim. 50}, 2 (2012), 703--719.

\bibitem{HertyYu2018}
{\sc Herty, M., and Yu, H.}
\newblock Feedback boundary control of linear hyperbolic equations with stiff
  source term.
\newblock {\em Internat. J. Control 91}, 1 (2018), 230--240.

\bibitem{hesse2012}
{\sc Hesse, H., and Palacios, R.}
\newblock Consistent structural linearisation in flexible-body dynamics with
  large rigid-body motion.
\newblock {\em Computers \& Structures 110-111\/} (2012), 1--14.

\bibitem{higdon1986initial}
{\sc Higdon, R.~L.}
\newblock Initial-boundary value problems for linear hyperbolic system.
\newblock {\em SIAM rev. 28}, 2 (1986), 177--217.

\bibitem{hodges1990}
{\sc Hodges, D.~H.}
\newblock A mixed variational formulation based on exact intrinsic equations
  for dynamics of moving beams.
\newblock {\em Int. J. Solids Struct. 26}, 11 (1990), 1253--1273.

\bibitem{hodges2003geometrically}
{\sc Hodges, D.~H.}
\newblock Geometrically exact, intrinsic theory for dynamics of curved and
  twisted anisotropic beams.
\newblock {\em AIAA Journal 41}, 6 (2003), 1131--1137.

\bibitem{horn2012matrix}
{\sc Horn, R.~A., and Johnson, C.~R.}
\newblock {\em Matrix analysis}, second~ed.
\newblock CUP, Cambridge, 2013.

\bibitem{hu2019minimal}
{\sc Hu, L., and Olive, G.}
\newblock Minimal time for the exact controllability of one-dimensional
  first-order linear hyperbolic systems by one-sided boundary controls.
\newblock {\em J. Math. Pures Appl. 148\/} (2021), 24--74.

\bibitem{Zwart2012bluebook}
{\sc Jacob, B., and Zwart, H.~J.}
\newblock {\em {L}inear {P}ort-{H}amiltonian {S}ystems on
  {I}nfinite-{D}imensional {S}paces}, vol.~223 of {\em Operator Theory:
  Advances and Applications}.
\newblock Birkh\"{a}user/Springer Basel AG, Basel, 2012.
\newblock Linear Operators and Linear Systems.

\bibitem{kimrenardy87}
{\sc Kim, J.~U., and Renardy, Y.}
\newblock Boundary control of the {T}imoshenko beam.
\newblock {\em SIAM J. Control Optim. 25}, 6 (1987), 1417--1429.

\bibitem{Kmit_classical_nonlin}
{\sc Kmit, I.}
\newblock Classical solvability of nonlinear initial-boundary problems for
  first-order hyperbolic systems.
\newblock {\em Int. J. Dyn. Syst. Differ. Equ. 1}, 3 (2008), 191--195.

\bibitem{LLS}
{\sc Lagnese, J.~E., Leugering, G., and Schmidt, E. J. P.~G.}
\newblock {\em {M}odeling, {A}nalysis and {C}ontrol of {D}ynamic {E}lastic
  {M}ulti-{L}ink {S}tructures}.
\newblock Systems \& Control: Foundations \& Applications. Birkh\"{a}user
  Boston, Inc., Boston, MA, 1994.

\bibitem{LeugeringSchmidt2002}
{\sc Leugering, G., and Schmidt, E. J. P.~G.}
\newblock On the modelling and stabilization of flows in networks of open
  canals.
\newblock {\em SIAM J. Control Optim. 41}, 1 (2002), 164--180.

\bibitem{Li_1994_global}
{\sc Li, T.}
\newblock {\em Global classical solutions for quasilinear hyperbolic systems},
  vol.~32 of {\em Research in Applied Mathematics}.
\newblock John Wiley \& Sons, 1994.

\bibitem{Li_blue_book}
{\sc Li, T.}
\newblock {\em Global classical solutions for quasilinear hyperbolic systems},
  vol.~32 of {\em RAM: Research in Applied Mathematics}.
\newblock Masson, Paris; John Wiley \& Sons, Ltd., Chichester, 1994.

\bibitem{li2010controllability}
{\sc Li, T.}
\newblock {\em {C}ontrollability and {O}bservability for {Q}uasilinear
  {H}yperbolic {S}ystems}, vol.~3 of {\em AIMS Ser. Appl. Math.}
\newblock Am. Inst. Math. Sci., Springfield, MO; Higher Education Press,
  Beijing, 2010.

\bibitem{li2010nodal}
{\sc Li, T.}
\newblock Exact boundary controllability of nodal profile for quasilinear
  hyperbolic systems.
\newblock {\em Math. Methods Appl. Sci. 33\/} (2010), 2101--2106.

\bibitem{LiJin2001_semiglob}
{\sc Li, T., and Jin, Y.}
\newblock Semi-global {$C^1$} solution to the mixed initial-boundary value
  problem for quasilinear hyperbolic systems.
\newblock {\em Chinese Ann. Math. Ser. B 22}, 3 (2001), 325--336.

\bibitem{LiRao2002_cam}
{\sc Li, T., and Rao, B.}
\newblock Local exact boundary controllability for a class of quasilinear
  hyperbolic systems.
\newblock {\em Chinese Ann. Math. Ser. B 23}, 2 (2002), 209--218.

\bibitem{LiRao2003_sicon}
{\sc Li, T., and Rao, B.}
\newblock Exact boundary controllability for quasi-linear hyperbolic systems.
\newblock {\em SIAM J. Control Optim. 41}, 6 (2003), 1748--1755.

\bibitem{li2010no}
{\sc Li, T., Rao, B., and Wang, Z.}
\newblock Exact boundary controllability and observability for first order
  quasilinear hyperbolic systems with a kind of nonlocal boundary conditions.
\newblock {\em Discrete Contin. Dyn. Syst. 28}, 1 (2010), 243--257.

\bibitem{li2016book}
{\sc Li, T., Wang, K., and Gu, Q.}
\newblock {\em {E}xact {B}oundary {C}ontrollability of {N}odal {P}rofile for
  {Q}uasilinear {H}yperbolic {S}ystems}.
\newblock SpringerBriefs in Mathematics. Springer, Singapore, 2016.

\bibitem{Li_Duke85}
{\sc Li, T., and Yu, W.}
\newblock {\em {B}oundary {V}alue {P}roblems for {Q}uasilinear {H}yperbolic
  {S}ystems}.
\newblock Duke University Mathematics Series, V. Duke University, Mathematics
  Department, Durham, NC, 1985.

\bibitem{Zhuang2021}
{\sc Li, T., and Zhuang, K.}
\newblock A cut-off method to realize the exact boundary controllability of
  nodal profile for {S}aint-{V}enant systems on general networks with loops.
\newblock {\em J. Math. Pures Appl. 151\/} (2021), 1--27.

\bibitem{Macchelli2004Timo}
{\sc Macchelli, A., and Melchiorri, C.}
\newblock Modeling and control of the {T}imoshenko beam. {T}he distributed port
  {H}amiltonian approach.
\newblock {\em SIAM J. Control Optim. 43}, 2 (2004), 743--767.

\bibitem{Macchelli2007}
{\sc Macchelli, A., Melchiorri, C., and Stramigioli, S.}
\newblock Port-based modeling of a flexible link.
\newblock {\em IEEE Transactions on Robotics 23}, 4 (2007), 650--660.

\bibitem{Macchelli2009}
{\sc {Macchelli}, A., {Melchiorri}, C., and {Stramigioli}, S.}
\newblock Port-based modeling and simulation of mechanical systems with rigid
  and flexible links.
\newblock {\em IEEE Transactions on Robotics 25}, 5 (2009), 1016--1029.

\bibitem{Maschke1992}
{\sc Maschke, B., and {van der Schaft}, A.}
\newblock Port-controlled {H}amiltonian systems: Modelling origins and
  systemtheoretic properties.
\newblock {\em IFAC Proceedings Volumes 25}, 13 (1992), 359--365.

\bibitem{Mattioni2020Timo}
{\sc Mattioni, A., Wu, Y., and Le~Gorrec, Y.}
\newblock Infinite dimensional model of a double flexible-link manipulator: the
  port-{H}amiltonian approach.
\newblock {\em Appl. Math. Model. 83\/} (2020), 59--75.

\bibitem{morgul91}
{\sc Morg\"{u}l, O.}
\newblock Boundary control of a {T}imoshenko beam attached to a rigid body:
  planar motion.
\newblock {\em Internat. J. Control 54}, 4 (1991), 763--791.

\bibitem{Munoz2020}
{\sc Muñoz-Simón, A., Wynn, A., and Palacios, R.}
\newblock Unsteady and three-dimensional aerodynamic effects on wind turbine
  rotor loads.
\newblock In {\em AIAA Scitech Forum\/} (2020).

\bibitem{myshkis2008global}
{\sc Myshkis, A.~D., and Filimonov, A.~M.}
\newblock On the global continuous solvability of the mixed problem for
  one-dimensional hyperbolic systems of quasilinear equations.
\newblock {\em Differential Equations 44}, 3 (2008), 413--427.

\bibitem{NicaiseValein2007}
{\sc Nicaise, S., and Valein, J.}
\newblock Stabilization of the wave equation on 1-{D} networks with a delay
  term in the nodal feedbacks.
\newblock {\em Netw. Heterog. Media 2}, 3 (2007), 425--479.

\bibitem{Palacios2017modes}
{\sc Palacios, R.}
\newblock {\em Invariant manifolds in beam dynamics: free vibrations and
  nonlinear normal modes}.
\newblock Springer Berlin Heidelberg, 2017, pp.~1--8.

\bibitem{Palacios2011intrinsic}
{\sc Palacios, R., and Epureanu, B.}
\newblock An intrinsic description of the nonlinear aeroelasticity of very
  flexible wings.
\newblock In {\em 52nd AIAA/ASME/ASCE/AHS/ASC Structures, Structural Dynamics
  and Materials Conference\/} (2011).

\bibitem{Palacios2010aero}
{\sc Palacios, R., Murua, J., and Cook, R.}
\newblock Structural and aerodynamic models in nonlinear flight dynamics of
  very flexible aircraft.
\newblock {\em AIAA Journal 48}, 11 (2010), 2648--2659.

\bibitem{pavel_2013}
{\sc Pavel, L.}
\newblock Classical solutions in {Sobolev} spaces for a class of hyperbolic
  {Lotka--Volterra} systems.
\newblock {\em SIAM Journal on Control and Optimization 51}, 3 (2013),
  2132--2151.

\bibitem{pazy}
{\sc Pazy, A.}
\newblock {\em Semigroups of linear operators and applications to partial
  differential equations}, vol.~44 of {\em Applied Mathematical Sciences}.
\newblock Springer, 1983.

\bibitem{Prieur_Winkin_2018}
{\sc Prieur, C., and Winkin, J.~J.}
\newblock Boundary feedback control of linear hyperbolic systems: Application
  to the {Saint-Venant--Exner} equations.
\newblock {\em Automatica 89\/} (2018), 44--51.

\bibitem{quinnrussel78}
{\sc Quinn, J.~P., and Russell, D.~L.}
\newblock Asymptotic stability and energy decay rates for solutions of
  hyperbolic equations with boundary damping.
\newblock {\em Proc. Roy. Soc. Edinburgh Sect. A 77}, 1-2 (1977), 97--127.

\bibitem{racke1992lectures}
{\sc Racke, R.}
\newblock {\em Lectures on Nonlinear Evolution Equations. Initial Value
  Problems}.
\newblock Birkhäuser, Cham, 2015.

\bibitem{reissner1981finite}
{\sc Reissner, E.}
\newblock On finite deformations of space-curved beams.
\newblock {\em Zeitschrift f{\"u}r angewandte Mathematik und Physik ZAMP 32}, 6
  (1981), 734--744.

\bibitem{Russel_1978}
{\sc Russell, D.~L.}
\newblock Controllability and stabilizability theory for linear partial
  differential equations: recent progress and open questions.
\newblock {\em SIAM Rev. 20}, 4 (1978), 639--739.

\bibitem{Sarac2021}
{\sc Sarac, Y., and Zuazua, E.}
\newblock Sidewise control of 1-d waves, 2021.
\newblock arXiv preprint arXiv:2101.00473.

\bibitem{simo1985finite}
{\sc Simo, J.}
\newblock A finite strain beam formulation. {T}he three-dimensional dynamic
  problem. {P}art {I}.
\newblock {\em Comput. Methods in Appl. Mech. and Engrg. 49}, 1 (1985), 55--70.

\bibitem{Simo1988}
{\sc Simo, J.~C., Marsden, J.~E., and Krishnaprasad, P.~S.}
\newblock The {H}amiltonian structure of nonlinear elasticity: the material and
  convective representations of solids, rods, and plates.
\newblock {\em Arch. Rational Mech. Anal. 104}, 2 (1988), 125--183.

\bibitem{strohm_dissert}
{\sc Strohmeyer, C.}
\newblock {\em Networks of nonlinear thin structures - theory and
  applications}.
\newblock PhD thesis, FAU University Press, 2018.

\bibitem{tartar1981some}
{\sc Tartar, L.~C.}
\newblock Some existence theorems for semilinear hyperbolic systems in one
  space variable.
\newblock Tech. rep., Wisconsin Univ-Madison Mathematics Research Center, 1981.

\bibitem{tucsnak_weiss}
{\sc Tucsnak, M., and Weiss, G.}
\newblock {\em Observation and control for operator semigroups}.
\newblock Springer Science \& Business Media, 2009.

\bibitem{turo1997mixed}
{\sc Turo, J.}
\newblock Mixed problems for quasilinear hyperbolic systems.
\newblock {\em Nonlinear Analysis, Theory, Methods \& Applications 30}, 4
  (1997), 2329--2340.

\bibitem{ValeinZuazua2009}
{\sc Valein, J., and Zuazua, E.}
\newblock Stabilization of the wave equation on 1-{D} networks.
\newblock {\em SIAM J. Control Optim. 48}, 4 (2009), 2771--2797.

\bibitem{Schaft2002}
{\sc van~der Schaft, A.~J., and Maschke, B.~M.}
\newblock Hamiltonian formulation of distributed-parameter systems with
  boundary energy flow.
\newblock {\em J. Geom. Phys. 42}, 1-2 (2002), 166--194.

\bibitem{flotow_spacecraft}
{\sc von Flotow, A.~H.}
\newblock Traveling wave control for large spacecraft structures.
\newblock {\em Journal of Guidance, Control, and Dynamics 9}, 4 (1986),
  462--468.

\bibitem{vrabie2004}
{\sc Vrabie, I.~I.}
\newblock {\em {D}ifferential {E}quations: {A}n {I}ntroduction to {B}asic
  {C}oncepts, {R}esults, and {A}pplications}.
\newblock World Scientific, 2004.

\bibitem{kw2011}
{\sc Wang, K.}
\newblock Exact boundary controllability of nodal profile for 1-d quasilinear
  wave equations.
\newblock {\em Frontiers Math. China 6\/} (2011), 545--555.

\bibitem{kw2014}
{\sc Wang, K., and Gu, Q.}
\newblock Exact boundary controllability of nodal profile for quasilinear wave
  equations in a planar tree-like network of strings.
\newblock {\em Math. Methods Appl. Sci. 37\/} (2014), 1206--1218.

\bibitem{wang2014windturbine}
{\sc Wang, L., Liu, X., Renevier, N., Stables, M., and Hall, G.~M.}
\newblock Nonlinear aeroelastic modelling for wind turbine blades based on
  blade element momentum theory and geometrically exact beam theory.
\newblock {\em Energy 76\/} (2014), 487--501.

\bibitem{YWang2021_NP_HUM}
{\sc Wang, Y., Leugering, G., and Li, T.}
\newblock {HUM} method to the exact boundary controllability of nodal profile
  for vibrating strings, 2021.
\newblock In preparation.

\bibitem{YWang2019partialNP}
{\sc Wang, Y., and Li, T.}
\newblock Exact boundary controllability of partial nodal profile for network
  of strings.
\newblock {\em Nonlinear Analysis: Real World Applications 62\/} (2021).

\bibitem{wang2006exact}
{\sc Wang, Z.}
\newblock Exact controllability for nonautonomous first order quasilinear
  hyperbolic systems.
\newblock {\em Chinese Ann. Math. Ser. B 27}, 6 (2006), 643--656.

\bibitem{weiss99}
{\sc Weiss, H.}
\newblock {\em Zur Dynamik geometrisch nichtlinearer Balken}.
\newblock PhD thesis, Technische Universit\"at Chemnitz, 1999.

\bibitem{xu2005}
{\sc Xu, G.~Q.}
\newblock Boundary feedback exponential stabilization of a {T}imoshenko beam
  with both ends free.
\newblock {\em Internat. J. Control 78}, 4 (2005), 286--297.

\bibitem{XuHanYung2007}
{\sc Xu, G.~Q., Han, Z.~J., and Yung, S.~P.}
\newblock Riesz basis property of serially connected {T}imoshenko beams.
\newblock {\em Internat. J. Control 80}, 3 (2007), 470--485.

\bibitem{ZhangXuMastorakis2009}
{\sc Zhang, K.~T., Xu, G.~Q., and Mastorakis, N.~E.}
\newblock Stability of a complex network of {E}uler-{B}ernoulli beams.
\newblock {\em WSEAS Trans. Syst. 8}, 3 (2009), 379--389.

\bibitem{ZhangXu2013}
{\sc Zhang, Y., and Xu, G.}
\newblock Exponential and super stability of a wave network.
\newblock {\em Acta Appl. Math. 124\/} (2013), 19--41.

\bibitem{Zhuang2018}
{\sc Zhuang, K., Leugering, G., and Li, T.}
\newblock Exact boundary controllability of nodal profile for {S}aint-{V}enant
  system on a network with loops.
\newblock {\em J. Math. Pures Appl. 129\/} (2019), 34--60.

\bibitem{Zuazua2012}
{\sc Zuazua, E.}
\newblock Control and stabilization of waves on 1-d networks.
\newblock In {\em Modelling and optimisation of flows on networks}, vol.~2062
  of {\em Lecture Notes in Math.} Springer, Heidelberg, 2013, pp.~463--493.

\end{thebibliography}

\part{Reprints of published articles}
\label{part:reprints}

\renewcommand*{\chapappifchapterprefix}{Paper}

\chapter{Boundary feedback stabilization for the intrinsic geometrically exact beam model}

\textbf{Charlotte Rodriguez, and G\"unter Leugering.}

\noindent In: \textit{SIAM Journal on Control and Optimization} 58 (6), pp. 3533--3558 (2020). 

\noindent \texttt{DOI}: \href{https://doi.org/10.1137/20M1340010}{\texttt{10.1137/20M1340010}}.

\chapter{Networks of geometrically exact beams: well-posedness and stabilization}

\textbf{Charlotte Rodriguez.}

\noindent In: \textit{Mathematical Control and Related Fields} (2021). 

\noindent \texttt{DOI}: \href{https://www.aimsciences.org/article/doi/10.3934/mcrf.2021002}{\texttt{10.3934/mcrf.2021002}}. Advance online publication.

\chapter{Nodal profile control for networks of geometrically exact beams}

\textbf{G\"unter Leugering, Charlotte Rodriguez, and Yue Wang.}

\noindent In: \textit{Journal de Mathématiques Pures et Appliquées} 155, pp. 111--139 (2021).

\noindent \texttt{DOI}: \href{https://doi.org/10.1016/j.matpur.2021.07.007}{\texttt{10.1016/j.matpur.2021.07.007}}. In press.

\backmatter


\end{document}